\documentclass{amsart}

\usepackage{amssymb}
\usepackage{tensor}
\usepackage{graphicx}
	\graphicspath{ {./images/} }
\usepackage{hyperref}
\hypersetup{
    colorlinks,
    citecolor=black,
    filecolor=black,
    linkcolor=black,
    urlcolor=black
}
\usepackage[final]{changes}
\usepackage{tikz}
\usetikzlibrary{shapes.geometric, arrows, tqft}

\def\XXint#1#2#3{{\setbox0=\hbox{$#1{#2#3}{\int}$ }
\vcenter{\hbox{$#2#3$ }}\kern-.6\wd0}}

\newtheorem{thm}{Theorem}[section]
\newtheorem{prop}[thm]{Proposition}
\newtheorem{lem}[thm]{Lemma}
\newtheorem{cor}[thm]{Corollary}

\theoremstyle{definition}
\newtheorem{definition}[thm]{Definition}

\newtheorem{assumption}[thm]{Assumption}

\theoremstyle{remark}
\newtheorem{remark}[thm]{Remark}

\numberwithin{equation}{section}

\newcommand{\R}{\mathbb{R}}  			
\renewcommand{\phi}{\varphi}
\newcommand{\tl}[1]{\tilde{#1}} 		
\newcommand{\ol}[1]{\overline{#1}} 		
\newcommand{\cl}[1]{\mathcal{#1}} 
\newcommand{\cone}{\mathcal{C}} 		

\DeclareMathOperator{\dist}{dist} 		
\DeclareMathOperator{\tr}{tr}			
\let \div \relax
\DeclareMathOperator{\div}{div}			
\DeclareMathOperator{\supp}{supp}		
\DeclareMathOperator{\inj}{inj}			

\title[Ricci Flow Singularities Modeled on AC Shrinkers]{Closed Ricci Flows with Singularities Modeled on Asymptotically Conical Shrinkers}
\author{Maxwell Stolarski}
\email{max.stolarski@warwick.ac.uk}
\urladdr{https://homepages.warwick.ac.uk/~u2175999/}
\address{Warwick Mathematics Institute, University of Warwick}

\begin{document}
\begin{abstract}
	Given an asymptotically conical, shrinking, gradient Ricci soliton, we show that there exists a Ricci flow solution on a closed manifold that forms a finite-time singularity modeled on the given soliton.
	No symmetry or K{\"a}hler assumptions on the soliton are required.
	The proof provides a precise asymptotic description of the singularity formation.
\end{abstract}

\maketitle
\tableofcontents

\section{Introduction}

A smooth collection of Riemannian metrics $\{ g(t) \}_{t \in [t_0, T)}$ on a manifold $M$ is said to evolve by Ricci flow if
	$$\partial_t g = -2 Rc.$$
The Ricci flow equation is invariant under parabolic scaling and pullback by diffeomorphisms.
Considered as a dynamical system on the space of metrics modulo scaling and pullback by diffeomorphism, the Ricci flow has generalized fixed points given by Ricci solitons, self-similar Ricci flow solutions that evolve by scaling and pullback by diffeomorphism.
Ricci solitons $(M, g, X, \lambda)$ may be specified by a Riemannian manifold $(M, g)$, a vector field $X$ on $M$, and a real number $\lambda \in \R$ such that
	$$Rc + \cl{L}_X g = \lambda g.$$
As generalized fixed points, Ricci solitons model singularities of the Ricci flow.
For example, \cite{EMT11} showed that any type I singularity of the Ricci flow is modeled by a nontrivial gradient shrinking Ricci soliton.
More generally, \cite{Bamler21} showed that parabolic rescalings of the flow based at any singular point subsequentially converge in a weak sense to a possibly non-smooth soliton.

However, a certain converse question has no clear answer.
Namely, given a smooth, complete, \emph{noncompact} Ricci soliton $(M, g, X, \lambda)$, it is not always clear if there exists a Ricci flow $\{ G(t) \}_{t \in [ 0, T)}$ on a \emph{closed} manifold $\cl{M}$ that forms a singularity modeled on the Ricci soliton $(M, g, X, \lambda)$.
Of course, the self-similar Ricci flow solution $(M = \cl{M}, (1 - 2 \lambda t)\phi_t^* g)$ associated to a shrinking Ricci soliton $(M, g, X, \lambda)$ forms a finite-time singularity modeled on $(M, g, X, \lambda)$, 
so the restriction that $M$ is noncompact and $\cl{M}$ compact here is key.
One obstruction comes from Perelman's no local collapsing Theorem \cite{Perelman02}, which in particular shows a singularity for a Ricci flow solution on a closed manifold cannot be modeled by the product of Hamilton's cigar soliton with Euclidean space.

In this article, we show that, given any asymptotically conical, shrinking, gradient Ricci soliton, there exists a Ricci flow solution on a closed manifold that forms a singularity modeled by that soliton.
\begin{thm} \label{Main Thm}
	Let $(M, \ol{g}, \ol{\nabla} f, \lambda = \frac{1}{2} )$ be a smooth, complete, asymptotically conical, shrinking, gradient Ricci soliton.
	Then there exists a closed manifold $\cl{M}$ and
	a Ricci flow solution $\{ G(t) \}_{t \in [t_0, 1)}$ on $\cl{M}$
	that forms a local, type I singularity at time $t=1$
	such that,
	for any sequence $t_j \nearrow 1$ and $p_\infty \in M$, 
	there exists a sequence $p_j \in \cl{M}$
	such that the parabolically rescaled Ricci flows
		$$\left( \cl{M}, G_j(t) \doteqdot \frac{1}{1 - t_j} G( t_j + t ( 1 - t_j)), p_j \right)$$
	converge to $( M, ( 1 - t) \phi_t^* \ol{g}, p_\infty)$ in the pointed Cheeger-Gromov sense as $j \nearrow \infty$.
	
	Moreover, there exists an open subset $U \subset \cl{M}$ and $\Phi : U \times [t_0, 1) \to M$
	such that:
	\begin{enumerate}

	 	\item for every $t \in [t_0, 1)$, 
			$\Phi_t = \Phi( \cdot, t) : U \to M$ is a diffeomorphism onto its image,
		\item for any compact subset $K \subset M$,
			$K \subset \Phi_t (U)$ for all $t$ in a neighborhood of $1$ with $t < 1$,
		\item  
		$$\frac{1}{1 -t} (\Phi_t^{-1})^*  G(t) 
		\xrightarrow[t \nearrow 1]{C^\infty_{loc}(M)} \ol{g},
		$$
		and
		\item $Rm$ is uniformly bounded outside $U$
			$$\sup_{(x,t) \in (\cl{M} \setminus U) \times  [t_0, 1)} | Rm|_g(x,t) < \infty.$$
	\end{enumerate}
\end{thm}
We shall obtain a more precise asymptotic description of the singularity formation in the course of the proof of Theorem \ref{Main Thm}.
We note that, unlike many previous constructions of Ricci flow singularity formation (e.g. \cite{AngenentKnopf04, Maximo14, Wu14, AIK15, Appleton19, Stolarski19, DiGiovanni20} among others), no symmetry or K{\"a}hler condition on $(M, \ol{g})$ or $(\cl{M}, G(t))$ is assumed.
Recently, \cite{LeeZhao24} adapted the proof of Theorem \ref{Main Thm} to the mean curvature flow setting.
Precise definitions of asymptotically conical, shrinking gradient Ricci soliton are provided in Section \ref{Sect Prelims on Shrinkers}.

Examples of complete, noncompact Ricci solitons constructed in \cite{AngenentKnopf21} provide additional motivation for Theorem \ref{Main Thm}.
In dimensions $n \ge 5$,
\cite{AngenentKnopf21} construct a self-similar Ricci flow solution $(M^n_-, g(t))_{t \in (-\infty, 0)}$ that converges to a cone $( \cl{C}, g_{\cl{C}} )$ as $t \nearrow 0$,
and distinct self-similar Ricci flow solutions $(M^n_+, g(t) )_{t \in (0, +\infty) }$, $(M'^n_+, g'(t) )_{t \in (0, +\infty) }$
that converge to the cone $( \cl{C}, g_{\cl{C}})$ as $t \searrow 0$.
Concatenating the time intervals, these solutions demonstrate nonuniqueness of the Ricci flow through singularities in dimensions $n \ge 5$.
These nonuniqueness examples, however, occur on \emph{noncompact} topologies $M_-, \cl{C}, M_+$.
\cite{AngenentKnopf21} thus leaves open whether this nonuniqueness of the Ricci flow through singularities can \emph{only} occur on noncompact topologies, as is the case for the heat equation, 
or if there exist Ricci flows on closed manifolds that demonstrate such nonuniqueness behavior.

The shrinking Ricci solitons in the above example \cite{AngenentKnopf21} are smooth, complete, asymptotically conical, and gradient.
By demonstrating that these shrinking solitons model singularities of Ricci flows on closed manifolds, Theorem \ref{Main Thm} thereby provides a first step toward showing that nonuniqueness of the Ricci flow through singularities occurs on compact topologies.
Indeed, \cite{AIV, IlmanenWhite24, LeeZhao24} have analogous results for the mean curvature flow.
We expect that such nonuniqueness behavior will inform the nascent theory of weak Ricci flows in high dimensions.

Further applications of Theorem \ref{Main Thm} come from taking the soliton $(M, \ol{g}, \ol{\nabla} f, \lambda = \frac{1}{2} )$ to be the FIK shrinkers \cite{FIK03} or their generalizations constructed in \cite{DW11}.
Theorem \ref{Main Thm} thereby provides an alternative proof of \cite[Theorem 1.2]{Maximo14}.
The construction of the Ricci flow $G(t)$ in Theorem \ref{Main Thm} is sufficiently flexible as to conclude $G(t)$ is non-K{\"a}hler as well (see Remark \ref{Rmk Initially Non-Kahler}).
While the only known examples of smooth, complete, asymptotically conical, gradient shrinking Ricci solitons come from \cite{FIK03, DW11, AngenentKnopf21}, we expect Theorem \ref{Main Thm} may provide an additional dynamical perspective that can contribute to the construction of new shrinking Ricci solitons.

The outline of the paper is as follows: 
We begin in Section \ref{Sect Prelims on Shrinkers} by recording properties of shrinking Ricci solitons that will be used throughout the rest of the paper.
Section \ref{Sect Strat of Proof} then outlines the Wa{\.z}ewski retraction principle argument \cite{Wazewski47} used to prove Theorem \ref{Main Thm}.
The remaining Sections \ref{Sect Setup}--\ref{Sect Proof of Main Thm} carry out the proof in detail.
First, Section \ref{Sect Setup} describes the manifold $\cl{M}$, initial metric $G(t_0)$, and map $\Phi$ that arise in the statement of Theorem \ref{Main Thm} as well as sets $\cl{B}, \cl{P}$ used to carry out the Wa{\.z}ewski retraction principle argument.
Section \ref{Sect Prelim Ests} records some coarse preliminary estimates on the Ricci flow of $G(t_0)$,
and then Section \ref{Sect Preserve and Improve C^2 Bounds} provides more precise $C^2$ control on the evolving metrics.
We use these estimates in Section \ref{Sect Behavior of L^2_f} to describe the behavior of the evolving metrics in a weighted $L^2$ space.
Finally, Section \ref{Sect Proof of Main Thm} combines the estimates of the prior sections to complete the proof of Theorem \ref{Main Thm}.
Appendix \ref{App Rounding Out Cones} provides explicit examples of Riemannian manifolds $(\cl{M}, G(t_0))$ considered in Section \ref{Sect Setup}.
Appendix \ref{Appdix HarMapFlow} records the properties of the harmonic map heat flow used throughout the paper.
Appendix \ref{Holder Ests in Eucl Space} states interior H{\"o}lder estimates for parabolic equations that are used for the interior estimates in Section \ref{Sect Preserve and Improve C^2 Bounds}.
Appendix \ref{App Constants} lists the constants used throughout the proof and their dependencies for the readers' convenience.

\noindent \textbf{Acknowledgements:} I would like to thank Dan Knopf for initially suggesting a version of this problem and Brett Kotschwar for useful discussions on asymptotically conical shrinking solitons.
I am grateful for Richard Bamler's observation that led to the removal of a technical assumption from Theorem \ref{Main Thm}.

The author is supported by a Leverhulme Trust Early Career Fellowship (ECF-2023-182).

\section{Preliminaries on Shrinking Solitons} \label{Sect Prelims on Shrinkers}

\begin{definition}
	A \emph{shrinking, gradient Ricci soliton} is a triple $(M, g, f)$ consisting of a smooth Riemannian manifold $(M, g)$ and a smooth function $f : M \to \R$, called the \emph{potential function}, such that 
	\begin{equation} \label{Soliton Eqns}
		Rc + \nabla^2 f = \frac{1}{2} g \quad   \text{ and }  \quad
		R + | \nabla f|^2 = f
		\quad \text{ on } M.
	\end{equation}
	Henceforth, we will simply call $(M, g, f)$ a \emph{shrinker}.
\end{definition}

\begin{remark} \label{Rmk Normalization}
	More generally, a shrinking, gradient Ricci soliton could be defined by the equation
		$$Rc + \nabla^2 f = \lambda g.$$
	The case of $\lambda = \frac{1}{2}$ may then be achieved by rescaling $g$.
	
	Additionally, $Rc + \nabla^2 f = \frac{1}{2} g$ implies that $ \nabla ( R + | \nabla f|^2 - f ) \equiv 0 $ throughout $M$ (see for example \cite[Proposition 1.15]{ChowEtAl07}).
	Hence, the second equation in \eqref{Soliton Eqns} may be achieved from the first by adding an appropriate constant to $f$ on each connected component of $M$.
\end{remark}

By taking the trace of $Rc + \nabla^2 f = \frac{1}{2} g$, we see
\begin{prop} \label{Prop Lapl f Eqn}
	If $(M^n, g, f)$ is a shrinking, gradient Ricci soliton, then
		$$R + \Delta f = \frac{n}{2} \qquad \text{ on } M.$$
	
	In particular, $\Delta f - \nabla_{\nabla f} f + f = \frac{n}{2}$ on $M$.
\end{prop}

\begin{definition} \label{Defn phi}
	Given a shrinker $(M, g, f)$, $\phi: M \times ( -\infty, 1) \to M$
	will denote the one-parameter family of diffeomorphisms generated by $\frac{1}{ 1 - t} \nabla f $, that is
	\begin{equation} \label{phi Evol Eqn}
		\partial_t \phi = \frac{1}{1-t} \nabla f \circ \phi \text{ on } M \times ( -\infty, 1)
		\qquad \text{ with } 
		\phi( \cdot, 0) = Id_M.
	\end{equation}
\end{definition}

\cite{Zhang09} ensures that $\phi$ exists when $(M, g)$ is complete.
Using $\phi$, one can construct a self-similar Ricci flow solution from a shrinker.

\begin{prop} \label{Prop RF from Shrinker}
	If $(M, \ol{g} , f)$ is a shrinker and $(M, \ol{g} )$ is complete,
	then $g(t) \doteqdot (1 - t) \phi_t^* \ol{g}$ evolves by Ricci flow $\partial_t g = -2Rc$ for all $t \in ( - \infty, 1)$.
\end{prop}

\begin{prop} \label{Prop Nonneg Scalar Curv}
	If $(M, g, f)$ is a complete shrinker, then the scalar curvature $R \ge 0$ is nonnegative throughout $M$.
	
	If $M$ is also connected, then $R = 0$ somewhere if and only if $(M, g)$ is Ricci-flat.
\end{prop}
\begin{proof}
	The fact that $R \ge 0$ is a direct consequence of \cite[Corollary 2.5]{Chen09} and Proposition \ref{Prop RF from Shrinker}.
	
	The statement regarding equality follows from the fact that $R$ satisfies 
		$$\Delta R - \nabla_{\nabla f} R   - \frac{1}{2} R = - 2 |Rc|^2 \le 0	\qquad \text{ on } M$$
	(see for example \cite[Proposition 1.13]{ChowEtAl07}) and the strong maximum principle.
\end{proof}

The following result of \cite{CaoZhou10} says that the potential function of a complete, non-compact shrinker grows quadratically at infinity: 
\begin{prop}[Theorem 1.1 of \cite{CaoZhou10}] \label{Prop Potential Func Grows Quadratically}
	Let $(M, g,f)$ be a complete, non-compact, connected shrinker and let $p_0 \in M$.
	Then the potential function $f : M \to \R$ satisfies 
		$$\frac{1}{4} \left( \dist_{g} ( p_0, x) - c_1 \right)^2 \le f(x) \le 	\frac{1}{4} \left( \dist_{g} ( p_0, x) + c_2 \right)^2		\qquad \text{for all } x \in M$$
	where $c_1,c_2$ are positive constants depending only on $n$ and the geometry of the unit ball $B_{g}( p_0, 1) \subset M$ centered at $p_0$.
\end{prop}

\cite[theorem 1.2]{CaoZhou10} also shows that complete, non-compact, connected shrinkers have at most Euclidean volume growth. Combined with Proposition \ref{Prop Potential Func Grows Quadratically}, this fact yields the following moment estimates for $f$ (see also \cite[Corollary 1.1]{CaoZhou10}):

\begin{prop}[Moment estimates] \label{Prop Moment Ests}
	Let $(M, g, f)$ be a complete, connected shrinker.
	Then, for any $c > 0$ and  $\kappa \ge 0$,
		$$\int_M f^{\kappa} e^{- c f} dV_{g} < \infty.$$
\end{prop}

\subsection{Functional Analysis on Shrinking Solitons}
\begin{definition}
Let $(M, g, f)$ be a shrinker. For any tensor $T$, define the operators
	\begin{gather*}
		\div_f (T) \doteqdot \div(T) - \langle \nabla f, T \rangle 	\\
		\Delta_f T \doteqdot \div_f (\nabla T) = \Delta T - \nabla_{\nabla f} T.
	\end{gather*}
\end{definition}
Observe that these operators satisfy the integration by parts formula
	$$\int_M \langle S, \div_f(T) \rangle e^{-f} dV_g = - \int_M \langle \nabla S, \nabla T \rangle e^{-f} dV_g$$
for any smooth, compactly supported tensors $S,T$ of the same type.

\begin{prop} \label{RcInKerDiv_f}
	If $(M, g, f)$ is a shrinker, then $\div_f Rc \equiv 0$.
\end{prop}
\begin{proof}
	Tracing $Rc + \nabla^2 f = \frac{1}{2} g$ and taking $\nabla$ yields
		$$R + \Delta f = \frac{n}{2} \implies \nabla_j R + \nabla_j \Delta f = 0.$$
	The contracted Bianchi identity then implies that
	\begin{equation} \label{Eqn RcInKerDiv_f 1}
		2 \nabla^i Rc_{ij} + \nabla_j \Delta f = 0.
	\end{equation}
	On the other hand, taking the divergence of $Rc + \nabla^2 f = \frac{1}{2} g$ and commuting covariant derivatives on $f$ yields
	\begin{equation} \label{Eqn RcInKerDiv_f 2}
		0 = \nabla^i Rc_{ij} + \nabla^i \nabla_i \nabla_j f  = \nabla^i Rc_{ij} + \nabla_j \Delta f + Rc_{jk} \nabla^k f. 
	\end{equation}
	Subtracting \eqref{Eqn RcInKerDiv_f 2} from \eqref{Eqn RcInKerDiv_f 1} shows that $\div_f Rc \equiv 0$.	
\end{proof}

\begin{remark}
	The proof of Proposition \ref{RcInKerDiv_f} generalizes to steady and expanding gradient Ricci solitons as well.
\end{remark}

\begin{definition}[Weighted Sobolev Spaces]
	Let $(M, g, f)$ be a shrinker.
	Let $E$ be a vector bundle on an open set $\Omega \subset M$ and assume $g$ induces a metric on $E$.
	Let $\Gamma( \Omega, E)$ denote the set of smooth sections of $E$.
	$L^2_f(\Omega, E)$ is the completion of 
		$$ \left\{ T \in \Gamma(\Omega, E)  : \int_\Omega | T|^2_g e^{-f} dV_g < \infty  \right\}$$
	with respect to the inner product
		$$( S , T )_{L^2_f(\Omega)} \doteqdot  \int_\Omega \langle S, T \rangle e^{-f} dV_g .$$
	
	For any $m \in \mathbb{N}$, $H^m_f(\Omega, E)$ is the completion of 
		$$ \left\{ T \in \Gamma(\Omega, E)  : \sum_{j=0}^m \int_\Omega | \nabla^j T|^2_g e^{-f} dV_g < \infty  \right\}$$	
	with respect to the inner product
		$$( S , T )_{H^m_f(\Omega) } \doteqdot  \sum_{j=0}^m \int_\Omega \langle \nabla^j S,  \nabla^j T \rangle e^{-f} dV_g .$$
\end{definition}

	Note $L^2_f (\Omega, E) = H^0_f( \Omega, E)$.
	In practice, we will often consider the case where $\Omega = M$ and simplify the notation by writing $\Gamma (E) = \Gamma(M, E),$ $L^2_f (E) = L^2_f( M, E)$, $( \cdot, \cdot)_{L^2_f }  = ( \cdot, \cdot)_{L^2_f(M) }$, and similarly for $H^m_f(E)$.
	If also the vector bundle $E$ is implicit from context, we may analogously write $L^2_f(M)$ or $L^2_f$.
	
%

\begin{lem}		\label{multGradfEst}
	Let $(M^n, g, f)$ be a complete shrinker.
	Then there exists a constant $C > 0$ such that 
		$$\| \nabla f \otimes T \|_{L^2_f} \le C \| T \|_{H^1_f}$$
	for any smooth tensor $T$ on $M$.
\end{lem}
\begin{proof}
	By density, we can additionally assume that $T$ is compactly supported.
	For a constant $C > 0$ to be determined, consider 
		$$0 \le | C \nabla T -  \nabla f \otimes T |^2$$
	Integrating both sides yields
	\begin{gather*} \begin{aligned} 
		0 	
		&\le \int_M | C \nabla T -  \nabla f \otimes T |^2 e^{-f} dV \\
		&= \int_M \left[ C^2 | \nabla T|^2 + | \nabla f|^2 |T|^2 \right] e^{-f} dV	
		- C \int_M \langle 2 \nabla T , \nabla f \otimes T \rangle e^{-f} dV\\
		&= \int_M \left[ C^2 | \nabla T|^2 + | \nabla f|^2 |T|^2 \right] e^{-f} dV	
		- C \int_M \left \langle  \nabla \left( |T|^2 \right) , \nabla f  \right \rangle e^{-f} dV\\
		&= \int_M \left[ C^2 | \nabla T|^2 + | \nabla f|^2 |T|^2 \right] e^{-f} dV	
		+ C \int_M   |T|^2  \div_f(\nabla f) e^{-f} dV\\
		&= \int_M \left[ C^2 | \nabla T|^2 + | \nabla f|^2 |T|^2 \right] e^{-f} dV	
		+ C \int_M   |T|^2  \left( \Delta f - |\nabla f|^2 \right) e^{-f} dV\\
		&= \int_M \left[ C^2 | \nabla T|^2 + | \nabla f|^2 |T|^2 \right] e^{-f} dV	
		+ C \int_M   |T|^2  \left( - R + \frac{n}{2} - |\nabla f|^2 \right) e^{-f} dV\\
		&\le \int_M \left[ C^2 | \nabla T|^2 + | \nabla f|^2 |T|^2 \right] e^{-f} dV	
		+ C \int_M   |T|^2  \left(   \frac{n}{2} - |\nabla f|^2 \right) e^{-f} dV\\
	\end{aligned} \end{gather*}	
	where the last inequality follows from Proposition \ref{Prop Nonneg Scalar Curv}.
	Therefore,
		$$(C -1) \int_M |T|^2 |\nabla f|^2 e^{-f} dV 
		\le  \int_M \left[C^2 | \nabla T|^2+C \frac{n}{2}|T|^2 \right]  e^{-f} dV $$
	Taking $C = 2$ yields the desired conclusion.
\end{proof}

\begin{lem}	\label{multRootfEst}
	Let $(M, g , f)$ be a complete shrinker with bounded curvature.
	Then there exists a constant $C> 0$ such that
		$$\| \sqrt{f} T \|_{L^2_f} \le C \| T \|_{H^1_f}$$
	for any smooth tensor $T$ on $M$.
\end{lem}
\begin{proof}
	By density, we can additionally assume that $T$ is compactly supported.
	By Proposition \ref{Prop Nonneg Scalar Curv},
		$$0 \le R + | \nabla f|^2  = f .$$
	Multiplying both sides by $|T|^2$ and integrating, it follows that
	\begin{gather*}\begin{aligned}
		& \qquad \| \sqrt{f} T \|_{L^2_f}^2	\\
		&= \int_M f | T|^2 e^{-f} dV_g	\\
		&= \int_M  R |T|^2  e^{-f} dV_g + \int_M | \nabla f|^2 |T|^2 e^{-f} dV_g	\\
		&\le \left( \sup_{M} |R| \right)   \| T\|^2_{L^2_f} + \| \nabla f \otimes T \|^2_{L^2_f}	\\
		&\le \left( \sup_{M} |R| \right) \| T \|_{L^2_f}^2 + \| T \|_{H^1_f}^2		&& (\text{Lemma \ref{multGradfEst}}).
	\end{aligned} \end{gather*}
\end{proof}

Combining Proposition \ref{Prop Potential Func Grows Quadratically} and Lemma \ref{multRootfEst} yields the following result:
\begin{cor} \label{momentEst}
	Let $(M, g, f)$ be a complete, connected shrinker with bounded curvature
	and let $p_0 \in M$.
	Then there exists a constant $C > 0$ such that
		$$\| \dist_g(\cdot, p_0) T\|_{L^2_f} \le C \| T \|_{H^1_f}$$
	for any smooth tensor $T$ on $M$.
\end{cor}

\begin{lem} \label{weightedRellich}
	Let $(M, g, f)$ be a complete, connected shrinker with bounded curvature
	and let $E$ be a tensor bundle on $M$.
	Then the embedding 
		$H^1_f(M, E) \to L^2_f(M, E)$
	is compact.
\end{lem}
\begin{proof}
	Let $\{ T_i \}_{i \in \mathbb{N}}$ be a sequence of tensors of the same type and assume the sequence is uniformly bounded in $H^1_f(M)$.
	After passing to a subsequence obtained through a diagonalization argument,
	the usual Rellich theorem implies that there exists $T_\infty \in H^1_{f, loc}(M)$  
	such that
		$$ T_i \xrightarrow[i\to \infty]{L^2_{f,loc} } T_\infty$$
	By weak compactness of balls in $H^1_f(M)$, in fact $T_\infty \in H^1_f(M)$.
	
	Fix $p_0 \in M$ and let $\ol{r}(p) \doteqdot \dist_g(p , p_0)$. 
	For any $R> 0$,
	\begin{gather*} \begin{aligned}
		&\qquad \| T_i - T_\infty \|_{L^2_f(M)}^2 \\
		&= \int_M | T_i - T_\infty|^2 e^{-f} dV_g	\\
		&= \int_{B_g(p_0, R)} | T_i - T_\infty|^2 e^{-f} dV 
		+ \int_{M \setminus B_g(p_0, R)} | T_i - T_\infty|^2 e^{-f} dV	\\
		&\le \| T_i - T_\infty \|_{L^2_f( B_g(p_0, R)) }^2
		+\frac{1}{R^2} \int_{M \setminus B_g(p_0, R)} \ol{r}^2 | T_i - T_\infty|^2 e^{-f} dV	\\
		&\le \| T_i - T_\infty \|_{L^2_f ( B_g(p_0, R) )}^2 
		+ \frac{1}{R^2} \| T_i - T_\infty \|_{H^1_f(M) }^2 		&& (\text{Corollary \ref{momentEst}})	\\
		&\le \| T_i - T_\infty \|_{L^2_f ( B_g(p_0, R) )}^2 
		+ \frac{C}{R^2} 
	\end{aligned} \end{gather*}
	where $C = \sup_i \| T_i \|_{H^1_f}^2 + \| T_\infty \|_{H^1_f}^2 < \infty$.
	Since the last line may be made arbitrarily small by taking $R$ and $i$ large,
	it follows that in fact 
		$$T_i  \xrightarrow[i\to \infty]{L^2_f } T_\infty.$$
\end{proof}

\subsection{The $\Delta_f + 2 Rm$ Operator}

Let $(M, g, f)$ be a complete shrinker.
In this subsection, we analyze the weighted Lichnerowicz Laplacian $\Delta_f + 2 Rm$ associated to the soliton that acts on symmetric 2-tensors $h \in \Gamma( M, Sym^2 T^*M )$ in coordinates by
	$$\Delta_f h_{ij} + 2 Rm[h]_{ij} = \nabla^a \nabla_a h_{ij} - \nabla^a f \nabla_a h_{ij} + 2 R\indices{^a_i^b_j} h_{ab}.$$

Define
	$$D(\Delta_f ) \doteqdot \left \{ h \in H^1_f(M, Sym^2 T^*M) \ | \ \Delta_f h \in L^2_f(M, Sym^2 T^*M)  \right\}$$
where here $\Delta_f h$ is taken in the sense of distributions.
By \cite[theorem 4.6]{Grigoryan09}, $\Delta_f$ extends to a self-adjoint operator $\Delta_f : D( \Delta_f) \to L^2_f(M, Sym^2 T^*M)$ that agrees with the usual $\Delta_f$ operator on smooth, compactly supported 2-tensors.
If the shrinker $(M, g, f)$ has bounded curvature, then $\Delta_f + 2Rm$ also defines a self-adjoint operator
	$$\Delta_f + 2 Rm : D( \Delta_f ) \to L^2_f ( M, Sym^2 T^*M).$$

\begin{lem} \label{Lem Op Bdd Above}
	Let $(M, g, f)$ be a complete shrinker with bounded curvature.
	Then $\Delta_f + 2 Rm$ is bounded above on $L^2_f(M, Sym^2 T^*M)$. In other words,  there exists a constant $C$ such that
		$$(h , \Delta_f h + 2 Rm[h])_{L^2_f} \le C \| h \|_{L^2_f}		\qquad 
		\text{ for all } h \in D(\Delta_f ).$$
\end{lem}
\begin{proof}
	Observe that
	\begin{gather*} \begin{aligned}
		& \left( h, \Delta_f h + 2 Rm[h] \right)_{L^2_f}	\\
		={}&  - \int_M | \nabla h|^2 e^{-f} dV_g + \int_M 2 \langle Rm[h], h \rangle e^{-f} dV_g	\\
		\le{}& 2 \left( \sup_M | Rm | \right) \int_M  | h|^2 e^{-f} dV_g	.
	\end{aligned}	\end{gather*}
	Since $Rm$ is uniformly bounded, taking 
		$C = 2 \left( \sup_M | Rm | \right) < \infty$
	completes the proof.
\end{proof}

\begin{lem} \label{Lem Compact Resolvent}
	Let $(M, g, f)$ be a complete, non-compact shrinker with bounded curvature.
	Then the resolvent operator
		$$(-\lambda + \Delta_f + 2Rm )^{-1} : L^2_f(M, Sym^2 T^*M) \to L^2_f(M, Sym^2 T^*M)$$
	is a compact operator for all $\lambda \in \rho( \Delta_f + 2 Rm)$, the resolvent of $\Delta_f + 2 Rm$.
\end{lem}
\begin{proof}
	Let $k_j$ be a bounded sequence in $L^2_f$ and define $h_j$ such that
		$$(-\lambda + \Delta_f + 2Rm)^{-1} k_j = h_j		\qquad \text{or} \qquad k_j = \Delta_f h_j + 2 Rm[h_j] - \lambda h_j.$$
	By definition of the resolvent,  $\lambda \in \rho( \Delta_f + 2 Rm)$ implies
		$(-\lambda + \Delta_f + 2Rm)^{-1} : L^2_f \to L^2_f$
	is bounded and thus
	$h_j$ is a bounded sequence in $L^2_f$.
	By integrating both sides of
		$$k_j = \Delta_f h_j + 2 Rm[h_j] - \lambda h_j$$	
	 against $-h_j$ and using the uniform bounds on $\| h_j \|_{L^2_f}$ and $\| k_j \|_{L^2_f}$, it follows that $h_j$ is a bounded sequence in $H^1_f$.
	By Lemma \ref{weightedRellich}, $h_j$ then has a convergent subsequence in $L^2_f$.
\end{proof}

Lemmas \ref{Lem Op Bdd Above} and \ref{Lem Compact Resolvent} show that $\Delta_f + 2 Rm : D( \Delta_f) \to L^2_f$ is a self-adjoint operator that's bounded above and has compact resolvent.
\cite[theorem XIII.64]{ReedSimonIV} therefore yields the following description of its spectrum:

\begin{thm} \label{Thm Op Spectrum}
	Assume $(M, g, f)$ is a complete, connected shrinker with bounded curvature.
	Then there exists an orthonormal basis $\{ h_j \}_{j = 1}^\infty$ of $L^2_f( M, Sym^2 T^*M)$ 
	such that, for all $j \in \mathbb{N}$, $h_j \in D(\Delta_f)$ is an eigenmode of $\Delta_f + 2Rm$ with eigenvalue $\lambda_j \in \R$,
	and the eigenvalues $\{ \lambda_j \}_{j= 1}^\infty$ satisfy $\lambda_1 \ge \lambda_2 \ge \dots$.
	
	Moreover, the eigenvalues $\{ \lambda_j \}_{j = 1}^\infty$ are each of finite-multiplicity, are given by the min-max principle, and tend to $-\infty$ as $j \to \infty$.
	The spectrum $\sigma ( \Delta_f + 2 Rm )$ of $\Delta_f + 2 Rm : D( \Delta_f) \to L^2_f$ equals $\{\lambda_j \}_{j = 1}^\infty$.
\end{thm}

Note also that classical elliptic regularity theory implies the eigenmodes $h_j$ in Theorem \ref{Thm Op Spectrum} are smooth as well.

\subsubsection{Geometric Eigenmodes}

We can specify a few of the eigenmodes and eigenvalues from Theorem \ref{Thm Op Spectrum} explicitly.

\begin{prop} \label{Prop Rc +1 Eigmode}
	Let $(M, g, f)$ be a shrinker.
	Then
		$$\Delta_f Rc  + 2 Rm [Rc] = Rc		\qquad \text{ on } M.$$
	In other words, $Rc$ is an eigenmode of the weighted Lichnerowicz Laplacian with eigenvalue $+1$ so long as $Rc \not \equiv 0$.
\end{prop}
\begin{proof}
	We first note that Proposition \ref{RcInKerDiv_f} and the contracted Bianchi identity imply
	\begin{equation} \label{equation2.1}
		 \nabla_i R = 2 R_{ij} \nabla^j f
	\end{equation} 
	Applying $\Delta$ to the soliton equation \eqref{Soliton Eqns} 
		$R_{ij} = \frac{1}{2} g_{ij} - \nabla_i \nabla_j f$
	and commuting derivatives, we obtain
	\begin{equation*} \begin{aligned}
		\Delta R_{ij} =& - \nabla^k \nabla_k \nabla_i \nabla_j f \\
		=& - \nabla^k \left( \nabla_i \nabla_k \nabla_j f + R_{kij} {}^l \nabla_l f \right)\\
		=& -\nabla^k \nabla_i \nabla_j \nabla_k f - R_{kijl} \nabla^k \nabla^l f - (\nabla^k R_{kijl} ) \nabla^l f\\\
		=& -\nabla^k \nabla_i \nabla_j \nabla_k f - R_{kijl} \nabla^k \nabla^l f + ( \nabla_l R_{ji} - \nabla_j R_{li}) \nabla^l f.\\
	\end{aligned} \end{equation*}
	Also, using the contracted Bianchi identity and (\ref{equation2.1}),
	\begin{equation*} \begin{aligned}
		\nabla^k \nabla_i \nabla_j \nabla_k f =& \nabla_i \nabla^k \nabla_k \nabla_j f + R_{kij} {}^l \nabla^k \nabla_l f + R_{il} \nabla_j \nabla^l f \\
		=& -\nabla_i \nabla^k R_{kj} + R_{kij} {}^l \nabla^k \nabla_l f + R_{il} \nabla_j \nabla^l f \\
		=& - \frac{1}{2} \nabla_i \nabla_j R + R_{kij} {}^l \nabla^k \nabla_l f + R_{il} \nabla_j \nabla^l f  \\
		=& - \nabla_j ( R_{ik} \nabla^k f) + R_{kij} {}^l \nabla^k \nabla_l f + R_{il} \nabla_j \nabla^l f  \\
		=& - \nabla_j R_{ik} \nabla^k f - R_{ik} \nabla_j \nabla^k f  + R_{kij} {}^l \nabla^k \nabla_l f + R_{il} \nabla_j \nabla^l f  \\
		=&- \nabla_j R_{ik} \nabla^k f  + R_{kijl} \nabla^k \nabla^l f.\\
	\end{aligned} \end{equation*}
	Therefore,
	\begin{equation*} \begin{aligned}
		\Delta R_{ij} =&  \nabla_j R_{ik} \nabla^k f - 2R_{kijl} \nabla^k \nabla^l f + ( \nabla_l R_{ji} - \nabla_j R_{li}) \nabla^l f\\
		=& - 2R_{kijl} ( -R^{kl} + \frac{1}{2} g^{kl} ) + \nabla_l R_{ji} \nabla^l f\\
		=& -2R_{kilj} R^{kl} + \nabla^l f \nabla_l R_{ij} + R_{ij}.
	\end{aligned} \end{equation*}
	This completes the proof.
\end{proof}

Proposition \ref{Prop Rc +1 Eigmode} and the soliton equations \eqref{Soliton Eqns} also imply the following corollary:

\begin{cor} \label{Cor nabla^2 f 0 Eigmode}
	If $(M, g, f)$ is a shrinker then 
		$$\Delta_f \nabla^2 f + 2 Rm [ \nabla^2 f ] = 0		\qquad \text{ on } M.$$
	In other words, $\nabla^2 f$ is an eigenmode of the weighted Lichnerowicz Laplacian with eigenvalue $0$ so long as $\nabla^2 f \not \equiv 0$.
\end{cor}

\begin{remark}
	For a shrinker $(M, g, f)$, the metric $g$ is \emph{not} an eigenmode of $\Delta_f + 2 Rm $ in general.
\end{remark}

\subsubsection{The Eigenmode Growth Condition} \label{Subsubsect Eigmode Growth Assumption}

Recall that Theorem \ref{Thm Op Spectrum} says that the spectrum of the weighted Lichnerowicz Laplacian on a complete, connected shrinker with bounded curvature consists entirely of finite multiplicity eigenvalues.
We now introduce the eigenmode growth condition which amounts to a pointwise bound at infinity for the eigenmodes $h_j$.

\begin{definition}[Eigenmode Growth Condition] \label{Defn Eigmode Growth Assumption}
	A complete, connected shrinker $(M^n, g, f)$ with bounded curvature is said to satisfy the \emph{eigenmode growth condition} if, for all $\lambda \in \mathbb{R}$ and $\delta > 0$,
	there exists $C = C(n, M, g, f, \lambda, \delta)$ such that 
	if $h \in D( \Delta_f )$
	has $\| h \|_{L^2_f}  = 1$ 
	and satisfies
	\begin{equation} \label{Eigval Eqn}
		\Delta_{f} h + 2 Rm[h] = \lambda h	,
	\end{equation}
	then
	\begin{equation} 
		| h |_{g}(x) \le C f(x)^{ \max \{ -\lambda, 0 \} + \delta }		\qquad \text{ for all } x \in M.
	\end{equation}
\end{definition}

We make a few remarks about this condition before continuing:

\begin{itemize}

\item Informally, one can think of the eigenmode growth condition as saying,
	$$``\Delta_f h + 2 Rm[h] = \lambda h \implies | h| \le Cf^{- \lambda}"$$
	but where we instead allow for slightly larger growth $f^{- \lambda + \delta}$ and add a caveat in the case $\lambda > 0$.

\item By Theorem \ref{Thm Op Spectrum}, the eigenmode growth condition holds so long as it holds for all $\lambda = \lambda_j \in \sigma ( \Delta_f + 2 Rm)$ and all $h = h_j$.

\item When $(M, g, f)$ is also not Ricci-flat, the eigenmode growth condition automatically holds \emph{locally}.
Indeed, Proposition \ref{Prop Nonneg Scalar Curv} and \eqref{Soliton Eqns} imply $0 < R + | \nabla f|^2 = f$.
Since $f$ grows quadratically at infinity (Proposition \ref{Prop Potential Func Grows Quadratically}), $\inf_M f > 0$. 
Thus, for any compact $\Omega \subset M$ and continuous $h$, there exists a constant $C = C( M, g, f, \lambda, \delta, h, \Omega)$ such that
	$$| h |_{g} (x) \le C f(x)^{\max\{ -\lambda, 0 \} + \delta } 	\qquad \text{ for all } x \in \Omega.$$	

\item Shrinking generalized cylinder solitons $M^n = \R^{n-k} \times \mathbb{S}^k$ do \emph{not} satisfy the eigenmode growth condition.
Indeed, the weighted Lichnerowicz Laplacian $\Delta_f + 2 Rm$ has eigenmodes of the form $h = ( \mathbf z \cdot \mathbf e) g_{\mathbb{S}^k}$ where $\mathbf z$ denotes the usual coordinates on $\R^{n-k}$ and $\mathbf e$ is any fixed vector in $\mathbb{R}^{n-k}$.
These eigenmodes have eigenvalue $\lambda = + \frac{1}{2}$ but grow linearly at infinity, and thus
	$$| h | \sim f^{ \frac{1}{2} } \not \le C f^\delta = C f^{ \max \{ - \frac{1}{2} , 0 \} + \delta }
	\qquad \text{ for $0 < \delta < \frac{1}{2}$}.$$

\item In the next subsection, we shall see that $(M^n, g, f)$ satisfies the eigenmode growth condition whenever it is asymptotically conical (Proposition \ref{Prop Asymp Conical Implies Eigmode Growth Assumption}).
This fact follows from recent results in \cite{CM21} which uses a weighted version of the frequency functions introduced in \cite{Almgren79}.

\item It is worth mentioning some additional related work.
When $(M, g, f)$ is asymptotically conical (see Definition \ref{Defn Asymply Conical} below), \cite[Theorem 1.2]{Bernstein17} showed that a stronger version of the eigenmode growth condition holds for scalar-valued $h = u : M \to \R$.
Also, \cite{DS21} have a related result for tensors on expanding solitons when \eqref{Eigval Eqn} contains an additional Lie derivative term.

\end{itemize}

\subsection{Asymptotically Conical Shrinking Solitons}

In this subsection, we introduce the notion of an asymptotically conical manifold and state some basic results contained in \cite{KotschwarWang15}.

\begin{definition} \label{Defn Cone}
	Let $(\Sigma, g_\Sigma)$ be a smooth, closed Riemannian manifold (not necessarily connected).
	The (regular) \emph{cone} on $\Sigma$, denoted $(\mathcal{C}( \Sigma), g_{\mathcal{C}})$, is the space $\mathcal{C}(\Sigma) = ( 0, \infty) \times \Sigma$ with the metric
		$$g_{\mathcal{C}} = dr^2 + r^2 g_\Sigma$$
	where $r$ is the standard coordinate on $(0, \infty)$.

	We write $\mathcal{C}_R (\Sigma) = (R, \infty) \times \Sigma$ for any $R \ge 0$.
	When the \emph{link} $\Sigma$ of the cone is clear from context, we shall simply write $\mathcal{C}_R$ or $\mathcal{C} = \mathcal{C}_0$.
	
	For any $\lambda > 0$, the \emph{dilation map} is defined by
	\begin{gather*}
		\rho_\lambda : \mathcal{C}_0 \to \mathcal{C}_0		\\
		\rho_\lambda( r, \sigma ) = (\lambda r,  \sigma).
	\end{gather*}
\end{definition}

\begin{definition} \label{Defn End}
	An \emph{end} of a smooth Riemannian manifold $(M, g)$ is an unbounded, connected component of $M \setminus \Omega$ for some compact $\Omega \subset M$.
\end{definition}

\begin{definition} \label{Defn Asymply Conical}
	Let $(M,g)$ be a smooth Riemannian manifold and let $V$ be an end of $M$.
	We say that $V$ is \emph{asymptotic to the cone} $(\mathcal{C}(\Sigma), g_\mathcal{C} )$ if 
	for some $R \ge 0$
	there is a diffeomorphism $\Psi: \mathcal{C}_R \to V$
	such that
		$$\lim_{\lambda \to +\infty} \lambda^{-2} \rho_\lambda^* \Psi^* g = g_\mathcal{C}		\qquad
		 \text{ in } C^2_{loc} \left( \mathcal{C}_0, g_{\mathcal{C}} \right).$$
	We say $V$ is an \emph{asymptotically conical end} if it's asymptotic to some regular cone.
	
	We say $(M, g)$ is \emph{asymptotically conical} if $M$ is noncompact and has finitely many ends 
	all of which are asymptotically conical ends.
\end{definition}

\begin{remark}
	By taking the disjoint union of the links and diffeomorphisms for each end $V$ in Definition \ref{Defn Asymply Conical}, one can show that
	 $(M, g)$ is asymptotically conical if and only if there exists some compact $\Omega \subset M$ and some smooth, closed $( \Sigma, g_\Sigma)$ (not necessarily connected) such that 
	for some $R \ge 0$ there is a diffeomorphism $\Psi : \cl{C}_R \to M \setminus \Omega$
	such that 
		$$\lim_{\lambda \to + \infty} \lambda^{-2} \rho_\lambda^* \Psi^* g = g_{\cl{C}} 
		\qquad \text{ in } C^2_{loc}( \cl{C}_0 , g_{\cl{C}} ).$$
	In this case, $\Sigma$ is connected if and only if $M$ has exactly one end.
\end{remark}

To conclude this subsection, we record some basic facts regarding asymptotically conical shrinkers.
The first result here is from \cite{KotschwarWang15}.

\begin{lem}[Lemma A.1 in \cite{KotschwarWang15}] \label{Lem A.1 KW15}
	Let $V$ be an end of a smooth Riemannian manifold $(M,g)$.
	Let $(\mathcal{C}(\Sigma), g_{\mathcal{C}} )$ be a (regular) cone and let
	$\Psi_{R_0} : \mathcal{C}_{R_0} \to V$ be a diffeomorphism for some $R_0 \ge 0$.
	For all $k \ge 0$, define the proposition
		$$(AC_k) \qquad \lim_{\lambda \to +\infty} \lambda^{-2} \rho_\lambda^* \Psi^* g = g_{\mathcal{C}} \quad
		 \text{ in } C^k_{loc} ( \mathcal{C}_0, g_\mathcal{C} ).$$
	Then
	\begin{enumerate}
		\item \label{Lem A.1 KW15 1}
		$(AC_k)$ holds if and only if
			$$\lim_{R \to +\infty} R^l \| \nabla_{g_\mathcal{C}}^{(l)} ( \Psi^* g - g_\mathcal{C} ) \|_{C^0 ( \mathcal{C}_R, g_\mathcal{C} )} = 0$$ 
			for every $l = 0, 1, \dots, k$.
	
		\item \label{Lem A.1 KW15 2}
		If $(AC_0)$ holds, then 
		the metrics $\Psi^* g$ and $g_\cone$ are uniformly equivalent on $\overline{\cone_R}$ for any $R > R_0$,
		and, for any $\epsilon> 0$ there exists $R > R_0$ such that 
			$$(1 - \epsilon) |r-R| \le
			  \ol{r}_R( r,\sigma) 
			  \le (1+\epsilon) |r - R|$$
		for all $(r, \sigma) \in \cone_R$.
		Here and below $\ol{r}_R( r,\sigma)  = \dist_{\Psi^*g} ( (r, \sigma) , \partial \cone_R )$.
		
		\item \label{Lem A.1 KW15 3}
		If $(AC_2)$ holds, then
		for any $R > R_0$
		there exists a constant $C$ depending on $R$ and $g_\Sigma$ such that
			$$\left( 1 + \ol{r}^2_R (p) \right) | Rm [{\Psi^*g} ] |_{\Psi^*g}(p) \le C$$
		for all $p \in \cone_R$.
	\end{enumerate}
\end{lem}

When $(M, g, f)$ is also a shrinker, the quadratic curvature decay in Lemma \ref{Lem A.1 KW15} \eqref{Lem A.1 KW15 3} improves to decay of all derivatives of curvature.

\begin{lem} \label{Lem Soliton Curv Decay}
	Let $(M^n, g, f)$ be an asymptotically conical, complete, connected shrinker.
	Then for all $m \in \mathbb{N}$, there exists a constant $C_m = C_m ( n , M, g, f, m)$ such that
		$$\sup_M f^{1 + \frac{m}{2}} | \nabla^m Rm |_g \le C_m.$$
\end{lem}
\begin{proof}
	Choose an element $p_0$ in the compact set $\{ p \in M : f(p) \le 1 +  \inf_M f \}$.
	\cite[Proposition 2.1(3)]{KotschwarWang15} implies that for all $m \in \mathbb{N}$, there exists $C_m$ such that
		$$\sup_{p \in M} \left( 1 + \dist_g( p, p_0 )^{2 + m} \right) | \nabla^m Rm |_g \le C_m.$$	
	Since $f$ grows quadratically by Proposition \ref{Prop Potential Func Grows Quadratically}, the statement of the lemma follows.
\end{proof}

\begin{lem} \label{Lem Nonflat and f>0}
	Let $(M^n, g, f)$ be an asymptotically conical, complete, connected shrinker that is not flat.
	Then $(M^n, g)$ is not Ricci-flat, the scalar curvature $R > 0$ is positive, and $\inf_M f > 0$.
\end{lem}
\begin{proof}
	Suppose for the sake of contradiction that $(M, g)$ is Ricci-flat.
	Consider the family of diffeomorphisms $\phi( \cdot, t) = \phi_t : M \to M$ from Definition \ref{Defn phi}.
	Then $t \mapsto ( 1 - t) \phi_t^* g$ and $t \mapsto g$ are both Ricci flow solutions with initial condition $g$.
	Note that Lemma \ref{Lem A.1 KW15} implies asymptotically conical manifolds $(M, g)$ have bounded curvature and Euclidean volume growth.
	By the uniqueness of Ricci flow solutions that are complete with bounded curvature \cite{ChenZhu06}, it follows that $(1-t) \phi_t^* g = g$ for all $t \in ( - \infty, 1)$.
	Therefore, for all $ t \in ( - \infty , 1)$,
		$$\sup_M | Rm[g] |_{g} = \sup_M | Rm[ ( 1 - t) \phi_t^* g] |_{( 1 - t) \phi_t^* g} = \frac{1}{1 - t} \sup_M | Rm [g] |_g,$$
	which forces $(M, g)$ to be flat.
	
	Now, since $(M, g)$ is complete, connected, and not Ricci-flat, Proposition \ref{Prop Nonneg Scalar Curv} implies the scalar curvature $R > 0$ is positive.
	\eqref{Soliton Eqns} then yields
		$$0 < R + | \nabla f|^2 = f		\qquad \text{ on } M.$$
	Since $f$ also grows quadratically by Proposition \ref{Prop Potential Func Grows Quadratically}, 
	it follows that $\inf_M f > 0$.
\end{proof}

\begin{prop}[Weighted Volume Decay] \label{Prop Weighted Vol Decay}
	Let $(M^n, g, f)$ be an asymptotically conical, complete, connected shrinker.
	There exists a constant $C = C(n, M, g,f)$ such that 
	if $f_0 \gg 1$ is sufficiently large depending on $n, M, g, f$,
	then
		$$\int_{ \{ x \in M : f(x) \ge f_0 \} } e^{-f} dV_{\ol{g}} \le C f_0^{ \frac{n}{2} - 1} e^{-\frac{1}{8} f_0}.$$
\end{prop}
\begin{proof}
	If $f_0 \gg 1$ is sufficiently large depending on $n, M, g, f$, 
	then Proposition \ref{Prop Potential Func Grows Quadratically} and Lemma \ref{Lem A.1 KW15} (2) imply that, after pulling back to the cone, the integral can be estimated by
	\begin{align*}
		\int_{\{ x \in M : f(x) \ge f_0 \}} e^{ - f} dV_g
		&\le C \int_{\frac{1}{2} \sqrt{4 f_0}}^\infty r^{n-1} e^{- \frac{1}{2} \frac{ r^2 }{4} } dr	\\
		&\le C \sqrt{ f_0}^{n - 2} e^{- \frac{1}{8} \sqrt{f_0}^2 } && ( f_0 \gg 1)	\\
		&= C f_0^{\frac{n}{2} - 1} e^{ - \frac{1}{8} f_0 }.
	\end{align*}
	where the constants $C$ above may change from line to line but depend only on $n, M, g, f$.
\end{proof}

\begin{remark}
	One ought to be able to improve Proposition \ref{Prop Weighted Vol Decay} to 
		$$\int_{ \{ x \in M : f(x) \ge f_0 \} } e^{-f} dV_{\ol{g}} \le C f_0^{ \frac{n}{2} - 1} e^{- f_0}.$$
	However, Proposition \ref{Prop Weighted Vol Decay} as written will suffice for our applications.
\end{remark}

\begin{lem} \label{Lem Flow Est 1}
	Let $(M^n, g, f)$ be an asymptotically conical, complete, connected shrinker.
	Recall that the one-parameter family $\{ \phi( \cdot, t) = \phi_t : M \to M \}_{t \in ( -\infty, 1)}$ of diffeomorphisms from Definition \ref{Defn phi} is defined by
	\begin{equation} \tag{\ref{phi Evol Eqn}}
		\partial_t \phi_t = \frac{1}{1 - t} \nabla f \circ \phi_t	,	\qquad \phi_0 = Id_M.
	\end{equation}
	There exists a constant $C = C(n, M, g, f)$ such that
		$$\max \left\{ 0, \phi_t^* f - \frac{ C}{ \phi_t^* f} \right\}\le ( 1 - t) \partial_t \left( \phi_t^* f \right) \le \phi_t^* f
		\qquad \text{ for all } t \in ( - \infty, 1).$$
\end{lem}
\begin{proof}
	First, observe that
	\begin{align*}
		\partial_t ( \phi_t^* f ) 
		&= \frac{ \partial}{\partial t} \left( f \circ \phi_t \right)	\\
		&= \langle \nabla f \circ \phi_t , \partial_t \phi_t \rangle	\\
		&= \frac{1}{ 1 -t} | \nabla f|^2_g \circ \phi_t	
		&& \eqref{phi Evol Eqn}\\
		&\ge 0.
	\end{align*}
	By \eqref{Soliton Eqns} and Proposition \ref{Prop Nonneg Scalar Curv}, 
		$$0 \le | \nabla f |^2 \le R + | \nabla f |^2 = f.$$
	Hence,
		$$\partial_t ( \phi_t^* f ) =  \frac{1}{1 - t} | \nabla f|_g^2 \circ \phi_t \le \frac{1}{1 - t} f \circ \phi_t = \frac{1}{1 - t} \phi_t^* f$$
	which proves the upper bound on $( 1 - t) \partial_t \phi_t^* f$.
	
	For the other lower bound, 
	first note that Lemma \ref{Lem Soliton Curv Decay} and Proposition \ref{Prop Potential Func Grows Quadratically} imply there exists $C = C(n, M , g ,f )$ such that
		$$0 \le R(x) \le \frac{C}{f(x)} 		\qquad \text{ for all } x \in M.$$
	Thus, by \eqref{Soliton Eqns},
		$$| \nabla f|^2_g = f  - R \ge f - \frac{C}{f},$$
	and so
		$$\partial_t ( \phi_t^* f ) = \frac{1}{1 - t} | \nabla f|^2_g \circ \phi_t \ge \left( f -  \frac{C}{f} \right) \circ \phi_t 
		= \phi_t^*f - \frac{ C}{ \phi_t^* f} .$$ 
\end{proof}

\begin{prop} \label{Prop Asymp Conical Implies Eigmode Growth Assumption}
	If $(M^n, g, f)$ is an asymptotically conical, complete, connected shrinker that is not flat,
	then $(M^n, g, f)$ satisfies the eigenmode growth condition (Definition \ref{Defn Eigmode Growth Assumption}).
\end{prop}
\begin{proof}
	Let $\delta > 0$ and
	$j \in \mathbb{N}$.
	As in Theorem \ref{Thm Op Spectrum},
	denote the corresponding eigenvalue $\lambda_j \in \sigma ( \Delta_f + 2 Rm )$ 
	and its eigenmode 
	$h_j$ with $\| h_j \|_{L^2_f} = 1$ and
	\begin{equation} \tag{\ref{Eigval Eqn}}
		\Delta_f h_j + 2 Rm [h_j] = \lambda_j h_j.
	\end{equation}
	By standard elliptic regularity and unique continuation results,
	$h_j$ is smooth and cannot vanish on any open set.
	
	Taking the inner product of \eqref{Eigval Eqn} with $h_j$ implies
		$$\langle \Delta_f h_j, h_j \rangle_g 
		= ( \lambda_j - C_n | Rm|_g ) | h_j|^2 
		\ge - ( \max \{ - \lambda_j, 0 \} + C_n |Rm|_g ) | h_j|^2 $$
	for some constant $C_n> 0$ that depends only on $n$.
	Since $|Rm|_g$ decays quadratically (Lemma \ref{Lem A.1 KW15} \eqref{Lem A.1 KW15 3}),
	there exists a compact set $\Omega \subset M$ such that
		$$C_n  | Rm|_g ( x) \le \delta 	\qquad \text{for all } x \in M \setminus \Omega.$$
	Thus,
		$$\langle \Delta_f h_j, h_j \rangle_g \ge - ( \max \{ - \lambda_j, 0 \} + \delta ) | h |_g^2
		\qquad \text{for all } x \in M \setminus \Omega.$$
	\cite[Theorem 0.7]{CM21}	now applies to show $|h_j|$ grows at most polynomially with rate $2 \max \{ - \lambda_j , 0 \} + 2\delta $ at infinity.
	Because $ \inf_M f > 0$ (Lemma \ref{Lem Nonflat and f>0}) and $f$ grows quadratically (Proposition \ref{Prop Potential Func Grows Quadratically}),
	it follows that, for some $C_j = C_j(n, M, g, f, j, \delta)$,
		$$|h_j |_g(x) \le C_j f(x)^{\max \{ -\lambda_j, 0 \} + \delta} 		\qquad \text{for all } x \in M.$$
	Since the eigenspaces of $\Delta_f + 2Rm$ have finite dimension (Theorem \ref{Thm Op Spectrum}),
	the eigenmode growth condition (Definition \ref{Defn Eigmode Growth Assumption}) for general $\lambda \in \R$ follows with $C = C( n, M, g, f, \lambda, \delta ) = \max_{\lambda_j = \lambda} C_j$.
\end{proof}

For asymptotically conical shrinkers, the pointwise bound in the eigenmode growth condition can be bootstrapped to derivative bounds of any order.

\begin{prop} \label{Prop Eigmode Derivative Growth}
	Let $(M^n, g, f)$ be an asymptotically conical, complete, connected shrinker that is not flat.
	Let $h_j$ be an eigenmode of the weighted Lichnerowicz Laplacian $\Delta_f + 2 Rm$ with eigenvalue $\lambda_j$ as in Theorem \ref{Thm Op Spectrum}.
	Then for all $\delta > 0$ and all $m \in \mathbb{N}$, there exists $C_m = C_m( n, M, g, f, j, \delta, m)$ such that
		$$| \nabla^m h_j |_{g} (x) \le C_m f(x)^{\max\{-\lambda_j, 0 \} + \delta} 		\qquad \text{ for all } x \in M.$$
\end{prop}
\begin{proof}
	We shall add a time variable $t$ and apply a rescaling argument to bootstrap the $C^0$-estimates 
		$$| h_j |_{g} \le C_0 (n, M, g, f, j, \delta) f^{ \max \{ -\lambda_j , 0 \} + \delta}$$
	from Proposition \ref{Prop Asymp Conical Implies Eigmode Growth Assumption} to $C^m$-estimates.
	Consider the family $\cl{H}(t)$ of symmetric 2-tensors on $M$ given by
		$$\cl{H}(t) \doteqdot (1 - t) \phi_t^* h_j$$
	where $\phi_t : M \to M$ is the family of diffeomorphisms solving
	\begin{equation} \tag{\ref{phi Evol Eqn}}
		\partial_t \phi_t = \frac{1}{1 -t} \nabla f  \circ \phi_t 
		\qquad \text{with}	\qquad
		\phi_0 = Id_M : M \to M.
	\end{equation}
	Then $\cl{H}$ satisfies the evolution equation
	\begin{align*}
		&\partial_t \cl{H} 	\\
		={}& - \phi_t^* h + ( 1 - t) \phi_t^* \cl{L}_{\partial_t \phi \circ \phi_t^{-1} } h_j	\\
		={}& - \phi_t^* h +  \phi_t^* \cl{L}_{\nabla f} h_j	
		+ \phi_t^* \left( \Delta_f h_j + 2 Rm [ h_j ] - \lambda_j h_j \right)\\
		={}& - \phi_t^* h
		+ \phi_t^* \left( 
		 \nabla_{\nabla f} h_j + (h_j)_{ac} \nabla_b \nabla^c f +  (h_j)_{bc} \nabla_a \nabla^c f \right)	\\
		&+ \phi_t^* \left( \Delta h_j - \nabla_{\nabla f} h_j + 2 Rm[h_j] 
		- \lambda_j h_j \right) \\
		={}& \phi_t^* \left( \Delta h_j + 2 Rm[h_j] - (h_j)_{ac} Rc_b^c - (h_j)_{bc} Rc_a^c - \lambda_j h_j \right)
	\end{align*}
	where we use the soliton equation \eqref{Soliton Eqns} $Rc + \nabla^2 f = \frac{1}{2} g$ in the last equality.
	With respect to the metrics $\check g(t) \doteqdot ( 1 - t) \phi_t^* g$, it follows that
	\begin{equation} \label{Deriv Eigmode Growth Rescaled Eqn}
		\partial_t \cl{H} = \check \Delta \cl{H} + 2 \check{Rm} [ \cl{H} ] - \check{Rc} * \cl{H} - \frac{ \lambda_j}{ 1 - t} \cl{H}
	\end{equation}
	where $\check \Delta = \Delta_{\check g} , \check{Rm} = Rm_{\check g},$ and $\check{Rc} = Rc_{\check g}$.
	Note that $\check g(t) = (1- t) \phi_t^* g$ is the Ricci flow solution with $\check g(0) = g$
	and is smooth with uniformly bounded curvature for all $t \in [ -1, 0]$.
	Hence, for all $t \in [-1, 0]$, the coefficients in \eqref{Deriv Eigmode Growth Rescaled Eqn} and their derivatives of order at most $m$ can be uniformly bounded by constants that depend only on $n, M, g, f, j, m$. 
	
	Since $|Rm_g |_g$ is uniformly bounded, the conjugate radius of $( M, g )$ is bounded below by say $2 r = 2 r( n, M, g, f) \in \left( 0, 1 \right]$.
	After possibly passing to a local cover and pulling back to $\R^n$ using the exponential map (with respect to $g$) based at an arbitrary $p \in M$,
	Lemma \ref{Lem Lin Int Est} applies to equation \eqref{Deriv Eigmode Growth Rescaled Eqn} 
	with $\Omega' = B_{2r} \times [ -4 r^2, 0 ] \subset \R^n \times [-1,0]$
	and no inhomogeneous term.
	Since the coefficients of \eqref{Deriv Eigmode Growth Rescaled Eqn} and $r$ have bounds that depend only on $n, M, g, f, j$, it follows from Lemma \ref{Lem Lin Int Est} that
	\begin{align*}
		& \quad \, \, | \nabla^m h_j |_g (p)	\\
		&= | \check \nabla^m \cl{H} |_{\check g} (p, t = 0 )	\\
		&\le C(n, M, g, f, j, m) \sup_{ (x,t) \in B_g (p, 2r) \times [ -1, 0 ] } | \cl{H} |_{\check g} (x,t)	
		&& (\text{Lemma \ref{Lem Lin Int Est}})\\
		&= C(n, M, g, f, j, m) \sup_{ (x,t) \in B_g (p, 2r) \times [ -1, 0 ] } | h_j |_g ( \phi_t (x) )	\\
		&\le C(n, M, g, f, j, \delta, m) \sup_{ (x,t) \in B_g (p, 2r) \times [ -1, 0 ] } f( \phi_t(x))^{ \max \{ - \lambda_j, 0 \} + \delta}
		&& (\text{Proposition \ref{Prop Asymp Conical Implies Eigmode Growth Assumption}}).
	\end{align*}
	By Lemma \ref{Lem Flow Est 1}, $0 \le \partial_t \phi_t^* f$ and thus $f( \phi_t(x) ) \le f( \phi_0 (x) ) = f(x)$ for all $t \le 0$.
	Hence,
	\begin{align*}
		| \nabla^m h_j |_g (p)
		&\le C(n, M, g, f, j, \delta, m) \sup_{ (x,t) \in B_g (p, 2r) \times [ -1, 0 ] } f( \phi_t(x))^{ \max \{ - \lambda_j, 0 \} + \delta}	\\
		&\le  C(n, M, g, f, j, \delta, m) \sup_{ x \in B_g (p, 2r)  } f( x)^{ \max \{ - \lambda_j, 0 \} + \delta}	\\
		&\le C(n, M, g, f, j, \delta, m) f(p)^{ \max \{ - \lambda_j, 0 \} + \delta}
	\end{align*}
	where the last equality follows from the fact that $f$ grows quadratically (Proposition \ref{Prop Potential Func Grows Quadratically}) and has $\inf_M f > 0$ (Lemma \ref{Lem Nonflat and f>0}).
	This completes the proof since $p \in M$ was arbitrary and no constants depend on $p$.	
\end{proof}

\section{Strategy of the Proof of Theorem \ref{Main Thm}} \label{Sect Strat of Proof}

Having completed our preliminary discussion of gradient shrinking Ricci solitons, we shall now fix a few choices for the remainder of the paper.
\begin{assumption} \label{Assume Shrinker}
Henceforth, $(M, \ol{g}, f)$ will denote an asymptotically conical, complete, connected shrinker that is not flat.
We shall use overlines to refer to geometric objects associated to $(M, \ol{g}, f)$, e.g. $\ol{Rm}, \ol{\nabla},$ etc.
As in Theorem \ref{Thm Op Spectrum}, fix an $L^2_f$-orthonormal basis $\{ h_j \}_{j \in \mathbb{N} }$ of eigenmodes with corresponding eigenvalues $\{ \lambda_j \}_{j \in \mathbb{N}}$ for the weighted Lichnerowicz Laplacian $\ol{\Delta}_f  + 2 \ol{Rm}$.
Fix $\lambda_* \in (-\infty, 0) \setminus \{ \lambda_j \}_{j \in \mathbb{N}}$ and let $K \in \mathbb{N}$ denote the eigenmode such that $\lambda_K > \lambda_* > \lambda_{K+1}$.
\end{assumption}

We now proceed to sketch the proof of Theorem \ref{Main Thm}.
Outside of some compact set, $(M, \ol{g})$ looks approximately like a cone $(\cl{C}( \Sigma) , g_{\cl{C}})$.
We can therefore find a suitably large precompact set $M' \subset M$ whose closure $\ol{M'}$ is a smooth manifold with boundary.
Identifying two copies of $\ol{M'}$ along their boundary yields a closed manifold $\cl{M}$ known as the double of $\ol{M'}$.

After adjusting the metric $\ol{g}$ near $\partial \ol{M'}$, $\ol{g}$ induces a $\mathbb{Z}_2$-invariant Riemannian metric on $\cl{M}$ which equals $\ol{g}$ away from $\partial \ol{M'} \subset \cl{M}$.
Here, the $\mathbb{Z}_2$ symmetry interchanges the two copies of $M'$ in the double $\cl{M}$.

We can repeat this same process with the family of metrics 
	$$G_{\mathbf p }(t_0) = ( 1 - t_0 ) \phi_{t_0}^* \left[ \ol{g} +  \sum_{j= 1}^{K} p_j h_j  \right]$$
indexed by $\mathbf p = ( p_1, p_2, \dots, p_K )  \in \R^{K}$ and $t_0 \in [0, 1)$ 
to obtain a family of $\mathbb{Z}_2$-invariant metrics $G_{\mathbf p} (t_0)$ on $\cl{M}$.

Consider the evolution of each of these metrics by Ricci flow.
In other words, let $ \{ G(t) = G_{\mathbf p } ( t) \}_{t \in [t_0, T( \mathbf p ) ) }$ denote the maximal solution to the Ricci flow on $\cl{M}$
	$$\partial_t G = - 2 Rc 		\qquad \text{ with initial condition } G(t_0) =G_{\mathbf p} (t_0) .$$
These metrics on $\cl{M}$ remain $\mathbb{Z}_2$-invariant and can thus be identified with a Ricci flow solution on  $M' \subset M$.
We extend these to \emph{approximate} Ricci flow solutions $\acute G_{\mathbf p }(t)$ on $M$ by extending $G_{\mathbf p } ( t )$ by $(1 - t) \phi_t^* \ol{g}$ on $M \setminus M'$ and smoothing near $\partial \ol{M'}$.

We then rescale and reparametrize $\acute G_{\mathbf p} (t ) $ to
	$$g_{\mathbf p }( t ) \doteqdot \frac{1}{1 - t} (\Phi_t^{-1} )^* \acute G_{\mathbf p } ( t)$$
where $\Phi_t : M \to M$ 
is given by $\Phi_t = \phi_t \circ \tl{\Phi}_t$.
Here $\phi_t : M \to M$ is the family of diffeomorphisms associated to the soliton
	$$\partial_t \phi_t = \frac{1}{1  -t} \nabla f \circ \phi_t
	\qquad \text{ with } \phi_0 = Id_M$$
and $\tl{\Phi}_t : M \to M$ solves the harmonic map heat flow
	$$\partial_t \tl{\Phi}_t = \Delta_{\acute G_{\mathbf p}(t), ( 1 - t) \phi_t^* \ol{g}} \tl{\Phi}_t
	\qquad \text{ with } \tl{\Phi}_{t_0} = Id_M.$$

\begin{remark}
Note that $g_{\mathbf p}(t)$ is only well-defined when $\Phi_t$ or equivalently $\tl{\Phi}_t$ is well-defined and a diffeomorphism.
We address this issue by, in effect, developing well-posedness results for the solution $\tl{\Phi}_t$ of the harmonic map heat flow that hold so long as $h_{\mathbf p } = g_{\mathbf p} - \ol{g}$ is small.
Subsection \ref{Subsect Must Exit Through clB} and Appendix \ref{Appdix HarMapFlow} establish these well-posedness results.
This well-posedness when $h_{\mathbf p}$ is small will be sufficient for our Wa{\. z}ewski box argument below to show that $h_{\mathbf p }$ must exit the box $\cl{B}$ before $\tl{\Phi}_t$ fails to exist, be unique, or be a diffeomorphism.
\end{remark}

In order to prove Theorem \ref{Main Thm}, we will show that for some $\mathbf p$ and $t_0$, $g_{\mathbf p } ( t )$ converges to the soliton metric $\ol{g}$ as $t \nearrow 1$.
To do so, we consider the difference $h_{\mathbf p} ( t ) \doteqdot g_{\mathbf p } (t ) - \ol{g}$.
This difference satisfies an equation of the form
	$$(1 - t) \partial_t h_{\mathbf p} = \ol{\Delta}_f h_{\mathbf p} + 2 \ol{Rm} [h_{\mathbf p}] + \cl{E}_1 + \cl{E}_2$$
where $\cl{E}_1, \cl{E}_2$ are small error terms that arise from linearizing the Ricci tensor at $\ol{g}$ and the fact that $\acute G_{\mathbf p }(t )$ is only an approximate Ricci flow solution, respectively.
The details of the construction thus far are provided in detail in Subsections \ref{Subsect The Manifold}--\ref{Subsect The Flow}.

Because the linearized operator $\ol{\Delta}_f  + 2 \ol{Rm}$ has positive eigenvalues (Proposition \ref{Prop Rc +1 Eigmode}), one should not expect $h_{\mathbf p}(t)$ to converge to $0$ for general choices of $\mathbf p$ and $t_0$.
If the flow were given purely by the linearized equation
	$$(1 - t) \partial_t h_{\mathbf p } \approx \ol{\Delta}_f h_{\mathbf p } + 2 \ol{Rm} [ h_{\mathbf p }]$$
one could simply take $\mathbf p = \mathbf 0$ and show $h_{\mathbf 0}(t)$ limits to $0$.
However, the nonlinear error terms $\cl{E}_1, \cl{E}_2$ preclude such an argument from going through.

Beginning in Subsection \ref{Subsect Defn of the Box}, we implement a Wa{\.z}ewski box argument \cite{Wazewski47} to show that $h_{\mathbf p }(t)$ converges to $0$ as $t \nearrow 1$ for \emph{some} small $\mathbf p$ and $t_0 < 1$ sufficiently close to $1$.
Informally, the general strategy is as follows: we design a time-dependent set $\cl{B}$ of time-dependent symmetric 2-tensors that becomes more restrictive as $t \nearrow 1$.
Namely, the set $\cl{B}$ consists of time-dependent symmetric 2-tensors $h(t)$ on $M$ that display an $L^2_f(M)$ decay and uniform $C^2(M)$ bounds on a given time interval.
Through Sections \ref{Sect Prelim Ests}--\ref{Sect Proof of Main Thm}, we prove a series of estimates designed to show that, for a suitable choice of parameters, if $h_{\mathbf p}$ exits the set $\cl{B}$ before time $t = 1$ then it must exit through a $(K-1)$-dimensional sphere $S^{K-1}$ in the boundary $\partial \cl{B}$ of $\cl{B}$.
If, for all $\mathbf p$ in $K$-dimensional ball $B^K \subset \R^K$, $h_{\mathbf p}(t)$ were to exit $\cl{B}$ at some time $t^*( \mathbf p ) < 1$, then the exit map $\mathbf p \mapsto h_{\mathbf p}(t^*( \mathbf p))$ would yield a continuous map $B^K \to S^{K-1}$ that, by construction, restricts to a map $S^{K - 1} \approx \partial B^K \to S^{K-1}$ on the boundary which is homotopic to the identity.
However, such a map cannot exist as the identity map $Id : S^{K-1} \to S^{K-1}$ is not null-homotopic.
Therefore there must exist some $\mathbf p^*$ such that $h_{\mathbf p^*}(t)$ remains in $\cl{B}$ for all $t \in [t_0, 1)$.
By the construction of $\cl{B}$, the corresponding Ricci flow $G_{\mathbf p^*}(t)$ then yields the desired Ricci flow in Theorem \ref{Main Thm}.

Sections \ref{Sect Prelim Ests}--\ref{Sect Behavior of L^2_f} provide the estimates necessary to make the above argument rigorous.
We complete the proof of Theorem \ref{Main Thm} in Section \ref{Sect Proof of Main Thm}.
\cite{AV97, AIK15, LeeZhao24} have successfully used a similar strategy in other singularity constructions for the mean curvature flow and Ricci flow.
\cite{HV, V94, Mizoguchi04, BS17, Stolarski19} among others have also applied a closely related technique to exhibit singularity formation for parabolic partial differential equations and geometric flows.

\begin{remark} \label{remark generalizing topologies}
For concreteness, we outlined the proof strategy above when $\cl{M}$ is the double of $\ol{M}' \subset M$ and the metrics $G_{\mathbf p}(t)$ on $\cl{M}$ are $\mathbb{Z}_2$-invariant.
However, because the analysis mainly depends on the behavior of the rescaled and reparameterized metrics $g_{\mathbf p}(t)$ on the soliton manifold $M$, the proof in fact generalizes to a broader class of closed manifolds $\cl{M}$ and symmetries $\cl{S}$ defined in Subsection \ref{Subsect The Manifold} below.
\end{remark}

\section{Setup} \label{Sect Setup}

\subsection{The Manifold} \label{Subsect The Manifold}
Consider the asymptotically conical shrinker $(M^n, \ol{g}, f)$.
For any $\Gamma_0 > 1$, define the open submanifold $M' = M'_{\Gamma_0} \doteqdot \{ x \in M : f(x) < 100 \Gamma_0 \} \subset M$.

Assume $\cl{M}^n$ is a closed $n$-manifold that contains (finitely many) disjoint diffeomorphic copies $\cl{M}_\omega'' \subsetneq \cl{M}$ of $M$ indexed by $\omega$ in some (finite) set $A$.
Fix diffeomorphisms $\iota_\omega'' : \cl{M}_\omega'' \to M$ for each $\omega \in A$ and let $\iota''$ be the naturally induced diffeomorphism $\iota'' : \sqcup_{\omega \in A} \cl{M}_\omega'' \to \sqcup_{\omega \in A} M$.
For all $\omega \in A$, these maps restrict to diffeomorphisms 
$$\iota_\omega : \cl{M}_\omega' \to M' = \{ x \in M : f(x) < 100 \Gamma_0 \}$$
 on some open subset $\cl{M}_\omega' \subset \cl{M}_\omega'' \subset \cl{M}$.
Let 
$$\cl{M}' \doteqdot \bigsqcup_{\omega \in A} \cl{M}_\omega' \subset \cl{M} \quad \text{and} \quad \iota : \cl{M}' \to \bigsqcup_{\omega \in A} M'$$
denote the diffeomorphism naturally obtained from the $\iota_\omega$.
Note that the $\iota_\omega$ and $\iota$ depend on $\Gamma_0$, but the $\iota_\omega''$ and $\iota''$ are independent of $\Gamma_0$.

While $f: M \to \R$ is strictly speaking a function on $M$, we will abuse notation and use $f$ to also refer to the continuous, piecewise smooth function $\cl{M} \to [0, 100 \Gamma_0] \subset \R$ defined by $\iota_\omega^* f$ on every $\cl{M}_\omega' \approx M'$ and extended by $100 \Gamma_0$ on $\cl{M} \setminus \cl{M}'$.
Note that $f : \cl{M} \to \R$ is smooth on the complement of $\partial \cl{M}' \subset \cl{M}$ and also depends on $\Gamma_0$.

Assume furthermore that $\cl{S} \le \text{Diff}(\cl{M})$ is a subgroup of the diffeomorphism group of $\cl{M}$ that respects the identifications $\iota_\omega'' : \cl{M}''_\omega \to M$.
Namely, for all $\psi \in \cl{S}$ and all $\omega \in A$, there exists $\tilde \omega \in A$ such that
\begin{equation} \label{eqn S respects the soliton identifications 1}
	\psi(\cl{M}''_\omega) = \cl{M}''_{\tilde \omega} \text{ and } \iota_{\tilde \omega}'' \circ \psi|_{\cl{M}''_\omega} = \iota_\omega'' : \cl{M}''_\omega \to  M .
\end{equation}
By restriction, the same properties hold for every $\cl{M}'_\omega \subset \cl{M}''_\omega$, that is, for any $\psi \in \cl{S}$ and any $\omega \in A$ there exists $\tilde \omega \in A$ such that 
\begin{equation} \label{eqn S respects the soliton identifications 2}
	\psi(\cl{M}'_\omega) = \cl{M}'_{\tilde \omega} \text{ and } \iota_{\tilde \omega} \circ \psi|_{\cl{M}'_\omega} = \iota_\omega : \cl{M}'_\omega \to  M .
\end{equation}
In particular, $\psi( \cl{M}' ) = \cl{M}'$, $\iota \circ \psi = \iota : \cl{M}' \to \sqcup_{\omega \in A} M'$, and $\psi^* f = f : \cl{M} \to [0, 100\Gamma_0]$.
It follows that there is also an induced action of $\cl{S}$ on $A$ by permutations where, given $\psi \in \cl{S}$ and $\omega \in A$, $\psi \cdot \omega$ is the unique element of $A$ such that $\psi(\cl{M}_\omega'') = \cl{M}_{\psi \cdot \omega}''$.
By replacing $A$ with an orbit of this $\cl{S}$ action, we shall additionally assume without loss of generality that $\cl{S}$ acts transitively on $A$.

This data is summarized in Figure \ref{figure closed manifold} below.

\begin{remark}
	For the simplest situation, one can consider the case where $\cl{S} = \{ Id \}$ and $A$ is a set of one element.
	In this case, $\cl{M}$ is a closed manifold that contains a diffeomorphic copy $\cl{M}''$ of $M$.
	There is a fixed diffeomorphism $\iota'' : \cl{M}'' \to M$ and the restriction $\iota : \cl{M}' \to M'$.
	The function $f: \cl{M} \to [0 ,100 \Gamma_0]$ is defined to be $\iota^* f$ on $\cl{M}'$ and $100 \Gamma_0$ on $\cl{M} \setminus \cl{M}'$.

	In Appendix \ref{App Rounding Out Cones}, we present a construction where $\cl{S} \cong \mathbb{Z}_2$ and $A$ is a set of two elements.
	The manifold $\cl{M}$ in this case is the double of a large subset of $M$, and $\cl{S}$ permutes the two pieces of the double. (See also Remark \ref{remark G_0(t_0) exists}.)
\end{remark}

\tikzset{
  soliton/.pic={   
  \begin{scope}[transform shape, rotate=90]
    \pic[tqft/cap, name = a,
    at={(0,6)}, circle x radius = 0.5cm,
    cobordism edge/.style={draw},
    ];
    
    \pic[tqft/cylinder, name = b, genus =1,
    anchor=incoming boundary 1, at=(a-outgoing boundary 1),
    cobordism height = 0.5cm,
    circle x radius = 0.5cm,
    hole 1/.style = {draw, transform shape, rotate=-90, xshift=0.25cm, yshift=0.25cm},
    cobordism edge/.style={draw},
    ];
    
    \draw (-0.5,3.5) .. controls (-0.5,3) .. 
        node (-d90) [pos=0.95, coordinate] {}
        (-2,0) coordinate (-d100);
    \draw (0.5, 3.5) .. controls (0.5,3)  .. (2,0)
        node (-a90) [pos=0.95, coordinate] {}
        coordinate (-a100);
    \draw (-0.5,0) coordinate (-c100) .. controls (0,3.5) ..
        node (-c90) [pos=0.05, coordinate] {}
        node (-b90) [pos=0.95, coordinate] {}
        (0.5,0) coordinate (-b100);
    
    \draw[style=dashed] (-c90) .. controls (-1,0.6) .. (-d90);
    \draw[style=dashed] (-b90) .. controls (1, 0.6) .. (-a90);

    \node (-upper_pt) at (2,2.25) [coordinate] {};
    \node (-lower_pt) at (-2,2.25) [coordinate] {};
    \node (-left_pt) at (0, 4.5) [coordinate] {};
    \node (-right_pt) at (0,0) [coordinate] {};

    \end{scope}
  }
}

\tikzset{
    hole/.pic={
    \begin{scope}[scale=0.5]
        \draw (-1,0) to[bend left] (1,0);
        \draw (-1.2,0.1) to[bend right] (1.2,0.1);
    \end{scope}
    }
}

\begin{figure} \label{figure closed manifold} 
	\centering
\begin{tikzpicture}[transform shape, scale = 1.5]
    \pic[scale=0.4] (manifold0) at (0,-2.5) {soliton};

    \pic[scale=0.4] (manifold1) at (0,0) {soliton};
    \pic[scale=0.4] (manifold2) at (0,2) {soliton};
    \pic[scale=0.4] (manifold3) at (0,4) {soliton};

    \draw (manifold1-c100) to [out=15, in=-15] (manifold1-b100);
    \draw (manifold1-a100) to [out=15, in=-15] (manifold2-d100);

    \draw (manifold2-c100) to [out=15, in=-15] (manifold2-b100);
    \draw (manifold2-a100) to [out=15, in=-15] (manifold3-d100);

    \draw (manifold3-c100) to [out=15, in=-15] (manifold3-b100);
    \draw (manifold3-a100) to [out=15, in=90] 
    node [midway, anchor=south west] {$\mathcal{M}$}
    (3,2)
    to [out=-90, in=-15]
    (manifold1-d100);

    \draw[<-, thick] (manifold0-upper_pt) -- 
    node (mid_arrow) [coordinate] {}
    (manifold1-lower_pt);
    \node at (mid_arrow) [anchor=east] {\small $\approx$};
    \node at (mid_arrow) [anchor=west] {\small $\iota_\omega : \mathcal{M}_\omega' \to M'$};

    \node at (manifold1-upper_pt) [] {\tiny $\mathcal{M}'_\omega \subset \mathcal{M}''_\omega$};
    \node at (manifold0-lower_pt) [] {\small $M' \subset M$};
    \node (clM'_label) at (manifold3-upper_pt) [anchor=south] {\tiny $\mathcal{M}' = \bigsqcup_{\omega \in A} \mathcal{M}_\omega'$};

    \draw[->, thick] (manifold2-left_pt) ++(-0.5,-0.5) arc 
        (300:60:0.5) node (mid_arc) [pos=0.5, anchor=east] {$\mathcal{S}$};

    \pic[scale=0.8, rotate=30] at (1,2.2) {hole};
    \pic[scale=0.5, rotate=0]  at (1.5,4) {hole};
    \pic[scale=0.5, rotate=0]  at (2, 2) {hole};
    \pic[scale=0.4, rotate=-30] at (1.8, 1) {hole};
    \pic[scale=0.5, rotate=00] at (1.2,0.25) {hole};

	\node at (manifold2-a90) [anchor=south west, xshift=1cm] {\tiny $f \equiv 100 \Gamma_0$};
    \node at (manifold2-a90) [anchor=south east] {\tiny $0 < f < 100 \Gamma_0$};
\end{tikzpicture}
	\caption{The closed manifold $\cl{M}$}
\end{figure}

\subsection{The Initial Data} \label{Subsect The Initial Data}

Here and for the remainder of the paper, we continue to let 
	$$M' = M'_{\Gamma_0} = \{ x \in M : f(x) < 100 \Gamma_0 \} \subset M,$$
	$$\cl{M}' = \bigsqcup_{\omega \in A} \cl{M}'_\omega \subset \cl{M}'' = \bigsqcup_{\omega \in A} \cl{M}_\omega'' \subset \cl{M}$$
	$$\left\{ \iota_\omega'' : \cl{M}_\omega'' \xrightarrow[]{\approx} M \right\}_{\omega \in A} , \qquad \iota_\omega = \iota_\omega'' |_{\cl{M}'_\omega} : \cl{M}'_\omega \xrightarrow[]{\approx} M',$$
	$$f : \cl{M} \to [0, 100 \Gamma_0] , \qquad \cl{S} \le \text{Diff}(\cl{M})$$
be data as specified in Subsection \ref{Subsect The Manifold}.

\begin{definition} \label{Defn G_0(t_0)}
For every $\Gamma_0 \gg 1$ sufficiently large (depending on $n, M, \ol{g}, f$) and every $0 \le t_0 < 1$,
assign a Riemannian metric $G_{\mathbf 0}( t_0) = G_{\mathbf 0}( \Gamma_0, t_0)$ to the closed manifold $\cl{M}$ as in Subsection \ref{Subsect The Manifold} above such that:
\begin{enumerate}
	\item \label{Defn G_0(t_0) Soliton Metric}
	(Soliton Metric Near ``Singular Points'')\footnote{In this item and throughout the paper, we shall slightly abuse notation and use $\phi_{t_0}^* \ol{g}$ to also mean the metric $\phi_{t_0}^* \ol{g}$ on each copy of $M$ in $\bigsqcup_{\omega \in A} M$. A similar abuse of notation applies to general tensors on $M$.}
		$$G_{\mathbf 0} (t_0) = ( 1 -t_0) \iota^* \phi_{t_0}^* \ol{g}
		\qquad \text{on } \{ x \in \cl{M} : f(x) \le \Gamma_0 \} \subset \cl{M}',$$

	\item \label{Defn G_0(t_0) Time Invt}
	(Independent of $t_0$ Away from ``Singular Points'')
	throughout the subset $\{ x \in \cl{M} :  f(x) \ge 16 \Gamma_0 \}$,
		$$G_{\mathbf 0 } ( t_0) = G_{\mathbf 0 } ( t_0 ' )	\qquad 
	\text{for all } 0 \le t_0, t_0 ' < 1,$$

	\item \label{Defn G_0(t_0) Cone Cgce}
	(Convergence to Cone in an Annular Neighborhood)
	for every $\Gamma_0$, the assignment $t_0 \mapsto G_{\mathbf 0} (t_0)$ is smooth in $t_0$ 
	and, for every $\omega \in A$, there exists an embedding $\Psi_\omega : \{ x \in \cl{M}'_\omega : \frac{\Gamma_0}{2} < f(x) < 32 \Gamma_0 \} \to \cone_1( \Sigma)$
	such that $G(t_0)$ smoothly converges to $\Psi^*_\omega g_{\cone}$ on this region as $t_0 \nearrow 1$,

	\item \label{Defn G_0(t_0) Curv Ests}
	(Curvature Estimates Away from ``Singular Points'')
	for all $m \in \mathbb{N}$, there exists $C$ depending only on $n, M, \ol{g}, f, m$ 
	such that,
	if $0 < 1 - t_0 \ll 1$ is sufficiently small,
	then the curvature $Rm = Rm[ G_{\mathbf 0 }(t_0) ]$ satisfies
		$$| \nabla^m Rm | \le  \frac{C}{ \Gamma_0^{1 + m/2} } \qquad \text{on } \{x \in \cl{M} : f(x) \ge \Gamma_0 /2 \} ,$$

	\item \label{Defn G_0(t_0) Non-Collapsed}
	(Non-Collapsed Away from ``Singular Points'')
	there exists $c > 0$ depending only on $n, M, \ol{g}, f$ such that
		$$Vol_{G_{\mathbf 0}(t_0)} ( B_{G_{\mathbf 0 }(t_0)} (x, 1) ) \ge c >0 \qquad \text{for all } x \in \{ y \in \cl{M} : f(y) > \Gamma_0 /2 \}$$
	for all $0 < 1-t_0 \ll 1$ sufficiently small,
	and 
	\item \label{Defn G_0(t_0) S Invariant}
	($\cl{S}$-Invariant)
		$$\psi^* G_{\mathbf 0}(t_0) = G_{\mathbf 0}(t_0) \qquad \text{for all } \psi \in \cl{S}.$$
\end{enumerate}
\end{definition}

Informally, these metrics spatially interpolate from the soliton metric $(1-t_0) \phi_{t_0}^* \ol{g}$ to a cone metric to a $t_0$-independent metric at scales determined by $\Gamma_0$.
Definition \ref{Defn G_0(t_0)} is summarized in Figure \ref{Figure G_0(t_0)}.
For the remainder of the paper, $G_{\mathbf 0}(t_0) = G_{\mathbf 0 }(\Gamma_0, t_0)$ always denotes an assignment of metrics as in Definition \ref{Defn G_0(t_0)}.

\tikzset{
    solitonLong/.pic ={
    \begin{scope}[transform shape, rotate=90]

    \pic[tqft/cap, name = a,
    at={(0,6)}, circle x radius = 0.5cm,
    cobordism edge/.style={draw},
    ];
    
    \pic[tqft/cylinder, name = b, genus =1,
    anchor=incoming boundary 1, at=(a-outgoing boundary 1),
    cobordism height = 0.5cm,
    circle x radius = 0.5cm,
    hole 1/.style = {draw, transform shape, rotate=-90, xshift=0.25cm, yshift=0.25cm},
    cobordism edge/.style={draw},
    ];

    \draw[rounded corners] (0.5, 3.5) 
    to [out=-90, in=117]
    (1,2) --
    node (-a0)  [at start, coordinate] {}
    node (-a/4) [pos=0.1, coordinate] {}
    node (-a/2) [pos=0.2, coordinate] {}
    node (-a1)  [pos=0.4, coordinate] {}
    node (-a16) [pos=0.6, coordinate] {}
    node (-a32) [pos=0.8, coordinate] {}
    node (-a100) [at end, coordinate] {}
    (3,-2) ;
    
    \draw[rounded corners] (-0.8,-2) 
    -- 
    node (-c100) [at start, coordinate] {}
    node (-c32)  [pos=0.2, coordinate] {}
    node (-c16)  [pos=0.4, coordinate] {}
    node (-c1)   [pos=0.6, coordinate] {}
    node (-c/2)  [pos=0.8, coordinate] {}
    node (-c0)   [at end, coordinate] {}
    (-0.2,2)
    -- (0,3) 
    -- (0.2, 2)
    -- 
    node (-b0) [at start, coordinate] {}
    node (-b/4) [pos=0.1, coordinate] {}
    node (-b/2) [pos=0.2, coordinate] {}
    node (-b1)  [pos=0.4, coordinate] {}
    node (-b16) [pos=0.6, coordinate] {}
    node (-b32) [pos=0.8, coordinate] {}
    node (-b100) [at end, coordinate] {}
    (0.8,-2);

    \draw[rounded corners] (-0.5, 3.5) 
    to [out=-90, in=63]
    (-1,2) --
    node (-d0) [at start, coordinate] {}
    node (-d/4) [pos=0.1, coordinate] {}
    node (-d/2) [pos=0.2, coordinate] {}
    node (-d1)  [pos=0.4, coordinate] {}
    node (-d16) [pos=0.6, coordinate] {}
    node (-d32) [pos=0.8, coordinate] {}
    node (-d100) [at end, coordinate] {}
    (-3,-2) ;

    \end{scope}
    }
}

\begin{figure} \label{Figure G_0(t_0)}
	\centering
\begin{tikzpicture}[transform shape, scale = 1.5]
    \pic (manifold1) {solitonLong};
    \draw[dotted] (manifold1-a/2) node (a/2) [anchor=south] {}
        to [out=240, in=120] 
        node (mid_ab/2) [midway, coordinate] {}
        node (mid2_ab/2) [pos=0.75, coordinate] {}
        (manifold1-b/2);
    \draw[dotted] (manifold1-a1) node (a1) [anchor=south] {}
        to [out=240, in=120] 
        node (mid_ab1) [midway, coordinate] {}
        (manifold1-b1);
    \draw[dotted] (manifold1-a16) node (a16) [anchor=south] {}
        to [out=240, in=120] 
        node (mid_ab16) [pos=0.4, coordinate] {}
        (manifold1-b16);
    \draw[dotted] (manifold1-a32) 
        node (a16) [anchor=south] {}
        to [out=240, in=120] 
        node (mid2_ab32) [pos=0.75, coordinate] {}
        (manifold1-b32);
    \path[] (manifold1-a100) to 
    node (mid_ab100) [pos=0.4, coordinate] {} 
    (manifold1-b100);

    \draw[thick, blue] (a1) -- 
    node [sloped, above, near end] {\tiny $(1-t_0) \iota^* \varphi_{t_0}^* \overline{g}$}
    ++(-4.2,0);

    \draw[thick, blue] (mid2_ab/2) -- 
    node [sloped, above] {\tiny convergence to cone}
    (mid2_ab32);

    \draw[thick, blue, ->] (mid_ab16) -- node [sloped, above] {\tiny independent of $t_0$} (mid_ab100);

    \path[] (manifold1-b/2) to 
    node (mid_ab/2) [coordinate] {}
    (manifold1-c/2);

    \path[] (manifold1-b100) to 
    node (mid_ab100) [coordinate] {}
    (manifold1-c100);
    
    \draw[thick, blue, ->] (mid_ab/2) -- 
    node [sloped, above] {\tiny curvature estimates}
    node [sloped, below] {\tiny \& non-collapsed}
    (mid_ab100);
   
    \draw[dotted] (manifold1-d/2) node [anchor=north] {\small $\frac12 \Gamma_0$}
        to [out=120, in=240] (manifold1-c/2);
    \draw[dotted] (manifold1-d1) node [anchor=north] {\small $\Gamma_0$}
        to [out=120, in=240] (manifold1-c1);
    \draw[dotted] (manifold1-d16) node [anchor=north] {\small $16 \Gamma_0$} 
        to [out=120, in=240] (manifold1-c16);
    \draw[dotted] (manifold1-d32) node [anchor=north] {\small $32 \Gamma_0$}
        to [out=120, in=240] 
        (manifold1-c32) node (c32) [anchor=south] {};
\end{tikzpicture}
\caption[]{The metrics $G(\Gamma_0, t_0)$ on $\cl{M}$}
\end{figure}

\begin{remark}
	Since $\cl{M}$ is closed, the metrics $G_{\mathbf 0}(t_0)$ are necessarily complete.
	Because any $\psi \in \cl{S}$ respects the idenitifcations with $M$ \eqref{eqn S respects the soliton identifications 1} and is an isometry (Definition \ref{Defn G_0(t_0)} \eqref{Defn G_0(t_0) S Invariant}), it follows that if $\psi_1 , \psi_2 \in \cl{S}$ and 
	$\psi_1 ( \cl{M}_\omega'') = \psi_2( \cl{M}_\omega'')$ for some $\omega \in A$, then $\psi_1 = \psi_2$.
	In other words, the action of $\cl{S}$ on $A$ is free.
	In particular, $\cl{S}$ must be a finite group.
\end{remark}

\begin{remark} \label{remark G_0(t_0) exists}
	It is not clear a priori that such families of metrics $G_{\mathbf0}(\Gamma_0, t_0)$ exist.
	Appendix \ref{App Rounding Out Cones} provides a detailed construction of such an assignment of metrics when $\cl{M}$ is the double of a suitable truncation of $M$ and $\cl{S} \cong \mathbb{Z}_2$ is the action that exchanges the two pieces of the double.
	The assignment of metrics $G_0(t_0)$ informally looks like the soliton metric $(1-t_0)\phi_{t_0}^* \ol{g}$ on each piece of the double but has been smoothed near the doubling boundary.
\end{remark}

For $\gamma_0 > 0$, let $\eta_{\gamma_0} : M \to [0,1]$ denote a compactly supported bump function
such that:
\begin{enumerate}
	\item $\eta_{\gamma_0} (x) = 1$ for all $x \in M$ such that $f(x) \le \frac{\gamma_0}{2 ( 1 - t_0)}$,
	\item $\supp \eta_{\gamma_0} \subset \left \{ x \in M : f(x) < \frac{\gamma_0}{1 - t_0} \right \}$, 
	\item $\overline{\{x \in M :  0 < \eta_{\gamma_0} (x) < 1 \}} \subset \left\{x \in M :  \frac{\gamma_0}{2(1 - t_0)}    < f(x) < \frac{\gamma_0}{1 - t_0}  \right\}$, and
	
	\item for all $m \in \mathbb{N}$, there exists $C$ depending only on $n, M, \ol{g},f, m$ such that, if $0 < 1 - t_0 \ll1$ is sufficiently small, then 
		$$ | \ol{\nabla}^m \eta_{\gamma_0} |_{\ol{g}} \le C.$$
\end{enumerate}
Note that $\eta_{\gamma_0}$ depends on $t_0$ as well but we will typically elide this dependence.
Such bump functions $\eta_{\gamma_0} : M \to [0,1]$ can be obtained by taking a suitable bump function $\hat \eta : ( 0 , \infty ) \to [0, 1]$ and letting 
	$$\eta_{\gamma_0} ( x ) = \hat \eta \left(  \frac{ 1 - t_0}{\gamma_0}  f(x) \right).$$
In particular, Equations \eqref{Soliton Eqns}, Corollary \ref{Cor nabla^2 f 0 Eigmode}, and the proof of Proposition \ref{Prop Eigmode Derivative Growth} adapted to the case of $\delta = 0$ imply derivative bounds for $f$,
which then yield the desired derivative bounds for $\eta_{\gamma_0}(x) = \hat \eta \left(  \frac{ 1 - t_0}{\gamma_0}  f(x) \right)$.

For all $\mathbf p = ( p_1, \dots , p_{K} ) \in \R^K$ with $| \mathbf p | \le \ol{p} (1 - t_0)^{| \lambda_*|}$, define a (possibly non-smooth) symmetric 2-tensor $G_{\mathbf p }(t_0) = G_{\mathbf p } ( \Gamma_0 , \gamma_0, t_0)$ on $\cl{M}$ by
	\begin{equation} \label{G_p(t_0) Defining Eqn}
	G_{\mathbf p} (t_0 ) \doteqdot G_{\mathbf 0 }(t_0 ) 	
	+ ( 1 - t_0) \iota^* \phi_{t_0}^* \left( \eta_{\gamma_0}     \sum_{j = 1}^{K } p_j  h_{j}   \right).
	\end{equation}
Note that a priori these $G_{\mathbf p}(t_0)$ need not be metrics on $\cl{M}$ nor even smooth.
To conclude this subsection, we show when the symmetric 2-tensors $G_{\mathbf p} (t_0)$ are (smooth) metrics on $\cl{M}$ and estimate their difference from $G_{\mathbf 0 } (t_0)$.
	
\begin{lem} \label{Lem Support of eta_gamma}
	If $\Gamma_0 \gg 1$ is sufficiently large (depending on $n, M, \ol{g}, f$),
	$0 < \gamma_0 \le 1$, 
	and $0 \le t_0 < 1$,
	then
		$$\supp \phi_{t_0}^* \eta_{\gamma_0} \subset \{x \in M :  f(x) < \Gamma_0 \}.$$
\end{lem}	
\begin{proof}
	By Lemma \ref{Lem Flow Est 1}, 
	there exists a positive constant $C = C(n, M, \ol{g}, f) > 0$ such that 
		$$\max \left\{ 0, \phi_t^* f - \frac{ C}{ \phi_t^* f} \right\} \le ( 1 - t) \partial_t ( \phi_t^* f ) 
		\qquad \text{for all } t < 1.$$
	In terms of the reparametrized time coordinate $\tau = - \ln ( 1 - t)$, this estimate reads
		$$\max \left\{ 0, \phi_{t (\tau) } ^* f - \frac{ C}{ \phi_{t( \tau) }^* f} \right\} \le  \partial_\tau ( \phi_{t ( \tau) }^* f ) 
		\qquad		\text{for all } \tau \in \R $$
	where $t(\tau) = 1 - e^{- \tau}$.
	Integrating with respect to $\tau$ and using that $\phi_0 = Id_M$, it follows that 
		$$f(x) \le \phi_{t( \tau)}^* f(x) 
		\qquad \text{ for all } (x, \tau) \in M \times \R, \text{ and} $$
		$$\sqrt{ C + (  f(x)^2 - C ) e^{2 \tau} } \le \phi_\tau^* f (x)
		\qquad \text{ for all } (x, \tau) \in \left\{ x \in M : f(x) \ge \sqrt{C} \right\} \times \R.$$
	
	Assume $\Gamma_0 > \sqrt{ C + 1 }$.
	Observe
	\begin{gather*} \begin{aligned}
		\supp \phi_{t_0}^* \eta_{\gamma_0} 	
		& = \phi_{t_0}^{-1} ( \supp \eta_{\gamma_0} ) 	\\
		& \subset \phi_{t_0}^{-1} \{ x \in M : f \le \gamma_0 e^{\tau_0} \} \\
		& = \{ x \in M : f \circ  \phi_{t_0}(x)  \le \gamma_0 e^{\tau_0} \}	\\
		& = \{ x \in M : [\phi_{t_0}^* f ]  (x)  \le \gamma_0 e^{\tau_0} \}	.
	\end{aligned} \end{gather*}
	Let $x \in M$ be such that $[\phi_{t_0}^* f ]  (x)  \le \gamma_0 e^{\tau_0}$.
	If $f(x) \le \sqrt{C}$, then $f(x) < \Gamma_0$.
	If $f(x) \ge \sqrt{C}$, then 
		$$\sqrt{ C + (  f(x)^2 - C ) e^{2 \tau_0} } \le \phi_{t_0}^* f(x) \le \gamma_0 e^{\tau_0}$$
	which implies	
		$$  f(x) \le \sqrt{ C + \gamma_0^2 - C e^{-2 \tau_0} } \le \sqrt{ C + 1 - C e^{-2 \tau_0} }  \le \sqrt{ C + 1} < \Gamma_0.$$
	In either case, it follows that $f(x) < \Gamma_0$.
\end{proof}
\begin{remark}
	In practice, the parameters $\Gamma_0, \gamma_0, t_0$ will be taken so that Lemma \ref{Lem Support of eta_gamma} implies
		$$\supp \phi_{t_0}^* \eta_{\gamma_0} \subset \{ f < \Gamma_0 / 2 \}.$$
	We make this condition explicit in Assumption \ref{Assume Smooth Metric and eta_gamma Supp}.
\end{remark}

Before continuing, we make a few notational simplifications that will be used throughout the rest of the paper.

	We use the notation ``$A \lesssim B$" to mean that there exists a constant $C$ such that $A \le CB$.
	We shall write ``$A \lesssim_{a,b} B$" when the constant $C$ depends on $a$ and $b$.
	
	We say ``$A \sim B$" when $A \lesssim B$ and $B \lesssim A$.
	Similarly, ``$A \sim_{a,b} B$" means $A \lesssim_{a,b} B$ and $B \lesssim_{a,b} A$.

	We shall write ``$\{ f < A \}$" for example to mean either 
	$\{ x \in M : f(x) < A \}$ or $\{ x \in \cl{M} : f(x) < A \}$.
	It will often be clear from context whether we refer to the former or the latter, but we will write $\{ f < A \} \subset M$ for example for additional specificity.

\begin{prop} \label{Prop Difference of G_p from G_0}
	If $\Gamma_0 \gg 1$ is sufficiently large (depending on $n, M, \ol{g}, f$),
	$0 < \gamma_0 \le 1$, 
	and $0 \le t_0 < 1$,	
	then 
		$$| G_{\mathbf p} (t_0) - G_{\mathbf 0} (t_0) |_{G_{\mathbf 0} (t_0)} \lesssim_{n, M, \ol{g}, f, \lambda_*}  \ol{p}  \gamma_0^{|\lambda_*|}$$
	for all $| \mathbf p | \le \ol{p} ( 1 - t_0)^{|\lambda_*|} $.
\end{prop}
\begin{proof}
	By Lemma \ref{Lem Support of eta_gamma}, we may assume $\supp \phi_{t_0}^* \eta_{\gamma_0} \subset \{ f \le \Gamma_0 \} $.
	It then follows that $G_{\mathbf p}(t_0) - G_{\mathbf 0 }(t_0)$ is supported in $\{ f \le \Gamma_0 \}$.
	Moreover, $G_{\mathbf 0 }(t_0) = (1 - t_0) \iota^* \phi_{t_0}^* \ol{g}$ in this region by Definition \ref{Defn G_0(t_0)} \eqref{Defn G_0(t_0) Soliton Metric}.
	Set $\tau_0 \doteqdot - \ln ( 1 - t_0) $.
	It follows that
	\begin{align*}
		 & | G_{\mathbf p} (t_0) - G_{\mathbf 0} (t_0) |_{G_{\mathbf 0} (t_0)}	\\
		={}& \left| ( 1- t_0) \iota^*  \phi_{t_0}^* \left( \eta_{\gamma_0}     \sum_{j = 1}^{K } p_j  h_{j}  
		  \right) \right|_{(1-t_0) \iota^* \phi_t^* \ol{g}}	\\
		={}& \left| \eta_{\gamma_0}     \sum_{j = 1}^{K } p_j  h_{j}  
		   \right|_{\ol{g}}	\\
		\le{}& \eta_{\gamma_0}  \sum_{j = 1}^{K } \ol{p} e^{\lambda_* \tau_0} | h_j |_{\ol{g}} .
	\end{align*}		
	By the eigenmode growth condition (Definition \ref{Defn Eigmode Growth Assumption} and Proposition \ref{Prop Asymp Conical Implies Eigmode Growth Assumption}),
	this quantity may then be bounded above by	
\begin{align*}		
		&\quad  \eta_{\gamma_0}  \sum_{j = 1}^{K } \ol{p} e^{\lambda_* \tau_0} C(n, M, \ol{g}, f, j, \delta) f^{\max \{ - \lambda_j , 0 \} + \delta}		\\
		&\lesssim_{n, M, \ol{g}, f, \lambda_*, \delta} \eta_{\gamma_0} f^{- \lambda_* }   \ol{p}  e^{\lambda_* \tau_0} 
			\sum_{j = 1}^{K } f^{\max \{ - \lambda_j , 0 \} + \delta + \lambda_*} 
		\\
		&\le \ol{p} \gamma_0^{| \lambda_* |} \sum_{j = 1}^{K } f^{\max \{ \lambda_* - \lambda_j + \delta, \lambda_* + \delta \} } .
	\end{align*}
	Since $\lambda_* < 0$ and $\lambda_* < \lambda_K \le \dots \le \lambda_1$,
	there exists $\delta = \delta(n, M, \ol{g}, f, \lambda_*) > 0$ such that
		$$ \lambda_* - \lambda_j + \delta < 0 \text{ and }  \lambda_* + \delta  < 0
		\qquad \text{for all } 1 \le j \le K.$$
	Together with the fact that $f > 0$ by Lemma \ref{Lem Nonflat and f>0}, it follows that
		$$\ol{p} \gamma_0^{| \lambda_* |} \sum_{j = 1}^{K } f^{\max \{ \lambda_* - \lambda_j + \delta, \lambda_* + \delta \} } 
		\lesssim_{n, M, \ol{g}, f,  \lambda_*} \ol{p} \gamma_0^{| \lambda_* |}.$$
\end{proof}

\begin{cor} \label{Cor Init Data is Metric}
	If 
	$\Gamma_0 \gg 1$ is sufficiently large (depending on $n, M, \ol{g}, f$),
	$0 < \gamma_0  \ll 1$ is sufficiently small (depending on $n, M, \ol{g}, f, \lambda_*$),
	$0 < \ol{p} \le 1$, and
	$0 \le t_0 < 1$, 
	then $G_{\mathbf p}(t_0)$ is a smooth $\cl{S}$-invariant metric on the closed manifold $\cl{M}$ for all $|\mathbf p | \le \ol{p} ( 1 - t_0)^{| \lambda_*|} $.
\end{cor}
\begin{proof}
	By Lemma \ref{Lem Support of eta_gamma}, we may assume $\supp \phi_{t_0}^* \eta_{\gamma_0} \subset \{ f \le \Gamma_0 \} \subset M'$.
	In particular, $\iota^* \phi_{t_0}^* \eta_{\gamma_0}^*$ has support on the interior of $\cl{M}'$.
	Hence, $G_{\mathbf p}(t_0)$ is a smooth symmetric 2-tensor since the eigenmodes $h_j$ are smooth as well.
	Proposition \ref{Prop Difference of G_p from G_0} implies that $G_{\mathbf  p}(t_0)$ is positive definite for all $| \mathbf p | \le \ol{p} ( 1 - t_0)^{| \lambda_* |} $, $\ol{p} \le 1$, $0 < \gamma_0 \ll 1$ sufficiently small, and $0 \le t_0 < 1$.
	
	Finally, for any $\psi \in \cl{S}$,
	\begin{multline*}
		\psi^* G_{\mathbf p}(t_0) 
		= \psi^* G_{\mathbf 0}(t_0) + (1 -t_0) \psi^* \iota^* \phi_{t_0}^* \left( \eta_{\gamma_0} \sum_{j=1}^K p_j h_j \right) \\
		=  G_{\mathbf 0}(t_0) + (1 -t_0)  \iota^* \phi_{t_0}^* \left( \eta_{\gamma_0} \sum_{j=1}^K p_j h_j \right)
		= G_{\mathbf p}(t_0)
	\end{multline*}
	by Definition \ref{Defn G_0(t_0)} \eqref{Defn G_0(t_0) S Invariant} and Equation \eqref{eqn S respects the soliton identifications 1}.
\end{proof}

\begin{assumption} \label{Assume Smooth Metric and eta_gamma Supp}
	Henceforth, we will implicitly assume that 
	$\Gamma_0 \gg 1$ is sufficiently large (depending on $n, M, \ol{g}, f$), 
	$0 < \gamma_0 \ll1$ is sufficiently small (depending on $n, M, \ol{g}, f, \lambda_*)$,
	$0 < \ol{p} \le 1$, and 
	$0 \le t_0 < 1$
	so that Corollary \ref{Cor Init Data is Metric} implies $G_{\mathbf p} (t_0)$ is a smooth $\cl{S}$-invariant metric on $\cl{M}$ for all $| \mathbf p | \le \ol{p} ( 1 - t_0)^{| \lambda_* |}$ and
	Lemma \ref{Lem Support of eta_gamma} implies $\supp \phi_{t_0}^* \eta_{\gamma_0} \subset \{ f < \Gamma_0 / 2 \}$.
\end{assumption}

\begin{remark} \label{Rmk Initially Non-Kahler}
	By perturbing the initial metrics $G_{\mathbf 0}(\Gamma_0, t_0)$ on $\cl{M} \setminus \ol{\cl{M}'} \neq \emptyset$
	and averaging over $\cl{S}$ via $G \mapsto \frac1{\# \cl{S}} \sum_{\psi \in \cl{S}} \psi^* G$,
	we may additionally assume without loss of generality that the metrics $G_{\mathbf p}(t_0)$ are non-K{\"a}hler.
\end{remark}

\subsection{The Flow} \label{Subsect The Flow}

We let $G_{\mathbf p}(t) = G_{\mathbf p} ( t; \Gamma_0, \gamma_0, t_0)$ denote the maximal solution to the Ricci flow
	$$\partial_t G_{\mathbf p} (t) = -2 Rc [ G_{\mathbf p} (t) ]$$
on $\cl{M}$ with initial data $G_{\mathbf p}(t_0)$ given by \eqref{G_p(t_0) Defining Eqn} at time $t = t_0$.
Let $T(\mathbf p ) = T(\mathbf p; \Gamma_0, \gamma_0, t_0)$ denote the minimum of $1$ and the maximal existence time of $G_{\mathbf p }( t ; \Gamma_0, \gamma_0 , t_0)$.

Let $\eta_{\Gamma_0} : M \to [0,1]$ be a compactly supported bump function such that:
\begin{enumerate}
	\item $\eta_{\Gamma_0} (x) = 1$ for all $x \in M$ such that $f(x) \le \frac{1}{2} \Gamma_0$,
	\item $\text{supp}(\eta_{\Gamma_0}) \subset \{ f < \Gamma_0  \}$,
	\item $\overline{\{ x \in M : 0 < \eta_{\Gamma_0}(x) < 1 \}} \subset \{ \frac{4}{6} \Gamma_0 < f < \frac{5}{6} \Gamma_0 \}	 \subset \{ \frac{1}{2} \Gamma_0  < f < \Gamma_0   \}$, 
	\item for all $m \in \mathbb{N}$, $| {}^{(1 - t) \phi_t^* \ol{g} } \nabla^m \eta_{\Gamma_0} |_{(1-t) \phi_t^* \ol{g} } \lesssim_{n, M, \ol{g}, f, m} 1 $,
	and
	\item $| {}^{(1 - t) \phi_t^* \ol{g} } \nabla \eta_{\Gamma_0} |_{(1-t) \phi_t^* \ol{g} } \lesssim_{n, M, \ol{g}, f} \Gamma_0^{-1/2} $.
\end{enumerate}
Such bump functions can be obtained by taking a suitable bump function $\hat \eta : (0, \infty) \to [0,1]$ and letting
	$$\eta_{\Gamma_0} (x) \doteqdot \hat \eta \left( \frac{ f(x)}{ \Gamma_0 } \right).$$
The derivative bounds (4), (5) follow from the fact $(1 - t) \phi_t^* \ol{g}$ converges to (a pullback of) the cone metric $g_{\cl{C}}$ on $\{ \frac{1}{2} \Gamma_0 < f < \Gamma_0 \}$ as $t \nearrow 1$ so long as $\Gamma_0 \gg 1$ is sufficiently large depending on $n, M, \ol{g}, f$ (see Proposition \ref{Prop 2.1 in KW15}).
The coarser derivative bounds in (4) suffice for most applications.
The finer derivative bounds (5) yield Lemma \ref{Lem difference from RF} in Section \ref{Sect Prelim Ests}, which in turn simplifies estimates later in the paper.

Recall the diffeomorphisms $\iota_\omega : \cl{M}'_\omega \to M' = \{ f < 100 \Gamma_0 \} \subset M$ for every $\omega \in A$.
Thus, the push-forwards
	$$(\iota_\omega)_* \left( G_{\mathbf p} (t) |_{\cl{M}'_\omega} \right) = ( \iota_\omega^{-1} )^*  \left( G_{\mathbf p} (t) |_{\cl{M}'_\omega} \right)$$
define Riemannian metrics on $M' \subset M$ for each $\omega \in A$.
Since $G_{\mathbf p }(t_0)$ is $\cl{S}$-invariant (Definition \ref{Defn G_0(t_0)} \eqref{Defn G_0(t_0) S Invariant}), the metrics $G_{\mathbf p } (t)$ are also $\cl{S}$-invariant on $\cl{M}$ for all $t \in [t_0, T( \mathbf p) )$.
Because $\cl{S}$ acts transitively on $A$ and respects the identifications $\iota_\omega $ \eqref{eqn S respects the soliton identifications 2}, it follows that in fact 
	$$(\iota_\omega)_* \left( G_{\mathbf p} (t) |_{\cl{M}'_\omega} \right) = ( \iota_{\tilde \omega})_* \left( G_{\mathbf p} (t) |_{\cl{M}'_{\tilde \omega}} \right) \qquad \forall \omega, \tilde \omega \in A.$$
Therefore, we may identify the metric $\iota_* G_{\mathbf p}(t) = \iota_* G_{\mathbf p}(t) |_{\cl{M}'}$ on $\bigsqcup_{\omega \in A} M'$ with a metric on $M'$, which will still be denoted as $\iota_* G_{\mathbf p}(t)$.
We extend these metrics $\iota_* G_{\mathbf p }(t)$ on $M' = \{ x \in M : f(x) < 100 \Gamma_0 \}$ to metrics $\acute G_{\mathbf p}(t)$ on all of $M$ by setting
\begin{equation}
	\acute G_{\mathbf p}(t) 
	\doteqdot \eta_{\Gamma_0} \iota_* G_{\mathbf p}(t) + (1- \eta_{\Gamma_0}) (1 - t) \phi_t^* \ol{g} \qquad \text{for all } t \in [t_0, T(\mathbf p) ).
\end{equation}
	
\begin{remark}	 \label{Rem acute G Observations}
We make some elementary observations before continuing:
\begin{itemize}
	\item At $t = t_0$,
	\begin{gather*} \begin{aligned}
		\acute G_{\mathbf p}(t_0) 
		&= \eta_{\Gamma_0} \iota_* G_{\mathbf 0}(t_0) 
		+ (1-\eta_{\Gamma_0} ) ( 1 - t_0)  \phi_{t_0}^* \ol{g} 	
		+ \eta_{\Gamma_0}  ( 1 - t_0)  \phi_{t_0}^* \left( \eta_{\gamma_0}     \sum_{j = 1}^{K } p_j  h_{j}   \right)
	\end{aligned} \end{gather*}
	Recall $\supp \phi_{t_0}^* \eta_{\gamma_0} \subset \{ f < \Gamma_0 / 2 \}$ by Assumption \ref{Assume Smooth Metric and eta_gamma Supp} and $\iota_* G_{\mathbf 0 } (t_0) = ( 1 - t_0)  \phi_{t_0}^* \ol{g}$ on $\{ f \le \Gamma_0 \} \supset \supp \eta_{\Gamma_0}$ by Definition \ref{Defn G_0(t_0)}.
	Therefore,
		$$\acute G_{\mathbf p}(t_0) 
			= (1 - t_0) \phi_{t_0}^* \ol{g} 
			+  ( 1 - t_0)  \phi_{t_0}^* \left( \eta_{\gamma_0}    \sum_{j = 1}^{K } p_j  h_{j}  \right)$$
		throughout $M$.
			
	\item On the set $\{ x \in M : \eta_{\Gamma_0}(x)  =  1 \}$,
		$$\acute G_{\mathbf p}(t) = \iota_* G_{\mathbf p}(t) 	\qquad
		\text{for all } t \in [t_0, T( \mathbf p ) ).$$
	
	\item On the set $\{ x \in M : \eta_{\Gamma_0}(x) =  0\}$,
		$$\acute G_{\mathbf p}(t) = ( 1 - t) \phi_{t}^* \ol{g}\qquad
		\text{for all } t \in [t_0, T( \mathbf p ) ).$$
		
	\item $\acute G_{\mathbf p}(t)$ does \emph{not} solve Ricci flow.
	However, on open subsets of $\eta_{\Gamma_0}^{-1} ( \{ 1 \} )$ and $\eta_{\Gamma_0}^{-1} ( \{ 0 \} )$, $ \acute G_{\mathbf p}(t)$ \emph{does} solve Ricci flow since it's equal to the Ricci flow solution $\iota_* G_{\mathbf p}(t)$ or $ (1 - t) \phi_t^* \ol{g}$.
	
	In general,
	\begin{equation} \label{eqn almost RF error term}
		\partial_t \acute G_{\mathbf p} + 2 Rc [ \acute G_{\mathbf p}(t)] 
		= 2 Rc [ \acute G_{\mathbf p}(t)] -2 \eta_{\Gamma_0} Rc[ \iota_* G_{\mathbf p}(t) ] - 2 (1-\eta_{\Gamma_0}) Rc [ \phi_t^* \ol{g} ] ,
	\end{equation}
	which is supported on the closure of $\{ x \in M : 0 < \eta_{\Gamma_0}(x) < 1 \}$.
\end{itemize}
\end{remark}

\begin{definition} \label{Diffeo Defns}
	Recall from Definition \ref{Defn phi} that, for $t \in (-\infty, 1)$, $ \phi( \cdot, t) = \phi_t : M \to M$ denote the soliton diffeomorphisms given as the solution to
	\begin{equation} \tag{\ref{phi Evol Eqn}}
		\partial_t \phi_t 
		= \frac{1}{1 - t} \ol{\nabla} f \circ \phi_t 
		\qquad \text{ with initial condition } \phi_{0} = Id_M.
	\end{equation}
	Define one-parameter families of functions
	\begin{equation} \label{Phi Defn}
		\Phi_t, \tl{\Phi}_t: M \to M \qquad \text{ by } \Phi_t = \phi_t \circ \tilde{\Phi}_t
	\end{equation}
	where $\tl{\Phi}_t : M \to M$ solves the time-dependent harmonic map heat flow
	\begin{equation} \label{tlPhi Evol Eqn}
		\partial_t \tl{\Phi}_t = \Delta_{\acute{G}_{\mathbf{p}}(t), (1-t) \phi_t^* \ol{g} } \tl{\Phi}_t
		\qquad \text{ with initial condition }
		\tl{\Phi}_{t_0} = Id_M.
	\end{equation}
	Let $T_\Phi( \mathbf p) \in [t_0,  T( \mathbf p)]$ denote the maximal time for which $\tl{\Phi}_t$, or equivalently $\Phi_t$, exists, is unique, and is a diffeomorphism.
\end{definition}

\begin{remark}
	Given that $\tl{\Phi}_t$ is defined as a harmonic map heat flow on \emph{non-compact} manifolds, there is some subtlety to establishing its short-time existence and uniqueness.
	Appendix \ref{Appdix HarMapFlow} addresses such well-posedness issues in detail,
	and Subsection \ref{Subsect Must Exit Through clB} applies the Appendix \ref{Appdix HarMapFlow} results to our particular setting to prove Theorem \ref{Main Thm}.
\end{remark}

\begin{lem}
	For all $t \in (t_0, T_\Phi(\mathbf p) )$, 
	$\Phi_t : M \to M$ satisfies
	\begin{equation} \label{Phi Evol Eqn}
		 \partial_t \Phi_t = \Delta_{   \acute G_{\mathbf p}(t)  ,  \ol{g} }  \Phi_t+ \frac{1}{1 - t} \ol{\nabla} f \circ \Phi_t
		\qquad \text{ with } \quad \Phi_{t_0} = \phi_{t_0} : M \to M.
	\end{equation}
\end{lem}
\begin{proof}
	The proof uses the following general facts about the Laplacian that hold for all real numbers $\lambda > 0$, all  Riemannian metrics $g, \ol{g}$ on $M$, all smooth functions $f : M \to \mathbb{R}$, and all diffeomorphisms $\psi : M \to M$ :
	\begin{gather}
		 \Delta_{g, \lambda \ol{g} } f = \Delta_{ g, \ol{g} } f = \lambda \Delta_{\lambda g, \ol{g}} f,
		 \tag{1} \label{Lapl Scale Codomain}	\\
		( \Delta_{g, \ol{g} } f) \circ \psi = \Delta_{\psi^* g, \ol{g} } ( f \circ \psi), \text{ and}
		\tag{2} \label{Lapl Precompose} 	\\
		\psi_* \Delta_{g, \ol{g} } Id_M = \Delta_{( \psi^{-1})^* g, ( \psi^{-1})^* \ol{g} } Id_M.
		\tag{3}  \label{Pushforward Lapl Id}
	\end{gather}

	It follows that, with $\acute G = \acute G_{\mathbf p }(t)$,
	\begin{align*}
		& \partial_t \Phi_t 	\\
		={}& \partial_t ( \phi_t \circ \tl{\Phi}_t )
		&& (\text{equation \eqref{Phi Defn}})	\\
		={}& \left\{ ( \phi_t)_* [ \partial_t \tl{\Phi}_t \circ \tl{\Phi}_t^{-1} ] \right\} \circ \Phi_t 
		+ ( \partial_t \phi_t) \circ \tl{\Phi}_t
		&& (\text{chain rule}) \\
		={}& \left\{ ( \phi_t)_* \left[ \left( \Delta_{\acute{G}, (1-t) \phi_t^* \ol{g} } \tl{\Phi}_t \right)  \circ \tl{\Phi}_t^{-1} \right] \right\} \circ \Phi_t 
		&& (\text{equation \eqref{tlPhi Evol Eqn}} ) \\
		& \quad
		+ \frac{1}{1 - t} \ol{\nabla} f \circ \phi_t \circ \tl{\Phi}_t
		&&( \text{equation \eqref{phi Evol Eqn}} ) \\
		={}& \left\{ ( \phi_t)_* \left[ \left( \Delta_{\acute{G}, \phi_t^* \ol{g} } \tl{\Phi}_t \right)  \circ \tl{\Phi}_t^{-1} \right] \right\} \circ \Phi_t 
		&& (\text{item \eqref{Lapl Scale Codomain}}) \\
		& \quad
		+ \frac{1}{1 - t} \ol{\nabla} f \circ \Phi_t 
		&& ( \text{equation \eqref{Phi Defn}}) 	\\
		={}& \left\{ ( \phi_t)_* \left[  \Delta_{(\tl{\Phi}_t^{-1})^* \acute{G}, \phi_t^* \ol{g} } Id_M  \right] \right\} \circ \Phi_t 
		+ \frac{1}{1 - t} \ol{\nabla} f \circ \Phi_t 
		&& (\text{item \eqref{Lapl Precompose}}) \\
		={}& \left\{    \Delta_{( \phi_t^{-1})^* (\tl{\Phi}_t^{-1})^* \acute{G},  \ol{g} } Id_M   \right\} \circ \Phi_t 
		+ \frac{1}{1 - t} \ol{\nabla} f \circ \Phi_t 
		&& (\text{item \eqref{Pushforward Lapl Id}}) 	\\
		={}& \left\{    \Delta_{( \Phi_t^{-1})^*  \acute{G},  \ol{g} } Id_M   \right\} \circ \Phi_t 
		+ \frac{1}{1 - t} \ol{\nabla} f \circ \Phi_t 
		&& (\text{equation \eqref{Phi Defn}}) 	\\
		={}&    \Delta_{  \acute{G},  \ol{g} } \Phi_t
		+ \frac{1}{1 - t} \ol{\nabla} f \circ \Phi_t 
		&& (\text{item \eqref{Lapl Precompose}}) 	.
	\end{align*}
\end{proof}

For all $t \in [t_0, T_\Phi( \mathbf p) )$, define rescaled and reparametrized metrics $g_{\mathbf p} ( t) = g_{\mathbf p}(t; \Gamma_0, \gamma_0, t_0)$ on $M$ by
	\begin{equation} \label{g Defn}
		g_{\mathbf p}(t ) = \frac{1}{1 - t} \left(\Phi_t^{-1} \right)^* \acute G_{\mathbf p} (t).
	\end{equation}

\begin{lem} \label{Lem g Evol Eqn}
	For all $ t \in ( t_0, T_{\Phi} ( \mathbf p, t_0 ) )$, $g = g_{\mathbf p}(t; t_0)$ satisfies
	\begin{equation} \label{g Evol Eqn}
		 (1 - t) \partial_t g	
		= -2 Rc [g] 
		+ \cl{L}_{B_{\ol{g}}(g)} g - \cl{L}_{\ol{\nabla} f} g + g 	
		+ \left(\Phi_t^{-1} \right)^* \left\{ \partial_t \acute G_{\mathbf p} + 2 Rc[ \acute G_{\mathbf p} ] \right\} ,
	\end{equation}
	where $B_{\ol{g}} (g)$ is the vector field on $M$ from DeTurck's trick defined by
	\begin{equation} \label{DeTurck Vector Field Eqn}	
		B_{\ol{g}} (g) = - \Delta_{g, \ol{g} } Id_M
	\end{equation}
	or, in local coordinates, 
		$$B_{\ol{g}} (g)^k = g^{ij} \left( \Gamma(g)_{ij}^k - \Gamma( \ol{g} )_{ij}^k \right).$$
\end{lem}

\begin{remark}
	The vector field $B_{\ol{g}} (g)$ is also given by the equation
		$$\ol{g} \left\langle B_{\ol{g}} (g) , \cdot \right\rangle
		= - \div_{g} \ol{g} + \frac{1}{2} d( \tr_{g} \ol{g}).$$
\end{remark}

\begin{proof}
	Differentiating equation \eqref{g Defn} with respect to $t$, it follows that
	\begin{align*}
		& \partial_t g	\\
		={}& \frac{1}{ (1-t)^2} ( \Phi_t^{-1} )^* \acute{G}_{\mathbf p}
		+ \frac{1}{ 1 - t} ( \Phi_t^{-1})^* \cl{L}_{  \partial_t \Phi_t^{-1} \circ \Phi_t } \acute{G}_{\mathbf p} 
		+ \frac{1}{ 1 - t} ( \Phi_t^{-1} )^* \partial_t \acute{G}_{\mathbf p}	\\
		&  - \frac{2}{ 1 - t} ( \Phi_t^{-1} )^* Rc[ \acute{G}_{\mathbf p} ]
		+ \frac{2}{ 1 - t} ( \Phi_t^{-1} )^* Rc[ \acute{G}_{\mathbf p} ] \\
		={}& \frac{1}{1-t} g 
		+ \frac{1}{ 1 - t}  \cl{L}_{ ( \Phi_t)_* [ \partial_t \Phi_t^{-1} \circ \Phi_t ] } ( \Phi_t^{-1})^* \acute{G}_{\mathbf p} 
		- \frac{2}{ 1 - t}  Rc[ ( \Phi_t^{-1} )^* \acute{G}_{\mathbf p} ]	\\
		&  + \frac{1}{ 1 - t} (\Phi_t^{-1})^* \left\{ \partial_t \acute{G}_{\mathbf p} + 2 Rc[ \acute{G}_{\mathbf p} ] \right\} \\
		={}& \frac{1}{1-t} g 
		+ \cl{L}_{ ( \Phi_t)_* [ \partial_t \Phi_t^{-1} \circ \Phi_t ] } g
		- \frac{2}{ 1 - t}  Rc[ g ]	\\
		&  + \frac{1}{ 1 - t} (\Phi_t^{-1})^* \left\{ \partial_t \acute{G}_{\mathbf p} + 2 Rc[ \acute{G}_{\mathbf p} ] \right\} .
	\end{align*}
	
	Next, note that differentiating both sides of $\Phi_t \circ \Phi_t^{-1} = Id_M$ with respect to $t$ implies
		$$( \Phi_t)_* [ \partial_t \Phi_t^{-1} \circ \Phi_t ] = - ( \partial_t  \Phi_t ) \circ \Phi_t^{-1}.$$
	Hence,
	\begin{align*}
		& ( \Phi_t)_* [ \partial_t \Phi_t^{-1} \circ \Phi_t ]	\\
		={}& - ( \partial_t  \Phi_t ) \circ \Phi_t^{-1}	\\
		={}& - ( \Delta_{\acute{G}_{\mathbf p} , \ol{g} } \Phi_t ) \circ \Phi_t^{-1} - \frac{1}{1-t} \ol{\nabla} f 
		&& (\text{equation \eqref{Phi Evol Eqn}}) \\
		={}& -  \Delta_{ (\Phi_t^{-1})^* \acute{G}_{\mathbf p} , \ol{g} } Id_M   - \frac{1}{1-t} \ol{\nabla} f 	
		&& (\text{item \eqref{Lapl Precompose} above}) \\
		={}& - \frac{1}{1 - t} \Delta_{ \frac{1}{1-t} (\Phi_t^{-1})^* \acute{G}_{\mathbf p} , \ol{g} } Id_M   - \frac{1}{1-t} \ol{\nabla} f 	
		&&  (\text{item \eqref{Lapl Scale Codomain} above})	\\
		={}& - \frac{1}{1 - t} \Delta_{ g , \ol{g} } Id_M   - \frac{1}{1-t} \ol{\nabla} f 	
		&& ( \text{equation \eqref{g Defn}} )\\
		={}& \frac{1}{1-t} B_{\ol{g}} (g) - \frac{1}{1 - t} \ol{\nabla} f .
	\end{align*}

	Therefore,
	\begin{align*}
		( 1 - t) \partial_t g 	
		={}&  g 
		+ (1-t) \cl{L}_{ ( \Phi_t)_* [ \partial_t \Phi_t^{-1} \circ \Phi_t ] } g
		- 2  Rc[ g ]	
		+  (\Phi_t^{-1})^* \left\{ \partial_t \acute{G}_{\mathbf p} + 2 Rc[ \acute{G}_{\mathbf p} ] \right\} \\
		={}& - 2  Rc[ g ]	
		+  \cl{L}_{B_{\ol{g}} (g)}  g - \cl{L}_{ \ol{\nabla} f} g
		+ g 
		+  (\Phi_t^{-1})^* \left\{ \partial_t \acute{G}_{\mathbf p} + 2 Rc[ \acute{G}_{\mathbf p} ] \right\} .
	\end{align*}
\end{proof}

\begin{cor} \label{Cor h(t) Evol Eqn}
	For all $ t \in ( t_0, T_{\Phi} ( \mathbf p ) )$, 
	the tensor
		$$h = h_{\mathbf p} (t; \Gamma_0, \gamma_0, t_0) \doteqdot g_{\mathbf p} (t ; \Gamma_0, \gamma_0, t_0) - \ol{g}$$
	satisfies an evolution equation given in local coordinates by
	\begin{gather} \label{h(t) Evol Eqn}   \begin{aligned}
		 ( 1 - t) \partial_t h_{ij}	
		={}& g^{ab} \ol{\nabla}_a \ol{\nabla}_b h_{ij} - \ol{\nabla}_{\ol{\nabla} f} h
		+ 2 \ol{R} \indices{_i^k_j^l} h_{kl}	\\
		&  + \ol{R} \indices{_{ja}^p_b} \left( \ol{g}_{ip} \tl{h}^{ab}  + \hat{h}^{ab} h_{ip} \right) 
		+ \ol{R} \indices{_{ia}^p_b} \left( \ol{g}_{jp} \tl{h}^{ab}  + \hat{h}^{ab} h_{jp} \right) 	\\
		& + g^{ab} g^{pq} ( \ol{\nabla} h * \ol{\nabla} h )_{abpqij}		
		+ \left(\Phi_t^{-1} \right)^* \left\{ \partial_t \acute G_{\mathbf p} + 2 Rc[ \acute G_{\mathbf p} ] \right\}	,
	\end{aligned} 	\end{gather}
	where the tensors $\hat{h}, \tl{h}$ are defined by
	\begin{equation} \label{Inverse h Tensors}
		(g_{\mathbf p})^{ab} = \ol{g}^{ab} - \hat{h}^{ab}		\quad \text{and} \quad 
		\hat{h}^{ab} = \ol{g}^{ak} \ol{g}^{bl} h_{kl} + \tl{h}^{ab},
	\end{equation} 
	$( \ol{\nabla} h * \ol{\nabla} h )_{abpqij}$ here is given by
	\begin{align*}
		2 ( \ol{\nabla} h * \ol{\nabla} h )_{abpqij}
		={}& ( \ol{\nabla}_i h_{pa} ) ( \ol{\nabla}_j h_qb)	
		+ 2 (\ol{\nabla}_a h_{jp})(\ol{\nabla}_q h_{ib} )
		- 2 (\ol{\nabla}_a h_{jp} )( \ol{\nabla}_b h_{iq} )	\\
		& 
		- 2 ( \ol{\nabla}_j h_{pa} ) ( \ol{\nabla}_b h_{iq} )
		- 2 ( \ol{\nabla}_i h_{pa} ) ( \ol{\nabla}_b h_{jq} ),
	\end{align*}
	and the raising and lowering of indices on the curvature terms in equation \eqref{h(t) Evol Eqn} is done with respect to $\ol{g}$.
\end{cor}
\begin{proof}
	Throughout this proof, any raising or lowering of indices is done with respect to $\ol{g}$.
	However, we shall only raise or lower indices on terms involving the curvature $\ol{Rm}$.
	Throughout the proof, we shall also write $g$ instead of $g_{\mathbf p}(t; \Gamma_0, \gamma_0,  t_0)$ to simplify the notation.
	
	By \cite[Lemma 2.1]{Shi89},
	\begin{gather*} \begin{aligned}
		& \quad -2 Rc[g] + \cl{L}_{B_{\ol{g}} ( g) } g	\\
		&= g^{ab} \ol{\nabla}_a \ol{\nabla}_b g_{ij} 
		- g^{ab} g_{ip} \ol{R} \indices{_{ja}^p_b} - g^{ab} g_{jp} \ol{R} \indices{_{ia}^p_b}
		+ g^{ab} g^{pq} ( \ol{\nabla} g * \ol{\nabla} g )	\\
		& = g^{ab} \ol{\nabla}_a \ol{\nabla}_b h_{ij} 
		- g^{ab} g_{ip} \ol{R} \indices{_{ja}^p_b} - g^{ab} g_{jp} \ol{R} \indices{_{ia}^p_b}
		+ g^{ab} g^{pq} ( \ol{\nabla} h * \ol{\nabla} h )
		&& (\text{since } \ol{\nabla} \ol{g} = 0).
	\end{aligned} \end{gather*}
	Next, consider the tensors $\hat{h}, \tl{h}$ given as in equation \eqref{Inverse h Tensors} by
		$$g^{ab} = \ol{g}^{ab} - \hat{h}^{ab}		\quad \text{and} \quad 
		\hat{h}^{ab} = \ol{g}^{ak} \ol{g}^{bl} h_{kl} + \tl{h}^{ab}.$$
	We expand the Riemann tensor terms using this notation to obtain
	\begin{gather*} \begin{aligned}
		& \quad 
		 g^{ab} g_{ip} \ol{R} \indices{_{ja}^p_b} \\
		& = ( \ol{g}^{ab} - \hat{h}^{ab} ) ( \ol{g}_{ip} + h_{ip} ) \ol{R} \indices{_{ja}^p_b} 	\\
		& = \ol{g}^{ab}\ol{g}_{ip} \ol{R} \indices{_{ja}^p_b} 	
		- \hat{h}^{ab} \ol{g}_{ip}\ol{R} \indices{_{ja}^p_b} 	
		+ \ol{g}^{ab} h_{ip} \ol{R} \indices{_{ja}^p_b} 	
		- \hat{h}^{ab} h_{ip}\ol{R} \indices{_{ja}^p_b} 	 \\
		& = \ol{Rc}_{ji}
		- \ol{g}_{ip} \ol{g}^{ak} \ol{g}^{bl} h_{kl} \ol{R} \indices{_{ja}^p_b} 	
		- \ol{g}_{ip} \tl{h}^{ab} \ol{R} \indices{_{ja}^p_b} 	
		+ h_{ip} \ol{Rc}_j^p
		- \hat{h}^{ab} h_{ip}\ol{R} \indices{_{ja}^p_b} 	 \\
		& = \ol{Rc}_{ji} 
		 - \ol{R} \indices{_j^k_i^l} h_{kl}
		+ h_{ip} \ol{Rc}_j^p
		- \ol{R} \indices{_{ja}^p_b} \left( \ol{g}_{ip} \tl{h}^{ab}  + \hat{h}^{ab} h_{ip} \right) .
	\end{aligned} \end{gather*}
	
	It follows that	
	\begin{gather*} \begin{aligned}
		& \quad (1 - t) \partial_t h	\\
		& = ( 1 - t) \partial_t g		\\
		& = -2 Rc[g] + \cl{L}_{B_{\ol{g}}(g)} g -  \cl{L}_X g+ g	
		+ \left( \Phi_t^{-1} \right)^* \left\{ \partial_t \acute G_{\mathbf p} + 2 Rc[ \acute G_{\mathbf p} ] \right\} 
		&& ( X \doteqdot \ol{\nabla} f)\\
		& = g^{ab} \ol{\nabla}_a \ol{\nabla}_b h_{ij} 
		- 2 \ol{Rc}_{ij} 
		+ 2 \ol{R} \indices{_i^k_j^l} h_{kl}
		-  \ol{Rc}_j^p h_{ip}		-  \ol{Rc}_i^p h_{jp}\\
		& \quad + \ol{R} \indices{_{ja}^p_b} \left( \ol{g}_{ip} \tl{h}^{ab}  + \hat{h}^{ab} h_{ip} \right) 
		+ \ol{R} \indices{_{ia}^p_b} \left( \ol{g}_{jp} \tl{h}^{ab}  + \hat{h}^{ab} h_{jp} \right) 		\\
		& \quad + g^{ab} g^{pq} ( \ol{\nabla} h * \ol{\nabla} h ) -  \cl{L}_X g+ g
		+ \left( \Phi_t^{-1} \right)^* \left\{ \partial_t \acute G_{\mathbf p} + 2 Rc[ \acute G_{\mathbf p} ] \right\} \\
		& = g^{ab} \ol{\nabla}_a \ol{\nabla}_b h_{ij} 
		+ 2 \ol{R} \indices{_i^k_j^l} h_{kl}
		-  \ol{Rc}_j^p h_{ip}		-  \ol{Rc}_i^p h_{jp}\\
		& \quad + \ol{R} \indices{_{ja}^p_b} \left( \ol{g}_{ip} \tl{h}^{ab}  + \hat{h}^{ab} h_{ip} \right) 
		+ \ol{R} \indices{_{ia}^p_b} \left( \ol{g}_{jp} \tl{h}^{ab}  + \hat{h}^{ab} h_{jp} \right) 		\\	
		& \quad + g^{ab} g^{pq} ( \ol{\nabla} h * \ol{\nabla} h ) -  \cl{L}_X g + \cl{L}_X \ol{g} + g - \ol{g}
		&& \left(\ol{Rc} + \frac{1}{2} \cl{L}_X \ol{g} = \frac{1}{2} \ol{g} \right)	\\
		& \quad + \left( \Phi_t^{-1} \right)^* \left\{ \partial_t \acute G_{\mathbf p} + 2 Rc[ \acute G_{\mathbf p} ] \right\} \\
		& = g^{ab} \ol{\nabla}_a \ol{\nabla}_b h_{ij} 
		+ 2 \ol{R} \indices{_i^k_j^l} h_{kl}	\\
		& \quad + \ol{R} \indices{_{ja}^p_b} \left( \ol{g}_{ip} \tl{h}^{ab}  + \hat{h}^{ab} h_{ip} \right) 
		+ \ol{R} \indices{_{ia}^p_b} \left( \ol{g}_{jp} \tl{h}^{ab}  + \hat{h}^{ab} h_{jp} \right) 		\\	
		& \quad + g^{ab} g^{pq} ( \ol{\nabla} h * \ol{\nabla} h )
		-  \ol{Rc}_j^p h_{ip}	-  \ol{Rc}_i^p h_{jp}-  \cl{L}_X h + h	\\
		& \quad + \left( \Phi_t^{-1} \right)^* \left\{ \partial_t \acute G_{\mathbf p} + 2 Rc[ \acute G_{\mathbf p} ] \right\}.
		\end{aligned} \end{gather*}
	Now, expand the Lie derivative term and use the soliton equations \eqref{Soliton Eqns} to compute that
	\begin{align*}
		\cl{L}_X h_{ij}	
		={}& \ol{\nabla}_X h_{ij} + h_{ip} \ol{\nabla}_j X^p + h_{jp} \ol{\nabla}_i X^p 	\\
		={}&  \ol{\nabla}_X h_{ij} + h_{ip} \ol{\nabla}_j \ol{\nabla}^p f + h_{jp} \ol{\nabla}_i \ol{\nabla}^p f	\\
		={}&  \ol{\nabla}_X h_{ij} - \ol{Rc}_j^p h_{ip}    - \ol{Rc}_i^p h_{jp}  + \frac{1}{2}  \ol{g}_j^p h_{ip} +  \frac{1}{2} \ol{g}_i^p h_{jp}	\\
		={}&  \ol{\nabla}_X h_{ij} - \ol{Rc}_j^p h_{ip}    - \ol{Rc}_i^p h_{jp}  + h_{ij} .
	\end{align*}
	It follows that
	\begin{gather*} \begin{aligned}
		( 1 - t) \partial_t h
		& = g^{ab} \ol{\nabla}_a \ol{\nabla}_b h_{ij} - \ol{\nabla}_X h
		+ 2 \ol{R} \indices{_i^k_j^l} h_{kl}	\\
		& \quad + \ol{R} \indices{_{ja}^p_b} \left( \ol{g}_{ip} \tl{h}^{ab}  + \hat{h}^{ab} h_{ip} \right) 
		+ \ol{R} \indices{_{ia}^p_b} \left( \ol{g}_{jp} \tl{h}^{ab}  + \hat{h}^{ab} h_{jp} \right) 		\\
		& \quad + g^{ab} g^{pq} ( \ol{\nabla} h * \ol{\nabla} h )		
 		+ \left( \Phi_t^{-1} \right)^* \left\{ \partial_t \acute G_{\mathbf p} + 2 Rc[ \acute G_{\mathbf p} ] \right\}.
	\end{aligned}	\end{gather*}
\end{proof}

\begin{remark} \label{Remark clE_2 Observations}
		The term
		$$\left( \Phi_t^{-1} \right)^* \left\{ \partial_t \acute G_{\mathbf p} + 2 Rc[ \acute G_{\mathbf p} ] \right\}$$
		captures how far $\acute{G} = \acute{G}_{\mathbf p}(t)$ is from being a Ricci flow.
		By \eqref{eqn almost RF error term}, it may also be written as
		$$(\Phi_t^{-1})^* \left\{  \partial_t \acute G + 2 Rc [ \acute G ] \right\} =
		-2(\Phi_t^{-1})^* \left \{ \eta_{\Gamma_0} Rc[\iota_* G ] + (1- \eta_{\Gamma_0}) Rc [ \phi_t^* \ol{g} ]- Rc [ \acute G ] \right \}$$
		where $G = G_{\mathbf p}(t )$.

		As noted in Remark \ref{Rem acute G Observations}
			$$\left \{ \eta_{\Gamma_0} Rc[ \iota_* G ] + (1- \eta_{\Gamma_0}) Rc [ \phi_t^* \ol{g} ]- Rc [ \acute G ] \right \}$$
		is supported on the closure of $\Omega =  \{x \in M:  0 < \eta_{\Gamma_0}(x) < 1 \}$.
		Hence, 
		$$ \supp \left( (\Phi_t^{-1})^* \left\{  \partial_t \acute G + 2 Rc [ \acute G ] \right\}  \right) \subset  \Phi_t \left( \ol{\Omega} \right) = \phi_t \left( \tl{\Phi}_t \left( \ol{\Omega} \right) \right).$$
		Later, in Lemma \ref{Lem Drift of the Grafting Region}, we will estimate the location of this region in $M$ in order to obtain estimates on $(\Phi_t^{-1})^* \left\{  \partial_t \acute G + 2 Rc [ \acute G ] \right\} $.
		
\end{remark}

\begin{remark} \label{Remark tau}
	Given the derivatives $( 1 - t) \partial_t$ that appear in the above evolution equations, 
	it will be convenient to work with respect to the reparameterized time coordinate
		$$\tau \doteqdot - \ln ( 1 - t).$$
	We also denote $\tau_0 \doteqdot - \ln ( 1 - t_0)$ or equivalently $(1 - t_0) = e^{- \tau_0}$.
	Note that $(1 - t_0) \ll 1$ being sufficiently small is equivalent to $\tau_0 \gg1 $ being sufficiently large.
	
	Observe that $( 1 - t) \frac{ \partial}{ \partial t} = \frac{ \partial }{ \partial \tau}$.
	Henceforth, we may simplify notation by writing $h( \tau)$ or $\phi_{\tau}$ for example to denote $h( t( \tau) )$ or $\phi_{t ( \tau)}$, respectively, where $t(\tau) = 1 - e^{- \tau}$.

\end{remark}

\subsection{Definition of the Box} \label{Subsect Defn of the Box}
Recall from assumption \ref{Assume Shrinker} that 
we have fixed a $L^2_f$-orthonormal basis $\{ h_j \}_{j \in \mathbb{N}}$ of eigenmodes 
with corresponding eigenvalues $\{ \lambda_j \}_{j \in \mathbb{N}}$ 
for the weighted Lichnerowicz Laplacian $\ol{\Delta}_f + 2 \ol{Rm} = \ol{\Delta} - \ol{\nabla}_{\ol{\nabla} f } + 2 \ol{Rm}$.
Moreover, $\lambda_* < 0$ was chosen so that $\lambda_* \notin \{ \lambda_j \}_{j \in \mathbb{N}}$ 
and $K = K (n, M, \ol{g}, f, \lambda_*)$ denotes the index such that $\lambda_{K+1} < \lambda_* < \lambda_K$.

\begin{definition} \label{Defn Projs}
For general $h \in L^2_f(M)$, define projections
\begin{gather*} \begin{aligned}
	h_u \doteqdot & \sum_{j=1}^K ( h, h_k)_{L^2_f} h_k,	\text{ and}\\
	h_s \doteqdot & \sum_{j=K+1}^\infty ( h, h_k)_{L^2_f} h_k = h - h_u.	\\
\end{aligned} \end{gather*}
\end{definition}

\begin{definition} \label{Defn Box}
For $\lambda_*< 0$  as in assumption \ref{Assume Shrinker},
constants $ \mu_u, \mu_s, \epsilon_0, \epsilon_1, \epsilon_2 \in (0,1)$, and an interval $I \subset \R$, 
define 
	$$\cl{B} = \cl{B}[ \lambda_* , \mu_u, \mu_s, \epsilon_0, \epsilon_1, \epsilon_2,   I ]$$
to be the set of smooth sections $h$ of $\pi^* Sym^2 T^*M \to M \times I$ where $\pi : M \times I \to M$ is the projection map
such that:
\begin{enumerate}
	\item $h(\cdot, \tau)$ is compactly supported in $M$ for all $\tau \in I$,

	\item (Fine $L^2_f$ Estimates)
\begin{gather*} \begin{aligned}
	\left \| h_u( \cdot, \tau) \right \|_{L^2_f} &\le \mu_u e^{\lambda_* \tau}		
	&& \text{for all } \tau \in I, \\
	\left \| h_s( \cdot, \tau) \right \|_{L^2_f} &\le \mu_s e^{\lambda_* \tau}		
	&& \text{for all } \tau \in I, \\
\end{aligned} \end{gather*}
	and
	\item (Coarse $C^2$ Estimates)
\begin{align*}
	| h |_{\ol{g}} 
	&\le \epsilon_0	
	&& \text{for all } \tau \in I, \\
	| \ol{\nabla} h |_{\ol{g}} &\le \epsilon_1
	&& \text{for all } \tau \in I, \text{ and} \\
	| \ol{\nabla}^2 h |_{\ol{g}} &\le \epsilon_2
	&& \text{for all } \tau \in I. 
\end{align*}
\end{enumerate}
\end{definition}

With a slight abuse of notation, we will also write
	$$``\cl{B}[ \lambda_* , \mu_u, \mu_s, \epsilon_0, \epsilon_1, \epsilon_2 , [\tau(t_0), \tau(t_1) ) ]"$$
	$$ \text{as } ``\cl{B}[ \lambda_* , \mu_u, \mu_s, \epsilon_0, \epsilon_1, \epsilon_2 , [t_0, t_1) ]"$$
for example.

\begin{definition} \label{Defn clP}
	For $\lambda_*$ as in assumption \ref{Assume Shrinker};
	$\ol{p} > 0$;
	$\Gamma_0 \gg 1$ and $0 < \gamma_0 \ll 1$ as in assumption \ref{Assume Smooth Metric and eta_gamma Supp};
	$\mu_u, \mu_s, \epsilon_0, \epsilon_1, \epsilon_2 \in (0,1)$;
	and $0 \le t_0 < t_1 \le 1$;
	let 
		$$\cl{P} = \cl{P} [\lambda_*, \ol{p}, \Gamma_0,   \gamma_0, \mu_u, \mu_s, \epsilon_0, \epsilon_1, \epsilon_2, t_0,  t_1 ] \subset \R^K$$
	denote the set of $\mathbf p  = ( p_1, \dots, p_{K } ) \in \mathbb{R}^{K}$
	such that:
	\begin{enumerate}
	\item 	$| \mathbf p | \le \ol{p} e^{\lambda_* \tau_0}$,
	\item 	$T( \mathbf p ) \in[ t_1, 1]$, and
	\item there exists a smooth function $\tl{\Phi} : M \times [t_0, t_1) \to M$ 
	solving \eqref{tlPhi Evol Eqn} 
	\begin{equation*}
		\partial_t \tl{\Phi} = \Delta_{\acute G_{\mathbf p }(t), ( 1- t) \phi_t^* \ol{g} } \tl{\Phi}
		\text{ on } M \times (t_0, t_1)
		\qquad \text{with initial condition }
		\tl{\Phi}( \cdot, t_0) = Id_M
	\end{equation*}
	where
	$\tl{\Phi}_t = \tl{\Phi}( \cdot, t) : M \to M$ is a diffeomorphism for all $t \in [t_0, t_1)$ and 
		$$h_{\mathbf p }(t) \doteqdot \frac{1}{1-t} ( \Phi_t^{-1} )^* \acute G_{\mathbf p } (t) - \ol{g} 
		\in \cl{B}[ \lambda_* ,  \mu_u, \mu_s, \epsilon_0, \epsilon_1, \epsilon_2 , [t_0, t_1) ].$$
	\end{enumerate}
\end{definition}

\section{Preliminary Estimates} \label{Sect Prelim Ests}

The goal of this section is to estimate and control the error term
	$$( \Phi_t^{-1})^* \left\{ \partial_t \acute G_{\mathbf p} + 2 Rc[ \acute G_{\mathbf p} ] \right\}$$
that appears in \eqref{h(t) Evol Eqn}.
First, we obtain bounds for $\partial_t \acute G + 2 \acute{Rc}$.
Such estimates are obtained from Perelman's pseudolocality theorem applied to our setting.
Second, we estimate the location of the support of $( \Phi_t^{-1})^* \left\{ \partial_t \acute G_{\mathbf p} + 2 Rc[ \acute G_{\mathbf p} ] \right\}$.
As noted in Remark \ref{Remark clE_2 Observations}, such estimates are obtained through drift estimates on $\tl{\Phi}_t, \phi_t$ which are recorded in Lemma \ref{Lem Drift of the Grafting Region}.

We begin by recording the following version of Perelman's pseudolocality theorem (Theorem 10.1 in \cite{Perelman02})\footnote{See also \cite[Ch. 21]{ChowEtAl10} and \cite{Lu10}.}
and its consequences for our setting.

\begin{thm} [Perelman's Pseudolocality Theorem] \label{Thm Perelman's Pseudolocality}
	For any dimension $n$, there exist $\epsilon_n, \delta_n > 0$ with the following property:
	If $(M^n , g(t) )$ is a smooth Ricci flow on a closed manifold $M$ defined for $t \in [t_0, T)$ that, at initial time $t_0$, satisfies
	\begin{align*}
		|Rm|_{g(t_0)}(x) &\le \frac{1}{ r_0^2} \quad \text{ for all  } x \in B_{g(t_0)} ( x_0, r_0) \\
		\text{and } \text{Vol}_{g(t_0) } B_{g(t_0)} (x_0, r_0) &\ge ( 1 - \delta_n ) \omega_n r_0^n 
	\end{align*}
	for some $x_0 \in M$ and $r_0 > 0$,
	then
		$$|Rm|_{g(t)}(x,t) \le \frac{1}{\epsilon_n^2 r_0^2} \qquad \text{for all } x \in B_{g(t)} (x_0, \epsilon_n r_0), t \in [t_0, \min \{ T , t_0 +  (\epsilon_n r_0)^2 \} ).$$
\end{thm}

\begin{prop}[Consequences of Pseudolocality] \label{Prop Pseudolocality App}
	Let $\Omega \Subset  \{ f > \Gamma_0/2 \} \subset \cl{M}$.\footnote{$U \Subset V$ means that the closure of $U$ is a compact subset of $V$.}
	If 
	$\Gamma_0 \gg 1$ is sufficiently large (depending on $n, M, \ol{g}, f$),
	then there exists $\cl{K}_0$ (depending on $n, M, \ol{g}, f, \Gamma_0, \Omega$) such that if
	$0 < 1 - t_0 \ll 1$ is sufficiently small (depending on $n, M, \ol{g}, f, \Gamma_0, \Omega$), 
	then
	for all $| \mathbf p | \le \ol{p} e^{\lambda_* \tau_0}$
		$$| Rm[ G_{\mathbf p } (t) ] |_{G_{\mathbf{p}}(t)} (x,t) \le \cl{K}_0 	\qquad \text{for all } (x,t) \in \Omega \times [t_0, T( \mathbf p ) ).$$
\end{prop}
\begin{proof}
	Let $ \epsilon_n, \delta_n  > 0$ be as in the statement of Perelman's pseudolocality Theorem \ref{Thm Perelman's Pseudolocality} above. 
	
	Throughout $\{ f > \Gamma_0 / 2\}$,
	$G_{\mathbf p } (t_0) = G_{\mathbf 0 } (t_0)$ for all $\mathbf p$
	and, for some $C = C(n, M, \ol{g}, f)$, 
	\begin{equation} \label{proof conseqs of pseudoloc eqn 1}
		|Rm[ G_{\mathbf p }(t_0) ] |_{G_{\mathbf p }(t_0) } \le \frac{C}{\Gamma_0}
	\end{equation}
	by Definition \ref{Defn G_0(t_0)} (\ref{Defn G_0(t_0) Curv Ests}) if $0 < 1 - t_0 \ll 1$ is sufficiently small.
	Additionally, Definition \ref{Defn G_0(t_0)} (\ref{Defn G_0(t_0) Time Invt}) \& (\ref{Defn G_0(t_0) Cone Cgce}) imply that the restriction of the metrics $G_{\mathbf p } (t_0)$ to $\{ f > \Gamma_0 / 2\}$ stay within an arbitrarily small $C^0$-neighborhood for all $0 < 1 - t_0 \ll 1$ sufficiently small.
	
	Since $\Omega \Subset \{ f > \Gamma_0 / 2 \}$, 
	there exists $0 < r_0 \ll 1$ sufficiently small 
	depending on $n, M , \ol{g}, f, \Gamma_0, \Omega$
	so that, for all $0 < 1 - t_0 \ll 1$ sufficiently small,
	\begin{equation*}
		\frac{C}{\Gamma_0} \le \frac{1}{r_0^2} \qquad \text{and} \qquad 
		B_{G_{\mathbf p } (t_0)}(x, r_0) \subset \{ f > \Gamma_0/2 \}
		\qquad \text{for all } x \in \Omega.
	\end{equation*}
	Moreover, by taking $0 < r_0 \ll 1$ possibly smaller (depending on $n, M, \ol{g}, f, \Gamma_0,\Omega$), we claim that additionally
	\begin{equation} \label{proof conseqs of pseudoloc eqn 2}
		\text{Vol} \, B_{G_{\mathbf p} (t_0)} (x, r_0) \ge ( 1 - \delta_n) \omega_n r_0^n
		\qquad \text{for all } x \in \Omega \text{ and all } 0 < 1-t_0 \ll 1.
	\end{equation}
	Indeed, if not, then there exists a sequence $x_i \in \Omega$, $t_i \nearrow 1$, and $r_i \searrow 0$ where \eqref{proof conseqs of pseudoloc eqn 2} fails.
	By the curvature estimates \eqref{proof conseqs of pseudoloc eqn 1} and non-collapsing condition (Definition \ref{Defn G_0(t_0)} \eqref{Defn G_0(t_0) Non-Collapsed}), the blow-up sequence $r_i^{-2} G_{\mathbf p}(t_i)$ based at $x_i$ has a subsequence which converges in the pointed $C^{1, \alpha}$-topology to Euclidean space $(\R^n, g_{Euc})$.\footnote{See for example \cite[proof of Lemma 11.4.9, p. 433]{Petersen16} and references therein for a detailed justification of this statement.}
	Thus, the inequality \eqref{proof conseqs of pseudoloc eqn 2} holds for $i \gg 1$, which contradicts the choice of the sequence.
	This contradiction proves \eqref{proof conseqs of pseudoloc eqn 2}.

	If $0 < 1 - t_0 \ll 1$ is also small enough so that
		$$T( \mathbf p) - t_0 \le 1 - t_0 \le \epsilon_n^2 r_0^2,$$
	then Perelman's pseudolocality Theorem \ref{Thm Perelman's Pseudolocality}
	applies for all $x_0 = x \in \Omega$ and implies that
		$$| Rm |(x,t)  \le \frac{1}{ \epsilon_n^2 r_0^2 } \qquad \text{for all } (x,t) \in \Omega \times [t_0, T(\mathbf p)).$$
	Taking $\cl{K}_0 = ( \epsilon_n r_0)^{-2}$ completes the proof.
\end{proof}

\begin{remark}
	As a corollary of Proposition \ref{Prop Pseudolocality App}, it follows that no singularity of the flow $\{ G_{\mathbf p }(t) \}_{t \in [t_0, T( \mathbf p))}$ 
	occurs on $\Omega \Subset \{ f > \Gamma_0 /2\} \subset \cl{M}$, so long as the parameters $\Gamma_0 \gg 1$ and $0 < 1- t_0 \ll 1$ are chosen appropriately and $|\mathbf p | \le \ol{p} e^{\lambda_* \tau_0}$.
\end{remark}

\subsection{Estimates on the Grafting Region} \label{Subsect Ests on Grafting Region}
Next, we obtain estimates on the \textit{grafting region}
\begin{equation} \label{defn grafting region}
	\Omega_{\eta_{\Gamma_0}} \doteqdot \{x \in M :  0 < \eta_{\Gamma_0}(x) < 1 \}.
\end{equation}
Recall from Subsection \ref{Subsect The Flow} that $\eta_{\Gamma_0}$ was chosen so that $\Omega_{\eta_{\Gamma_0}}  \subset \{ \frac{4}{6} \Gamma_0 < f < \frac{5}{6} \Gamma_0 \} \Subset \{  \frac{1}{2} \Gamma_0  < f < \Gamma_0 \} \subset M$.
Therefore, Proposition \ref{Prop Pseudolocality App} applied to $\Omega = \iota^{-1} \left( \bigsqcup_{\omega \in A} \Omega_{\eta_{\Gamma_0} } \right) $ implies that 
we may assume without loss of generality that
\begin{equation} \label{eqn curv ests for ests on grafting region}
	| Rm[G_{\mathbf p}(t)]  |_{G_{\mathbf p }(t) } (x,t) \le \cl{K} _0		
	\qquad \text{for all }  (x,t) \in \iota^{-1} \left( \bigsqcup_{\omega \in A}  \Omega_{\eta_{\Gamma_0}} \right)  \times [ t_0, T( \mathbf p) )
\end{equation}
for some $\cl{K}_0$ (depending on $n, M , \ol{g}, f, \Gamma_0$)
when $0 < 1 - t_0 \ll 1$ is sufficiently small (depending on $n, M, \ol{g}, f, \Gamma_0$).

We also note that Definition \ref{Defn G_0(t_0)} implies
	\begin{itemize}
		\item
		on $\ol{\iota^{-1} \left( \bigsqcup_{\omega \in A} \Omega_{\eta_{\Gamma_0}} \right) } \subset \cl{M}$
		\begin{equation} \label{eqn metric equality on grafting region at initi time}
			\iota^* \acute G_{\mathbf p}(t_0) = G_{\mathbf p }(t_0) =G_{\mathbf 0 }(t_0)  =  ( 1 - t_0)\iota^* \phi_{t_0}^* \ol{g} ,
		\end{equation}
		and
		\item 
		there exists an an embedding $\Psi : \{ x \in M : \Gamma_0 / 2 < f(x) < 32 \Gamma_0\} \to \cone_1(\Sigma)$ such that
			$$\acute{G}_{\mathbf p }( t_0) \xrightarrow[t_0 \nearrow 1]{C^\infty\left ( \ol{\Omega_{\eta_{\Gamma_0}  } } \right )} \Psi^* g_{\mathcal C}.$$
	\end{itemize}
We combine these facts to estimate how much $\acute G$ differs from a Ricci flow.
\begin{lem} \label{Lem difference from RF}
	There exists $C$ (depending only on $n, M, \ol{g}, f$) such that 
	if $0 < 1 - t_0 \ll 1$ is sufficiently small (depending on $n, M, \ol{g}, f, \Gamma_0$),
	then
	\begin{align*}
		\left|  \partial_t \acute G_{\mathbf p} + 2 Rc [ \acute G_{\mathbf p }(t)]  \right|_{\acute G_{\mathbf p}(t)} 
		&\le \frac{C}{\Gamma_0} 
		\quad \text{and}		\\
		\left|  {}^{\acute G_{\mathbf p}(t)} \nabla \left( \partial_t \acute G_{\mathbf p} + 2 Rc [ \acute G_{\mathbf p }(t)] \right)  \right|_{\acute G_{\mathbf p}(t)} 
		&\le \frac{C}{\Gamma_0^{3/2}}
	\end{align*}
	for all $x \in M$,  $t \in [t_0, T(\mathbf p) )$, and all $| \mathbf p | \le \ol{p} e^{\lambda_* \tau_0}$.
\end{lem}

\begin{proof}
	Throughout this proof, we shall omit the subscript $\mathbf p$ to simplify notation.
	First, note that by Remark \ref{Remark clE_2 Observations}
		$$\partial_t \acute G = - 2 Rc [ \acute G] \qquad \text{ on } M \setminus \ol{\Omega_{\eta_{\Gamma_0}}}$$
	so it suffices to estimate $\partial_t \acute G + 2 Rc [ \acute G]$ on $\ol{\Omega_{\eta_{\Gamma_0}}}$.
	
	Throughout the proof, set 
		$$\Omega \doteqdot \iota^{-1} \left( \bigsqcup_{\omega \in A} \Omega_{\eta_{\Gamma_0}} \right) \subset \cl{M}.$$
	By \eqref{eqn curv ests for ests on grafting region},
	$|Rm[G(t)] |_{G(t)}(x,t) \le \cl{K}_0$ for all  $(x,t) \in \ol{\Omega} \times [t_0, T( \mathbf p) )$, which implies that 
		$$| G(t) - G(t_0) |_{G(t) } \le C_n \cl{K}_0 ( t - t_0) \le C_n \cl{K}_0 ( 1 - t_0) \qquad 
		\text{on } \ol{\Omega } \times [ t_0, T(\mathbf p ) ).$$
	In particular,
	for any $C^0 ( \ol{\Omega } )$-neighborhood of $G(t_0)$, $G(t)$ remains in that neighborhood for all $t \in [t_0, T(\mathbf p) )$ if $0 < 1 - t_0 \ll 1$ is sufficiently small.
	
	By applying Proposition \ref{Prop Pseudolocality App} to a slightly larger neighborhood $\Omega'$ containing $\ol{\Omega}$,
	the curvature bounds $| Rm |_G \le \cl{K}_0$ on $\Omega' \times [ t_0, T( \mathbf p) )$,
	Shi's local derivative estimates \cite{Shi89},\footnote{See also \cite[Theorem 14.16]{ChowEtAl07}.}
	and the regularity of the initial data $G(t_0)$ on $\{ \Gamma_0/ 2 < f < \Gamma_0 \}$
	imply that 
	for all $m \in \mathbb{N}$ there exists $\cl{K}_m$ (depending only on $n, M, \ol{g}, f, \Gamma_0, \cl{K}_0$)
	such that 	
		$$| {}^{G(t)} \nabla^m Rm |_{G(t) } \le \cl{K}_m 		\qquad \text{ on } \ol{\Omega} \times [ t_0, T( \mathbf p) ).$$
	It follows that, for any $m \in \mathbb{N}$,
	$G(t)$ can be made to stay in an arbitrarily small $C^m( \ol{ \Omega } )$-neighborhood 
	of $G(t_0)$ for all $t \in [t_0, T(\mathbf p))$
	by taking $0 < T(\mathbf p ) - t_0 \le 1 - t_0 \ll 1$ sufficiently small.

	Therefore, for all $t \in [t_0, T(\mathbf p ))$,
	$\iota_* G(t), \iota_* G(t_0), ( 1- t) \phi_t^* \ol{g}, (1- t_0) \phi_{t_0}^* \ol{g}$ can all be made to stay within an arbitrarily small $C^3( \ol{\Omega_{\eta_{\Gamma_0}} })$-neighborhood of each other by taking $0 <  1- t_0 \ll 1$ sufficiently small depending on $n, M, \ol{g}, f, \Gamma_0$.
	Hence,
		$$\acute G (t) - \acute G (t_0) 
		=  \eta_{\Gamma_0} \iota_* \left[ G( t) - G( t_0) \right] 
		+ (1- \eta_{\Gamma_0}) \left[ (1-t) \phi_t^* \ol{g} - (1-t_0) \phi_{t_0}^* \ol{g} \right]$$
	can be made $C^3$-small on $\ol{ \Omega_{\eta_{\Gamma_0}} } \times [ t_0, T( \mathbf p) )$
	by taking $0 < 1 - t_0 \ll 1$ sufficiently small.
	Recall also \eqref{eqn metric equality on grafting region at initi time} that $\iota_* G(t_0) = \acute G(t_0) = (1-t_0) \phi_{t_0}^* \ol{g}$ on $\ol{\Omega_{\eta_{\Gamma_0}}}$.
	In particular, $0 < 1 - t_0 \ll 1$ sufficiently small implies that on $\ol{\Omega_{\eta_{\Gamma_0}} } \times [t_0, T( \mathbf p ))$
	\begin{gather*}
		| Rc [ \acute G(t) ] |_{\acute G(t)} \le 2  | Rc [ \acute G(t_0) ] |_{\acute G(t_0)} 
		= 2 | Rc [ G_{\mathbf 0 } (t_0) ] |_{ G_{\mathbf 0 } (t_0) } \lesssim_{n, M, \ol{g}, f} \Gamma_0^{-1}
		\quad \text{and} 
		\\
		| {}^{\acute G(t)} \nabla Rc [ \acute G(t) ] |_{\acute G(t)} 
		\le 2 | {}^{\iota^* \acute G_{\mathbf 0} (t_0)} \nabla Rc [ G_{\mathbf 0 } (t_0) ] |_{ \iota^* \acute G_{\mathbf 0 } (t_0) } 
		\lesssim_{n, M, \ol{g}, f} \Gamma_0^{-3/2}
	\end{gather*}
	where we have used Definition \ref{Defn G_0(t_0)} (\ref{Defn G_0(t_0) Curv Ests}) in the last inequality of each line above.
	
	Similarly, $0 < 1 - t_0 \ll 1$ sufficiently small implies that on $\ol{\Omega_{\eta_{\Gamma_0} } } \times [t_0, T( \mathbf p ))$
	\begin{align*}
		& \quad \left| \partial_t \acute G \right|_{\acute G(t)}	\\
		&= \left|  \eta_{\Gamma_0} \iota_* \partial_t G	
		+ ( 1 - \eta_{\Gamma_0} ) \partial_t \left( ( 1- t) \phi_t^* \ol{g} \right)  \right|_{\acute G(t)}	\\
		&= \left| - 2 \eta_{\Gamma_0} \iota_* Rc[ G(t) ] 
		-2 ( 1 - \eta_{\Gamma_0} ) Rc \left[ ( 1- t) \phi_t^* \ol{g} \right]  \right|_{\acute G(t)}	\\
		&\le 2 \left|  Rc[ G(t) ] \right|_{\iota^* \acute G(t)}
		+ 2 \left| Rc[ ( 1 - t) \phi_t^* \ol{g} ] \right|_{\acute G(t)}		\\
		&\le 4 \left| Rc[ G(t_0) ] \right|_{ \iota^* \acute G(t_0) }
		+ 4 \left| Rc[ ( 1 - t_0) \phi_{t_0}^* \ol{g} ] \right|_{\acute G(t_0) }		\\
		&= 8 \left| Rc[ G_{\mathbf 0}(t_0) ] \right|_{  G_{\mathbf 0} (t_0) }	\\
		&\lesssim_{n, M, \ol{g}, f} \Gamma_0^{-1} ,
	\end{align*}
	and similarly
	\begin{align*}
		& \quad \left| {}^{\acute G(t)} \nabla \partial_t \acute G \right|_{\acute G(t)}	\\
		&= \left|  {}^{\acute G(t)} \nabla \left( -2 \eta_{\Gamma_0} \iota_* Rc[ G(t)] 
		-2 ( 1 - \eta_{\Gamma_0} )  Rc[ ( 1-t)\phi_t^* \ol{g} ]  \right) \right|_{\acute G(t)}	\\
		&\lesssim_n  \left| {}^{\acute G(t)} \nabla \eta_{\Gamma_0} \right|_{\acute G(t)}
		\left( \left| Rc[ G(t) ] \right|_{\iota^* \acute G(t)}+  \left| Rc[ ( 1 - t) \phi_t^* \ol{g} ] \right|_{\acute G(t)} \right) \\
		& \qquad +  \left| {}^{\iota^* \acute G(t)} \nabla Rc[ G(t) ] \right|_{\iota^* \acute G(t)}
		+  \left|  {}^{\acute G(t)} \nabla Rc[ ( 1 - t) \phi_t^* \ol{g} ] \right|_{\acute G(t)}	\\
		&\le 2 \left| {}^{\acute G(t_0)} \nabla \eta_{\Gamma_0} \right|_{\acute G(t_0)}
		\left( \left| Rc[ G(t_0) ] \right|_{\iota^* \acute G(t_0)}+  \left| Rc[ ( 1 - t_0) \phi_{t_0}^* \ol{g} ] \right|_{\acute G(t_0)} \right) \\
		& \qquad + 2 \left| {}^{\iota^* \acute G(t_0)} \nabla Rc[ G(t_0) ] \right|_{\iota^* \acute G(t_0) }
		+  2\left|  {}^{\acute G(t_0)} \nabla Rc[ ( 1 - t_0) \phi_{t_0}^* \ol{g} ] \right|_{\acute G(t_0)}	\\
		&= 4 \left| {}^{\acute G_{\mathbf 0}(t_0)} \nabla \eta_{\Gamma_0} \right|_{\acute G_{\mathbf 0}(t_0)}
		 \left| Rc[ G_{\mathbf 0}(t_0) ] \right|_{\iota^* \acute G_{\mathbf 0}(t_0)}  
		+ 4 \left| {}^{\iota^* \acute G_{\mathbf 0}(t_0)} \nabla Rc[ G_{\mathbf 0}(t_0) ] \right|_{\iota^* \acute G_{\mathbf 0}(t_0)} \\
		&\lesssim_{n, M, \ol{g}, f} \left| {}^{\acute G_{\mathbf 0}(t_0)} \nabla \eta_{\Gamma_0} \right|_{\acute G_{\mathbf 0}(t_0)} \frac{ 1}{\Gamma_0} + \frac{1}{\Gamma_0^{3/2} }\\
		&\lesssim_{n, M, \ol{g}, f} \Gamma_0^{-3/2}.
	\end{align*}
\end{proof}

\begin{lem}[Drift of the Grafting Region] \label{Lem Drift of the Grafting Region}
	Assume $0 \le t_0 < t_1 \le T( \mathbf p )$,
	$\tl{\Phi} : M \times [ t_0, t_1 ) \to M$ is a smooth solution of the harmonic map heat flow
		$$\partial_t \tl{\Phi}_t = \Delta_{\acute G_{\mathbf p}(t), ( 1 - t) \phi_t^* \ol{g} } \tl{\Phi} 
		\qquad \text{with initial condition } \tl{\Phi}( \cdot, t_0) = Id_M,$$
	$\tl{\Phi}_t = \tl{\Phi}(\cdot, t) : M \to M$ is a diffeomorphism for all $t \in [t_0, t_1)$,
	and 
		$$h = h_{\mathbf p}(t) = 
		\frac{1}{1 - t} ( \Phi_t^{-1})^* \acute G_{\mathbf p }(t) - \ol{g}
		\in \cl{B} [ \lambda_*, \mu_u, \mu_s, \epsilon_0, \epsilon_1, \epsilon_2 , [t_0, t_1) ].$$
		
	If $\Gamma_0 \gg 1$ is sufficiently large (depending on $n, M, \ol{g}, f$),
	$0 < \epsilon_0 , \epsilon_1 \ll 1$ are sufficiently small (depending on $n$),
	and $0 < 1 - t_0 \ll 1$ is sufficiently small (depending on $n, M, \ol{g}, f, \Gamma_0$),
	then, for all $t \in [t_0, t_1)$,
	\begin{align*}
		\tl{\Phi}_t ( \ol{ \Omega_{\eta_{\Gamma_0}}} ) \subset{} & \{ \Gamma_0 / 2 < f < \Gamma_0 \},	\\
		\Phi_t ( \ol{ \Omega_{\eta_{\Gamma_0}}} ) \subset{} & \{ (1 - t)^{-1} < f < \Gamma_0(1 - t)^{-1}  \},	\\
		\supp \left( ( \Phi_t^{-1} )^* \left\{ \partial_t \acute G + 2 Rc[ \acute G ] \right\} \right)
		\subset{}& \{ (1 - t)^{-1} < f < \Gamma_0(1 - t)^{-1}  \},
		\text{ and}	\\
		\supp h \subset{}& \{ f < \Gamma_0 ( 1 - t)^{-1} \}.
	\end{align*}
\end{lem}
\begin{proof}
	Recall 
		$$\Omega_{\eta_{\Gamma_0}} \subset \left\{ \frac{4}{6} \Gamma_0 < f < \frac{5}{6} \Gamma_0 \right\} \Subset \Omega' \doteqdot \{ \Gamma_0 / 2 < f < \Gamma_0 \} \Subset M$$
	and $(1 - t) \phi_t^* \ol{g}$ solves Ricci flow on $M \times (-\infty, 1)$.
	If $\Gamma_0 \gg 1$ is sufficiently large, 
	$(1 - t) \phi_t^* \ol{g}$ has uniformly bounded curvature on $\Omega' \times [ 0, 1)$ and
	converges to a pullback of the cone metric on $\Omega'$ as $t \nearrow 1$.
	In particular, for all $t_0 \in [0,1)$, $\dist_{(1 - t_0) \phi_{t_0}^* \ol{g} } ( \Omega_{\eta_{\Gamma_0}} , M \setminus \Omega')$ is bounded below by a constant depending only on $n, M, \ol{g}, f, \Gamma_0$.
	
	Since $h \in \cl{B}[ \lambda_* , \mu_u, \mu_s , \epsilon_0, \epsilon_1,\epsilon_2 , [t_0, t_1) ]$,
		$$\left| (\tl{\Phi}_t^{-1})^*\acute G_{\mathbf p }(t) - ( 1 -t) \phi_t^* \ol{g} \right|_{( 1 - t) \phi_t^* \ol{g}}
		= \left| \frac{1}{1 - t} (\Phi_t^{-1})^*\acute G_{\mathbf p }(t) -  \ol{g} \right|_{ \ol{g}}
		= \left| h_{\mathbf p }(t) \right|_{\ol{g}} 
		\le \epsilon_0$$
	and similarly
		$$\sqrt{1 - t} \left| {}^{(1 - t) \phi_t^* \ol{g} } \nabla \left(  (\tl{\Phi}_t^{-1})^*\acute G_{\mathbf p }(t) - ( 1 -t) \phi_t^* \ol{g} \right) \right|_{( 1 - t) \phi_t^* \ol{g}}
		=  | \ol{\nabla} h_{\mathbf p}(t) |_{\ol{g}} \le \epsilon_1.$$
	If $\epsilon_0 + \epsilon_1 \ll 1$ is sufficiently small depending on $n$
	and $0 < 1 - t_0 \ll 1$ is sufficiently small depending on $n, M, \ol{g}, f, \Gamma_0$,
	then the local drift estimates of Lemma \ref{Lem Local Drift Control} imply
		$$\tl{\Phi}_t (x) \in \Omega' = \{ \Gamma_0 /2 < f < \Gamma_0 \} 	\qquad \text{for all } x \in \ol{ \Omega_{\eta_{\Gamma_0}} }, t \in [t_0, t_1)$$
	since $\tl{\Phi}_{t_0} (x) = x \in  \ol{ \Omega_{\eta_{\Gamma_0}}}$.
	This proves the first claim.
	
	For the remainder of the proof, we follow Remark \ref{Remark tau} and work in terms of $\tau = -\ln( 1 - t)$.
	To prove the next claim, 
	it suffices to show that
		$$\phi_\tau ( \{ \Gamma_0 / 2 < f < \Gamma_0 \} ) \subset \{ e^\tau < f < \Gamma_0 e^\tau \}$$
	when $\Gamma_0 \gg 1$ is sufficiently large depending on $n, M, \ol{g}, f$.
	
	By Lemma \ref{Lem Flow Est 1}, there exists a constant $C = C(n, M, \ol{g}, f)$ such that 
		$$\phi_\tau^* f - \frac{C}{ \phi_\tau^* f} \le \partial_\tau ( \phi_\tau^* f) \le \phi_\tau^* f
		\qquad \text{ for all } \tau \in \R.$$
	Assume $\Gamma_0 /2 > \sqrt{C+1}$ and 
	let $x \in M$ such that $\sqrt{C+1} < \Gamma_0 / 2 < f(x) < \Gamma_0.$
	Integrating the inequalities for $\partial_\tau ( \phi_\tau^* f)$ from $0$ to $\tau$ and using that $\phi_0 = Id_M$  then implies
		$$		f( \phi_\tau (x) ) = (\phi_\tau^* f) (x) \le  f(x) e^\tau < \Gamma_0 e^\tau		\qquad \text{and}$$
		$$ f( \phi_\tau(x)) \ge \sqrt{ C + ( f(x)^2 - C ) e^{2 \tau } } > \sqrt{ C + e^{2 \tau} } \ge e^\tau.$$
	This completes the proof that
		$$\Phi_t ( \ol{ \Omega_{\eta_{\Gamma_0}}} ) \subset \{ (1 - t)^{-1} < f < \Gamma_0(1 - t)^{-1}  \}.$$
	
	The remaining two claims then follow from Remarks \ref{Rem acute G Observations}, \ref{Remark clE_2 Observations}, and the definition of $h$ as in Corollary \ref{Cor h(t) Evol Eqn}.
\end{proof}

\section{Preservation \& Improvement of $C^2$ Bounds} \label{Sect Preserve and Improve C^2 Bounds}

In what follows, we consider $h = h_{\mathbf p }$ as a function of the reparametrized time variable $\tau = - \ln ( 1 - t)$ and write its evolution equation \eqref{h(t) Evol Eqn} as
\begin{equation} \label{h(tau) Evol Eqn}
	\partial_\tau h = \ol{\Delta}_f h + 2 \ol{Rm} [h] + \cl{E}_1(\tau) + \cl{E}_2(\tau)
\end{equation}
where
	\begin{align} \label{clE_1 Eqn}
	\begin{split}
	 \mathcal{E}_1(\tau)  \doteqdot{}&  
		 \ol{R} \indices{_{ja}^p_b} \left( \ol{g}_{ip} \tl{h}^{ab}  + \hat{h}^{ab} h_{ip} \right) 
		+ \ol{R} \indices{_{ia}^p_b} \left( \ol{g}_{jp} \tl{h}^{ab}  + \hat{h}^{ab} h_{jp} \right) 	\\		
		&  + g^{ab} g^{pq} ( \ol{\nabla} h * \ol{\nabla} h )	 -\hat{h}^{ab} \ol{\nabla}_a \ol{\nabla}_b h	
		\qquad \text{and} 	
	\end{split} \\
		\label{clE_2 Eqn}
	  \mathcal{E}_2(\tau (t)) \doteqdot{}& ( \Phi_t^{-1} )^* \{ \partial_t \acute G + \acute{Rc} \}.
	\end{align}
In light of Lemmas \ref{Lem difference from RF} and \ref{Lem Drift of the Grafting Region}, we shall often assume
	$$| \cl{E}_2 |_{\ol{g}} \le \frac{C_0}{\Gamma_0} e^{- \tau}, \,
			| \ol{\nabla} \cl{E}_2 |_{\ol{g}} \le \frac{C_0}{\Gamma_0^{3/2}} e^{- \frac{3}{2}\tau}, \text{ and } 
			\supp \cl{E}_2 \subset \{ e^\tau \le f \le \Gamma_0 e^\tau \}$$
where $C_0  = C_0 ( n, M, \ol{g}, f) \ge 1$ is a constant that depends on $n, M, \ol{g}, f$.

In this section, we aim to show the following theorem which says that $h$ cannot exit through the sides of the box $\cl{B}$ defined by the coarse $C^2$ bounds:
\begin{thm} \label{Thm C^2 Bounds Preserved}
	Assume $\lambda_*<0$ is taken as in assumption \ref{Assume Shrinker},
	$\ol{p} \in (0,1]$ is taken as in assumption \ref{Assume Smooth Metric and eta_gamma Supp}, and that 
	$h= h_{\mathbf p }$ solves the equation 
	\begin{equation} \tag{\ref{h(t) Evol Eqn}}
		\partial_\tau h = \ol{\Delta}_f h + 2 \ol{Rm} [h] + \cl{E}_1 + \cl{E}_2
		\qquad \text{ on } M \times (\tau_0, \tau_1 )
	\end{equation}
	where $| \mathbf p | \le \ol{p} e^{\lambda_* \tau_0}$.
	Assume 
		\begin{gather*}
			\| h \|_{L^\infty ( M \times [ \tau_0, \tau_1] )} \le \epsilon < 1,	\\
			h(\tau_0) = \eta_{\gamma_0}  \sum_{j = 1}^{K} p_j h_j ,
			\\
			\| h (\tau)  \|_{L^2_f (M) } \le \mu e^{\lambda_* \tau} ,  \\
			| \cl{E}_2 |_{\ol{g}} \le \frac{C_0}{\Gamma_0} e^{- \tau}, \,
			| \ol{\nabla} \cl{E}_2 |_{\ol{g}} \le \frac{C_0}{\Gamma_0^{3/2}} e^{- \frac{3}{2}\tau}, \text{ and } 
			\supp \cl{E}_2 \subset \{ e^\tau \le f \le \Gamma_0 e^\tau \}
		\end{gather*}
	where $C_0 = C_0 ( n, M, \ol{g}, f) \ge 1$ depends only on $n, M, \ol{g} , f$.
	
	Let $0 < \delta \le \epsilon$.
	If $0 < \ol{p} \ll 1$ is sufficiently small (depending on $n, M, \ol{g},f, \lambda_*$),
	$0 < \epsilon \ll 1$ is sufficiently small (depending on $n, M, \ol{g}, f, \lambda_*$),
	$0 < \mu \ll 1$ is sufficiently small (depending on $n, M, \ol{g}, f, \lambda_*$), 
	$0 < \gamma_0 \ll 1$ is sufficiently small (depending on $n, M, \ol{g}, f, \lambda_* , \delta$),
	$\tau_0 \gg 1$ is sufficiently large (depending on $n, M, \ol{g}, f, \lambda_*, \gamma_0, \ol{p}, \epsilon , \delta$),
	then $h$ satisfies the estimates 
		$$| h | + | \ol{\nabla} h | + | \ol{\nabla}^2 h| \le W f^{| \lambda_*|} e^{\lambda_* \tau} \text{ and}$$
		$$| h | + | \ol{\nabla} h | + | \ol{\nabla}^2 h| \le W' \delta$$
	throughout $M \times [ \tau_0 , \tau_1 ]$
	where $W = W(n, M, \ol{g}, f, \lambda_* )$
	and $W' = W'( n, M, \ol{g}, f)$.
\end{thm}

The proof is based on partitioning $M$ into three regions or scales $\{ f < \Gamma \}$, $\{ \Gamma < f < \gamma e^\tau \}$, and $\{ \gamma e^\tau < f < \Gamma e^\tau \}$,
with $\Gamma \gg 1$ sufficiently large and $0 < \gamma \ll 1$ sufficiently small.
For regions $\{ \Gamma < f < \gamma e^\tau \}$ and $\{ \gamma e^\tau < f < \Gamma e^\tau \}$, 
barriers provide $C^0$ estimates for $h$ and interior estimates bootstrap these $C^0$ estimates to $C^2$ estimates.
The estimates in the former region $\{ \gamma e^\tau < f < \Gamma e^\tau \}$ are simplified by the fact that $\cl{E}_2$ is supported on $\{ \gamma e^\tau < f < \Gamma e^\tau \}$ for suitable choices of the parameters $\gamma, \Gamma$.
To estimate $h$ on $\{ f < \Gamma \}$, we use local derivative estimates and the $L^2_f$ decay of $h$.

\subsection{Local Estimates}
\begin{lem}[Local Interior Estimates: $C^0_{loc}$ to $C^m_{loc}$] \label{Local Int Est C^0 to C^m_loc}
	Let $m \in \mathbb{N}$, $\Omega \Subset \Omega' \Subset M$, and $\tau_0 \le \tau_1 < \tau_2$.
	Assume 
		$$\partial_\tau h = \ol{\Delta}_f h + 2 \ol{Rm} [h] + \cl{E}_1 + \cl{E}_2	\qquad
		\text{on } \Omega' \times (\tau_0, \tau_2)$$
	and that $\cl{E}_2 \equiv 0$ on $\Omega' \times [\tau_0, \tau_2]$.

	If $\tau_1 > \tau_0$, then there exists $\epsilon > 0$ (depending on $n, m$) and $C$ (depending on $n, M , \ol{g}, f, m, \sup_{\Omega'} f,  \dist_{\ol{g}} ( \Omega, M \setminus \Omega' ) , \tau_1 - \tau_0$) 
	such that 
	\begin{gather*}
		\| h \|_{L^\infty( \Omega' \times [\tau_0, \tau_2])} < \epsilon	 \implies	
		\| h \|_{C^m( \Omega \times [\tau_1, \tau_2])} \le C \| h \|_{L^\infty( \Omega' \times [\tau_0, \tau_2])}.
	\end{gather*}
		
	If $\tau_1 = \tau_0$,
	then there exists $\epsilon > 0$ (depending on $n, m$) and
	$C$ (depending on $n, M, \ol{g}, f, m,\sup_{\Omega'} f, \dist_{\ol{g}} ( \Omega, M \setminus \Omega' )$)	
	such that
	 \begin{gather*}
	 	\| h \|_{L^\infty ( \Omega' \times [ \tau_0, \tau_2] ) } + \| h \|_{C^{m+1} ( \Omega' \times \{ \tau_0 \} ) } < \epsilon \implies	\\
		\| h \|_{C^m ( \Omega \times [ \tau_0, \tau_2] )} \le 
		C \left( \| h \|_{L^\infty( \Omega' \times [\tau_0, \tau_2])} + \| h \|_{C^{m+1} ( \Omega' \times \{ \tau_0 \} ) } \right).
	\end{gather*}
\end{lem}

\begin{proof}
	By assumption,
	\begin{gather*}   \begin{aligned}
		\partial_\tau h_{ij}	
		={}& \ol{\Delta} h_{ij} - \ol{\nabla}_{\ol{\nabla} f} h
		+ 2 \ol{R} \indices{_i^k_j^l} h_{kl}	\\
		&  + \ol{R} \indices{_{ja}^p_b} \left( \ol{g}_{ip} \tl{h}^{ab}  + \hat{h}^{ab} h_{ip} \right) 
		+ \ol{R} \indices{_{ia}^p_b} \left( \ol{g}_{jp} \tl{h}^{ab}  + \hat{h}^{ab} h_{jp} \right) 	\\
		& + g^{ab} g^{pq} ( \ol{\nabla} h * \ol{\nabla} h )_{abpqij}		
		- \hat h^{ab} \ol{\nabla}_a \ol{\nabla}_b h
	\end{aligned} 	\end{gather*}
	on $\Omega' \times (\tau_0, \tau_2) $.
	
	In the case where $\tau_1 > \tau_0$, define
		$$r_0 \doteqdot \frac{1}{2} \min \left\{ \dist_{\ol{g}} ( \Omega, M \setminus \Omega' ) , \sqrt{ \tau_1 - \tau_0 } \right\} > 0.$$
	By Klingenberg's Theorem, there exists $0 < r \le r_0$ depending on $n$ and $\sup_M |\ol{Rm}|_{\ol{g}} < \infty$ such that, for all $p \in M$, $B_{\ol{g}} ( p, 2r)$ contains no points conjugate to $p$.
		
	At each point $p \in \Omega$ and each $\tau \in [ \tau_1, \tau_2]$, consider the domains $B_{\ol{g}} ( p, r ) \times [ \tau - r^2, \tau ] \subset B_{\ol{g}} ( p, 2 r ) \times [ \tau - (2 r)^2, \tau ] \subset \Omega' \times [ \tau_0, \tau_2]$.
	By passing to a local cover if necessary, we may assume the exponential map at $p$ is a diffeomorphism to $B_{\ol{g}} ( p, 2r)$.
	After pulling back under this diffeomorphism, $h$ satisfies a differential equation for which Proposition \ref{Prop Nonlin Int Est} applies with $F = 0$.
	Since $\ol{Rc} + \ol{\nabla}^2 f = \frac{1}{2} \ol{g}$ and $\ol{R} + | \ol{\nabla} f|^2_{\ol{g}} = f$, the coefficients in this differential equation and their derivatives can be bounded in terms of $\ol{Rm}$, its derivatives $\ol{\nabla}^l \ol{Rm}$, and $\sup_{\Omega'} f$.
	Proposition \ref{Prop Nonlin Int Est} with $F = 0$ thus yields the desired result.
	
	In the case where $\tau_1 = \tau_0$, define
		$$r_0 \doteqdot \frac{1}{2} \dist_{\ol{g}} ( \Omega, M \setminus \Omega' ) > 0.$$
	Again, take $0 < r \le r_0$ depending on $n$ and $\sup_M | \ol{Rm} |_{\ol{g}} < \infty$ such that, for all $p \in M$, $B_{\ol{g}} ( p, 2r)$ contains no points conjugate to $p$.
	
	At each point $p \in \Omega$ and $\tau \in [ \tau_0 + (2r)^2, \tau_2]$, we can invoke the previous case with $\tau_1 = \tau_0 + (2r)^2 > \tau_0$.
	If instead $\tau \in [ \tau_0, \tau_0 + (2r)^2 ] \cap [ \tau_0, \tau_2] \doteqdot [ \tau_0, \tau_2']$, the result follows from similar logic as the previous case using instead the domains $B_{\ol{g}} ( p, r ) \times [ \tau_0, \tau_2' ] \subset B_{\ol{g}} ( p , 2r) \times [ \tau_0 , \tau_2' ] \subset \Omega' \times [\tau_0, \tau_2]$ and Proposition \ref{Prop Nonlin Int Est+} with $F = 0$. 
\end{proof}

\begin{lem}[Local Interior Integral Estimates] \label{Local Int Integral Est+}
	Let $m \in \mathbb{N}$, $\Omega \Subset \Omega' \Subset M$,
	and $\tau_0 \le \tau_1 < \tau_2$.
	Assume 
		$$\partial_\tau h = \ol{\Delta}_f h + 2 \ol{Rm} [h] + \cl{E}_1 + \cl{E}_2	\qquad
		\text{on } \Omega' \times (\tau_0, \tau_2)$$
	and that $\cl{E}_2 \equiv 0$ on $\Omega' \times [\tau_0, \tau_2]$.
	
	If $\tau_1 > \tau_0$, then there exists $\epsilon > 0$ (depending only on $n, m$)
	and $C$ (depending on $n, M , \ol{g}, f, m , \Omega, \Omega', \tau_1 - \tau_0$) 
	such that
	\begin{gather*}
		\| h \|_{C^{m+1}( \Omega' \times [\tau_0, \tau_2] )} < \epsilon
		  \implies \\
		\sup_{\tau \in [ \tau_1, \tau_2] } \| \nabla^{m}  h(\tau)   \|_{L^2_f ( \Omega ) } 
		+ \| h  \|_{L^2 H^{m+1}_f ( \Omega \times [ \tau_1, \tau_2] ) }	
		\le C 
		 \| h  \|_{L^2 H^m_f ( \Omega' \times [ \tau_0, \tau_2]) }.
	\end{gather*}
		
	If $\tau_1 = \tau_0$, then there exists $\epsilon > 0$ (depending only on $n, m$) and 
	$C$ (depending only on $n, M, \ol{g}, f, m , \Omega, \Omega' $) such that
	\begin{gather*}
		\| h \|_{C^{m+1}( \Omega' \times [\tau_0, \tau_2] )} < \epsilon \
		\implies \\
		\sup_{\tau \in [ \tau_0, \tau_2] } \| \nabla^m  h( \tau)   \|_{L^2_f ( \Omega  ) }
		+ \| h   \|_{L^2 H^{m+1}_f ( \Omega \times [ \tau_0, \tau_2] ) } 	\\
		\le C   \left( \| h( \tau_0 )  \|_{H^m_f ( \Omega' ) } + 
		 \| h  \|_{L^2 H^m_f ( \Omega' \times [ \tau_0, \tau_2]) } \right).
	\end{gather*}
\end{lem}
\begin{proof}
	For the proof we will omit overlines and it will be understood that all derivatives and curvatures are taken with respect to $\ol{g}$.
	
	We first show that $h$ and its derivatives satisfy evolution equations of the form
		$$\partial_\tau \nabla^m h = \Delta_f \nabla^m h + \sum_{l = 0}^{m+1} B_l^m * \nabla^l h
		- \div_f ( \hat h * \nabla^{m+1} h ) $$
	where the tensors $B_l^m$ satisfy certain estimates specified later in the proof.
	
	First, recall that by equation \eqref{h(tau) Evol Eqn} with $\cl{E}_2 \equiv 0$ on $\Omega'$
	\begin{gather*}
	\begin{aligned}
		\partial_\tau h ={}& \Delta_f h - \hat{h}^{ab} \nabla_a \nabla_b h  + 2 Rm[h] \\
		& + (g+h)^{-1}  (g+h)^{-1} ( \nabla h * \nabla h ) + Rm * \tl{h} + Rm * \hat{h} * h	,
	\end{aligned} \\
		\text{where } (g+h)^{ab} = \ol{g}^{ab} - \hat{h}^{ab}		\ \text{and} \
		\hat{h}^{ab} = \ol{g}^{ak} \ol{g}^{bl} h_{kl} + \tl{h}^{ab}.
	\end{gather*}	
	
	This evolution equation can be rewritten as
	\begin{equation} \label{h Evol Eqn with B}
		\partial_\tau h = \Delta_f h  + B^0_1 * \nabla h + B^0_0 *  h  - \div_f ( \hat h * \nabla h) 
	\end{equation}
	where
	\begin{gather*}
		(\hat h * \nabla h)\indices{^i_{jk}} = \hat{h}^{ip} \nabla_p h_{jk}	,	\\
		B^0_1 \doteqdot \nabla h - \hat{h} * \nabla h - \hat{h} * \hat{h} * \nabla h + \nabla \hat{h} - \nabla f * \hat{h}, 
		\text{ and}	\\
		B^0_0 \doteqdot  2 Rm* + Rm * \hat{h} + Rm * T			\\ \text{ where } 
		T^{akbl}  \doteqdot \ol{g}^{ak} ( \ol{g} + h )^{bl} -  \ol{g}^{ak} \ol{g}^{bl}  \text{ so that } 
		 \tl{h}^{ab} = T^{akbl}  h_{kl}.		
	\end{gather*}

	Observe that if $0 < \epsilon \ll 1$ is sufficiently small depending only on $n$,
	then there exists a constant $C_0$ depending on $n, M, \ol{g}, f, \Omega'$ such that
	\begin{gather*} \begin{aligned}
		\| h(\tau) \|_{C^1 ( \Omega' ) } < \epsilon
		\implies{}& \| B^0_1\|_{L^\infty( \Omega')} + \| B^0_0\|_{L^\infty ( \Omega')} 
		< C_0 	.
	\end{aligned}	\end{gather*}
	
	Taking derivatives $\nabla^m$ of the evolution equation \eqref{h Evol Eqn with B} for $h$ yields
	\begin{align*}
		&\partial_\tau \nabla^m h \\
		={}& \nabla^m \Delta_f h + \nabla^m ( B^0_1 * \nabla h ) + \nabla^m ( B^0_0 * h) 
		- \nabla^m \div_f ( \hat{h}^{ip} \nabla_p \cl{H}_{jk} )	\\
		={}& \Delta_f \nabla^m h + \sum_{l=0}^m \nabla^l Rm * \nabla^{m-l} h	\\
		&+ \sum_{l = 1}^m \nabla f * \nabla^{l-1} Rm * \nabla^{m-l} h	
		+ \sum_{l =1}^m \nabla^{l+1} f * \nabla^{m - l + 1} h	\\
		&+ \sum_{l = 0}^m \nabla^l B^0_1 * \nabla^{m - l + 1 }  h
		+ \sum_{l = 0}^m \nabla^l B^0_0 * \nabla^{m- l  }  h		\\
		&- \div_f \left( \sum_{l=0}^m \nabla^l \hat{h} * \nabla^{m - l + 1} h \right)
		+ \sum_{l = 0}^{m-1} \nabla^l Rm * \nabla^{m -1- l} ( \hat{h} * \nabla h )	.
	\end{align*}
	By the definition of $B_0^0$ and $B_0^1$, this equation may be rewritten as
		$$\partial_\tau h = \Delta_f \nabla^m h + \sum_{l = 0}^{m+1} B_l^m * \nabla^l h
		- \div_f ( \hat h * \nabla^{m+1} h ) $$
	where the tensors $B_l^m$ are continuous functions of $h, \nabla h, \dots , \nabla^{m+1} h$, and $x \in M$.
	Hence, if $0 < \epsilon \ll 1$ is sufficiently small depending on $n, m$, there exists a constant $C_m$ depending on $n, M, \ol{g}, f, m, \Omega'$ such that
	\begin{gather*}
		\| h(\tau) \|_{C^{m+1}( \Omega' ) } < \epsilon
		\implies
		\sum_{l = 0}^{m+1} \| B^m_l\|_{L^\infty( \Omega')}  < C_m 	.
	\end{gather*}	
		
	Now, assume we are in the case where $\tau_1 > \tau_0$.
	Take a bump function $\eta : M \times \R \to [0,1]$ with
		$$\eta \equiv 1 \text{ on } \Omega \times [ \tau_1, \tau_2] 
		\, \text{ and } \, \supp( \eta ) \subset \Omega' \times [ \tau_0, +\infty ) $$
	and assume $\eta$ is suitably smooth so that its derivatives $\nabla^m \eta$ are bounded by constants that depend only on $n, M, \ol{g}, m, \Omega, \Omega', \tau_1 - \tau_0$.
	
	From the evolution equation for $\nabla^m h$, we now estimate the time derivative of 
		$$\int \eta^2 | \nabla^m h |^2 e^{-f} = \int_M \eta^2 | \nabla^m h |^2 e^{-f} dV_{\ol{g}}.$$
	\begin{gather*} \begin{aligned}
		&  \frac{d}{d \tau} \frac{1}{2} \int \eta^2  | \nabla^m h |^2 e^{-f}
		\\
		={}& \int \eta^2 \langle \nabla^m h, \partial_\tau \nabla^m h \rangle e^{-f} 
		+ \int \eta \eta_\tau | \nabla^m h |^2 e^{-f}		
		\\
		={}& \int \eta^2 \langle \nabla^m h , \Delta_f \nabla^m h \rangle e^{-f}
			+ \int \eta \eta_\tau | \nabla^m h |^2 e^{-f}		\\
		&  + \int \eta^2 \langle \nabla^m h, B^m_{m+1} * \nabla^{m+1} h \rangle e^{-f}
		+ \sum_{l=0}^{m} \int \eta^2 \langle \nabla^m h, B^m_{l} * \nabla^{l} h \rangle e^{-f}		\\
		&  - \int \eta^2 \langle \nabla^m h, \div_f ( \hat{h} * \nabla^{m+1} h) \rangle e^{-f}	\\
		\le{}& - \int \eta^2 | \nabla^{m+1} h |^2 e^{-f} +  C(n,m) \int \eta | \nabla \eta | | \nabla^m h | |\nabla^{m+1} h | e^{-f}
		+ \int \eta \eta_\tau | \nabla^m h |^2 e^{-f}		\\
		&  +C(n,m)  \int \eta^2 | \nabla^m h | | B^m_{m+1}| | \nabla^{m+1} h | e^{-f}
		+ C(n,m) \sum_{l=0}^{m} \int \eta^2 | \nabla^m h | | B^m_{l} | |\nabla^{l} h |  e^{-f}		\\
		&  + C(n,m) \int \eta^2 | \nabla^{m+1} h | | \hat h | | \nabla^{m+1} h | e^{-f}
		+ C(n,m) \int 2 \eta | \nabla \eta | | \nabla^{m} h | | \hat h | | \nabla^{m+1} h | e^{-f}.
	\end{aligned} \end{gather*}	
	Invoking Young's inequality, for any $\delta > 0$, the above expression can be bounded by
	\begin{gather*} \begin{aligned}	
		&  \frac{d}{d \tau} \frac{1}{2} \int \eta^2  | \nabla^m h |^2 e^{-f}
		\\
		\le{}& - \int \eta^2 | \nabla^{m+1} h |^2 e^{-f} 
		+ \delta \int \eta^2 | \nabla^{m+1} h  |^2 e^{-f} + \frac{C(n,m) }{4 \delta} \int | \nabla \eta|^2 | \nabla^m h  |^2 e^{-f} \\	
		& + \int \eta \eta_\tau | \nabla^m h |^2 e^{-f}	\\
		& + \delta \int \eta^2 | \nabla^{m+1} h |^2 e^{-f} 
		+ \frac{C(n,m)}{4 \delta} \| B^m_{m+1}(\tau) \|^2_{L^\infty( \Omega' )} \int \eta^2 | \nabla^m h |^2 e^{-f}	\\
		&  + C(n, m) \left(  \sum_{l=0}^m \| B^m_l(\tau)  \|_{L^\infty( \Omega' )} \right)  \| h (\tau) \|^2_{H^m_f( \Omega' )}	\\
		&  + C(n, m) \| \hat h(\tau) \|_{L^\infty( \Omega' ) }  \int \eta^2 | \nabla^{m+1} h |^2 e^{-f} 
		+ C(n, m)  \| \hat h(\tau) \|_{L^\infty( \Omega' ) } \int | \nabla \eta |^2 | \nabla^m h |^2  e^{-f}		\\
		\le{}& \left( - 1 + 2\delta + C(n,m) \| \hat h(\tau) \|_{L^\infty ( \Omega' )} \right)
		 \int \eta^2 | \nabla^{m+1} h |^2 e^{-f} \\
		& + C(n,m) \left( \frac{ \| \nabla \eta \|_\infty^2}{ \delta} + \| \eta \eta_\tau \|_\infty + \frac{  \| B^m_{m+1} \|_{L^\infty( \Omega' )}^2 }{ \delta}  + \| \hat h  \|_{L^\infty( \Omega' )} \| \nabla \eta \|_\infty^2	 \right) 
		  \int_{\Omega' } | \nabla^m h |^2 e^{-f} 	\\
		&  + C(n,m)  \left(  \sum_{l=0}^m \| B^m_l \|_{L^\infty( \Omega' )} \right)  \| h \|^2_{H^m_f ( \Omega' )}.
	\end{aligned} \end{gather*}
	Take $0 < \epsilon \ll 1$ sufficiently small depending on $n, m$ so that
		$$\| h (\tau) \|_{C^{m+1} ( \Omega')} < \epsilon $$
	implies
	\begin{gather*}
		C(n,m) \| \hat h(\tau)  \|_{L^\infty ( \Omega' ) } \le \frac{1}{2} \quad \text{ and } \quad
		\sum_{l = 0}^{m + 1} \| B^m_l \|_{L^\infty( \Omega' ) } \le C(n, M, \ol{g}, f, m , \Omega' ). 
	\end{gather*}	
	Taking $\delta = 1/8$, it then follows that 
	\begin{gather*} 
		\frac{d}{d\tau } \int \eta^2 | \nabla^m h |^2 e^{-f} dV
		+ \frac{1}{4} \int \eta^2 | \nabla^{m+1} h |^2 e^{-f}	
		\le C \sum_{l = 0}^m \int_{\Omega'} | \nabla^l h |^2 e^{-f} dV
	\end{gather*}
	for some constant $C = C( n, M, \ol{g}, f, m, \Omega', \Omega , \tau_1 - \tau_0 )$
	whenever $\| h(\tau)  \|_{C^{m+1} ( \Omega' ) } < \epsilon.$
	Integrating this estimate from $\tau_0$ to $\tau \in [\tau_1, \tau_2]$ 
	under the assumption that $\| h \|_{C^{m+1}( \Omega' \times [ \tau_0, \tau_2 ] ) } < \epsilon$
	then implies
	\begin{align*}
		& \quad \, \int_\Omega | \nabla^m h (\tau)|^2 e^{-f} dV
		+ \frac{1}{4} \int_{\tau_1}^\tau \int_\Omega | \nabla^{m+1} h |^2 e^{-f} dV d \tl{\tau} 	\\
		&\le \int_M \eta^2 | \nabla^m h (\tau)|^2 e^{-f} dV
		+ \frac{1}{4} \int_{\tau_0}^\tau \int_M \eta^2 | \nabla^{m+1} h |^2 e^{-f} dV d \tl{\tau} 	
		&& ( \text{choice of $\eta$} )\\
		&\le
		C \iint_{\Omega' \times [ \tau_0 , \tau_2 ]} \sum_{l = 0}^m | \nabla^l h |^2 e^{-f} dV d \tl{\tau} .
	\end{align*}
	Since $\tau \in [ \tau_1, \tau_2]$ was arbitrary, it follows that $\| h \|_{C^{m+1} ( \Omega' \times [ \tau_0, \tau_2] )} < \epsilon $ implies
	\begin{gather*}
		\sup_{\tau \in [ \tau_1, \tau_2] } \| \nabla^m h \|_{L^2_f( \Omega ) } 
		+ \| h \|_{L^2 H^{m+1}_f( \Omega \times [ \tau_1, \tau_2 ] ) }	
		\le C \| h  \|_{L^2 H^m_f ( \Omega' \times [ \tau_0, \tau_2] ) }.
	\end{gather*}
	for some $C$ depending on $n, M, \ol{g}, f, m, \Omega', \Omega, \tau_1 - \tau_0$.

	In the case where $\tau_1 = \tau_0$, take instead the bump function $\eta : M \to [ 0,1]$ to be independent of time $\tau$ with
		$$\eta \equiv 1 \text{ on } \Omega \text{ and } \supp ( \eta) \subset \Omega' .$$
	and assume $\eta$ is suitably smooth so that its derivatives $\nabla^m \eta$ are bounded by constants that depend only on $n, M , \ol{g}, m , \Omega, \Omega'$.
	Then the same estimate as above (with $\partial_\tau \eta = 0$) applies to show that 
	\begin{gather*}
		\frac{d}{d\tau } \int \eta^2 | \nabla^m h |^2 e^{-f} dV
		+ \frac{1}{4}  \int \eta^2 | \nabla^{m+1} h |^2 e^{-f} dV	
		\le C \sum_{l = 0}^m \int_{\Omega'} | \nabla^l h |^2 e^{-f} dV
	\end{gather*}
	for some constant $C = C( n, M, \ol{g}, f, m, \Omega', \Omega  )$
	whenever $\| h(\tau)  \|_{C^{m+1} ( \Omega' ) } < \epsilon.$
	Integrating from $\tau_0$ to $\tau \in [\tau_0, \tau_2]$, it follows that
	\begin{gather*}	 \begin{aligned}
		 & \int_{\Omega} | \nabla^m h (\tau) |^2 e^{-f}
		 + \frac{1}{4}  \int_{\tau_0}^\tau \int_{\Omega} | \nabla^{m+1} h |^2 e^{-f} dV d \tl{\tau}	\\
		 \le{}&  \int_{\Omega'}  | \nabla^m h (\tau_0) |^2 e^{-f}
		+  C \int_{\tau_0}^\tau \int_{\Omega'} \sum_{l=0}^m | \nabla^l h |^2 e^{-f} dV \, d \tl{\tau}
	\end{aligned} 	\end{gather*}
	whenever $\| h  \|_{C^{m+1} ( \Omega'  \times [ \tau_0, \tau_2] ) } < \epsilon$ as desired.
\end{proof}

\begin{lem} \label{Lem Local L^2_f to C^2 Est}
	Let $\Gamma > 0$ and $\tau_0 < \tau_1$.
	Assume
		$$\partial_\tau h = \ol{\Delta}_f h + 2 \ol{Rm} [ h] + \cl{E}_1 + \cl{E}_2
		\qquad \text{ on } M \times ( \tau_0, \tau_1 )$$
	and that $ \cl{E}_2 = 0 $ on $\{ f < 2 \Gamma \} \times [ \tau_0, \tau_1]$.
	Assume also that 
	\begin{gather*} \begin{aligned}
		\| h \|_{L^\infty ( M \times [\tau_0 , \tau_1] ) } &< \epsilon, \\
		\| h( \tau )  \|_{L^2_f ( M) } &\le \mu e^{\lambda_* \tau }	
		&& \text{ for all } \tau \in [\tau_0, \tau_1],
		\text{ and} \\
		\sum_{l = 0}^{4n+2}  | \ol{\nabla}^l h |_{\ol{g}} (x, \tau_0)  
		&\le  C_0 \ol{p} e^{\lambda_* \tau_0}
		&& \text{ for all } x \in \{ f < 2 \Gamma \}, \tau = \tau_0.
	\end{aligned} \end{gather*}
		
	If
	$\Gamma \gg 1$ is sufficiently large (depending on $n, M, \ol{g}, f$),
	$0 < \epsilon \ll 1$ is sufficiently small (depending on $n, M ,\ol{g}, f, \Gamma$), 
	$\tau_0 \gg 1$ is sufficiently large (depending on $n, M, \ol{g}, f, \lambda_* , \Gamma, C_0 \ol{p}$),
	then there exists $C_1$ (depending on $n, M, \ol{g}, f, \lambda_*, \Gamma$) such that
		$$ \sum_{l=0}^2 | \ol{\nabla}^l h |_{\ol{g}}(x, \tau)   
		\le C_1 (\mu +  C_0\ol{p}  )  e^{\lambda_* \tau} 
		\quad \text{ for all } (x, \tau) \in \{ f < \Gamma \}  \times [ \tau_0, \tau_1].$$
\end{lem}

\begin{proof}
	Consider the domains
		$$\Omega_i \doteqdot \left \{ x \in M : f(x) < \left( 1 + \frac{i}{4n+3} \right) \Gamma \right \} 
		\qquad \text{for } i = 0, 1, \dots, 4n+3.$$	
	If $\Gamma \gg 1$ is sufficiently large depending on $n, M, \ol{g} , f$, then 
	$$\emptyset \ne \{ f < \Gamma \} = \Omega_0 \Subset \Omega_1 \Subset \dots \Subset \Omega_{4n+3} = \{ f < 2 \Gamma \} \Subset M.$$
	Note that $\cl{E}_2 \equiv 0$ on $\Omega_{4n+3} \times [ \tau_0, \tau_1]$ by assumption.
	Observe that the Sobolev inequality\footnote{See \cite{Aubin98} for example.} on $\Omega_0 = \{ f < \Gamma \}$ implies
	\begin{equation} \label{Local Sobolev Ineq+}
		\sup_{x \in \Omega_0} \sum_{l = 0}^{2} 
		| \ol{\nabla}^l h | (x, \tau) 
		\lesssim_{n, M, \ol{g}, f, \Gamma} \| h(\tau) \|_{H^{4n}_f ( \Omega_0 ) } 
		\quad \text{for all } \tau \in [\tau_0, \tau_1].
	\end{equation}
	Hence, it suffices to prove
		$$\| h(\tau) \|_{H^{4n}_f( \Omega_0 ) } \lesssim_{n, M, \ol{g}, f, \lambda_*, \Gamma}
		(\mu +   C_0 \ol{p} ) e^{\lambda_* \tau} 
		\qquad \text{for all } \tau \in [ \tau_0, \tau_1 ].$$
	There are two cases to consider according to whether or not $\tau \ge \tau_0 + 1$ as well.
	
	Consider the first case where $\tau \in [ \tau_0 + 1, +\infty) \cap [ \tau_0 , \tau_1] $.
	Then
	\begin{equation*}
		\| h (\tau) \|_{H^{4n}_f ( \Omega_0 ) }
		\le \| h  \|_{L^\infty H^{4n}_f ( \Omega_1  \times [ \tau - \frac{1}{4n+3} , \tau ]  ) }
	\end{equation*}
	Observe that if $0 < \epsilon \ll 1$ is sufficiently small depending on $n$, then Lemma \ref{Local Int Est C^0 to C^m_loc} implies that for all $i \in \{ 2, 3, \dots, 4n+2 \}$
	\begin{align*}
		\| h \|_{C^{4n+3 - i} ( \Omega_i \times [ \tau - \frac{i}{4n+3}, \tau] ) }
		&\le \| h \|_{C^{4n+1} ( \Omega_{4n+2} \times [ \tau - \frac{ 4n+2 }{4n+3} , \tau ] ) }	\\
		&\lesssim_{n, M, \ol{g}, f, \Gamma} \| h \|_{L^\infty ( \Omega_{4n+3} \times [ \tau - \frac{ 4n+3 }{4n+3} , \tau ] ) }	
		&& (\text{Lemma \ref{Local Int Est C^0 to C^m_loc}} )\\
		&\le \epsilon.
	\end{align*}
	In particular, if $0 < \epsilon \ll 1$ is sufficiently small depending on $n, M, \ol{g}, f, \Gamma$, 
	then Lemma \ref{Local Int Integral Est+} applies repeatedly to show
	\begin{align*}
		&\| h (\tau) \|_{H^{4n}_f ( \Omega_0 ) }	\\
		\le& \| h  \|_{L^\infty H^{4n}_f ( \Omega_1  \times [ \tau - \frac{1}{4n+3} , \tau ]  ) } 	\\
		\lesssim&_{n, M, \ol{g}, f, \Gamma} \| h \|_{L^2 H^{4n}_f ( \Omega_2 \times [ \tau - \frac{2 }{4n+3} , \tau] )}
		&& ( \text{Lemma \ref{Local Int Integral Est+}} )	\\
		\lesssim&_{n, M, \ol{g}, f,  \Gamma } \| h \|_{L^2 H^{4n-1}_f ( \Omega_3 \times [ \tau - \frac{3}{4n+3} , \tau] )}
		&& ( \text{Lemma \ref{Local Int Integral Est+}} )	\\
		& \vdots	\\
		\lesssim&_{n, M, \ol{g}, f, \Gamma} 
		\| h \|_{L^2 H^0_f ( \Omega_{4n+2} \times [ \tau - \frac{4n+2}{4n+3} , \tau] )}
		&& ( \text{Lemma \ref{Local Int Integral Est+}} )	\\
		=& \| h \|_{L^2 L^2_f ( \Omega_{4n+2} \times [ \tau - \frac{4n+2}{4n+3} , \tau] )} \\
		\le& \| h \|_{L^\infty L^2_f ( \Omega_{4n+2} \times [ \tau - \frac{4n+2}{4n+3} , \tau] )}
		&& \left( \frac{4n+2}{4n+3} \le 1 \right)	\\ 
		\lesssim&_{\lambda_*}  \mu e^{\lambda_* \tau} .
	\end{align*}
	
	Now consider the second case where $\tau \in [ \tau_0, \tau_0 + 1 ] \cap [ \tau_0, \tau_1] = [ \tau_0, \min \{ \tau_0 +1, \tau_1 \} ]$. Denote $\tau_1' \doteqdot \min \{ \tau_0 + 1, \tau_1 \}$.
	Note that
		$$\| h \|_{C^{4n+2} ( \Omega_{4n+3} \times \{ \tau_0 \} ) } \le C_0 \ol{p} e^{\lambda_* \tau_0}$$
	by assumption.
	Hence, if $0 < \epsilon \ll 1$ is sufficiently small depending on $n$
	and $\tau_0 \gg 1$ is sufficiently large depending on $n, \lambda_*, C_0 \ol{p}$, 
	then Lemma \ref{Local Int Est C^0 to C^m_loc} implies that
	for all $i \in \{ 2, 3, \dots, 4n+2 \}$
	\begin{align*}
		\| h \|_{C^{4n+3 - i} ( \Omega_i \times [ \tau_0, \tau_1' ] ) }
		&\le \| h \|_{C^{4n+1} ( \Omega_{4n+2} \times [ \tau_0, \tau_1' ] ) }	\\
		&\lesssim_{n, M, \ol{g}, f, \Gamma}
		 \| h \|_{L^\infty ( \Omega_{4n+3} \times [ \tau_0, \tau_1' ] ) }	
		+ \| h \|_{C^{4n+2} ( \Omega_{4n+3} \times \{ \tau_0 \} ) }\\
		&\le \epsilon +   C_0 \ol{p} e^{\lambda_* \tau_0} 	.
	\end{align*}
	In particular, if $0 < \epsilon \ll 1$ is sufficiently small depending on $n, M, \ol{g}, f, \Gamma$
	and $\tau_0 \gg 1$ is sufficiently large depending on $n, M, \ol{g}, f, \lambda_*, \Gamma, C_0 \ol{p}$
	then Lemma \ref{Local Int Integral Est+} applies repeatedly to show
	\begin{align*}
		&\| h (\tau) \|_{H^{4n}_f ( \Omega_0 ) }	\\
		\le& \| h  \|_{L^\infty H^{4n}_f ( \Omega_1  \times [ \tau_0 , \tau_1' ]  ) } 	\\
		\lesssim&_{n, M, \ol{g}, f, \Gamma} 
		\| h(\tau_0) \|_{H^{4n}_f ( \Omega_2  ) } 
		+ \| h \|_{L^2 H^{4n}_f ( \Omega_2 \times [ \tau_0 ,  \tau_1' ] ) }
		&& ( \text{Lemma \ref{Local Int Integral Est+}} )\\
		\lesssim&_{n, M, \ol{g}, f, \Gamma} 
		\| h(\tau_0) \|_{H^{4n}_f ( \Omega_3  ) } 
		+ \| h \|_{L^2 H^{4n-1}_f ( \Omega_3 \times [ \tau_0 ,  \tau_1' ] ) }
		&& ( \text{Lemma \ref{Local Int Integral Est+}} )\\
		& \vdots \\
		\lesssim&_{n, M, \ol{g}, f, \Gamma} 
		\| h(\tau_0) \|_{H^{4n}_f ( \Omega_{4n+2}  ) } 
		+ \| h \|_{L^2 H^{0}_f ( \Omega_{4n+2} \times [ \tau_0 ,  \tau_1' ] ) }
		&& ( \text{Lemma \ref{Local Int Integral Est+}} )\\
		\le&  	\| h(\tau_0) \|_{H^{4n}_f ( \Omega_{4n+2}  ) } 
		+ \| h \|_{L^\infty L^2_f ( \Omega_{4n+2} \times [ \tau_0 ,  \tau_1' ] ) } \\
		\lesssim&_{n, M, \ol{g}, f, \Gamma} C_0 \ol{p} e^{\lambda_* \tau_0} 
		+ \mu e^{\lambda_* \tau_0} \\
		\lesssim&_{\lambda_*}
		C_0 \ol{p} e^{\lambda_* \tau} 
		+ \mu e^{\lambda_* \tau} .
	\end{align*}
\end{proof}

\subsection{Large Scale Barrier}

\begin{lem} \label{Lem |h| Evol Eqn}
	Assume $h$ satisfies \eqref{h(tau) Evol Eqn} on some spacetime domain $\Omega \times (\tau_0, \tau_1) \subset M \times ( \tau_0, \tau_1)$.
	If $0 < \epsilon \ll 1$ is sufficienly small (depending on $n$), then
		$$| h |_{\ol{g}}^2 \le \epsilon 	\qquad \text{ on } \Omega \times ( \tau_0, \tau_1)$$
	implies
	\begin{gather} \label{|h| Evol Eqn}  \begin{aligned} 
		\partial_\tau | h |_{\ol{g}}^2 
		&\le g^{ab} \ol{\nabla}_a \ol{\nabla}_b | h |_{\ol{g}}^2 - \ol{\nabla}_{\ol{\nabla} f} | h |_{\ol{g}}^2 
		+ 4 C_n ( 1 + \epsilon ) |\ol{Rm}|_{\ol{g}} | h |_{\ol{g}}^2 + | h|_{\ol{g}} | \cl{E}_2 |_{\ol{g}} \\
		&\le ( 1 + C_n | h |_{\ol{g}} ) \ol{\Delta} | h |_{\ol{g}}^2 - \ol{\nabla}_{\ol{\nabla} f} | h |_{\ol{g}}^2 + 4 C_n ( 1 + \epsilon ) |\ol{Rm}|_{\ol{g}} | h |_{\ol{g}}^2 + | h|_{\ol{g}} | \cl{E}_2 |_{\ol{g}} 
	\end{aligned}	\end{gather}
	on $\Omega \times ( \tau_0, \tau_1)$.
\end{lem}
\begin{proof}
	If $ 0 < \epsilon \ll 1$ is sufficiently small depending on $n$, then 
	$| h |_{\ol{g}}^2 \le \epsilon$ implies $g = \ol{g} + h$ and $\ol{g}$ are comparable.
	Moreover, \eqref{Inverse h Tensors} implies
	\begin{equation} \label{hat h, tilde h Ests}
		| \hat h |_{\ol{g}} \lesssim_n | h |_{\ol{g}} \qquad \text{ and } \qquad
		| \tilde h |_{\ol{g}} \lesssim_n | h |_{\ol{g}}^2 
	\end{equation}
	when $0 < \epsilon \ll 1$ is sufficiently small depending on $n$.
	
	We compute the evolution equation for $| h|^2_{\ol{g}}$ using \eqref{h(tau) Evol Eqn}:
	\begin{gather*} \begin{aligned}
		& \partial_\tau | h |^2	\\
		={}& 2 \langle h, \partial_\tau h \rangle_{\ol{g}}	\\
		={}& 2 \ol{g}^{ip} \ol{g}^{jq} h_{pq} g^{ab} \ol{\nabla}_a \ol{\nabla}_b h_{ij}
		- 2 \langle h, \ol{\nabla}_{\ol{\nabla} f} h \rangle_{\ol{g}} 
		+ 4 \ol{R} \indices{_i^k_j^l} h_{kl} h_{pq} \ol{g}^{ip} \ol{g}^{jq}		\\
		& \qquad
		+ 2 \ol{R} \indices{_{ja}^p_b} \left( \ol{g}_{ip} \tl{h}^{ab}  + \hat{h}^{ab} h_{ip} \right) h_{pq} \ol{g}^{ip} \ol{g}^{jq}
		+ 2 \ol{R} \indices{_{ia}^p_b} \left( \ol{g}_{jp} \tl{h}^{ab}  + \hat{h}^{ab} h_{jp} \right) 	h_{pq} \ol{g}^{ip} \ol{g}^{jq}	\\
		& \qquad
		+ 2 g^{ab} g^{pq} ( \ol{\nabla} h * \ol{\nabla} h )_{abpqij}	h_{pq} \ol{g}^{ip} \ol{g}^{jq}
		+ \langle h, \cl{E}_2 \rangle_{\ol{g}} \\
		={}& 
		g^{ab} \ol{\nabla}_a \ol{\nabla}_b | h |_{\ol{g}}^2 - 2 g^{ab} \langle \ol{\nabla}_a h, \ol{\nabla}_b h \rangle_{\ol{g}}
		- \ol{\nabla}_{\ol{\nabla} f} | h|^2_{\ol{g}} + 4 \ol{R}^{pkql}h_{kl} h_{pq}	\\
		& \qquad
		+ 2 \ol{R} \indices{_{ja}^p_b} \left( \ol{g}_{ip} \tl{h}^{ab}  + \hat{h}^{ab} h_{ip} \right) h_{pq} \ol{g}^{ip} \ol{g}^{jq}
		+ 2 \ol{R} \indices{_{ia}^p_b} \left( \ol{g}_{jp} \tl{h}^{ab}  + \hat{h}^{ab} h_{jp} \right) 	h_{pq} \ol{g}^{ip} \ol{g}^{jq}	\\
		& \qquad
		+ 2 g^{ab} g^{pq} ( \ol{\nabla} h * \ol{\nabla} h )_{abpqij}	h_{pq} \ol{g}^{ip} \ol{g}^{jq}
		+ \langle h, \cl{E}_2 \rangle_{\ol{g}} .
	\end{aligned} \end{gather*}	
	Invoking \eqref{hat h, tilde h Ests}, it follows that 	
	\begin{gather*} \begin{aligned}
		& \partial_\tau | h |^2 \\
		 \le{}& g^{ab} \ol{\nabla}_a \ol{\nabla}_b | h |_{\ol{g}}^2 - 2 ( 1 - C_n \epsilon) | \ol{\nabla} h |_{\ol{g}}^2
		- \ol{\nabla}_{\ol{\nabla} f} | h|^2_{\ol{g}} + 4 C_n | \ol{Rm} |_{\ol{g}} | h |_{\ol{g}}^2  
		+ 4 C_n | \ol{Rm} |_{\ol{g}} | h|_{\ol{g}}^3		\\
		& \qquad
		+ C_n | h|_{\ol{g}} | \ol{\nabla} h |_{\ol{g}}^2
		+ | h|_{\ol{g}} | \cl{E}_2 |_{\ol{g}}	\\
		={}& g^{ab} \ol{\nabla}_a \ol{\nabla}_b | h |_{\ol{g}}^2 - 2 ( 1 - 2C_n \epsilon) | \ol{\nabla} h |_{\ol{g}}^2
		- \ol{\nabla}_{\ol{\nabla} f} | h|^2_{\ol{g}} + 4 C_n ( 1 + \epsilon)  | \ol{Rm} |_{\ol{g}} | h |_{\ol{g}}^2 
		+ | h|_{\ol{g}} | \cl{E}_2 |_{\ol{g}}	.
	\end{aligned} 	\end{gather*}
	The first inequality in \eqref{|h| Evol Eqn} follows from assuming $\epsilon$ is sufficiently small so that $ 1 - 2 C_n \epsilon > 0$.
	\eqref{hat h, tilde h Ests} then implies the second inequality in \eqref{|h| Evol Eqn}.
\end{proof}

\begin{lem} \label{Lem Barrier Intermediate Scale}	
	For $A, B, \kappa > 0$, let
		$$u \doteqdot e^{-\kappa \tau} \left( A f^{\kappa} - B f^{\kappa - 1} \right).$$
	Then $u$ is a supersolution of
		$$\partial_\tau u \ge ( 1 + C_n \sqrt{u} ) \ol{\Delta} u - \ol{\nabla}_{ \ol{\nabla} f } u + 4 C_n ( 1 + \epsilon) | \ol{Rm} |_{\ol{g}} u $$
		$$\text{ on the region } \{ (x, \tau) \in M \times [\tau_0, \infty) :  \Gamma < f(x) < \gamma e^\tau \}$$
	if 
		$$ B  - A \left\{   \kappa \left( \kappa + \frac{n}{2} \right)  + 4 C_n ( 1 + \epsilon) \left( \sup_{M} |\ol{Rm} |_{\ol{g}} f \right) \right\} \ge \omega > 0, $$
		$$\Gamma \gg 1 \text{ is sufficiently large depending on } n , B,\kappa , \omega  \text{ and }$$
		$$0 < \gamma \ll 1 \text{ is sufficiently small depending on } n, M, \ol{g}, A, B, \kappa, \omega.$$
		
	Additionally, 
		$$u > C e^{- \kappa \tau} f^{\kappa} \quad \text{on } \{ f > \Gamma \}		\qquad
		\text{if } C + \frac{B}{\Gamma} \le A.$$
	In particular, $u > C e^{- \kappa \tau} f^{\kappa}$ throughout $\{ f > \Gamma \}$
	if $A > C$ and $\Gamma$ is sufficiently large (depending on $A,B,C$).
\end{lem}
\begin{proof}
	Throughout this proof, we omit overlines.	
	Recall by Proposition \ref{Prop Lapl f Eqn} that $\Delta_f f = \frac{ n}{2} - f$.
	It follows that
		$$ \nabla f^{\kappa} = \kappa f^{\kappa - 1} \nabla f,$$
		$$\nabla_i \nabla_j f^{\kappa} = \kappa f^{\kappa - 1} \nabla_i \nabla_j f + \kappa ( \kappa - 1) f^{\kappa - 2} \nabla_i f \nabla_j f,$$
	\begin{align*}
		\Delta f^{\kappa} &= \kappa f^{\kappa-1} \Delta f + \kappa ( \kappa - 1) f^{\kappa-2} | \nabla f |^2	\\
		 &= \frac{n}{2} \kappa f^{\kappa - 1} - R \kappa f^{\kappa - 1} + \kappa ( \kappa - 1) f^{\kappa - 1} \frac{ | \nabla f|^2}{f},
	\end{align*}
	and so
	\begin{align*}
		\Delta_f f^{\kappa} 
		&= \kappa f^{\kappa-1} \Delta f + \kappa ( \kappa - 1) f^{\kappa-2} | \nabla f |^2
		- \kappa f^{\kappa - 1} |\nabla f|^2	\\
		&= \kappa f^{\kappa - 1} \Delta_f f + \kappa ( \kappa - 1) f^{\kappa-2} | \nabla f |^2	\\
		&= \frac{n}{2} \kappa f^{\kappa -1} - \kappa f^{\kappa } 
		+ \kappa ( \kappa - 1) f^{\kappa-1}  \frac{| \nabla f |^2	}{f} .
	\end{align*}
	
	In particular,
	\begin{multline}
		\Delta_f u
		= - \kappa u + e^{- \kappa \tau } f^{\kappa - 1} \left[ A \frac{n}{2} \kappa - B + A \kappa ( \kappa - 1) \frac{| \nabla f|^2}{f} \right] 	\\
		- B e^{- \kappa \tau} f^{\kappa - 2}  \left[ \frac{n}{2} ( \kappa - 1) + ( \kappa - 1) ( \kappa - 2) \frac{ | \nabla f|^2}{f} \right].
	\end{multline}
	
	A straightforward computation reveals
	\begin{gather*} \begin{aligned}
		&\quad \partial_\tau u 
		- ( 1 + C_n \sqrt{u} ) \Delta u + \nabla_{\nabla f} u - 4 C_n ( 1 + \epsilon ) | Rm| u	\\
		& =  \partial_\tau u - \Delta_f u - 4 C_n( 1 + \epsilon)  |Rm| u - C_n \sqrt{u} \Delta u	\\
		& = - \kappa u + \kappa u + e^{- \kappa \tau} f^{\kappa - 1} \left[ B - A \kappa \frac{n}{2} - A \kappa ( \kappa - 1) \frac{ |\nabla f |^2}{f} \right] \\
		& \qquad 
		+ B e^{- \kappa \tau} f^{\kappa - 2}  \left[ \frac{n}{2} ( \kappa - 1) + ( \kappa - 1) ( \kappa - 2) \frac{ | \nabla f|^2}{f} \right]	\\
		& \qquad - 4 C_n( 1+ \epsilon)  | Rm | e^{- \kappa \tau} ( A f^{\kappa} - B f^{\kappa - 1}  ) \\
		& \qquad - C_n \sqrt{ u } \Delta u		\\
		& \ge e^{- \kappa \tau} f^{\kappa -1} \left[ B - A \kappa \frac{n}{2} - A \kappa ( \kappa - 1) \frac{ |\nabla f |^2}{f} \right] \\
		& \qquad + B e^{- \kappa \tau} f^{\kappa - 2}  \left[ \frac{n}{2} ( \kappa - 1) + ( \kappa - 1) ( \kappa - 2) \frac{ | \nabla f|^2}{f} \right]	\\
		& \qquad - 4 C_n( 1+ \epsilon)  | Rm | e^{- \kappa \tau}  A f^{\kappa}  \\
		& \qquad - C_n \sqrt{ u } A e^{- \kappa \tau} f^{\kappa - 1} 
		\left[ \frac{n}{2} \kappa - R \kappa + \kappa ( \kappa - 1) \frac{ | \nabla f|^2}{f} \right]		\\
		& \qquad + C_n \sqrt{ u } B e^{- \kappa \tau} f^{\kappa - 2} 
		\left[ \frac{n}{2} (\kappa-1) - R (\kappa - 1) + (\kappa-1) ( \kappa - 2) \frac{ | \nabla f|^2}{f} \right]		.
	\end{aligned} \end{gather*}
	Next, note that
		$$0 \le \frac{ | \nabla f |^2 }{f} \le 1$$
	by \eqref{Soliton Eqns} and Lemma \ref{Lem Nonflat and f>0}.
	Moreover, by the definition of $u$,
		$$u \le A e^{- \kappa \tau } f^\kappa .$$
	These inequalities imply
	\begin{gather*} \begin{aligned} 
		&  \partial_\tau u 
		- ( 1 + C_n \sqrt{u} ) \Delta u + \nabla_{\nabla f} u - 4 C_n ( 1 + \epsilon ) | Rm| u	\\
		 \ge{}& e^{- \kappa \tau} f^{\kappa - 1} \left[ B - A \kappa \frac{n}{2} - A \kappa^2 \right] 	\\
		&  + B e^{- \kappa \tau} f^{\kappa - 2} \left[- \frac{ n}{2}  - 3 \kappa \right]	\\
		&  - 4 C_n ( 1 + \epsilon) (\sup_M f |Rm |) A e^{- \kappa \tau} f^{\kappa-1} 	\\
		&  + C_n \sqrt{A} e^{- \kappa \tau / 2} f^{\kappa/2} A e^{- \kappa \tau} f^{\kappa - 1} \left[  -\frac{n}{2} \kappa - \kappa^2 \right]	\\
		&  + C_n \sqrt{A} e^{- \kappa \tau / 2} f^{\kappa/2} B e^{- \kappa \tau} f^{\kappa - 2} \left[ - \frac{n}{2} - R \kappa - 3 \kappa \right]	\\
		\ge {}& e^{-\kappa \tau} f^{\kappa - 1}
		\left[ B - A \kappa \frac{n}{2} - A \kappa^2 - 4C_n ( 1 + \epsilon) ( \sup_M f | Rm| ) A \right]	\\
		& - e^{-\kappa \tau} f^{\kappa - 1} \frac{B}{\Gamma}  \left[ \frac{n}{2} + 3 \kappa \right] 	
		- e^{-\kappa \tau} f^{\kappa - 1} \gamma^{\kappa / 2} C_n A^{3/2} \left[ \frac{n}{2} \kappa + \kappa^2 \right] \\
		&- e^{-\kappa \tau} f^{\kappa - 1} \gamma^{\kappa / 2} \frac{B}{\Gamma}  C_n \sqrt{A} \left[ \frac{n}{2} + R \kappa + 3 \kappa \right].
	\end{aligned} \end{gather*}
	This quantity is positive if 
		$$ B  - A \left\{   \kappa \left( \kappa + \frac{n}{2} \right)  + 4 C_n ( 1 + \epsilon) \left( \sup_{M} |Rm |_{g} f \right) \right\} \ge \omega > 0, $$
	$\Gamma \gg 1 $ is sufficiently large depending on $n, B, \kappa, \omega$,
	and $0 < \gamma \ll 1$ is sufficiently small depending on $n, M, \ol{g}, A, B, \kappa, \omega$.
	
	The final two statements of the lemma follow from straightforward estimates using the assumption that $f > \Gamma$.
\end{proof}

\begin{lem} \label{Lem Barrier Large Scale}
	For $a, B > 0$, let
		$$u \doteqdot a - \frac{B}{f}.$$
	Then $u$ is a supersolution of
		$$\partial_\tau u \ge ( 1 + C_n \sqrt{u} ) \ol{\Delta} u - \ol{\nabla}_X u + 4 C_n ( 1 + \epsilon ) |\ol{Rm}|_{\ol{g}} u + \frac{C_0}{\Gamma} e^{-\tau} \sqrt{u}$$
	$$\text{ on the region } \{ (x, \tau) \in M \times [ \tau_0, \infty) : \gamma e^\tau < f(x) < \Gamma e^\tau \} $$
	if $$B > 4 C_n ( 1 + \epsilon ) ( \sup f | Rm | ) a + C_0 \sqrt{a} $$
	and $\tau_0 \gg 1$ is sufficiently large (depending on $n, M, \ol{g}, f, C_0, \epsilon, \gamma, a, B$).
	
	Additionally,
		$$u > c \text{ on } \{ f > \gamma e^\tau \ge \gamma e^{\tau_0} \}
		\qquad \text{if } \frac{B}{\gamma} e^{- \tau_0} < a - c.$$
	In particular, for any $a > c$, $u > c$ throughout $\{ f > \gamma e^\tau > \gamma e^{\tau_0} \}$ if $\tau_0 \gg 1$ is sufficiently large (depending on $a , B, c, \gamma$).
\end{lem}
\begin{proof}
	We omit overlines throughout this proof.
	By mimicking the computations in the previous proof with $\kappa = 0$, it follows that
	\begin{gather*} \begin{aligned}
		&  \Delta_f u + 4C_n ( 1 + \epsilon ) | Rm | u + C_n \sqrt{ u } \Delta u	 + \frac{C_0}{\Gamma} e^{-\tau} \sqrt{ u} \\
		={}& \frac{ B n}{2} f^{-2} - B f^{-1} - 2 B f^{-2} \frac{ | \nabla f|^2}{f} 	\\
		&  + 4 (1+ \epsilon) C_n |Rm| a - 4 C_n ( 1 + \epsilon ) |Rm| \frac{B}{f}	\\
		&  - C_n \sqrt{ a - B/f} B \left\{ - \frac{n}{2} f^{-2} + R f^{-2} + 2 f^{-2} \frac{ |\nabla f|^2}{f} \right\} \\	
		&  + \frac{C_0}{\Gamma} e^{-\tau} \sqrt{ a - B/ f} \\
		\le{}& \frac{ B n}{2} f^{-2} - B f^{-1} + 4 C_n (1 + \epsilon) | Rm | a + B C_n \frac{ n}{2} f^{-2} \sqrt{ a } 
		+ \frac{C_0}{\Gamma} e^{-\tau} \sqrt{a}\\
		\le{}& f^{-1} \left[ - B + 4 C_n (1 + \epsilon ) \left( \sup f |Rm| \right) a \right] 
		+ f^{-2} \left[ \frac{B n}{2} + B C_n \frac{n}{2} \sqrt{ a} \right] 
		+ C_0 f^{-1} \sqrt{a} \\
		\le{}& f^{-1} \left[ - B + 4 C_n (1 + \epsilon ) \left( \sup f |Rm| \right) a + C_0 \sqrt{a} \right] 
		+ \frac{1}{\gamma e^{\tau_0}} f^{-1} \left[ \frac{B n}{2} + B C_n \frac{n}{2} \sqrt{ a} \right] .
	\end{aligned} \end{gather*}
	This quantity can be made negative by taking $B > 4 C_n ( 1 + \epsilon ) ( \sup f | Rm | ) a + C_0 \sqrt{a} $ and $\tau_0 \gg 1$ sufficiently large depending on $n, M, \ol{g}, f, C_0 , \epsilon , \gamma , a, B$.
	
	The final two statements of the lemma follow from straightforward estimates using the assumption that $f > \gamma e^{\tau_0}$.
\end{proof}

\subsection{Large Scale Interior Estimates}
The barriers from the previous subsection provide $C^0$ control of $h$ at large scales.
In this subsection, we prove interior estimates that permit us to bootstrap these $C^0$ bounds up to $C^2$ bounds on a smaller domain so long as the initial data is suitably smooth.
These interior estimates are based on interior estimates from Appendix \ref{Holder Ests in Eucl Space}.

\begin{lem} \label{Lem Int Est Intermediate Scale}
	Assume
		$$\partial_\tau h = \ol{\Delta}_f h + 2 \ol{Rm} [h] + \cl{E}_1  \qquad \text{on }
		\{ (x, \tau) \in M \times ( \tau_0, \tau_1) : \Gamma < f(x) < \gamma e^{\tau} \}.$$

	Assume also that 
		$$| h |_{\ol{g}} \le C_0 e^{- \kappa \tau} f^\kappa	\qquad 
		\text{on } \{ (x, \tau) \in M \times [ \tau_0, \tau_1] : \Gamma \le f(x) \le \gamma e^{\tau} \}$$
	and
		$$\sum_{m = 1}^3 | \ol{\nabla}^{m} h |_{\ol{g}} \le C_0 e^{- \kappa \tau_0 } f^\kappa	\qquad
		\text{on } \{ x \in M : \Gamma \le f(x) \le \gamma e^{\tau_0} \}.$$
	If $\Gamma \gg 1$ is sufficiently large (depending on $n, M, \ol{g}, f$) and $C_0 \gamma^\kappa \ll 1$ is sufficiently small (depending on $n, M, \ol{g}, f$) then
	\begin{multline*}
		| \ol{\nabla}^{m} h |_{\ol{g}} + | \ol{\nabla}^{m} h |_{\ol{g}} 
		\lesssim_{n, M, \ol{g}, f, \kappa} C_0 e^{- \kappa \tau} f^\kappa 		\\
		\text{on } \{ (x, \tau) \in M \times [ \tau_0, \tau_1 ] : 4 \Gamma \le f(x) \le \frac{1}{4} \gamma e^\tau \}.
	\end{multline*}
\end{lem}
\begin{proof}
	The unbounded coefficient $\ol{\nabla} f$ of the drift term $\ol{\nabla}_{\ol{\nabla} f } h$ precludes us from applying Proposition \ref{Prop Nonlin Int Est} or Proposition \ref{Prop Nonlin Int Est+} directly.
	To circumvent this issue, we appeal to a rescaling argument.
	
	Fix $\tau_* \in ( \tau_0, \tau_1]$.
	Consider the family $\cl{H}(s)$ of symmetric 2-tensors on $M$ given by
		$$\cl{H}(s) \doteqdot ( 1 - s) \phi_s^* \left[ h( \tau_* - \ln ( 1- s) ) \right]$$
	where $\phi_s: M \to M$ is defined as the family of diffeomorphisms solving \eqref{phi Evol Eqn}, that is
		$$\partial_s \phi_s = \frac{1}{1 - s} \ol{\nabla} f \circ \phi_s \qquad \text{ with }  
		\qquad \phi_0 = Id_M : M \to M.$$

	Then $\cl{H}(s)$ satisfies the evolution equation
	\begin{gather*}  \begin{aligned}
		& \quad \partial_s \cl{H} 	
		\\
		& = - \frac{1}{1-s} \cl{H} 
		+ (1-s) \frac{1}{1-s} \phi_s^* \partial_\tau h 
		+ (1-s)  \phi_{s}^* \cl{L}_{\partial_s \phi_s \circ \phi_s^{-1} } h		
		\\
		& = - \frac{1}{1-s} \cl{H} 
		+ \phi_s^* \partial_\tau h 
		+ \phi_{s}^* \cl{L}_{\ol{\nabla} f } h	
		\\
		& = - \frac{1}{1-s} \cl{H} 
		+  \phi_s^*
		\left\{ g^{ab} \ol{\nabla}_a \ol{\nabla}_b h - \ol{\nabla}_{\ol{\nabla} f} h + 2 \ol{Rm}[h] 	\right.	\\
		& \qquad + \ol{Rm}_{jab}^p \left( \ol{g}_{ip} \tl{h}^{ab} + \hat{h}^{ab} h_{ip} \right)	
		+ \ol{Rm}_{iab}^p \left( \ol{g}_{jp} \tl{h}^{ab} + \hat{h}^{ab} h_{jp} \right)				\\
		& \qquad + g^{ab} g^{pq} ( \ol{\nabla} h * \ol{\nabla} h)_{abpqij}  \\
		& \qquad \left. + \ol{\nabla}_{\ol{\nabla} f} h + h_{ip} \ol{\nabla}_j \ol{\nabla}^p f + h_{pj} \ol{\nabla}_i \ol{\nabla}^p f \right\}.
	\end{aligned} \end{gather*}
	Let $\check g \doteqdot (1-s) \phi_s^* \ol{g}$. Then 
	\begin{gather*} \begin{aligned}	
		& \quad \partial_s \cl{H}	\\
		& = ( (1-s) \phi_s^* g)^{ab} \check \nabla_a \check \nabla_b \cl{H} + 2 \check{Rm} [ \cl{H} ] + \check{Rc} * \cl{H} 	\\
		& \qquad +  ( (1-s) \phi_s^* g)^{ab} ( (1-s) \phi_s^* g)^{pq} ( \check \nabla \cl{H} * \check \nabla \cl{H} )_{abpqij}			\\
		& \qquad +  \frac{ 1}{1-s} \check{Rm}_{jab}^p \check{g}_{ip} ( \phi_s^* \tl{h} )^{ab} 
		+ \frac{1}{1-s} \check{Rm}_{jab}^p ( \phi_s^* \hat{h})^{ab} \cl{H}_{ip}		\\
		& \qquad +  \frac{ 1}{1-s} \check{Rm}_{iab}^p \check{g}_{jp} ( \phi_s^* \tl{h} )^{ab} 
		+ \frac{1}{1-s} \check{Rm}_{iab}^p ( \phi_s^* \hat{h})^{ab} \cl{H}_{jp}		
		\\
		& = ( (1-s) \phi_s^* g)^{ab} \check \nabla_a \check \nabla_b \cl{H} + 2 \check{Rm} [ \cl{H} ] + \check{Rc} * \cl{H} 	\\
		& \qquad + ( (1-s) \phi_s^* g)^{ab} ( (1-s) \phi_s^* g)^{pq} ( \check \nabla \cl{H} * \check \nabla \cl{H} )_{abpqij}	\\
		& \qquad +  \frac{ 1}{1-s} \check{Rm}_{jab}^p \check{g}_{ip} ( \phi_s^* \tl{h} )^{ab} 
		+ \frac{1}{1-s} \check{Rm}_{jab}^p ( \phi_s^* \hat{h})^{ab} \cl{H}_{ip}		\\
		& \qquad +  \frac{ 1}{1-s} \check{Rm}_{iab}^p \check{g}_{jp} ( \phi_s^* \tl{h} )^{ab} 
		+ \frac{1}{1-s} \check{Rm}_{iab}^p ( \phi_s^* \hat{h})^{ab} \cl{H}_{jp}		.	
	\end{aligned} \end{gather*}
	From the definition of $\tl{h}$ and $\hat h$ in \eqref{Inverse h Tensors} as
		$$g^{ab} = \ol{g}^{ab} - \hat{h}^{ab} \qquad \text{ and } \qquad  \tl{h}^{ab} = \hat{h}^{ab} - \ol{g}^{ak} \ol{g}^{bl} h_{kl},$$
	it follows that 
		$$\frac{ 1}{1 -s} ( \phi_s^* \tl{h} )^{ab} = \tl{\cl{H}}^{ab}		\qquad \text{ and } \qquad 
		\frac{ 1}{1 -s} ( \phi_s^* \hat{h} )^{ab} = \hat{\cl{H}}^{ab}$$
	where $\tl{\cl{H}}, \hat{\cl{H}}$ are defined by
		$$( ( 1- s) \phi_s^* g) ^{ab} = \check{g}^{ab} - \hat{\cl{H}}^{ab} \qquad \text{ and } \qquad 
		 \tl{\cl{H}}^{ab} = \hat{\cl{H}}^{ab} - \check{g}^{ak} \check{g}^{bl} \cl{H}_{kl}.$$
	Note also that $\check g(s) = ( 1 - s) \phi_s^* \ol{g}$ is the Ricci flow solution corresponding to the soliton metric $\check g(0) = \ol{g}$
	and thus remains smooth for all $s \in (-\infty , 1)$.
	Additionally,
		$$( 1 - s) \phi_s^* g = ( 1 - s) \phi_s^* \ol{g} + ( 1 - s) \phi_s^* h = \check g(s) +  \cl{H}(s).$$
	Therefore, the evolution equation for $\cl{H}$ may be written as
	\begin{gather} \label{clH Evol Eqn 1} \begin{aligned}
		\partial_s \cl{H}
		={}& \check \Delta \cl{H} + 2 \check{Rm} [ \cl{H} ] + \check{Rc} * \cl{H}	\\
		& - \hat{\cl{H}}^{ab} \check \nabla_a \check \nabla_b \cl{H}
		 + ( \check g - \hat{\cl{H}} )^{ab} ( \check g - \hat{\cl{H}} )^{pq}  ( \check \nabla \cl{H} * \check \nabla \cl{H} )_{abpqij}	\\
		&+   \check{Rm}_{jab}^p \check{g}_{ip}  \tl{\cl{H}}^{ab} 
		+  \check{Rm}_{jab}^p   \hat{\cl{H}}^{ab} \cl{H}_{ip}		
		 +   \check{Rm}_{iab}^p \check{g}_{jp}   \tl{\cl{H}}^{ab} 
		+  \check{Rm}_{iab}^p   \hat{\cl{H}}^{ab} \cl{H}_{jp}	.
	\end{aligned}	\end{gather}

	Let $s_0 = s_0 ( \tau_*, \tau_0) < 0$ be defined so that 
		$$\tau_* - \ln ( 1 - s_0) = \tau_0$$
	and set $s_0' \doteqdot \max \{ s_0, -1 \}$. 
	Recall from Lemma \ref{Lem Flow Est 1} that
		$$0 \le ( 1 - s) \partial_s  \phi_t^* f \le \phi_s^* f		\qquad \text{for all } s < 1.$$
	Integrating this derivative estimate from $s \le 0$ to $0$ yields
		$$ \frac{1}{1 - s} f( x) \le f( \phi_s (x) ) \le f(x)		\qquad \text{ for all } s \le 0.$$
	It follows that if 	
		$$(x,s) \in \left\{ 2 \Gamma < f(x) < \frac{1}{2} \gamma e^{\tau_*} \right\} \times [ s_0', 0 ]$$
	then 
	\begin{gather*}
		f( \phi_s (x) ) \ge \frac{1}{1 - s} f(x) >  \frac{2}{ 1- s} \Gamma \ge \Gamma	\qquad \text{and } \\
		f( \phi_s( x) ) \le f(x) < \frac{1}{2} \gamma e^{\tau_*} \le \gamma e^{ \tau_* - \ln ( 1 - s) }.
	\end{gather*}
	In other words,
	\begin{gather*}
		(x,s) \in \left\{ 2 \Gamma < f(x) < \frac{1}{2} \gamma e^{\tau_*} \right\} \times [ s_0', 0 ] \quad \text{ implies}	\\
		( \phi_s(x), \tau_* - \ln ( 1 - s) ) \in \{ (y, \tau) \in M \times \R : \Gamma < f(y) < \gamma e^\tau \} ,
	\end{gather*}
	the region where the assumed bounds on $h$ apply.
	Therefore,
		$$| \cl{H} |_{\check g  } (x, s)  
		= \left| h  \right|_{\ol{g}} ( \phi_s (x), \tau_* - \ln ( 1 - s) ) 
		\le C_0 e^{\kappa ( \tau_* - \ln ( 1 - s) ) } f ( \phi_s( x) )^\kappa
		\le C_0 \gamma^\kappa$$
	for all $(x, s) \in \{ 2 \Gamma < f(x) < \frac{1}{2} \gamma e^{\tau_*} \} \times [ s_0', 0 ]$.

	Since $f$ grows quadratically at infinity, we may assume $\Gamma \gg 1$ is sufficiently large depending on $n, M, \ol{g}, f$ so that, for any $p \in M$,
		$$f( p) > \Gamma \implies B_{\ol{g}} ( p, 1) \subset \left \{ \frac{1}{2} f(p) < f < 2 f(p) \right \}.$$ 
	Thus, the radius 1 neighborhood of $\{ 4 \Gamma < f < \frac{1}{4} \gamma e^{\tau_*} \}$ with respect to $\ol{g}$ is contained in $\{ 2 \Gamma < f < \frac{1}{2} \gamma e^{\tau_*} \}$, that is
		$$\bigcup_{x \in \{ 4 \Gamma < f < \frac{1}{4} \gamma e^{\tau_*} \} } B_{\ol{g}} ( x, 1) \subset
			\left\{ 2 \Gamma < f < \frac{1}{2} \gamma e^{\tau_*} \right\}.$$
	Since $\ol{Rm}$ decays quadratically at infinity, we may additionally assume $\Gamma \gg 1$ is sufficiently large depending on $n, M, \ol{g}, f$ so that, for all $p \in M$ with $f(p) > \Gamma$, $B_{\ol{g}} (p,1)$ contains no points conjugate to $p$.
			
	Let $p_* \in M$ such that $4 \Gamma < f(p_*) < \frac{1}{4} \gamma e^{\tau_*}$.
	After possibly passing to a local cover and pulling back to $\R^n$ using the exponential map based at $p_*$, 
	Proposition \ref{Prop Nonlin Int Est} or Proposition \ref{Prop Nonlin Int Est+} applies to equation \eqref{clH Evol Eqn 1} on the unit ball in $\R^n$.
	Note that the terms in the equation \eqref{clH Evol Eqn 1} are bounded independently of $p_*$ and $\tau_*$,
	so the resulting constants will depend only on $n, M, \ol{g}, f$.
	There are two possible cases to consider depending on whether or not $s_0 \ge -1$.
	
	In the first case where $ s_0 \ge -1$,
	Proposition \ref{Prop Nonlin Int Est+} with $F = 0$ implies
	that if $C_0 \gamma^\kappa \ll 1$ is sufficiently small depending on $n, M, \ol{g}, f$,
	there exists $C_1 = C_1( n, M, \ol{g}, f)$ such that
	\begin{align*}
		& | \ol{\nabla} h |_{\ol{g}} ( p_*, \tau_*) + | \ol{\nabla}^2 h|_{\ol{g}} ( p_*, \tau_*) 	\\
		={}& | \check \nabla \cl{H} |_{\check g}(p_* , s= 0) + | \check{\nabla}^2 \cl{H} |_{\check g} ( p_* , s = 0) 	\\
		\le{}& C_1 \left(  \sup_{ B_{\ol{g}}( p_*, r_0 ) \times [s_0, 0] } | \cl{H} |(x,s)
		+  \sup_{B_{\ol{g}}( p_*, r_0 ) } \sum_{m =1}^3 | \check{\nabla}^{m} \cl{H} |(x, s_0)  	\right) \\
		\le{}& C_1 \left(  \sup_{ \{ \frac{1}{2} f( p_*) < f < 2 f( p_*) \} \times [s_0, 0] } | \cl{H} |(x,s)
		+  \sup_{ \{ \frac{1}{2} f( p_*) < f < 2 f( p_*) \} } \sum_{m =1}^3 | \check{\nabla}^{m} \cl{H} |(x, s_0)  	\right) \\
		\le{}& C_1 \left(  \sup_{ \{ \frac{1}{4} f( p_*) < f < 4 f( p_*) \} \times [\tau_0, \tau_*] } | h |_{\ol{g}} (x,s)
		+  \sup_{ \{ \frac{1}{4} f( p_*) < f < 4 f( p_*) \} } \sum_{m =1}^3 \frac{1}{ ( 1 - s_0)^{m/2} } | \ol{\nabla}^m h |_{\ol{g}}(x, \tau_0)  	\right) \\
		\le{}& 4 C_1 \sup_{ \{ \frac{1}{4} f( p_*) < f < 4 f( p_*) \} \times [\tau_0, \tau_*] } C_0 e^{- \kappa \tau} f^\kappa	\\
		\lesssim& {}_{\kappa} C_1 C_0 e^{- \kappa \tau_*} f( p_*)^\kappa 
	\end{align*}
	where in the last line we used the fact that $s_0  \ge  -1$ implies $\tau_* - \tau_0 = \ln ( 1 - s_0 ) \le \ln ( 2)$.
	
	In the second case where $s_0 < -1$, Proposition \ref{Prop Nonlin Int Est} with $F = 0$ implies 
	that if $C_0 \gamma^\kappa \ll 1$ is sufficiently small depending on $n, M, \ol{g}, f$, 
	there exists $C_1 = C_1( n, M, \ol{g}, f)$ such that 
	$$| \check \nabla \cl{H} |_{\check g}(p_* , s= 0) + | \check{\nabla}^2 \cl{H} |_{\check g} ( p_* , s = 0) 	
		\le C_2   \sup_{B_{\ol{g}}( p_*, r_0 ) \times [-r_0^2, 0] } | \cl{H} |(x,s)$$
	The remainder of the proof of the second case follows by similar logic as in the proof of the first case.
	
	In both cases,
		$$| \ol{\nabla} h |_{\ol{g}} ( p_*, \tau_*) + | \ol{\nabla}^2 h|_{\ol{g}} ( p_*, \tau_*) 	 
		\lesssim_{n, M, \ol{g}, f, \kappa} C_0 e^{- \kappa \tau_*} f( p_*)^\kappa$$
	when $C_0 \gamma^\kappa \ll 1$ is sufficiently small depending on $n, M, \ol{g}, f$.
	Since $\tau_* \in ( \tau_0 ,\tau_1]$ and $p_* \in \{ 4 \Gamma < f < \frac{1}{4} \gamma e^{\tau_*} \}$ were arbitrary and the constants depend only on $n, M, \ol{g}, f, \kappa$, the statement of the lemma follows.
\end{proof}

\begin{lem} \label{Lem Int Est Large Scale}
	Assume	
		$$\partial_\tau h = \ol{\Delta}_f h + 2 \ol{Rm} [h] + \cl{E}_1 + \cl{E}_2		\qquad
		 \text{ on } M \times (\tau_0, \tau_1 ).$$
	Assume also that
		$$| h |_{\ol{g}}(x, \tau)  \le \delta 	\qquad
		\text{ on } \{ (x,\tau) \in M \times [ \tau_0, \tau_1 ] : \gamma e^{\tau} < f(x) \},$$
		$$\sum_{m = 1}^3 | \ol{\nabla}^m h |_{\ol{g}}(x, \tau_0)  \le \delta		\qquad
		\text{ on } \{ x \in M : \gamma e^{\tau_0 } \le f(x) \},$$
	and		
		$$| \cl{E}_2 |_{\ol{g}}(x,\tau)  + | \ol{\nabla} \cl{E}_2 |_{\ol{g}}(x, \tau) \le \delta		\qquad
		\text{ on } \{ (x,\tau) \in M \times [ \tau_0, \tau_1 ] : \gamma e^{\tau} < f(x) \}.$$
		
	If $\delta \ll 1$ is sufficiently small (depending on $n, M, \ol{g}, f$) 
	and $\tau_0 \gg 1$ is sufficiently large (depending on $n, M, \ol{g}, f, \gamma$),
	then 
		$$| \ol{\nabla} h |_{\ol{g}} (x, \tau) + | \ol{\nabla}^2 h |_{\ol{g}} (x, \tau) \lesssim_{n, M, \ol{g}, f} \delta	\qquad
		\text{ on } \{ (x, \tau) \in M \times [ \tau_0, \tau_1 ] : 4 \gamma e^{\tau} < f(x) \}.$$
\end{lem}

\begin{proof}
	The proof is analogous to that of Lemma \ref{Lem Int Est Intermediate Scale}.
	We outline the argument nonetheless.
	
	Fix $\tau_* \in ( \tau_0, \tau_1]$
	and consider
		$$\cl{H}( s) \doteqdot ( 1 - s) \phi_s^* h( \tau_* - \ln ( 1 - s) ).$$
	Then $\cl{H}(s)$ satisfies the evolution equation 	
	\begin{gather} \label{clH Evol Eqn 2} \begin{aligned}
		\partial_s \cl{H}
		={}& \check \Delta \cl{H} + 2 \check{Rm} [ \cl{H} ] + \check{Rc} * \cl{H}	
		+ \phi_s^* \cl{E}_2 \\
		& - \hat{\cl{H}}^{ab} \check \nabla_a \check \nabla_b \cl{H}
		 + ( \check g - \hat{\cl{H}} )^{ab} ( \check g - \hat{\cl{H}} )^{pq}  ( \check \nabla \cl{H} * \check \nabla \cl{H} )_{abpqij}	\\
		&+   \check{Rm}_{jab}^p \check{g}_{ip}  \tl{\cl{H}}^{ab} 
		+  \check{Rm}_{jab}^p   \hat{\cl{H}}^{ab} \cl{H}_{ip}		
		 +   \check{Rm}_{iab}^p \check{g}_{jp}   \tl{\cl{H}}^{ab} 
		+  \check{Rm}_{iab}^p   \hat{\cl{H}}^{ab} \cl{H}_{jp}	
	\end{aligned}	\end{gather}
	where $\check g(s) = ( 1 - s) \phi_s^* \ol{g}$.
	
	Let $s_0 = s_0( \tau_*, \tau_0) < 0$ be defined by
		$$\tau_* - \ln ( 1 - s_0) = \tau_0$$
	and set $s_0' \doteqdot \max \{ s_0 , -1 \}$.
	Then
	\begin{gather*}
		(x,s) \in \{ f > 2 \gamma e^{\tau_*} \} \times [ s_0', 0 ] \quad \text{ implies}	\\
		( \phi_s(x), \tau_* - \ln ( 1 - s) ) \in \{ (y, \tau) \in M \times \R :   f(y) >  \gamma e^\tau \}.
	\end{gather*}
	Hence,
	\begin{gather*}
		| \cl{H} |_{\check g} (x, s) \le \delta,
		\qquad
		| \phi_s^* \cl{E}_2 |_{\check g} (x, s) 
		= \frac{1}{1 - s} | \cl{E}_2 |_{\ol{g}} (\phi_s(x) , \tau_* - \ln ( 1 - s) ) \le \delta	,\\
		\text{and} \quad 
		| \check{\nabla} \phi_s^* \cl{E}_2 |_{\check g} (x, s) 
		= \frac{1}{(1 - s)^{3/2}} | \ol{\nabla} \cl{E}_2 |_{\ol{g}} (\phi_s(x) , \tau_* - \ln ( 1 - s) )
		\le \delta
	\end{gather*}
	for all $(x,s) \in \{   f(x)  > 2  \gamma e^{\tau_*} \} \times [ s_0', 0 ]$.

	Assume $\tau_0 \gg 1$ is sufficiently large depending on $n, M, \ol{g}, f, \gamma$ so that
		$$f( p ) > \gamma e^{\tau_0} 
		\implies
		B_{\ol{g} } (p, 1 ) \subset \left \{ f > \frac{1}{2} f(p) \right \}$$
	and $f(p) > \gamma e^{\tau_0}$ implies $B_{\ol{g}} ( p, 1) $ contains no points conjugate to $p$.
	
	As in the proof of Lemma \ref{Lem Int Est Intermediate Scale},
	Propositions \ref{Prop Nonlin Int Est} and \ref{Prop Nonlin Int Est+} may be applied to \eqref{clH Evol Eqn 2} to deduce that if $\delta \ll 1$ is sufficiently small depending on $n, M, \ol{g}, f$,
	then
		$$| \ol{\nabla} h |_{\ol{g}} (p_*, \tau_* ) + | \ol{\nabla}^2 h |_{\ol{g}} (p_*, \tau_* )
		\lesssim_{n, M, \ol{g}, f} \delta$$
	for all $p_* \in M$ with $f(p_*) > 4 \gamma e^{\tau_*}$.
\end{proof}

\subsection{Proof of Theorem \ref{Thm C^2 Bounds Preserved}}

\begin{proof}[Proof of Theorem \ref{Thm C^2 Bounds Preserved}.]
	We first seek to prove a $C^0$ bound.
	The proof of Proposition \ref{Prop Difference of G_p from G_0} 
	shows that there exists $C_1 = C_1(n, M, \ol{g}, f, \lambda_*) \ge 1 \ge \ol{p}$ such that 
		$$\left|  \sum_{j = 1}^{K} p_j h_j   \right| 
		\le C_1 \ol{p} e^{\lambda_* \tau_0} f^{- \lambda_*} 
		\le C_1 e^{\lambda_* \tau_0} f^{- \lambda_*} $$
	for all $| \mathbf p | \le \ol{p} e^{\lambda_* \tau_0}$.
	
	By Lemma \ref{Lem |h| Evol Eqn}, $| h |_{\ol{g}} : M \times [\tau_0, \tau_1] \to \R$ satisfies the differential inequality
	\begin{align}	
		\partial_\tau | h|_{\ol{g}}^2 \le{}&  
		( 1 + C_n | h |_{\ol{g}} ) \ol{\Delta} | h |_{\ol{g}}^2 - \ol{\nabla}_{ \ol{\nabla} f } | h |_{\ol{g}}^2 + 4 C_n ( 1 + \epsilon ) | \ol{Rm} |_{\ol{g}} | h|_{\ol{g}}^2 +  | \cl{E}_2|_{\ol{g}} | h |_{\ol{g}} 	\label{|h|^2 Evol Eqn 0} \\
		\le{}& ( 1 + C_n | h |_{\ol{g}} ) \ol{\Delta} | h |_{\ol{g}}^2 - \ol{\nabla}_{ \ol{\nabla} f } | h |_{\ol{g}}^2 + 4 C_n ( 1 + \epsilon ) | \ol{Rm} |_{\ol{g}} | h|_{\ol{g}}^2 +  \frac{C_0}{\Gamma_0 } e^{-\tau} | h |_{\ol{g}} \label{|h|^2 Evol Eqn 1}
	\end{align}
	on $M \times ( \tau_0, \tau_1)$ so long as $0 < \epsilon \ll 1$ is sufficiently small depending on $n$.
	Moreover, the assumption that $\supp \cl{E}_2 \subset \{ e^\tau \le f \le \Gamma_0 e^\tau \}$
	implies
	\begin{equation}	\label{|h|^2 Evol Eqn 2}
		\partial_\tau | h|_{\ol{g}}^2 \le 
		( 1 + C_n | h |_{\ol{g}} ) \ol{\Delta} | h |_{\ol{g}}^2 - \ol{\nabla}_{ \ol{\nabla} f } | h |_{\ol{g}}^2 + 4 C_n ( 1 + \epsilon ) | \ol{Rm} |_{\ol{g}} | h|_{\ol{g}}^2 
	\end{equation}
	on $\{ (x, \tau) \in M \times (\tau_0, \tau_1) : f(x) < e^\tau \}$.
		
	We aim to use Lemmas \ref{Lem Barrier Intermediate Scale} and \ref{Lem Barrier Large Scale} to construct a supersolution for $|h|_{\ol{g}}^2$ outside of some compact subset of $M$.
	To this end, let $0 < \delta \le \epsilon$ and define
	\begin{gather*}
		A = A(n, M, \ol{g}, f, \lambda_*) \doteqdot C_1^2  + 1 >  0		\\
		B_0 = B_0(n, M, \ol{g}, f, \lambda_*) \doteqdot 
		A \left\{ 2 | \lambda_* | \left( 2 | \lambda_* | + \frac{n}{2} \right) + 4 C_n ( 1 + 1 ) \sup_M \left(| \ol{Rm} |_{\ol{g}} f \right) \right\} +1> 0, \\
		a \doteqdot \delta^2 \in (0, \epsilon^2 ] \subset (0, 1),	\text{ and }\\
		B_1 = B_1(n, M, \ol{g}, f) \doteqdot 4 C_n (1 +1) \left( \sup_M | \ol{Rm} |_{\ol{g}} f \right) + C_0 > 0.
	\end{gather*}
	By definition,
	\begin{gather*} \begin{aligned}
		&B _0- A  \left\{2  | \lambda_* | \left( 2 | \lambda_* | + \frac{n}{2} \right) + 4 C_n ( 1 + \epsilon ) \sup_M (| \ol{Rm} | f) \right\} 	\\
		\ge{}& B _0- A  \left\{2  | \lambda_* | \left( 2 | \lambda_* | + \frac{n}{2} \right) + 4 C_n ( 1 + 1 ) \sup_M (| \ol{Rm} | f) \right\}	\\
		={}& 1,
	\end{aligned} \end{gather*}
	and
	\begin{gather*} \begin{aligned}
		B_1 = 4 C_n (1 +1) \left( \sup_M | \ol{Rm} |_{\ol{g}} f \right)  + C_0 
		> 4 C_n ( 1 + \epsilon ) ( \sup f | Rm| ) a + C_0 \sqrt{a}.
	\end{aligned}	\end{gather*}
	
	Additionally, define
	\begin{gather*}	 \begin{aligned}
		\gamma_+ &\doteqdot \left( \frac{2 \delta^2}{A} \right)^{1 /(2 | \lambda_* |)} 
		\lesssim_{n, M, \ol{g}, f, \lambda_*} \epsilon^{1 /|\lambda_*| }  \text{ and} \\
		\gamma_- &\doteqdot \left( \frac{\delta^2}{2A} \right)^{1 / (2| \lambda_* |)} 
		\lesssim_{n, M, \ol{g}, f, \lambda*} \epsilon^{1 / |\lambda_*|}.
	\end{aligned} \end{gather*}
	Note that $0 < \gamma_- \le \gamma_+ \le 1$.
	
	By Lemma \ref{Lem Barrier Intermediate Scale} with $\kappa = - 2 \lambda_* > 0$,
	$e^{2 \lambda_* \tau} ( A f^{- 2 \lambda_* } - B_0 f^{- 2 \lambda_* -1} )$ is a supersolution of equation \eqref{|h|^2 Evol Eqn 2} on the domain $\{ \Upsilon < f < \gamma_+ e^\tau \}$ 
	for some large constant $\Upsilon = \Upsilon ( n, M, \ol{g}, f, \lambda_* )$ that depends only on $n, M, \ol{g}, f, \lambda_*$
	and so long as $\gamma_+ \ll 1$ is sufficiently small depending on $n, M, \ol{g}, f, \lambda_*$.
	In fact, since $\gamma_+ \lesssim_{n, M, \ol{g}, f, \lambda_*} \epsilon^{1 /| \lambda_* |}$, it instead suffices to assume $\epsilon \ll 1$ is sufficiently small depending on $n, M, \ol{g},f , \lambda_*$.
	Additionally, 
		$$e^{2\lambda_* \tau} ( A f^{-2\lambda_* } - B_0 f^{- 2\lambda_* -1} ) 
		\ge e^{2\lambda_* \tau} f^{2| \lambda_* |} ( A - B_0 / \Upsilon ) 
		\qquad \text{ on } \{ \Upsilon \le f \le \gamma_+ e^\tau \}.$$
	In particular, by taking $\Upsilon$ possibly larger depending on $n, M, \ol{g}, f, \lambda_*$, we may assume that 
		$$A - B_0 / \Upsilon > A - 1 = C_1^2  > 0,$$
	that Lemma \ref{Lem Local L^2_f to C^2 Est} holds with $\Gamma = \Upsilon$,
	and that Lemma \ref{Lem Int Est Intermediate Scale} holds with $\Gamma = \frac{1}{4} \Upsilon$.
	
	Similarly, by Lemma \ref{Lem Barrier Large Scale}, $a - B_1/f$ is a positive supersolution of equation \eqref{|h|^2 Evol Eqn 1}
	on the domain $\{ \gamma_- e^\tau < f < \Gamma_0 e^{\tau} \}$
	if $\tau_0 \gg 1$ is sufficiently large depending on $n, M, \ol{g} , f, \lambda_* ,  \epsilon, \delta$.
	
	Next, we claim that if $\tau_0 \gg 1$ is sufficiently large depending on $n, M, \ol{g}, f, \lambda_*, \delta$
	and $\tau \ge \tau_0$, then
	\begin{gather*} \begin{aligned}
		a - B_1/ f &< e^{2 \lambda_* \tau} f^{2 |\lambda_*|} ( A - B_0 / f) \text{ when } f = \gamma_+ e^\tau	\text{ and}\\
		a - B_1/ f &> e^{2 \lambda_* \tau} f^{2 |\lambda_*|} ( A - B_0 / f) \text{ when } f = \gamma_- e^\tau.
	\end{aligned}	\end{gather*}
	
	For the first inequality, 
	\begin{align*}
		a - B_1/f
		& \le a		&& ( B_1, f > 0)\\
		&= \delta^2	\\
		&= \frac{1}{2} A \gamma_+^{2 |\lambda_*|} 
			&& \left( \gamma_+ = (2 \delta^2 / A)^{1/ (2 |\lambda_*|)} \right)	\\
		&< A \gamma_+^{2 |\lambda_*|} - \frac{B_0}{\gamma_+ e^\tau } \gamma_+^{2|\lambda_*|}
		&& (\tau \ge \tau_0 \gg 1)	\\
		&= \left( A - \frac{B_0}{f} \right) e^{2 \lambda_* \tau} f^{2|\lambda_* |}
		&& ( f = \gamma_+ e^\tau ).
	\end{align*}
	
	For the second inequality,
	\begin{align*}
		( A - B_0 / f ) e^{2 \lambda_* \tau } f^{2 |\lambda_*|}
		&\le A e^{2 \lambda_* \tau} f^{2 |\lambda_*|}		&& ( B_0, f > 0 ) \\
		&= A \gamma_-^{2 |\lambda_*|} 		&& ( f = \gamma_- e^\tau )	\\
		&= \frac{\delta^2 }{2}		&& \left( \gamma_- = (\delta^2 / (2A) )^{1/ (2|\lambda_*|)} \right) \\
		&< \delta^2 - \frac{B_1}{\gamma_- e^\tau }		&& ( \tau \ge \tau_0 \gg 1)	\\
		&= a - \frac{B_1}{f}		&& ( f = \gamma_- e^\tau ) .
	\end{align*}
	
	Having proven the claim, it then follows that the function $u : M \times [\tau_0, \infty) \to \R$ given by
	\begin{equation*}
		u =  \begin{cases}
			\left(A - \frac{B_0}{f} \right) e^{2 \lambda_* \tau} f^{2 | \lambda_* |}, & \text{if } f \le \gamma_- e^\tau, 	\\
			\min \left\{ \left(A - \frac{B_0}{f} \right) e^{ 2 \lambda_* \tau} f^{2 | \lambda_* |}, a - \frac{B_1}{f} \right\}, 
				& \text{if } \gamma_- e^\tau \le f \le \gamma_+ e^\tau,	\\
			a - \frac{B_1}{f}, & \text{if } \gamma_+ e^\tau \le f ,
		\end{cases}
	\end{equation*}
	is a locally Lipschitz continuous function which is also positive $u > 0$ on $\{ \Upsilon \le f \}$
	when $\tau_0 \gg 1$ is sufficiently large (depending on $n, M, \ol{g}, f, \lambda_* , \epsilon, \delta$).
	Moreover, $u$ is a supersolution of equation \eqref{|h|^2 Evol Eqn 0} on $\{ \Upsilon \le f \le \Gamma_0 e^\tau \}$.
	Indeed, $u$ equals supersolutions on $\{ \Upsilon < f \le \gamma_- e^\tau \} \cup \{ \gamma_+ e^\tau \le f < \Gamma_0 e^\tau \}$,
	and, on the region $\{ \gamma_- e^\tau \le f \le \gamma_+ e^\tau \}$, $u$ is a minimum of two supersolutions which is again a supersolution.

	Notice that $\Upsilon = \Upsilon ( n, M, \ol{g}, f , \lambda_* )$ depends only on $n, M, \ol{g}, f, \lambda_*$.
	Since the eigenmodes $h_j$ are smooth,
	there exists $C_2 = C_2 ( n, M, \ol{g}, f , \lambda_* )$ such that, for all  $x \in \{ f < 2 \Upsilon \}$,
	\begin{align*}
		& \sum_{l = 0}^{4n+2} | \ol{\nabla}^l  h  |_{\ol{g}} ( x, \tau_0) \\
		={}& \sum_{l = 0}^{4n+2} | \ol{\nabla}^l \sum_{j=1}^{K} p_j h_j |_{\ol{g}} ( x, \tau_0) 
		&& \left( \tau_0 \gg 1 \implies 2 \Upsilon \le \frac{\gamma_0}{2} e^{\tau_0} \right)	\\
		\le{}& \ol{p} e^{\lambda_* \tau_0} \sum_{j = 1}^{K} \sum_{l =0}^{4n+2} | \ol{\nabla}^l h_j |	\\
		\le{}& C_2 \ol{p} e^{\lambda_* \tau_0} 
		&& ( \text{smoothness of the $h_j$} ) .
	\end{align*}
	By Lemma \ref{Lem Local L^2_f to C^2 Est}, if $\epsilon \ll 1$ is sufficiently small depending on $n, M, \ol{g},f , \lambda_*$ and
	$\tau_0 \gg 1$ is sufficiently large depending on $n, M, \ol{g}, f, \lambda_*, \ol{p}$,
	then
	\begin{align*}
		\sum_{l = 0}^2 | \ol{\nabla}^l h  |_{\ol{g}} (x, \tau) &\lesssim_{n, M, \ol{g}, f, \lambda_* }
		 ( \mu +  C_2 \ol{p} ) e^{\lambda_* \tau} 	\\
		 &\lesssim_{n, M, \ol{g}, f, \lambda_*}  ( \mu +  C_2 \ol{p} ) e^{\lambda_* \tau} f^{|\lambda_*|}
	\end{align*}
	for all $(x, \tau) \in \{ f < \Upsilon \} \times [ \tau_0, \tau_1]$.
	Let $C_3 = C_3 (n, M, \ol{g}, f, \lambda_*)$ be a constant such that 
		$$\sum_{l = 0}^2 | \ol{\nabla}^l  h |_{\ol{g}} (x, \tau) 
		\le C_3 ( \mu +  C_2 \ol{p} ) e^{\lambda_* \tau} f^{| \lambda_* |} $$
	for all $(x, \tau) \in \{ f < \Upsilon \} \times [ \tau_0, \tau_1]$.
	
	Since $C_1, C_2, C_3$ depend only on $n, M, \ol{g}, f, \lambda_*$,
	it follows that
		$$C_3 ( \mu + C_2 \ol{p} ) \le C_1 < \sqrt{ A - B_0 / \Upsilon} $$
	if $0 < \mu , \ol{p} \ll 1$ are sufficiently small depending on $n, M, \ol{g}, f, \lambda_*$.
	Under this condition,
		$$| h |^2_{\ol{g}} (x , \tau) < u(x, \tau)		\qquad 
		\text{ for all } (x, \tau) \in \{ f = \Upsilon \} \times [ \tau_0, \tau_1].$$
	Note additionally that
		$$| h |^2_{\ol{g}} (x , \tau) = 0 <  u(x, \tau)		\qquad 
		\text{ for all } (x, \tau) \in \{ f \ge \Gamma_0 e^{\tau} \}.$$
	
	Finally, for all $x \in \{ \Upsilon \le f  \}$,
	\begin{align*}
		 | h|^2_{\ol{g}} (x, \tau_0)	
		={}& \eta^2_{\gamma_0} \left| \sum_{j = 1}^{ K} p_j h_j  \right|^2	\\
		\le{}& C_1^2  e^{2\lambda_* \tau_0} f^{- 2\lambda_* }	\\
		<{}& ( A - B_0 / \Upsilon ) e^{2\lambda_* \tau_0} f^{- 2\lambda_* }	\\
		\le{}& ( A - B_0 / f ) e^{2\lambda_* \tau_0} f^{- 2\lambda_* },
	\end{align*}
	and, for all $x \in \{ f \ge \gamma_- e^{\tau} \}$,
	\begin{align*}
		| h|^2_{\ol{g}} (x, \tau_0	)
		={}& \eta^2_{\gamma_0} \left| \sum_{j = 1}^{ K } p_j h_j  \right|^2	\\
		\le{}& \eta^2_{\gamma_0} C_1^2  e^{2 \lambda_* \tau_0} f^{-2 \lambda_* }		\\
		\le{}& C_1^2  \gamma_0^{2 | \lambda_* |}
		&& ( \supp \eta_{\gamma_0} \subset \{ f \le \gamma_0 e^\tau \}	 ) \\
		\le{}& \frac{\delta^2}{2} 
		&& ( \gamma_0 \ll 1)	\\
		<{}& a - \frac{B_1}{f} 		&& ( \tau \ge \tau_0 \gg 1 ) 
	\end{align*}	
	if $0 < \gamma_0 \ll 1$ is sufficiently small depending on $n, M, \ol{g}, f, \lambda_*, \delta$
	and $\tau_0 \gg 1$ is sufficiently large depending on $n, M, \ol{g}, f, \lambda_*, \delta$.
	Therefore, $| h^2|_{\ol{g}} ( x, \tau_0 ) \le u(x, \tau_0)$ for all $x \in \{ \Upsilon \le f \le \Gamma_0 e^{\tau} \}$.
	
	Since $u$ is a supersolution and $| h |^2_{\ol{g}}$ a subsolution of equation \eqref{|h|^2 Evol Eqn 0}, 
	the maximum principle then implies that $| h |^2_{\ol{g}} \le u$ for all $(x, \tau) \in \{ \Upsilon \le f \le \Gamma_0 e^\tau \} \times [ \tau_0, \tau_1]$.
	Combined with the fact that $h \equiv 0 $ on $\{ f \ge \Gamma_0 e^\tau \}$ and the bound
		$$| h |_{\ol{g}} (x ,\tau) \le C_1  e^{\lambda_* \tau} f^{- \lambda_*}		\qquad
		\text{ for all } (x, \tau) \in \{ f \le \Upsilon \} \times [ \tau_0, \tau_1 ],$$
	it then follows that 
	\begin{align*}
		| h |_{\ol{g}} (x, \tau) &\lesssim_{n, M, \ol{g}, f, \lambda_* }  e^{\lambda_* \tau} f^{- \lambda_* } \text{ and} \\
		| h |_{\ol{g}} (x, \tau) &\le \delta 
	\end{align*}
	throughout $M \times [ \tau_0, \tau_1]$ so long as $\tau_0 \gg 1$ is sufficiently large depending on $n, M, \ol{g}, f, \lambda_*, \delta$.
	
	Now, we use the interior estimates to improve this $C^0$ bound to a $C^2$ bound.
	Since $\lambda_* < \min \{  \lambda_K, 0 \} = \min_{j \le K} \min \{  \lambda_j, 0 \}$,
	there exists some positive $\delta' = \delta' (n, M, \ol{g}, f, \lambda_* ) > 0$ such that 
	$\max \{ -\lambda_j, 0 \} + \delta' < |\lambda_*|$ for all $j \le K$.
	The initial data thus satisfies the estimate
	\begin{align*}
		& \quad \sum_{l = 0}^{3} | \ol{\nabla}^l h |_{\ol{g}} (x, \tau_0) 	\\
		&= \sum_{l = 0}^{3} \left| \ol{\nabla}^l 
		\left( \eta_{\gamma_0}  \sum_{j = 1}^{K} p_j h_j   \right)
		\right|_{\ol{g}} 	\\
		&\lesssim \left( \sum_{l = 0}^3 \left| \ol{\nabla}^l \eta_{\gamma_0} \right| \right)
		\left( \sum_{l = 0}^3  \sum_{j = 1}^{K } |p_j| | \ol{\nabla}^l h_j |   \right)	\\
		&\lesssim_{n, M, \ol{g},f, \lambda_*}  \left( \sum_{l = 0}^3 \left| \ol{\nabla}^l \eta_{\gamma_0} \right| \right)
		\left( \sum_{j = 1}^{K} |p_j| f^{\max \{ - \lambda_j, 0\} + \delta'}   \right) 	
		&& (\text{Proposition \ref{Prop Eigmode Derivative Growth}})		\\
		&\lesssim_{n, M, \ol{g},f, \lambda_*}  \left( \sum_{l = 0}^3 \left| \ol{\nabla}^l \eta_{\gamma_0} \right| \right)
		  \ol{p}  e^{\lambda_* \tau_0} f^{| \lambda_* |}
		&& (\text{choice of $\delta'$})	\\
		&\lesssim_{n, M, \ol{g}, f}  
		\ol{p} e^{\lambda_* \tau_0} f^{| \lambda_* |}	\\
		&\le  e^{\lambda_* \tau_0} f^{| \lambda_* |}.
	\end{align*}
	Thus, there exists some $0 < \gamma_1 \ll 1$ sufficiently small depending on $n, M, \ol{g}, f, \lambda_*$ so that Lemma \ref{Lem Int Est Intermediate Scale} applies with $\Gamma = \frac{1}{4} \Upsilon$ and $\gamma = 4 \gamma_1$ to show that
		$$| \ol{\nabla} h |_{\ol{g}} + | \ol{\nabla}^2 h |_{\ol{g}} \lesssim_{n, M, \ol{g}, f, \lambda_*} 
		 e^{\lambda_* \tau} f^{|\lambda_* |} 
		\qquad \text{ on } \{ \Upsilon \le f \le \gamma_1 e^{\tau} \} .$$
	
	The above estimate on the initial data also shows
	\begin{align*}
		\sum_{l = 0}^{3} | \ol{\nabla}^l h |_{\ol{g}} (x, \tau_0) 	
		\lesssim&_{n, M, \ol{g}, f, \lambda_*} 
		\left( \sum_{l = 0}^3 \left| \ol{\nabla}^l \eta_{\gamma_0} \right| \right)
		 \ol{p}  e^{\lambda_* \tau_0} f^{| \lambda_* |}	\\
		\lesssim&_{n, M, \ol{g}, f}  \ol{p}  \gamma_0^{| \lambda_* |}	\\
		\le{}&  \gamma_0^{| \lambda_* |} .	
	\end{align*}
	Hence, if $\gamma_0 \ll 1$ is sufficiently small depending on $n, M , \ol{g}, f, \lambda_*, \delta$ then
		$$ \sum_{l = 0}^{3} | \ol{\nabla}^l h |_{\ol{g}} (x, \tau_0) 	\le \delta.
		\qquad \text{for all } (x, \tau) \in M \times \{ \tau_0 \}.$$
	Moreover,
	\begin{gather*}
		| \cl{E}_2 | + | \ol{\nabla} \cl{E}_2 |
		\le \frac{C_0}{\Gamma_0} e^{- \tau_0} + \frac{C_0}{ \Gamma_0^{3/2} } e^{- \frac{3}{2} \tau_0} 
		\le \delta
	\end{gather*}
	if $\tau_0 \gg 1$ is sufficiently large depending on $n, M, \ol{g}, f, \delta$.
	Therefore, Lemma \ref{Lem Int Est Large Scale} applies to show that 
		$$| \ol{\nabla} h | + | \ol{\nabla}^2 h | \lesssim_{n, M, \ol{g}, f} \delta
		\qquad \text{ on } \{ \gamma_1 e^\tau \le f \}$$
	if $0 < \delta \le \epsilon \ll 1$ is sufficiently small depending on $n, M, \ol{g}, f$
	and $\tau_0 \gg 1$ is sufficiently large depending on $n, M, \ol{g} , f, \lambda_*, \delta$.
	
	Together with the local $C^2$ estimate
		$$\sum_{l = 0}^2 | \ol{\nabla}^l h | \lesssim_{n, M, \ol{g}, f, \lambda_*}  e^{\lambda_* \tau} f^{| \lambda_* |} 
		\qquad \text{ on } \{ f \le \Upsilon \},$$
	it therefore follows that
	\begin{align*}
		\sum_{l =0}^2 | \ol{\nabla}^l h |_{\ol{g}} (x, \tau) &\lesssim_{n, M, \ol{g}, f, \lambda_* } 
		 e^{\lambda_* \tau} f^{| \lambda_* |} \text{ and} \\
		\sum_{l = 0}^2 | \ol{\nabla}^l h |_{\ol{g}} (x, \tau) &\lesssim_{n, M, \ol{g}, f}  \delta 
	\end{align*}
	throughout $M \times [ \tau_0, \tau_1 ]$ if $\tau_0 \gg 1$ is sufficiently large depending on $n, M, \ol{g}, f, \lambda_* , \delta$.
	This completes the proof.
\end{proof}

\section{Behavior of $L^2_f(M)$ Quantities} \label{Sect Behavior of L^2_f}

In this section, we explore the dynamics of the $L^2_f(M)$ estimates included in Definition \ref{Defn Box} of $\cl{B}$.

\begin{lem}[$L^2_f$ Estimates for the Error Terms $\cl{E}_1, \cl{E}_2$] \label{Lem L^2_f Ests for Errors}
	Assume
		$$|h|_{\ol{g}} + | \ol{\nabla} h |_{\ol{g}} + | \ol{\nabla}^2 h |_{\ol{g}} \le W f^{|\lambda_* | } e^{\lambda_* \tau} \qquad (\text{for some } W \ge 1),$$
		$$|h|_{\ol{g}} + | \ol{\nabla} h |_{\ol{g}} + | \ol{\nabla}^2 h |_{\ol{g}} \le \epsilon ,$$
		$$| \cl{E}_2 |_{\ol{g} } \le C_0 e^{-\tau} \qquad \text{ and } \qquad 
		\supp \cl{E}_2 \subset \{ f \ge e^{ \tau} \}$$
	throughout $ M \times [ \tau_0, \tau_1]$.
	
	If $0 < \epsilon \ll 1$ is sufficiently small (depending on $n$)
	and $\tau_0 \gg 1$ is sufficiently large (depending on $n, M, \ol{g},f, \lambda_*, C_0$)
	then there exists a constant $C_1$ (depending on $n, M, \ol{g}, f, \lambda_*$) such that
		$$\| \cl{E}_1(\tau) \|_{L^2_f} + \| \cl{E}_2(\tau) \|_{L^2_f} \le C_1 W^2 e^{2 \lambda_* \tau} 
		\qquad \text{for all } \tau \in [ \tau_0, \tau_1].$$
\end{lem}

\begin{proof}
	Recall that $\cl{E}_1$ is defined in \eqref{clE_1 Eqn} as 
	\begin{multline} \tag{\ref{clE_1 Eqn}}
	 \mathcal{E}_1(\tau)  \doteqdot
		 \ol{R} \indices{_{ja}^p_b} \left( \ol{g}_{ip} \tl{h}^{ab}  + \hat{h}^{ab} h_{ip} \right) 
		+ \ol{R} \indices{_{ia}^p_b} \left( \ol{g}_{jp} \tl{h}^{ab}  + \hat{h}^{ab} h_{jp} \right) 	\\		
		  + g^{ab} g^{pq} ( \ol{\nabla} h * \ol{\nabla} h )	 -\hat{h}^{ab} \ol{\nabla}_a \ol{\nabla}_b h		.
	\end{multline}
	If $0 < \epsilon \ll 1$ is sufficiently small depending on $n$, then
		$$| \cl{E}_1 |_{\ol{g}} \lesssim_n ( 1 + | \ol{Rm} |_{\ol{g}} ) \left( |h|_{\ol{g}} + | \ol{\nabla} h |_{\ol{g}} + | \ol{\nabla}^2 h |_{\ol{g}} \right)^2
		\lesssim_{n, M, \ol{g}, f} \left( |h|_{\ol{g}} + | \ol{\nabla} h |_{\ol{g}} + | \ol{\nabla}^2 h |_{\ol{g}} \right)^2$$
	throughout $M \times [\tau_0, \tau_1]$.
	Therefore,
	\begin{gather*} \begin{aligned}
		& \qquad \| \cl{E}_1 \|_{L^2_f}^2 	\\
		& = \int_M | \cl{E}_1 |^2_{\ol{g} } e^{-f} dV_{\ol{g}} 	\\
		& = \int_{ \{ f \le  e^{\tau} \} }  | \cl{E}_1 |^2_{\ol{g} } e^{-f} dV_{\ol{g}} 
		+ \int_{ \{ f >  e^\tau \} }   | \cl{E}_1 |^2_{\ol{g} } e^{-f} dV_{\ol{g}}	\\
		& \lesssim_{n, M, \ol{g}, f} \int_{ \{ f \le  e^\tau \} } W^4 f^{ 4 | \lambda_* |} e^{ 4 \lambda_* \tau} e^{-f} dV_{\ol{g}}	\\
		& \qquad + \int_{ \{ f >  e^\tau \} }   \epsilon^4 e^{-f} dV_{\ol{g}}	\\
		& \le W^4 e^{4 \lambda_* \tau}  \int_M  f^{ 4 | \lambda_* |}  e^{-f} dV_{\ol{g}}
		+ \epsilon^4 \int_{ \{ f >  e^\tau \} }  e^{-f} dV_{\ol{g}}		\\
	\end{aligned} \end{gather*}
	The moment estimates for $f$ in Proposition \ref{Prop Moment Ests} imply that 
		$$\int_M f^{4 | \lambda_* | } e^{-f} dV_{\ol{g}} \le C(n, M, \ol{g}, f, \lambda_* ).$$
	Additionally, the weighted volume decay in Proposition \ref{Prop Weighted Vol Decay} implies
		$$\int_{ \{ f >  e^\tau \} }  e^{-f} dV_{\ol{g}} 
		\lesssim C(n, M, \ol{g}, f) (  e^\tau )^{\frac{n}{2} -1}  e^{-  \frac{1}{8} (  e^\tau) }
		\le  e^{4 \lambda_*  \tau}$$
	if $\tau \ge \tau_0 \gg 1$ is sufficiently large depending on $n, M, \ol{g}, f, \lambda_*$.
	The estimate for $\cl{E}_1$ then follows.
	
	We now estimate $\cl{E}_2$.
	\begin{gather*} \begin{aligned}
		 \| \cl{E}_2 \|_{L^2_f}^2	
		& = \int_M | \cl{E}_2 |_{\ol{g}}^2 e^{-f} dV_{\ol{g}}		\\
		& \le C_0^2 \int_{\supp \cl{E}_2 } e^{-2 \tau} e^{-f} dV_{\ol{g}}	\\	
		& \le C_0^2 e^{-2 \tau} \int_{ \{f \ge  e^\tau  \}} e^{-f} dV_{\ol{g}}
		&& (\supp \cl{E}_2 \subset \{ f \ge  e^\tau \} ) 	\\
		& \lesssim_{n, M, g, f} C_0^2 e^{-2 \tau} (  e^\tau)^{\frac{n}{2}-1} e^{- \frac{1}{8} (  e^\tau ) } 
		&& ( \text{Proposition \ref{Prop Weighted Vol Decay}} )	\\
		& \le 	W^4 e^{4 \lambda_* \tau}		
		&& ( \tau \ge \tau_0 \gg 1) .
	\end{aligned} \end{gather*}
\end{proof}

\begin{lem}[Evolutions of Projections] \label{Lem Evols of Projs}
	Assume $h$ is a smooth section of $\pi^* Sym^2 T^*M \to M \times [\tau_0, \tau_1]$ such that
		$$h(\cdot , \tau) \text{ is compactly supported in } M  \text{ for all } \tau \in [ \tau_0, \tau_1],$$
		$$\partial_\tau h = \ol{\Delta}_f h + 2 \ol{Rm} [ h] + \cl{E}_1 + \cl{E}_2 \qquad \text{ on } M \times ( \tau_0, \tau_1),$$
		$$\text{and } \| \cl{E}_1+ \cl{E}_2 \|_{L^2_f}  \le C e^{2 \lambda_* \tau}.$$
	Then the projections $ \pi_u h = h_u, \pi_s = h_s$ as in Definition \ref{Defn Projs} satisfy
	\begin{gather*} \begin{aligned}
		\frac{d}{d \tau} \| e^{- \lambda_* \tau} h_u \|_{L^2_f} 
		& \ge (\lambda_K - \lambda_*) \| e^{- \lambda_* \tau} h_u \|_{L^2_f}
		- C e^{\lambda_* \tau_0} 
		&&\text{and} \\
		\frac{d}{d \tau} \| e^{- \lambda_* \tau} h_s \|_{L^2_f} 
		& \le (\lambda_{K+1} - \lambda_* ) \| e^{- \lambda_* \tau} h_s \|_{L^2_f}
		+ C e^{\lambda_* \tau_0} .
	\end{aligned} \end{gather*}
\end{lem}

\begin{proof}
	To simplify notation, write
	\begin{gather*}	
		\partial_\tau h = \ol{\Delta}_f h + 2 \ol{Rm} [ h] + \cl{E}_1 + \cl{E}_2 = \cl{A} h + \cl{E}	\\
		\text{where } 
		\cl{A} = \ol{\Delta}_f + 2 \ol{Rm}
		\quad \text{ and } \quad
		\cl{E} = \cl{E}_1 + \cl{E}_2.
	\end{gather*}
	Note that the projections $\pi_u, \pi_s$ commute with $\cl{A}$ when applied to smooth compactly supported tensors.
	
	It follows that
	\begin{gather*} \begin{aligned}
		& \qquad \frac{d}{d\tau} \left(e^{-2 \lambda_* \tau} \| h_u \|^2_{L^2_f}\right)	\\
		& = -2 \lambda_* e^{-2 \lambda_* \tau} \| h_u \|^2_{L^2_f} 
		+ 2 e^{- 2 \lambda_* \tau} \| h_u \|_{L^2_f} ( \cl{A} h_u , h_u )_{L^2_f}
		+ 2 e^{- 2 \lambda_* \tau} ( \cl{E} , h_u)_{L^2_f}	\\
		& \ge -2 \lambda_* e^{-2 \lambda_* \tau} \| h_u \|^2_{L^2_f} 
		+ 2 e^{- 2 \lambda_* \tau} \| h_u \|^2_{L^2_f} \lambda_K \\
		& \qquad 
		- 2 e^{- 2 \lambda_* \tau} \| \cl{E} \|_{L^2_f}   \|  h_u \|_{L^2_f}	\\
		& =  2 e^{- 2 \lambda_* \tau} \| h_u \|^2_{L^2_f} (\lambda_K - \lambda_* ) 
		- 2 e^{- 2 \lambda_* \tau} \| \cl{E} \|_{L^2_f}   \|  h_u \|_{L^2_f}	\\
	\end{aligned} \end{gather*}
	Note that 
	$\lambda_K - \lambda_*   > 0$
	by assumption \ref{Assume Shrinker}
	and
		$$\frac{ d}{d\tau} e^{-2 \lambda_* \tau} \| h_u \|^2_{L^2_f} 
		= 2 e^{- \lambda_* \tau} \|  h_u \|_{L^2_f} \frac{d}{d \tau} \| e^{- \lambda_* \tau} h_u \|_{L^2_f} .$$
	Therefore,
	\begin{gather*} \begin{aligned}
		\frac{d}{d \tau} \| e^{ - \lambda_* \tau} h_u \|_{L^2_f} 
		& \ge  (\lambda_K - \lambda_*)    \| e^{- \lambda_* \tau} h_u \|_{L^2_f} 
		 - e^{ - \lambda_* \tau}  \| \cl{E} \|_{L^2_f}	\\
		 & \ge (\lambda_K - \lambda_*)   \| e^{- \lambda_* \tau} h_u \|_{L^2_f} 
		 - C e^{  \lambda_* \tau}	\\
		 & \ge (\lambda_K - \lambda_*)   \| e^{- \lambda_* \tau} h_u \|_{L^2_f} 
		 - C e^{  \lambda_* \tau_0}.	 
	\end{aligned} \end{gather*}

	A similar computation shows that 
	\begin{align*}
		\frac{ d}{d\tau} \| e^{- \lambda_* \tau} h_s \|_{L^2_f}
		 &\le 
		(\lambda_{K+1} - \lambda_* )  \| e^{- \lambda_* \tau} h_s \|_{L^2_f} 
		+ e^{ - \lambda_* \tau} \| \cl{E} \|_{L^2_f} 	\\
		&\le 
		(\lambda_{K+1} - \lambda_*)  \| e^{- \lambda_* \tau} h_s \|_{L^2_f} 
		+ C e^{  \lambda_* \tau_0} .
	\end{align*}
\end{proof}

\begin{lem} \label{Lem Immediately Exits}
	Let $\tau_0 <  \tau_1^+$.
	Assume $h$ is a smooth section of $\pi^* Sym^2 T^*M \to M \times [\tau_0, \tau_1^+)$ such that
	\begin{gather*}
		h(\tau) \doteqdot h(\cdot, \tau) \text{ is a compactly supported in } M \text{ for all }  \tau \in [\tau_0, \tau_1^+ ),	\\
		\partial_\tau h = \ol{\Delta}_f h + 2 \ol{Rm} [ h] + \cl{E}_1 + \cl{E}_2 \qquad \text{ on } M \times (\tau_0, \tau_1^+),	\\
		\text{and } \| \cl{E}_1+ \cl{E}_2 \|_{L^2_f}  \le C e^{2 \lambda_* \tau}.
	\end{gather*}
	For any $\mu_u, \mu_s > 0$,
	if $\tau_0 \gg 1$ is sufficiently large (depending on $n, M, \ol{g}, f, \lambda_*, C, \mu_u, \mu_s$)
	then the following two implications hold:
	\begin{enumerate}
	\item
	(Immediately Exits on Unstable Side)\\
	For all $\tau_1 \in [\tau_0, \tau_1^+)$,
		$$\| h_u(\tau_1)  \|_{L^2_f} = \mu_u e^{\lambda_* \tau_1} \implies 
		\| h_u( \tau ) \|_{L^2_f} > \mu_u e^{\lambda_* \tau}$$
		$$ \text{ for all $\tau$ in a neighborhood of $\tau_1$ with $\tau > \tau_1$},$$ 
	and	
	\item
	(Can't Exit through Stable Side)
		$$\| h_s ( \tau_0) \|_{L^2_f} \le \mu_s e^{\lambda_* \tau_0} 
		\implies
		\| h_s(\tau) \|_{L^2_f} \le \mu_s e^{\lambda_* \tau} 		\qquad \text{for all } \tau \in [\tau_0, \tau_1^+).$$
	\end{enumerate}
\end{lem}
\begin{proof}
	(Immediately Exits on Unstable Side)	\\
	At a time $\tau_1 \ge \tau_0$ when
		$$\| h_u(  \tau_1) \|_{L^2_f} = \mu_u e^{\lambda_* \tau_1},$$
	Lemma \ref{Lem Evols of Projs} implies
	\begin{gather*} \begin{aligned}
		 \left. \frac{ d }{d\tau} \right|_{\tau = \tau_1} \| e^{- \lambda_* \tau} h_u \|_{L^2_f} 	
		 \ge{}& (\lambda_K - \lambda_* ) \| e^{- \lambda_* \tau_1} h_u \|_{L^2_f} - C  e^{ \lambda_* \tau_0} 	\\
		={}& ( \lambda_K - \lambda_* )  \mu_u - C e^{ \lambda_* \tau_0} 	\\
		>{}& 0	&& (\lambda_* < \lambda_K )
	\end{aligned} \end{gather*}
	if $\tau_0 \gg 1$ is sufficiently large depending on $n, M, \ol{g}, f, \lambda_*, C, \mu_u$. \\

	(Can't Exit through Stable Side)\\
	At a time $\tau_1 \ge \tau_0$ when 
		$$\| h_s(\tau_1) \| = \mu_s e^{\lambda_* \tau_1},$$
	Lemma \ref{Lem Evols of Projs} implies
	\begin{gather*} \begin{aligned}
		\left. \frac{ d }{d\tau} \right|_{\tau = \tau_1} \| e^{- \lambda_* \tau} h_s \|_{L^2_f} 	
		 \le{}& (\lambda_{K+1} - \lambda_* ) \| e^{ - \lambda_* \tau} h_s \|_{L^2_f} + C  e^{\lambda_* \tau_0} 	\\
		={}&  (\lambda_{K+1} - \lambda_* ) \mu_s + C  e^{\lambda_* \tau_0} 	\\
		<{}& 0	&& ( \lambda_{K+1} < \lambda_* )
	\end{aligned} \end{gather*}
	if $\tau_0 \gg 1$ is sufficiently large depending on $n, M, \ol{g}, f, \lambda_*, C, \mu_s$.
	The statement follows.
\end{proof}

\section{Proof of Theorem \ref{Main Thm}} \label{Sect Proof of Main Thm}

In this section, we combine the results of the prior sections to complete the proof of Theorem \ref{Main Thm}.
First, we record some estimates on the $L^2_f$ and $C^m$ norms of the initial data $h_{\mathbf p }(t_0)$ that will be used in the remaining susbsections.
Second, we establish well-posedness results for the harmonic map heat flow solution $\tl{\Phi}$ that hold so long as $h = \frac{1}{1-t} ( \Phi_t^{-1})^* \acute G - \ol{g}$ remains small in a $C^2$ sense.
Informally, these well-posedness results imply that, at the first time $\mathbf p$ exits the set $\cl{P}$, it must be because $h_{\mathbf p}$ fails to satisfy the estimates defining $\cl{B}$ and not because the harmonic map heat flow solutions $\tl{\Phi}$ fail to exist.
Finally, we make rigorous the Wa{\.z}ewski box argument used to prove Theorem \ref{Main Thm}.

\subsection{Estimates on the Initial Data}

\begin{lem}[$L^2_f$ Estimates at $t_0$] \label{Lem L^2_f Ests at t_0}
	If $| \mathbf p| \le \ol{p}e^{\lambda_* \tau_0} \le e^{\lambda_* \tau_0}$ and 
	$\tau_0 \gg 1$ is sufficiently large (depending on $n,  M, \ol{g}, f , \gamma_0 $), 
	then the following estimates hold:
	\begin{align}
		\label{L^2_f Est at t_0 Unstable}
		\left| (h_{\mathbf p } ( t_0), h_j )_{L^2_f}  - p_j \right|
		&\lesssim_{n, M, \ol{g}, f, \lambda_*}  e^{ - \frac{ \gamma_0}{100} e^{\tau_0} } 
		&& \text{for } j \le K, \text{ and}
		\\
		\label{L^2_f Est at t_0 Stable}
		\| \pi_s h_{\mathbf p }(t_0) \|_{L^2_f}
		 &\lesssim_{n, M, \ol{g}, f, \lambda_*}  e^{ - \frac{ \gamma_0}{100} e^{\tau_0} } .
	\end{align}
\end{lem}
\begin{proof}
	First, observe that
	\begin{gather*}
		h_{\mathbf p }(t_0) 
		= \frac{ 1}{ 1 - t_0} (\phi_{t_0}^{-1} )^* \acute G_{\mathbf p} (t_0) - \ol{g}	
		= \eta_{\gamma_0} \sum_{i = 1}^{K} p_i h_i
	\end{gather*}
	by Remark \ref{Rem acute G Observations}.
	Hence, for any $j \in \mathbb{N}$,
	\begin{gather} \label{Decomp of h_p(t_0)} 
	\begin{aligned}
		( h_{\mathbf p } (t_0) , h_j )_{L^2_f } 
		&= \sum_{i = 1}^{K} p_i ( \eta_{\gamma_0} h_i , h_j )_{L^2_f}	\\
		&= \sum_{i = 1}^{K} p_i (  h_i , h_j )_{L^2_f}	
		 + \sum_{i = 1}^{K} p_i ( (\eta_{\gamma_0} - 1) h_i , h_j )_{L^2_f}	\\
		&= \sum_{i = 1}^{K} p_i \delta_{ij}
		 + \sum_{i = 1}^{K} p_i ( (\eta_{\gamma_0} - 1) h_i , h_j )_{L^2_f},
	\end{aligned} \end{gather}
	where here $\delta_{\cdot \cdot}$ denotes the Kronecker delta function.
	For any $j \in \mathbb{N}$ and any $i \le K$,
	\begin{align*}
		& \quad \left| ( ( \eta_{\gamma_0}  - 1 ) h_i , h_j )_{L^2_f} \right|	\\
		&\le \int_{\supp ( \eta_{\gamma_0}  - 1 ) } | \eta_{\gamma_0}  - 1 | | \langle h_i, h_j \rangle | e^{-f} dV \\
		&\le \int_{ \{ f \ge \frac{1}{2} \gamma_0 e^{\tau_0}  \} }  |  h_i | | h_j  | e^{-f} dV \\
		&\le \left( \int_{ \{ f \ge \frac{1}{2} \gamma_0 e^{\tau_0} \} } | h_i |^2 e^{-f} dV \right)^{1/2}
		 \| h_j \|_{L^2_f(M) }	
		&& ( \text{Cauchy-Schwartz} ) 	\\
		&\le \left( \int_{ \{ f \ge \frac{1}{2} \gamma_0 e^{\tau_0} \} } | h_i |^2 e^{-f} dV \right)^{1/2} 
		&& ( \| h_j \|_{L^2_f(M) }	=1).
	\end{align*}	
	Using the eigenmode growth condition (Definition \ref{Defn Eigmode Growth Assumption}) with $\delta = 1$,
	$| h_i | \le C_i f^{ \max \{ -\lambda_i, 0 \} + 1 }$ and 
	the integral may then be estimated by 
	\begin{gather} \label{L^2_f Est of h_i at Infty}
	\begin{aligned}	
		& \quad  \left( \int_{ \{ f \ge \frac{1}{2} \gamma_0 e^{\tau_0} \} } | h_i |^2 e^{-f} dV \right)^{1/2} \\
		&\le \left( \int_{ \{ f \ge \frac{1}{2} \gamma_0 e^{\tau_0} \} } C_i^2 f^{ 2 \max \{ - \lambda_i, 0 \} + 2} e^{-f} dV \right)^{1/2}	\\
		&\lesssim_{n, M, \ol{g}, f, \lambda_*}
		 \left( \int_{ \{ f \ge \frac{1}{2} \gamma_0 e^{\tau_0} \} }  f^{ 2 \max \{ -\lambda_K, 0 \} + 2 } e^{-f} dV \right)^{1/2}	
		 && ( i \le K = K(n, M, \ol{g}, f, \lambda_* ) ) \\
		&\le \left( \int_{ M }  f^{ 4 \max \{ - \lambda_K , 0 \} + 4 } e^{-f} dV \right)^{1/4}	
		\left( \int_{ \{ f \ge \frac{1}{2} \gamma_0 e^{\tau_0} \} }   e^{-f} dV \right)^{1/4}
		&& ( \text{Cauchy-Schwartz} ) 	\\
		&\lesssim_{n, M, \ol{g}, f, \lambda_*} \left( \frac{1}{2} \gamma_0 e^{\tau_0} \right)^{ \frac{n}{8} - \frac{1}{4} } e^{ - \frac{ 1}{32} \left( \frac{1}{2} \gamma_0 e^{\tau_0} \right) }
		&& ( \text{Propositions \ref{Prop Moment Ests} and \ref{Prop Weighted Vol Decay}} )	\\
		&\le e^{- \frac{ 1}{100} \gamma_0 e^{\tau_0} }	
		&& (\tau_0 \gg 1).
	\end{aligned} \end{gather}
	Recalling \eqref{Decomp of h_p(t_0)}, it follows that for any $j \in \mathbb{N}$,
	\begin{gather} \label{Eqn Doubly Exp Decay}
	\begin{aligned}
		 \left| ( h_{\mathbf p } (t_0) , h_j )_{L^2_f } 
		- \sum_{i = 1}^{K} p_i \delta_{ij}  \right|	
		&= \left| \sum_{i = 1}^{K} p_i ( (\eta_{\gamma_0} - 1) h_i , h_j )_{L^2_f}
		  \right|	\\
		&\le  \ol{p}  e^{\lambda_* \tau_0} \sum_{i =1}^{K}
		 \left| ( (\eta_{\gamma_0} - 1) h_i , h_j )_{L^2_f}\right|\\
		&\le \sum_{i =1}^{K}
		 \left| ( (\eta_{\gamma_0} - 1) h_i , h_j )_{L^2_f}\right|
		 \\
		&\lesssim_{n, M, \ol{g}, f , \lambda_*} e^{- \frac{1}{100} \gamma_0 e^{\tau_0}} 
		&& ( \tau_0 \gg 1),
	\end{aligned}
	\end{gather}
	and estimate \eqref{L^2_f Est at t_0 Unstable}  then follows from taking $j \le K$.
	
	To prove \eqref{L^2_f Est at t_0 Stable}, we estimate as follows:
	\begin{gather*} \begin{aligned}
		 \| \pi_s h_{\mathbf p} (t_0) \|_{L^2_f} 	
		&=  \left \| \pi_s \left( h_{\mathbf p} (t_0) - \sum_{j = 1}^{K} p_j h_j \right) \right \|_{L^2_f}	\\ 
		& \le \left\| h_{\mathbf p }(t_0) - \sum_{j=1}^{K} p_j h_j  \right\|_{L^2_f}	\\
		&= \left \| (  \eta_{\gamma_0} -1 ) \left( \sum_{j = 1}^{K} p_j h_j  \right) \right \|_{L^2_f}	\\
		&\le \sum_{j = 1}^{K}  \ol{p}  e^{\lambda_* \tau_0} \| (  \eta_{\gamma_0}  - 1) h_j \|_{L^2_f}	\\	
		&\le \sum_{j = 1}^K \| (  \eta_{\gamma_0}  - 1) h_j \|_{L^2_f}	
		&& ( \tau_0 \gg 1) 	\\
		&\le \sum_{j= 1}^K \left( \int_{\{ f \ge \frac{1}{2} \gamma_0 e^{\tau_0} \} } | h_j|^2 e^{-f} dV \right)^{1/2}	\\
		&\lesssim_{n, M, \ol{g}, f, \lambda_*} e^{- \frac{1}{100} \gamma_0 e^{\tau_0}}
		&& \eqref{L^2_f Est of h_i at Infty}. 
	\end{aligned} \end{gather*}
\end{proof}

\begin{lem}[$C^m$ Estimates at $t_0$]  \label{Lem C^m Ests at t_0}
	For any $m \in \mathbb{N}$ there exists a constant $C$ (depending on $n, M, \ol{g}, f, \lambda_*, m$) such that 
		$$| \ol{\nabla}^m h_{\mathbf p } (t_0) |_{\ol{g}} \le C \ol{p} \gamma_0^{| \lambda_K |}$$
	if $0 < 1 - t_0 \ll1 $ is sufficiently small (depending on $n, M, \ol{g}, f, \gamma_0, m$).
\end{lem}
\begin{proof}
	The proof is similar to the proof of Proposition \ref{Prop Difference of G_p from G_0} but we provide the details nonetheless.
	
	At $t = t_0$,
		$$h_{\mathbf p }(t_0) = 
		\frac{1}{1 - t_0} (\phi_{t_0}^{-1})^* \acute G_{\mathbf p }(t_0) - \ol{g}
		= \eta_{\gamma_0}   \sum_{j = 1}^{K} p_j h_j  $$
	by Remark \ref{Rem acute G Observations}.
	Let $m \in \mathbb{N}$. By Proposition \ref{Prop Eigmode Derivative Growth}, it follows that
	\begin{align*}
		& \quad | \ol{\nabla}^m h_{\mathbf p}(t_0) |_{\ol{g}}	\\
		&\lesssim_{n,m} \sum_{l = 0}^m | \ol{\nabla}^{m - l} \eta_{\gamma_0} |_{\ol{g}}
		\left( \ol{p} e^{\lambda_* \tau_0} \sum_{j = 1}^{K} | \ol{\nabla}^l h_j|_{\ol{g}} \right)	\\
		& \le \sum_{l = 0}^m | \ol{\nabla}^{m - l} \eta_{\gamma_0} |_{\ol{g}}
		\left(  \ol{p} e^{\lambda_* \tau_0} \sum_{j = 1}^{K}  C(n, M, \ol{g}, f, j, \delta,  l) f^{ \max \{ - \lambda_j, 0 \} + \delta} 
		  \right)  	\\
		& \lesssim_{n, M, \ol{g}, f, \lambda_*, \delta, m}  \sum_{l = 0}^m | \ol{\nabla}^{m - l} \eta_{\gamma_0} |_{\ol{g}}
		\left( \ol{p} e^{\lambda_* \tau_0} \sum_{j = 1}^{K}   f^{\max \{ - \lambda_j, 0 \} + \delta} \right)  .
	\end{align*}
	Since $\lambda_* < \min \{ \lambda_K, 0 \} = \min_{j \le K} \min \{ \lambda_j, 0 \}$,
	there exists positive $\delta = \delta ( n, M, \ol{g}, f, \lambda_*)$ such that 
	$\lambda_* + \delta < \min \{ \lambda_K, 0 \} = \min_{j \le K} \min \{ \lambda_j, 0 \}$.
	Combined with the fact that $f$ is bounded below on $M$ (Lemma \ref{Lem Nonflat and f>0}), we can then estimate
	\begin{align*}
		 &\quad \,   \sum_{l = 0}^m | \ol{\nabla}^{m - l} \eta_{\gamma_0} |_{\ol{g}}
		\left( \ol{p} e^{\lambda_* \tau_0} \sum_{j = 1}^{K}   f^{ \max \{ - \lambda_j, 0 \} + \delta } \right)  		\\
		&= \sum_{l = 0}^m | \ol{\nabla}^{m - l} \eta_{\gamma_0} |_{\ol{g}}
		\left( \ol{p} e^{\lambda_* \tau_0}\sum_{j = 1}^{K} 
		  f^{\max \{ - \lambda_j, 0 \} + \delta + \lambda_* } f^{- \lambda_*} \right)  		\\
		&\lesssim_{n, M, \ol{g}, f, \lambda_*}
		\sum_{l = 0}^m | \ol{\nabla}^{m - l} \eta_{\gamma_0} |_{\ol{g}}
		 \ol{p}    f^{- \lambda_*} e^{\lambda_* \tau_0}		\\
		&\le \sum_{l = 0}^m | \ol{\nabla}^{m - l} \eta_{\gamma_0} |_{\ol{g}}
		   \ol{p}    \gamma_0^{|\lambda_*|} 
	\end{align*}
	where the last inequality follows from the fact that $\supp \eta_{\gamma_0} \subset \{ f \le \gamma_0 e^{\tau_0} \}$ and $- \lambda_* = | \lambda_*|$.
	The derivative estimates for $\eta_{\gamma_0}$ (see Subsection \ref{Subsect The Initial Data}) then complete the proof.
\end{proof}

\subsection{Must Exit Through $\cl{B}$} \label{Subsect Must Exit Through clB}

In this subsection, we establish well-posedness results for the harmonic map heat flow solution $\tl{\Phi}$ that hold so long as $h = \frac{1}{1-t} ( \Phi_t^{-1})^* \acute G - \ol{g}$ remains small in a $C^2$ sense.
Informally, these well-posedness results imply that, at the first time $\mathbf p$ exits the set $\cl{P}$ (given by Definition \ref{Defn clP}), it must be because $h_{\mathbf p}$ fails to satisfy the estimates defining $\cl{B}$ and not because the harmonic map heat flow solutions $\tl{\Phi}$ fail to exist nor because we've reached the singular time of $G_{\mathbf p}$.

\begin{lem} \label{Lem t_1^* Well-Defined}
	If $\ol{p} \le 1$,
	$\Gamma_0 \gg 1$ is sufficiently large (depending on $n, M, \ol{g}, f$),
	$0 < \epsilon_1 \ll 1$ is sufficiently small (depending on $n$),
	$0 < \epsilon_0 \ll 1$ is sufficiently small (depending on $n, M, \ol{g}, f, \epsilon_1, \epsilon_2$),
	$0 < \gamma_0 \ll 1$ is sufficiently small (depending on $n, M, \ol{g}, f, \lambda_* , \epsilon_0$),
	and $\tau_0 \gg 1$ is sufficiently large (depending on $n, M, \ol{g}, f, \lambda_*, \Gamma_0, \gamma_0$),
	then, 
	for all $| \mathbf p| \le \ol{p} e^{\lambda_* \tau_0}$, 
	there exists $t_1 = t_1(\mathbf p )  > t_0$ such that 
	$$\mathbf p \in \cl{P} [ \lambda_* , \ol{p}, \Gamma_0, \gamma_0, \mu_u = 2, \mu_s =2, \epsilon_0, \epsilon_1, \epsilon_2, t_0, t_1 ].$$
	
	In particular, for all $| \mathbf p | \le \ol{p} e^{\lambda_* \tau_0}$,
		$$t_1^* = t_1^*(\mathbf p)  \doteqdot \sup \{ t_1 \in (t_0, 1] : \mathbf p \in \cl{P} [ \lambda_* , \ol{p}, \Gamma_0, \gamma_0, \mu_u = 2, \mu_s =2, \epsilon_0, \epsilon_1, \epsilon_2, t_0, t_1 ] \} $$
	is well-defined and $t_1^* > t_0$.
\end{lem}
\begin{proof}
	First note that $T( \mathbf p ) > t_0$ follows automatically from the short-time existence of the Ricci flow $G_{\mathbf p}(t)$.
	
	The existence of a harmonic map heat flow solution $\tl{\Phi}$ will follow from Proposition \ref{Prop Existence HarMapFlow} after we confirm that $\acute G_{\mathbf p}(t)$ and $(1 - t) \phi_t^* \ol{g}$ satisfy the assumptions of the proposition.
	Since $G_{\mathbf p}(t)$ has bounded curvature for short time, 
	$G_{\mathbf p}(t)$ remains in an arbitrarily small $C^{12}(M)$-neighborhood of $G_{\mathbf p }(t_0)$ for all $t \in [t_0, t_1]$ if $0 < t_1 - t_0 \ll 1$ is sufficiently small.
	Similarly, $(1 - t)\phi_t^* \ol{g}$ remains in an arbitrarily small $C^{12}(M)$-neighborhood of $(1 - t_0) \phi_{t_0}^* \ol{g}$ for all $t \in [t_0, t_1]$ if $0< t_1 - t_0 \ll 1$ is sufficiently small.
	Combined with the derivative bounds for $\eta_{\Gamma_0}$, it follows that
		$$\acute G_{\mathbf p }(t) = \eta_{\Gamma_0} \iota_* G_{\mathbf p }(t) + ( 1 - \eta_{\Gamma_0}) ( 1 - t) \phi_t^* \ol{g}$$
	remains in an arbitrarily small $C^{12}(M)$-neighborhood of $\acute G_{\mathbf p }(t_0)$ 
	for all $t \in [t_0, t_1]$ if $0< t_1 - t_0 \ll 1$ is sufficiently small.
	Thus, the derivatives of order at most $10$ of the curvature of $\acute G_{\mathbf p }(t)$ 
	are uniformly bounded for $t \in [t_0, t_1]$ if $0 < t_1 - t_0 \ll 1$ is sufficiently small.
	When $(1 - t_0) \ll 1$ is sufficiently small depending on $n, M, \ol{g}, f, \Gamma_0$, Lemma \ref{Lem difference from RF} implies uniform bounds on 
		$$\left| \partial_t \acute G_{\mathbf p }( t)+ 2 Rc[ \acute G_{\mathbf p } (t) ] \right|_{\acute G_{\mathbf p } (t) } $$
	for all $t \in [t_0, t_1 ] \subset [ t_0, T(\mathbf p ) )$.
	Finally, $\tl{\Phi}_{t_0} = Id_M$ is a diffeomorphism and
	\begin{gather*}
		 \left| \acute G_{\mathbf p }(t_0) - ( 1 - t_0) \phi_{t_0}^* \ol{g} \right|_{(1 - t_0) \phi_{t_0}^* \ol{g}} 	
		= \left| \eta_{\gamma_0} \sum_{j = 1}^{K} p_j h_j  \right|_{ \ol{g}} 	
		\lesssim_{n, M, \ol{g}, f, \lambda_*}  \ol{p}  \gamma_0^{| \lambda_*|}
	\end{gather*}
	by the proof of Proposition \ref{Prop Difference of G_p from G_0}.
	When additionally $\ol{p} \le 1$ and $\gamma_0 \ll 1$ is sufficiently small depending on $n, M, \ol{g}, f, \lambda_*,  \epsilon_0$,
	this estimate implies 
		$$\left| \acute G_{\mathbf p }(t_0) - ( 1 - t_0) \phi_{t_0}^* \ol{g} \right|_{(1 - t_0) \phi_{t_0}^* \ol{g}} 	
		\le \frac{1}{2} \epsilon_0.$$
	If $\epsilon_0 \ll 1$ is sufficiently small depending on $n$, then Proposition \ref{Prop Existence HarMapFlow}  applies to give the short-time existence of a smooth family of diffeomorphisms $\tl{\Phi} : M \times [ t_0, t_1 ) \to M$ solving the harmonic map heat flow with initial condition $\tl{\Phi}_{t_0} = Id_M$.
	Additionally, Proposition \ref{Prop Existence HarMapFlow} states that this $\tl{\Phi}$ may be chosen such that
		$$\left| \frac{1}{1- t}( \Phi_t^{-1} )^*\acute G_{\mathbf p }(t) -  \ol{g} \right|_{ \ol{g}} 
		=\left| ( \tl{\Phi}_t^{-1} )^*\acute G_{\mathbf p }(t) - ( 1 - t) \phi_{t}^* \ol{g} \right|_{(1 - t) \phi_{t}^* \ol{g}} 	
		\le  \epsilon_0 	$$
	on $M \times [ t_0, t_1).$
	
	Recall $\tau = - \ln ( 1 - t), \tau_0 = - \ln ( 1 - t_0)$, and denote $\tau_1 = - \ln ( 1 - t_1) < \infty $.
	By Corollary \ref{Cor h(t) Evol Eqn},
	$h(\tau) = h_{\mathbf p}(t(\tau))$ is a smooth solution of
	\begin{equation}	\tag{\ref{h(tau) Evol Eqn}}
		\partial_\tau h = \ol{\Delta}_f h + 2 \ol{Rm}[h] + \cl{E}_1 + \cl{E}_2	
		\qquad \text{ on } M \times (\tau_0, \tau_1) 
	\end{equation}
	where $\cl{E}_1, \cl{E}_2$ are defined as in \eqref{clE_1 Eqn}, \eqref{clE_2 Eqn} respectively.
	By Lemma \ref{Lem C^m Ests at t_0}, we may assume
		$$\sum_{m = 1}^{3} | \ol{\nabla}^m h(\tau_0) |_{\ol{g}}
		 \le  C(n, M, \ol{g}, f, \lambda_*) \ol{p} \gamma_0^{| \lambda_*|} \le \epsilon_0$$
	if $\ol{p} \le 1$, 
	$0 < \gamma_0 \ll 1$ is sufficiently small depending on $n, M, \ol{g}, f, \lambda_*, \epsilon_0$,
	and $\tau_0 \gg 1$ is sufficiently large depending on $n, M, \ol{g}, f, \gamma_0$.
	Since $| h(\tau) |_{\ol{g}} \le \epsilon_0$ for $\tau \in [\tau_0, \tau_1)$,
	Lemma \ref{Lem Int Est Large Scale} with $\delta = \epsilon_0$ therefore implies that 
		$$| \ol{\nabla} h|_{\ol{g}}(x, \tau) + | \ol{\nabla}^2 h|_{\ol{g}}(x, \tau) \lesssim_{n, M, \ol{g}, f} \epsilon_0
		\qquad \text{ on } \{ (x, \tau) \in M \times [\tau_0, \tau_1) : 4 \gamma_0 e^\tau < f(x) \}$$
	if also $\epsilon_0 \ll 1$ is sufficiently small depending on $n, M, \ol{g}, f$.
	Since $h$ is smooth and the complement of $\{ 4 \gamma_0 e^{\tau_1} < f \} \subset M$ is compact, it follows that, after possibly taking $\tau_1$ smaller, in fact
		$$| \ol{\nabla} h|_{\ol{g}} + | \ol{\nabla}^2 h|_{\ol{g}} \lesssim_{n, M, \ol{g}, f} \epsilon_0
		\qquad \text{ on } M \times [\tau_0, \tau_1).$$
	In particular, 
	\begin{align*}	
		| \ol{\nabla} h|_{\ol{g}} &\le \epsilon_1	 \text{ on } M \times [ \tau_0, \tau_1) \text{ and }	\\
		| \ol{\nabla}^2 h|_{\ol{g}} &\le \epsilon_2	 \text{ on } M \times [ \tau_0, \tau_1)
	\end{align*}
	if $0 < \epsilon_0 \ll 1$ is sufficiently small depending on $n, M, \ol{g}, f, \epsilon_1, \epsilon_2$.
	
	With these $C^2$ estimates for $h$, the proof of Lemma \ref{Lem Drift of the Grafting Region} now shows that 
		$$\supp h(\tau) \subset \{ f \le \Gamma_0 e^\tau \} \qquad \text{ for all } \tau \in [\tau_0, \tau_1)$$
	so long as $\Gamma_0 \gg 1$ is sufficiently large depending on $n, M, \ol{g}, f$,
	$0 < \epsilon_0, \epsilon_1 \ll 1$ are sufficiently small depending on $n$,
	and $\tau_0 \gg 1$ is sufficiently large depending on $n, M, \ol{g}, f, \Gamma_0$.
	In particular, $h(\tau)$ is compactly supported for all $\tau \in [ \tau_0, \tau_1)$.
	
	Finally, Lemma \ref{Lem L^2_f Ests at t_0} ensures that, at $\tau_0$,
	\begin{gather*}
		\| \pi_u h_{\mathbf p }( \tau_0) \|_{L^2_f} \le \frac{3}{2} e^{\lambda_* \tau_0}
		\quad 	\text{ and } 	\quad
		\| \pi_s h_{\mathbf p }( \tau_0) \|_{L^2_f} \le \frac{3}{2}  e^{\lambda_* \tau_0}	
	\end{gather*}
	if $\ol{p} \le 1$ and $\tau_0 \gg 1$ is sufficiently large depending on $n, M, \ol{g}, f, \lambda_* , \gamma_0$.
	Since $h_{\mathbf p }$ is smooth and supported in $\{ f \le \Gamma_0 e^{\tau_1} \}$ for all $\tau \in [ \tau_0, \tau_1)$, 
	these $L^2_f$ quantities are continuous in $\tau$.
	Therefore, after possibly taking $t_1 - t_0$ or equivalently $\tau_1 - \tau_0$ smaller, it follows that
	for all $\tau \in [\tau_0, \tau_1)$,
	\begin{gather*}
		\| \pi_u h_{\mathbf p }( \tau) \|_{L^2_f} \le  2 e^{\lambda_* \tau}	
		\quad \text{ and } \quad 
		\| \pi_s h_{\mathbf p }( \tau) \|_{L^2_f} \le 2 e^{\lambda_* \tau}	.
	\end{gather*}
	This completes the proof.
\end{proof}

\begin{remark}	\label{Rem t_1^* = inf}
	By Definition \ref{Defn clP} of $\cl{P}$, $t_1 < t_1'$ implies
		\begin{align*}
		&\cl{P} [ \lambda_* , \ol{p} , \Gamma_0, \gamma_0, \mu_u , \mu_s, \epsilon_0, \epsilon_1, \epsilon_2 , t_0, t_1' ]	\\
		\subset{}& \cl{P} [ \lambda_* , \ol{p} , \Gamma_0, \gamma_0, \mu_u , \mu_s, \epsilon_0, \epsilon_1, \epsilon_2 , t_0, t_1 ]	.
		\end{align*}
	Thus, $t_1^*( \mathbf p )$ may also be written as
	\begin{multline*}	
		t_1^*( \mathbf p ) \\
		= \inf (\{ t_1 \in (t_0, 1] : \mathbf p \notin \cl{P} [ \lambda_* , \ol{p}, \Gamma_0, \gamma_0, \mu_u = 2,  \mu_s =2, \epsilon_0, \epsilon_1, \epsilon_2, t_0, t_1 ] \}  \cup \{ 1 \}).
	\end{multline*}
\end{remark}

\begin{lem} \label{Lem Must Exit Through clB}
	Assume the constants $\ol{p}, \Gamma_0, \gamma_0,  \epsilon_0, \epsilon_1, \epsilon_2, \tau_0$ are taken sufficiently large or small as in the assumptions of Lemma \ref{Lem t_1^* Well-Defined} so that $t_1^*( \mathbf p)$ is well-defined.
	Additionally, assume that
	$0 < \epsilon_1, \epsilon_2 \ll 1$ are sufficiently small (depending on $n, M, \ol{g},f$).
	
	If $t_1^*( \mathbf p ) < 1$, then 
	$t_1^* < T( \mathbf p )$ and
	there exists $t_1^+ \in ( t_1^*, T)$ such that 
	there exists a unique smooth solution $\tl{\Phi} : M \times [t_0, t_1^+) \to M$ to the harmonic map heat flow 
	satisfying
	\begin{enumerate}
		\item $\tl{\Phi}_t = \tl{\Phi}( \cdot, t) : M \to M$ is a diffeomorphism for all $t \in [t_0, t_1^+)$, and 
		\item with $\Phi_t = \phi_t \circ \tl{\Phi}_t$,
			$$h_{\mathbf p }(t) = \frac{1}{1 - t} (\Phi_t^{-1})^* \acute G_{\mathbf p }(t) - \ol{g} 
			\in \cl{B}[ \lambda_*, \mu_u = 4, \mu_s = 4, 2 \epsilon_0, 2 \epsilon_1, 2 \epsilon_2, [t_0, t_1^+) ]$$
			where $\cl{B}$ is defined as in Definition \ref{Defn Box}.
	\end{enumerate}
	
	Moreover, 
	for all $t \in [t_0, t_1^+ )$,
	\begin{align*}
		\tl{\Phi}_t ( \ol{ \Omega_{\eta_{\Gamma_0}}} ) \subset{} & \{ \Gamma_0 / 2 < f < \Gamma_0 \},	\\
		\Phi_t ( \ol{ \Omega_{\eta_{\Gamma_0}}} ) \subset{} & \{ (1 - t)^{-1} < f < \Gamma_0(1 - t)^{-1}  \}, 
		\text{ and}	\\
		h_{\mathbf p }(t) = 0 		\, \text{ on } & \{ f \ge \Gamma_0 ( 1 - t)^{-1} \}
	\end{align*}
	where recall $\Omega_{\eta_{\Gamma_0}} = \{ x \in M : 0 < \eta_{\Gamma_0} (x) < 1 \}$ as in Subsection \ref{Subsect Ests on Grafting Region}.
\end{lem}

\begin{proof}
	To simplify the notation throughout this proof, we shall write
	\begin{gather*}
		\cl{B} [ \lambda_*, \mu_u = 2,  \mu_s= 2, \epsilon_0, \epsilon_1,\epsilon_2 , I ]	\quad 
		\text{as} \quad  \cl{B} [ \mu_{u,s} = 2, \epsilon_{0,1,2}, I ].
	\end{gather*}
	
	We first establish curvature estimates for $G_{\mathbf p }(t)$ and $\acute G_{\mathbf p}(t)$.
	Let $t_1 \in [t_0,  t_1^*(\mathbf p) )$.
	It follows from the definition of $t_1^*( \mathbf p )$ and $\cl{P}$ (Definition \ref{Defn clP}) that
	there exists $\tl{\Phi} : M \times [ t_0, t_1 ) \to M$ which is a smooth solution to the harmonic map heat flow \eqref{tlPhi Evol Eqn}
	such that $\tl{\Phi}_t : M \to M$ is a diffeomorphism for all $t \in [t_0, t_1)$ and
	$h_{\mathbf p } \in \cl{B}[  \mu_{u,s} = 2 , \epsilon_{0,1,2} , [t_0, t_1 ) ]$.

	Recall from the Definition \ref{Defn clP} of $\cl{P}$ that $h(t) = h_{\mathbf p}(t)$ therefore satisfies the estimates
		$$| \ol{\nabla}^{l} h(t)|_{\ol{g}} \le \epsilon_l 		\qquad \text{for all } t \in [t_0, t_1),  0 \le l \le 2.$$
	Hence, if $\epsilon_0 + \epsilon_1 + \epsilon_2 \ll 1$ is sufficiently small depending on $n, M, \ol{g}$, 
	then $g_{\mathbf p}(t) = \ol{g} + h_{\mathbf p}(t)$ is a family of complete metrics 
	such that 
		$$\sup_{M \times [ t_0, t_1) } |  Rm_{ g_{\mathbf p }(t)} |_{g_{\mathbf p }(t)} 
		\le 2 \sup_{M} |  \ol{Rm} |_{\ol{g}}. $$
	Since $t_1 < t_1^*( \mathbf p ) < 1$ and the $\tl{\Phi}_t$ are diffeomorphisms,
	$\acute G_{\mathbf p }(t) = ( 1- t) (\phi_t \circ \tl{\Phi}_t)^* g_{\mathbf p }(t)$ will also be a family of complete metrics with curvature bounds
	\begin{equation} \label{Must Exit Through B Curv Bounds}
		\sup_{M \times [ t_0, t_1) } |  Rm_{ \acute G_{\mathbf p }(t)} |_{\acute G_{\mathbf p }(t)} 
		\le 2 \sup_{M \times [t_0, t_1) } \frac{1}{1 - t} |  \ol{Rm} |_{\ol{g}} 	
		=  \frac{2}{1 - t_1^*( \mathbf p ) }\sup_{M  }  |  \ol{Rm} |_{\ol{g}}  < \infty.
	\end{equation}
	
	Let $U \subset M$ be the interior of the set where $\eta_{\Gamma_0} = 1$ and note that $\ol{ \{ f < \Gamma_0 / 2 \} } \subset U$ by construction of $\eta_{\Gamma_0}$.
	By construction, $\acute G_{\mathbf p }(t) = \iota_* G_{\mathbf p}(t)$ on $U \times [ t_0, t_1)$. 
	From the curvature bounds \eqref{Must Exit Through B Curv Bounds} combined with the pseudolocality Proposition \ref{Prop Pseudolocality App},
	it follows that there exists $\cl{K}_0$ depending only on $n, M, \ol{g}, f, \Gamma_0, t_1^*( \mathbf p )$ such that 
		$$\sup_{\cl{M} \times [ t_0, t_1) } | Rm_{ G_{\mathbf p }(t)} |_{G_{\mathbf p }(t)} 
		\le \cl{K}_0 < \infty$$
	as long as
	$\Gamma_0 \gg 1$ is sufficiently large depending on $n , M , \ol{g}, f$ and
	$0 < 1 - t_0 \ll 1$ is sufficiently small depending on $n, M, \ol{g}, f, \Gamma_0$.
	Because $G_{\mathbf p }(t)$ is a closed Ricci flow solution, 
	these curvature bounds imply the Ricci flow solution $G_{\mathbf p}(t)$ exists for slightly larger times.
	In other words, $T( \mathbf p) > t_1$
	and, by continuity,
		$$\sup_{\cl{M} \times [ t_0, t_1] } | Rm_{ G_{\mathbf p }(t)} |_{G_{\mathbf p }(t)} 
		\le \cl{K}_0.$$
	Note that $\cl{K}_0$ depends on $t_1^*( \mathbf p )$ but not $t_1$.
	Since $t_1 < t_1^*( \mathbf p)$ was arbitrary and $G_{\mathbf p}(t)$ is a closed Ricci flow solution,
	it follows that
		$$\sup_{\cl{M} \times [ t_0, t_1^*( \mathbf p) ] } | Rm_{ G_{\mathbf p }(t)} |_{G_{\mathbf p }(t)} 
		\le \cl{K}_0$$	
	and thus $ t_1^*( \mathbf p ) < T( \mathbf p ) $.	
	
	Derivative estimates for Ricci flow then imply that there exist
	constants $\cl{K}_1, \dots, \cl{K}_{12}$ (depending on $n$ and $\cl{K}_0$)
	and $T' \in (t_1^*( \mathbf p) , T( \mathbf p ) )$ such that
		$$\sup_{\cl{M} \times [ t_0, T' ] } | \nabla^m Rm_{ G_{\mathbf p }(t)} |_{G_{\mathbf p }(t)} 
		\le 2 \cl{K}_m$$	
	for all $0 \le m \le 12$.
	Taking $\cl{K}_0, \dots, \cl{K}_{12}$ possibly larger depending on $n, M, \ol{g}, f, t_1^*$ and $T' - t_1^*$ possibly smaller, we may additionally assume the soliton also satisfies
	$$\sup_{M \times [ t_0, T' ] } | \nabla^m Rm_{(1- t) \phi_t^* \ol{g}} |_{( 1 - t) \phi_t^* \ol{g}} 
		\le 2 \cl{K}_m$$	
	for all $0 \le m \le 12$.
	If $0 < 1- t_0 \ll 1$ is sufficiently small depending on $n, M, \ol{g}, f, \Gamma_0$,
	derivative estimates for $\eta_{\Gamma_0}$ and the fact that 
		$$\acute G_{\mathbf p }(t) = \eta_{\Gamma_0} \iota_* G + ( 1 - \eta_{\Gamma_0} )(1 -t)\phi_t^* \ol{g},$$
	then imply that the curvature of $\acute G_{\mathbf p } (t)$ and its derivatives of order $\le 10$ are uniformly bounded on $M \times [ t_0, T']$.

	Next, we establish the existence of the harmonic map heat flow solution $\tl{\Phi}$ on the time interval $[ t_0, t_1^+ )$.
	Again, let $t_1 < t_1^*(\mathbf p)$ and let
	$\tl{\Phi} : M \times [ t_0, t_1 ) \to M$ be a smooth solution to the harmonic map heat flow
	such that $\tl{\Phi}_t : M \to M$ is a diffeomorphism for all $t \in [t_0, t_1)$ and
	$h_{\mathbf p } \in \cl{B}[  \mu_{u,s} = 2, \epsilon_{0,1,2} , [t_0, t_1 ) ]$.
	Let $t_1^- \in ( t_0,  t_1)$.
	Note that
		$$\sup_M \left| (\tl{\Phi}_{t_1^-}^{-1} )^* \acute G_{\mathbf p}(t_1^- ) - ( 1 - t_1^-) \phi_{t_1^-}^* \ol{g} \right|_{(1 - t_1^-) \phi_{t_1^-}^* \ol{g}}
		= \sup_M \left| h_{\mathbf p }(t_1^-)  \right|_{\ol{g}}
		\le \epsilon_0.$$
	By Lemma \ref{Lem difference from RF} and the fact that the curvature of $\acute G_{\mathbf p} (t)$ has derivatives which are uniformly bounded up to order $10$ on the time interval $[ t_1^- , T' ] \subset [t_0, T']$, 
	the short-time existence Proposition \ref{Prop Existence HarMapFlow} applies to yield a smooth family of diffeomorphisms $\tl{\tl{\Phi}} : M \times [t_1^-, t_1^- + \delta ] \to M$ solving
		$$\partial_t \tl{\tl{\Phi}} = \Delta_{\acute G_{\mathbf p }(t), ( 1- t) \phi_t^* \ol{g} } \tl{ \tl{\Phi}}
		\qquad \text{ with initial condition } \tl{ \tl{\Phi} }_{t_1^-} = \tl{ \Phi}_{t_1^-}$$
	so long as $\epsilon_0 \ll 1$ is sufficiently small depending on $n$.
	Moreover, Proposition \ref{Prop Existence HarMapFlow} implies 
	$\delta  \in ( 0, T' - t_1^*]$ can be taken to be independent of $t_1^-$ and $t_1$
	and such that 
	\begin{equation} \label{C^0 Bounds on Extension Eqn}
		\left| \left( \tl{ \tl{\Phi}}_{t}^{-1} \right)^* \acute G_{\mathbf p }(t) - ( 1 - t) \phi_t^* \ol{g} \right|_{(1-t) \phi_t^* \ol{g}}
		\le 2 \epsilon_0 		
		\qquad \text{ on } M \times [ t_1^- , t_1^- + \delta ].
	\end{equation}
	
	Now, fix $t_1^-$ and $t_1$ sufficiently close to $t_1^*( \mathbf p )$ so that $t_1^- + \delta > t_1^*( \mathbf p)$ and consider the map $\tl{ \tl{\Phi}}$ as above.
	We shall use the uniqueness result for the harmonic map heat flow (Proposition \ref{Prop Uniqueness HarMapFlow}) to argue that in fact $\tl{ \tl{\Phi}} = \tl{\Phi}$ on the intersection of their domains and thereby obtain the desired extension to the time interval $[t_0, t_1^+ )$ for some $t_1^+ \in (t_1^*,  t_1^- + \delta)$.
	Suppose for the sake of contradiction that 
		$$\tl{ \Phi}(x,t) \ne \tl{ \tl{\Phi}} (x,t) 
		\qquad \text{ for some } (x, t) \in M \times [ t_1^-, t_1 )$$
	and consider
		$$t_2^* =  \inf \{ t \in [t_1^- , t_1 ) : \text{there exists } x \in M \text{ such that } \tl{ \Phi}(x,t) \ne \tl{ \tl{\Phi}} (x,t)  \} \in [t_1^- , t_1 ).$$
	Since $T' < 1$, $(1 - t) \phi_t^* \ol{g}$ has uniform curvature bounds and injectivity radius lower bounds on $M \times [ t_0, T'] \supset [ t_1^-, t_1 ] $.
	Let 
		$$\tl{h}(t) = (\tl{\Phi}_t^{-1} )^* \acute G_{\mathbf p }(t) - (1 - t) \phi_t^* \ol{g} 
		\quad \text{and} \quad
		\tl{ \tilde h}(t) = \left(\tl{\tilde \Phi}_t^{-1} \right)^* \acute G_{\mathbf p }(t) - (1 - t) \phi_t^* \ol{g} .$$
	Then
		$$\sup_{M \times [t_0, t_1)} \left| \tl{h}(t) \right|_{(1 - t) \phi_t^* \ol{g} } \le \epsilon_0,
		\qquad 
		\sup_{M \times [t_1^-, t_1^- + \delta] } \left| \tl{\tilde h}(t) \right|_{(1 - t) \phi_t^* \ol{g} } \le 2\epsilon_0,$$
		$$\text{and } \sup_{M \times [t_0, t_1)} \left| {}^{(1-t) \phi_t^* \ol{g} } \nabla \tl{h}(t) \right|_{(1 - t) \phi_t^* \ol{g} } \le \frac{\epsilon_1}{\sqrt{ 1  - t} }.$$	
	By similar logic as in the proof of Lemma \ref{Lem t_1^* Well-Defined}, one obtains derivative estimates for $\tl{\tl{h}}$ of the form
		$$ \left| {}^{(1-t) \phi_t^* \ol{g} } \nabla \tl{ \tl{h}}(t) \right|_{(1 - t) \phi_t^* \ol{g} } \lesssim_{n, M, \ol{g}, f} \frac{\epsilon_0 + \epsilon_1}{\sqrt{ 1  - t} }$$	
	for all $t$ in a neighborhood of $t_2^*$
	if $0 < \epsilon_0, \epsilon_1 \ll 1$ are sufficiently small depending on $n, M, \ol{g}, f$ and 
	$1 - t_0 \ll 1$ is sufficiently small depending on $n, M, \ol{g}, f$.
	If $0 < \epsilon_0 , \epsilon_1 \ll 1$ are sufficiently small depending on $n, M, \ol{g}, f$,
	then Proposition \ref{Prop Uniqueness HarMapFlow} implies that $\tl{\Phi}(x,t) = \tl{ \tl{\Phi}}(x,t)$ for all $x \in M$ and all $t$ in a neighborhood of $t_2^*$, which contradicts the definition of $t_2^*$.
	This contradiction implies $\tl{ \Phi} = \tl{ \tl{\Phi}}$ on the intersection of their domains, and we thereby obtain the desired extension to the time interval $[t_0, t_1^- + \delta ) \supsetneq [t_0, t_1^* )$ by gluing.

	Note in particular that 
	\begin{align*}
		\sup_{M \times [ t_0, t_1^- + \delta) }
		\left| h_{\mathbf p }(t) \right|_{\ol{g}}
		&=
		\sup_{M \times [ t_0, t_1^- + \delta)} 
		\left|  \frac{1}{1 -t} (\Phi_t^{-1})^* \acute G_{\mathbf p}(t) - \ol{g} \right|_{\ol{g}}	\\
		&=
		\sup_{M \times [ t_0, t_1^- + \delta)} 
		\left|   (\tl{\Phi}_t^{-1})^* \acute G_{\mathbf p}(t) -(1- t) \phi_t^* \ol{g} \right|_{(1 - t) \phi_t^* \ol{g}}		\\
		&\le 2 \epsilon_0
	\end{align*}
	by equation \eqref{C^0 Bounds on Extension Eqn}.
	Similar logic as in the proof of Lemma \ref{Lem t_1^* Well-Defined} then implies that in fact, for some $t_1^+ \in (t_1^* , t_1^- + \delta )$,
	$$h_{\mathbf p }(t) = \frac{1}{1 -t} (\Phi_t^{-1})^* \acute G_{\mathbf p}(t) - \ol{g} \in 
	\cl{B}[ \mu_{u,s} = 4, 2 \epsilon_0, 2 \epsilon_1, 2 \epsilon_2, [t_0, t_1^+) ]$$
	if the constants $\ol{p}, \Gamma_0, \gamma_0, \epsilon_0, \epsilon_1, \epsilon_2, \tau_0$ are taken sufficiently large or small as in the assumptions of Lemma \ref{Lem t_1^* Well-Defined}.
	This completes the proof of the existence statement for $\tl{\Phi}$.

	The uniqueness statement follows from Proposition \ref{Prop Uniqueness HarMapFlow}.
	Finally, the last statement of the lemma is a direct consequence of Lemma \ref{Lem Drift of the Grafting Region}.
\end{proof}

\subsection{Completing the Proof of \ref{Main Thm}}

In this subsection, we complete the proof of Theorem \ref{Main Thm} through the use of a Wa{\.z}ewski box argument.

\begin{assumption} \label{Assume for Contradiction}
	Throughout this subsection or equivalently the remainder of the paper, 
	we assume that the constants $\ol{p}, \Gamma_0, \gamma_0,  \epsilon_0, \epsilon_1, \epsilon_2, \tau_0$
	sufficiently large or small so that Lemmas \ref{Lem t_1^* Well-Defined} and \ref{Lem Must Exit Through clB} apply.
	In other words,
	$0 < \ol{p} \le 1$,
	$ \Gamma_0 \gg 1$ is sufficiently large (depending on $n, M, \ol{g}, f$),
	$0 < \epsilon_1, \epsilon_2 \ll 1$ are sufficiently small (depending on $n, M, \ol{g}, f$),
	$0 < \epsilon_0 \ll 1$ is sufficiently small (depending on $n, M, \ol{g}, f, \epsilon_1, \epsilon_2 $),
	$0 < \gamma_0 \ll 1$ is sufficiently small (depending on $n, M, \ol{g}, f, \lambda_*, \epsilon_0$), and
	$\tau_0 \gg 1$ is sufficiently large (depending on $n, M, \ol{g}, f, \lambda_*, \Gamma_0, \gamma_0$).
	
	Also, throughout this subsection, suppose for the sake of contradiction that $t_1^*( \mathbf p ) < 1$ for all $| \mathbf p| \le \ol{p} e^{\lambda_* \tau_0}$.
	We will arrive at a contradiction at the end of this subsection in the proof of \ref{Main Thm}.
\end{assumption}

For each $\mathbf p$, consider the smooth family of diffeomorphisms $\tl{\Phi}: M \times [ t_0, t_1^+( \mathbf p ) ) \to M$ as in Lemma \ref{Lem Must Exit Through clB} and let
	$$h_{\mathbf p }(t) = \frac{1}{1 -t} ( \Phi_t^{-1})^* \acute G_{\mathbf p }(t) - \ol{g}
	\qquad \text{ for all } t \in [t_0, t_1^+( \mathbf p )  ) .$$
	
\begin{lem} \label{Lem t_mu Well-Defined}
	Let $| \mathbf p | \le \ol{p} e^{\lambda_* \tau_0}$.
	For all $0 < \mu_u, \mu_s < 2$ and all $t_2 \in ( t_1^* (\mathbf p ), t_1^+( \mathbf p ))$,
	$$h_{\mathbf p }  \notin 
	\cl{B}[ \lambda_*, \mu_u   , \mu_s , \epsilon_0, \epsilon_1, \epsilon_2,  [t_0, t_2 ] ].$$
\end{lem}	
\begin{proof}
	Suppose not.
	Then there exists $0 < \mu_u, \mu_s < 2$ and $t_2 \in ( t_1^*, t_1^+ )$
	such that
	\begin{align*}
		h_{\mathbf p } &\in \cl{B} [ \lambda_* , \mu_u, \mu_s, \epsilon_0, \epsilon_1, \epsilon_2, [t_0, t_2] ]	
		\\
		&\subset  \cl{B} [ \lambda_* , 2, 2, \epsilon_0, \epsilon_1, \epsilon_2, [t_0, t_2) ].
	\end{align*}
	Hence,
		$$\mathbf p \in \cl{P} [\lambda_*, \ol{p}, \Gamma_0, \gamma_0, 2, 2, \epsilon_0, \epsilon_1, \epsilon_2, t_0, t_2 ]$$
	and $t_2 > t_1^*( \mathbf p )$, which contradicts the definition of $t_1^*( \mathbf p )$.
\end{proof}
		
For $0 < \mu_u , \mu_s\le  2$, define
\begin{equation}
	t_{\mu_u, \mu_s}^* ( \mathbf p ) 
	= \inf \{ t_2 \in [t_0, t_1^+ ( \mathbf p ) ) :
	h_{\mathbf p } \notin \cl{B} [ \lambda_*, \mu_u , \mu_s, \epsilon_0, \epsilon_1, \epsilon_2, [t_0, t_2] ] \}
\end{equation}
as the first time $h_{\mathbf p}$ exits the set $\cl{B}$ defined with parameters $\mu_u, \mu_s$.
By Lemma \ref{Lem t_mu Well-Defined}, $t_{\mu_u, \mu_s}^* ( \mathbf p )$ is well-defined
and 
	$$t_{\mu_u, \mu_s}^*( \mathbf p ) \le t_{2,2}^* ( \mathbf p ) =  t_1^*( \mathbf p )  <t_1^+( \mathbf p ) < 1.$$
Moreover,
	$$h_{\mathbf p} \in \cl{B} [ \lambda_* , \mu_u, \mu_s, \epsilon_0, \epsilon_1, \epsilon_2, [t_0, t_{\mu_u, \mu_s}^* ( \mathbf p )] ].$$
Indeed, this follows from the fact that $h_{\mathbf p}$ is smooth, its support can be controlled by Lemma \ref{Lem Must Exit Through clB}, and that Definition \ref{Defn Box} of $\cl{B}$ is defined by closed conditions.

For the closed ball $\ol{B_{\ol{p} e^{\lambda_* \tau_0} }} \subset \R^K$ of radius $\ol{p} e^{\lambda_* \tau_0}$ centered at the origin,
we consider the map
\begin{gather*}
	\cl{F} : \ol{B_{\ol{p} e^{\lambda_* \tau_0} }} \to \R^K	\qquad \text{ given by}	\\
	\cl{F}(\mathbf p ) 
	= 
	\left( \big( h_{\mathbf p }(t_{\mu_u, \mu_s}^*(\mathbf p) )  , h_1 \big)_{L^2_f(M)}, 
	\dots ,
	\big( h_{\mathbf p }(t_{\mu_u, \mu_s}^*( \mathbf p) )  , h_{K} \big)_{L^2_f(M)} \right).
\end{gather*}

\begin{lem} \label{Lem Homotopic to Id on Bndry}
	Assume $0 < \mu_u < \frac{1}{2} \ol{p}$ and 
	$\tau_0 \gg 1$ is sufficiently large (depending on $n, M, \ol{g}, f, \lambda_* , \gamma_0, \mu_u$).
	Consider the closed annulus $\ol{A} \subset \R^K$ given by
		$$\ol{A} = \{ \mathbf p \in \R^K : 2 \mu_u e^{\lambda_* \tau_0} \le | \mathbf p | \le \ol{p} e^{\lambda_* \tau_0} \}.$$
		
	For all $\mathbf p \in \ol{A}$ and any $0 < \mu_s < 2$, $t_{\mu_u, \mu_s}^* ( \mathbf p ) = t_0$.
	Moreover, $\cl{F}\left( \ol{A} \right) \subset \R^K \setminus \{ \mathbf 0 \}$ and the restricted map
		$$\cl{F}|_{\ol{A}} : \ol{A} \to \R^K \setminus \{ \mathbf 0 \}$$
	is homotopic to the identity $Id : \ol{A} \to \ol{A} \subset \R^K \setminus \{ 0 \}$.
\end{lem}
\begin{proof}
	Let $2 \mu_u e^{\lambda_* \tau_0} \le |\mathbf p | \le \ol{p} e^{\lambda_* \tau_0} $.
	Then
	\begin{align*}
		\| \pi_u h_{\mathbf p }(t_0) \|_{L^2_f}
		&= \sqrt{ \sum_{j = 1}^{K} ( h_{\mathbf p}( t_0) , h_j )_{L^2_f}^2 }	\\
		&\ge | \mathbf p | - C(n, M, \ol{g}, f, \lambda_*) e^{- \frac{\gamma_0}{100} e^{\tau_0}}	
		&& ( \text{Lemma \ref{Lem L^2_f Ests at t_0}} )	\\
		&\ge 2 \mu_ue^{\lambda_* \tau_0} - C(n, M, \ol{g}, f, \lambda_*) e^{- \frac{\gamma_0}{100} e^{\tau_0}}	\\
		&> \mu_ue^{\lambda_* \tau_0}
	\end{align*}	
	if $\tau_0 \gg 1$ is sufficiently large depending on $n, M, \ol{g}, f, \lambda_* , \gamma_0, \mu_u$.
	$\| \pi_u h_{\mathbf p}(t_0) \|_{L^2_f} > \mu_u e^{\lambda_* \tau_0}$
	therefore implies 
	$$h_{\mathbf p }(t_0) \notin \cl{B} [ \lambda_* , \mu_u, \mu_s, \epsilon_0, \epsilon_1 , \epsilon_2, \{ t_0 \} ]$$
	and thus $t_{\mu_u, \mu_s}^* ( \mathbf p )  = t_0$.
	
	Now, consider the straight-line homotopy from the identity to $\cl{F}$, that is
		$$\cl{F}_s ( \mathbf p ) \doteqdot \mathbf p + s ( \cl{F}(\mathbf p ) - \mathbf p ) ,
		\qquad  (s \in [0,1]).$$
	If $\tau_0 \gg 1$ is sufficiently large depending on $n, M, \ol{g}, f, \lambda_* ,\gamma_0, \mu_u$, then
	for any $\mathbf p \in \ol{A}$ and $s \in [0, 1]$,
	\begin{align*}
		& \quad \,     \left| \mathbf p + s( \cl{F}( \mathbf p ) - \mathbf p)  \right|	\\
		&= 
		\left| \mathbf p + s \left( \big( h_{\mathbf p }(t_0 )  , h_1 \big)_{L^2_f(M)} - p_1, 
		\dots ,
		\big( h_{\mathbf p }(t_0 )  , h_{K} \big)_{L^2_f(M)} - p_{K } \right)  \right|	\\
		&\ge | \mathbf p | - s \left| \left( \big( h_{\mathbf p }(t_0 )  , h_1 \big)_{L^2_f(M)} - p_1, 
		\dots ,
		\big( h_{\mathbf p }(t_0 )  , h_{K} \big)_{L^2_f(M)} - p_{K} \right)  \right|	\\
		&\ge 	2 \mu_u e^{\lambda_* \tau_0} 
		- C(n, M , \ol{g}, f, \lambda_*) e^{- \frac{ \gamma_0}{100} e^{\tau_0}}	
		&& \eqref{L^2_f Est at t_0 Unstable}		\\
		&> 0 && ( \tau_0 \gg 1).
	\end{align*}
	Therefore, $\cl{F}( \ol{A} ) = \cl{F}_{s = 1} ( \ol{A} ) \subset \R^K \setminus \{ 0 \}$
	and the straight-line homotopy $\cl{F}_s$ gives a homotopy from the identity $Id: \ol{A} \to \ol{A} \subset \R^K \setminus \{ \mathbf 0 \}$ to $\cl{F}|_{\ol{A}} : \ol{A} \to \R^K \setminus \{ \mathbf 0 \}$.
\end{proof}

\begin{lem} \label{Lem Exits Unstable Side}
	If $0 < \ol{p} \ll 1$ is sufficiently small (depending on $n, M, \ol{g}, f, \lambda_*$),
	$0 < \epsilon_0 \ll 1$ is sufficiently small (depending on $n, M, \ol{g}, f, \lambda_*$),
	$0 < \mu_u, \mu_s \ll 1$ are sufficiently small (depending on $n, M, \ol{g}, f, \lambda_*$),
	$0 < \gamma_0 \ll 1$ is sufficiently small (depending on $n, M, \ol{g}, f, \lambda_*, \epsilon_0, \epsilon_1, \epsilon_2$),
	and 
	$\tau_0 \gg 1$ is sufficiently large (depending on $n, M, \ol{g}, f, \lambda_*, \ol{p}, \epsilon_0, \epsilon_1, \epsilon_2, \gamma_0, \mu_u, \mu_s $),
	then
		$$t_{\mu_u, 1}^* ( \mathbf p ) = t_{\mu_u, \mu_s}^* ( \mathbf p) < t_1^* ( \mathbf p )$$
		for all $|\mathbf p | \le \ol{p} e^{\lambda_* \tau_0}$.	
\end{lem}
\begin{proof}
	We first show $t_{\mu_u, \mu_s}^* ( \mathbf p) < t_1^* ( \mathbf p )$.
	Observe that the assumptions of Theorem \ref{Thm C^2 Bounds Preserved} hold 
	with $\tau_1 = \tau ( t_{\mu_u, \mu_s}^* ) = - \ln ( 1 - t_{\mu_u, \mu_s}^* )$,
	$\epsilon = \epsilon_0$,
	and $\mu = \sqrt{ \mu_u^2 +  \mu_s^2 }.$
	Let $W = W(n, M, \ol{g}, f, \lambda_*)$ and $W' = W' ( n, M, \ol{g}, f) \ge 1$ be as in Theorem \ref{Thm C^2 Bounds Preserved} and set $\delta = \frac{ \min \{ \epsilon_0, \epsilon_1, \epsilon_2 \} }{2 W'} \in ( 0, \epsilon_0]$.
	Then Theorem \ref{Thm C^2 Bounds Preserved} implies
	\begin{gather}
		\label{Exits Unstable C^2 Est 1}
		| h_{\mathbf p } | + | \ol{\nabla} h_{\mathbf p } | + | \ol{\nabla}^2 h_{\mathbf p } | \le W f^{| \lambda_* |} e^{\lambda_* \tau} 	\text{ and } \\
		\label{Exits Unstable C^2 Est 2}
		| h_{\mathbf p } | + | \ol{\nabla} h_{\mathbf p } | + | \ol{\nabla}^2 h_{\mathbf p } | \le W' \delta = \frac{1}{2} \min \{ \epsilon_0, \epsilon_1, \epsilon_2 \}	
	\end{gather}
	throughout $M \times [ \tau_0, \tau(t_{\mu_u, \mu_s}^*) ]$
	so long as 
	$0 < \ol{p} \ll 1$ is sufficiently small depending on $n , M, \ol{g}, f, \lambda_*$;
	$0 < \epsilon = \epsilon_0 \ll 1$ is sufficiently small depending on $n, M, \ol{g}, f, \lambda_*$;
	$0 < \mu_u, \mu_s \ll1 $ are sufficiently small depending on $n, M, \ol{g}, f, \lambda_*$;
	$0 < \gamma_0 \ll 1$ is sufficiently small depending on $n, M, \ol{g}, f, \lambda_*, \epsilon_0, \epsilon_1, \epsilon_2$;
	and $\tau_0 \gg 1 $ is sufficiently large depending on $n, M, \ol{g}, f, \lambda_*, \ol{p} , \epsilon_0, \epsilon_1, \epsilon_2, \gamma_0$.
	
	In particular, \eqref{Exits Unstable C^2 Est 2} implies
		$$h_{\mathbf p } \in \cl{B} \left[ \lambda_* , \mu_u, \mu_s, \frac{1}{2} \epsilon_0, \frac{1}{2} \epsilon_1, \frac{1}{2} \epsilon_2 , [ t_0, t_{\mu_u, \mu_s}^* ] \right].$$
	By Lemma \ref{Lem Must Exit Through clB}, for all $t \in [t_0, t_1^+ )$,
		$$\supp h_{\mathbf p} (t) \subset \{ f \le \Gamma_0 ( 1 - t)^{-1} \} 
		\subset \{ f \le \Gamma_0 ( 1 - t_1^+)^{-1} \} \Subset M$$
	since $t_1^+ < 1$.
	Since $h$ is smooth,
	it follows that, for all $t_2$ in a neighborhood of $t_{\mu_u, \mu_s}^*$,
	\begin{gather} \label{h in clB Est}
	\begin{aligned} 
		h_{\mathbf p } &\in \cl{B} \left[ \lambda_*, \frac{3}{2} \mu_u, \frac{3}{2} \mu_s, 
		\frac{3}{4} \epsilon_0	, \frac{3}{4} \epsilon_1, \frac{3}{4} \epsilon_2, [ t_0, t_2) \right] 	\\
		&\subset \cl{B} [ \lambda_*, 2, 2, 
		 \epsilon_0	,  \epsilon_1,  \epsilon_2, [ t_0, t_2) ]  		
	\end{aligned} \end{gather}
	if $0 < \mu_u, \mu_s \le \frac{4}{3}$.
	Thus, $t_1^*( \mathbf p ) \ge t_2 > t_{\mu_u, \mu_s}^* ( \mathbf p )$.
	
	Next, we prove that $t_{\mu_u, 1}^* = t_{\mu_u, \mu_s}^*$.
	First, note that Lemma \ref{Lem difference from RF}, Lemma \ref{Lem Must Exit Through clB}, and equations \eqref{Exits Unstable C^2 Est 1} and \eqref{Exits Unstable C^2 Est 2} 
	imply the assumptions of Lemma \ref{Lem L^2_f Ests for Errors} hold.
	Thus, 
		$$\| \cl{E}_1 \|_{L^2_f} + \| \cl{E}_2 \|_{L^2_f} \lesssim_{n, M, \ol{g}, f, \lambda_*} e^{2 \lambda_* \tau}$$
	for all $\tau \in [ \tau_0, \tau( t_{\mu_u, \mu_s}^* ) ]$ 
	when $0 < \epsilon_0 \ll 1$ is sufficiently small depending on $n$ and $\tau_0 \gg 1$ is sufficiently large depending on $n, M, \ol{g}, f, \lambda_*, \gamma_0$.
	Additionally, Lemma \ref{Lem L^2_f Ests at t_0} implies that,
	if $\tau_0 \gg 1$ is sufficiently large depending on $n, M, \ol{g}, f, \lambda_*, \gamma_0, \mu_s$, then
	\begin{align*}
		\| \pi_s h_{\mathbf p }(\tau_0) \|_{L^2_f} &\le \frac{2}{3} \mu_s e^{\lambda_* \tau_0}.
	\end{align*}
	Lemma \ref{Lem Immediately Exits} on $[ \tau_0, \tau( t_{\mu_u, \mu_s}^* ) ]$ 
	now implies 
	\begin{align*}
		\| \pi_s h_{\mathbf p }(\tau) \|_{L^2_f} &\le \frac{2}{3} \mu_s e^{\lambda_* \tau}
	\end{align*}		
	for all $\tau \in [\tau_0, \tau( t_{\mu_u, \mu_s}^* ) ]$.
	By smoothness of $h$, it now follows that \eqref{h in clB Est} can be improved to
	\begin{gather} \label{h in clB Est+} \begin{aligned}
		h_{\mathbf p } &\in \cl{B} \left[ \lambda_*, 2 , \frac{3}{4} \mu_s, \frac{3}{4} \epsilon_0, \frac{3}{4} \epsilon_1, \frac{3}{4} \epsilon_2, t_0 ,t_2 \right]	\\
		&\subset \cl{B} \left[ \lambda_*, 2 ,  \mu_s,  \epsilon_0, \epsilon_1,  \epsilon_2, t_0 ,t_2 \right]
	\end{aligned}	\end{gather}
	for all $t_2$ in a neighborhood of $t_{\mu_u, \mu_s}^*$.
	Therefore, $t_{\mu_u, \mu_s}^* = t_{\mu_u, 1}^*$.
\end{proof}

\begin{lem} \label{Lem Image Off 0}
	Under the same assumptions as Lemma \ref{Lem Exits Unstable Side},
	$\cl{F} : \ol{B_{\ol{p} e^{\lambda_* \tau_0} } } \to \R^K$ is continuous and
		$$\cl{F}(\mathbf p ) \ne \mathbf 0 \qquad \text{for all } | \mathbf p | \le \ol{p} e^{\lambda_* \tau_0}.$$
\end{lem}
\begin{proof}
	By the definition of $t_{\mu_u, \mu_s}^* =t_{\mu_u, 1}^*$ and \eqref{h in clB Est} in the proof of Lemma \ref{Lem Exits Unstable Side},
	\begin{gather*}	\begin{aligned}
		h_{\mathbf p } &\in \cl{B} [\lambda_*, \mu_u,  1, \epsilon_0, \epsilon_1, \epsilon_2 , [ t_0, t_{\mu_u, \mu_s}^* ] ],	\\
		h_{\mathbf p } &\notin \cl{B} [\lambda_*, \mu_u, 1, \epsilon_0, \epsilon_1, \epsilon_2 , [ t_0, t_2 ] ],
		 \text{ and } \\
		h_{\mathbf p} &\in \cl{B} [\lambda_*, 2,2, \epsilon_0, \epsilon_1, \epsilon_2 , [ t_0, t_2 ] ]
	\end{aligned} \end{gather*}
	for all $t_2 $ in a neighborhood of $t_{\mu_u, \mu_s}^*$ with $t_2 > t_{\mu_u, \mu_s}^*$.
	By Definition \ref{Defn Box} of $\cl{B}$ and the smoothness of $h_{\mathbf p}$, it follows that
	\begin{equation} \label{F(p) = mu_u}
		|\cl{F}( \mathbf p ) |= \| \pi_u h_{\mathbf p }( t_{\mu_u, \mu_s}^* ) \|_{L^2_f} = \mu_u e^{\lambda_* \tau(t_{\mu_u, \mu_s}^*)}  > 0.
	\end{equation}		

	For the continuity of $\cl{F}$, 	
	observe that \eqref{F(p) = mu_u} and Lemma \ref{Lem Immediately Exits} implies 
	that for arbitrary $\tau', \tau''$ in a neighborhood of $\tau(t_{\mu_u, \mu_s}^*)$ 
	with $\tau' < \tau(t_{\mu_u, \mu_s}^*) < \tau''$ 
		$$\| \pi_u h_{\mathbf p }(\tau' ) \|_{L^2_f} < \mu_u e^{\lambda_* \tau'}
		\qquad \text{ and }		\qquad 
		\| \pi_u h_{\mathbf p }(\tau'' ) \|_{L^2_f} > \mu_u e^{\lambda_* \tau''}.$$
	\eqref{h in clB Est+} therefore implies that 	
	$h_{\mathbf p}( \tau')$ is in the interior of
	$\cl{B} [ \lambda_* , \mu_u, \mu_s, \epsilon_0, \epsilon_1, \epsilon_2 , \{ \tau' \} ]$
	and $h_{\mathbf p}( \tau'')$ is in the interior of the complement of
	$\cl{B} [ \lambda_* , \mu_u, \mu_s, \epsilon_0, \epsilon_1, \epsilon_2 , \{ \tau'' \} ]$.
	The same holds for all $\mathbf p'$ in a neighborhood of $\mathbf p$ by continuity and thus
	$t( \tau') \le t_{\mu_u, \mu_s}^*( \mathbf p ') \le t( \tau'')$.
	Continuity of $\cl{F}$ then follows from the definition of $\cl{F}$.
\end{proof}

We may now prove main theorem of the paper.

\begin{proof}[Proof of Theorem \ref{Main Thm}]
	By Lemmas \ref{Lem Homotopic to Id on Bndry}--\ref{Lem Image Off 0} with parameters chosen sufficiently large or small as in the assumptions of these lemmas,
	$\cl{F} : \ol{B_{\ol{p} e^{\lambda_* \tau_0} } } \to \R^K \setminus \{ 0 \}$
	is a continuous function whose restriction to $\ol{A} = \{ \mathbf p \in \R^K : 2 \mu_u e^{\lambda_* \tau_0} \le | \mathbf p | \le \ol{p} e^{\lambda_* \tau_0} \}$
	is homotopic to the identity $Id : \ol{A} \to \ol{A} \subset \R^K \setminus \{ 0 \}$.
	However, such a map cannot exist as $Id : \ol{A} \to \ol{A} \subset \R^K \setminus \{  0 \} $ is not null-homotopic.
	Therefore, the assumption (in assumption \ref{Assume for Contradiction}) that $t_1^*( \mathbf p) < 1$ for all $| \mathbf p | \le \ol{p} e^{\lambda_* \tau_0}$ must be false.
	Hence, there exists $\mathbf p^*$ with $| \mathbf p^* | \le \ol{p} e^{\lambda_* \tau_0}$ such that
		$$1 = t_1^*( \mathbf p^*) 
		=  \sup \{ t_1 \in (t_0, 1] : \mathbf p \in \cl{P} [ \lambda_* , \ol{p}, \Gamma_0, \gamma_0, \mu_u = 2, \mu_s =2, \epsilon_0, \epsilon_1, \epsilon_2, t_0, t_1 ] \} .$$

	Therefore, there exists a sequence $t^{(k)} \nearrow 1$ and smooth families of diffeomorphisms $\tl{\Phi}^{(k)} : M \times [ t_0, t^{(k)} ) \to M$ solving the harmonic map heat flow \eqref{tlPhi Evol Eqn} such that
		$$h_{\mathbf p^*} ( t) = \frac{1}{1 - t} ((\Phi_t^{(k)})^{-1})^* \acute G_{\mathbf p^* }(t) - \ol{g} 
		\in \cl{B} [ \lambda_*, \mu_u = 2, \mu_s = 2, \epsilon_0, \epsilon_1, \epsilon_2, t_0, t^{(k)} ].$$
	With $0 < \epsilon_0, \epsilon_1 \ll 1$ sufficiently small depending on $n$,
	Proposition \ref{Prop Uniqueness HarMapFlow} implies that, for any $k, k'$,  $\tl{\Phi}^{(k)} = \tl{\Phi}^{(k')}$ on the intersection of their domains.
	Hence, the $\tl{\Phi}^{(k)}$ paste to define a smooth solution $\tl{\Phi} : M \times [ t_0, 1) \to M$ of the harmonic map heat flow \eqref{tlPhi Evol Eqn} which is a diffeomorphism on each time slice and satisfies
		$$h_{\mathbf p^*} ( t) = \frac{1}{1 - t} (\Phi_t^{-1})^* \acute G_{\mathbf p^* }(t) - \ol{g} 
		\in \cl{B} [ \lambda_*, \mu_u = 2, \mu_s = 2, \epsilon_0, \epsilon_1, \epsilon_2, t_0, 1 ].$$
	By Definition \ref{Defn Box} of $\cl{B}$, $h_{\mathbf p^*}(t)$ remains in a small $C^2(M)$-neighborhood of $0$ for all $t$ and converges in $L^2_f(M)$ to $0$ as $t \nearrow 1$.
	By similar logic as in the beginning of the proof of Lemma \ref{Lem Exits Unstable Side}, Theorem \ref{Thm C^2 Bounds Preserved} applies to show 
	\begin{equation*}
		| h_{\mathbf p^*}(\tau) | + | \ol{\nabla} h_{\mathbf p^*}(\tau) | + | \ol{\nabla}^2 h_{\mathbf p^*}(\tau) | \lesssim_{n, M, \ol{g}, f, \lambda_*} f^{| \lambda_* |} e^{\lambda_* \tau}
		\qquad 	\text{for all $\tau \in [\tau_0 , \infty)$.}
	\end{equation*}
	Therefore, $h_{\mathbf p^*}(t)$ converges in $C^2_{loc}(M)$ to $0$ as $t \nearrow 1$.
	The local interior estimates in Lemma \ref{Local Int Est C^0 to C^m_loc} then bootstrap this to convergence in $C^\infty_{loc}(M)$.

	Recall from Remark \ref{Rem acute G Observations} and the choice of $\eta_{\Gamma_0}$ that
		$$ \iota^* \acute G_{\mathbf p^*} (t)= G_{\mathbf p^*}(t)  \quad \text{ on }
		 \left\{ f < \frac{4}{6} \Gamma_0 \right\} \subset \iota^{-1}\left( \bigsqcup_{\omega \in A} \{ x \in M : \eta_{\Gamma_0} (x) = 1 \} \right) \subset \cl{M}.$$
	We claim that, for any $t \in [t_0, 1)$,
	\begin{equation} \label{Exhaustion}	
		( \phi_t \circ \tl{\Phi}_t \circ \iota ) \left( \left\{  f < \frac{4}{6} \Gamma_0 \right\} \right) 
		\supset \{ f < ( 1 - t)^{-1} \}.
	\end{equation}
	Indeed, let $t \in [t_0, 1)$ and let $x \in M$ with $f(x) \ge \frac{4}{6} \Gamma_0$.	
	By similar logic as in the beginning of the proof of Lemma \ref{Lem Drift of the Grafting Region}, 
	the local drift estimate of Lemma \ref{Lem Local Drift Control} can be applied to deduce 
		$$f( \tl{\Phi}_{t} (x) ) \ge \frac{1}{2} \Gamma_0 $$
	so long as $\Gamma_0 \gg 1$ is sufficiently large depending on $n, M, \ol{g}, f$;
	$0 < \epsilon_0, \epsilon_1 \ll 1$ are sufficiently small depending on $n$;
	and $0 < 1 - t_0 \ll 1$ is sufficiently small depending on $n, M, \ol{g}, f, \Gamma_0$.
	Therefore, for all $t \in [t_0, 1)$,
		$$\tl{\Phi}_t \left( \left\{ x \in M : f(x) \ge \frac{4}{6} \Gamma_0 \right\} \right) \subset \left\{ x \in M : f(x) \ge \frac{1}{2} \Gamma_0 \right\}.$$
	The proof of Lemma \ref{Lem Drift of the Grafting Region}  shows
		$$\phi_t ( \{ f \ge \Gamma_0  / 2 \} ) \subset \{ f \ge ( 1 - t)^{-1} \}
		\qquad \text{ for all } t \in [0, 1)$$	
	when $\Gamma_0 \gg 1$ is sufficiently large depending on $n, M, \ol{g}, f$.
	Hence, 
		$$(\phi_t \circ \tl{\Phi}_t) \left( \left\{ f \ge \frac{4}{6} \Gamma_0 \right\} \right)
		\subset \phi_t  \left( \left\{ f \ge \frac{1}{2} \Gamma_0 \right\} \right)
		\subset \{ f \ge ( 1 - t)^{-1} \}
		\quad \text{ for all } t \in [t_0, 1).$$
	Taking complements and using the fact that $\phi_t, \tl{\Phi}_t$ are bijections for all $t \in [t_0, 1)$
	then implies
		$$(\phi_t \circ \tl{\Phi}_t \circ \iota) \left( \left\{ f < \frac{4}{6} \Gamma_0 \right\} \right)
		= (\phi_t \circ \tl{\Phi}_t) \left( \left\{ f < \frac{4}{6} \Gamma_0 \right\} \right)
		\supset \{ f < ( 1 - t)^{-1} \}$$
	for all $t \in[ t_0, 1)$.
	This proves the claim \eqref{Exhaustion}.
	
	\eqref{Exhaustion}, the uniform $C^2(M)$ bounds on $h_{\mathbf p^*}(t)$, and the fact that $h_{\mathbf p^*}(t)$ converges to $0$ in $C^\infty_{loc}(M)$ as $t \nearrow 1$ 
	together show that
	the curvature of $G_{\mathbf p^*}(t)$ blows up at the type I rate on $\left\{ f < \frac{4}{6} \Gamma_0 \right\} \subset \cl{M}$, that is,
		$$0 < \limsup_{t \nearrow 1} \sup_{x \in \{ f < \frac{4}{6} \Gamma_0 \} \subset \cl{M} }
		(1 - t) | Rm_{G_{\mathbf p^*}} |_{G_{\mathbf p^*}} (x,t) < \infty,$$
	and also complete the proof of the second statement of Theorem \ref{Main Thm}
	with $U = \left\{ f < \frac{4}{6} \Gamma_0 \right\}$.
	Note that Proposition \ref{Prop Pseudolocality App} additionally applies to give curvature bounds on $\cl{M} \setminus \left\{ f < \frac{4}{6} \Gamma_0 \right\}$ and confirm the singularity is local.
	
	To conclude, we check that Cheeger-Gromov convergence follows.
	Let $t_j$ be a sequence of times with $t_j \nearrow 1$ and let $y_\infty \in M$.
	By \eqref{Exhaustion}, 
	the sets $(\phi_{t_j} \circ \tl{\Phi}_{t_j} \circ \iota) \left( \left\{ f < \frac{4}{6} \Gamma_0 \right\} \right) \subset M$
	give an exhaustion of $M$
	and there exists a sequence $x_j \in \cl{M}$ such that $x_j \in \{f < \frac{4}{6} \Gamma_0 \}$ and 
	$(\phi_{t_j} \circ \tl{\Phi}_{t_j} \circ \iota)(x_j) = y_\infty$ for all $j$ large enough.
	Consider the sequence of parabolically rescaled Ricci flows
		$$G_{\mathbf p^*}^{(j)}(t) = \frac{1}{1 - t_j} G_{\mathbf p^*} ( t_j + t ( 1 - t_j) )
		\qquad \text{ on } \cl{M} \times \left[ 1 - \frac{1- t_0}{1 - t_j} , 1 \right).$$	
	The fact that $h_{\mathbf p^*}(t)$ converges to $0$ in $C^\infty_{loc}(M)$ as $t \nearrow 1$ implies
		$$( \cl{M}, G_{\mathbf p^*}^{(j)}(0), x_j ) \xrightarrow[j \nearrow \infty]{} ( M, \ol{g}, y_\infty )$$
	in the pointed Cheeger-Gromov sense
	with the maps $(\phi_{t_j} \circ \tl{\Phi}_{t_j} \circ \iota)^{-1}$ defined on $(\phi_{t_j} \circ \tl{\Phi}_{t_j} \circ \iota) \left( \left\{ f < \frac{4}{6} \Gamma_0 \right\} \right) \subset M$.
	The type I curvature blow up rate of $(\cl{M}, G_{\mathbf p^*}(t))$ implies that, for any $-\infty < a < b < 1$, the sequence of parabolically rescaled Ricci flows $\left(\cl{M}, G^{(j)}_{\mathbf p^*} (t) \right)$ has curvature bounds on $\cl{M} \times [a,b]$ that are uniform in $j$.
	By \cite[Chapter 3, Section 2]{ChowEtAl07}, this then gives sufficient compactness to deduce the \emph{flows} $\left(\cl{M}, G^{(j)}_{\mathbf p^*} (t), x_j\right)$ subsequentially converge to the flow $(M , (1-t) \phi_t^* \ol{g}, y_\infty)$ in the pointed Cheeger-Gromov sense.
\end{proof}

\appendix

\section{Rounding Out Cones} \label{App Rounding Out Cones}

Recall from Definition \ref{Defn Cone} that
	$$ \cone_R ( \Sigma) = ( R, \infty) \times \Sigma \quad \text{and} \quad
	g_{\cone} = dr^2 + r^2 g_{\Sigma}$$
where $(\Sigma, g_\Sigma)$ is a closed Riemannian manifold.
The goal of this appendix is to prove Proposition \ref{prop Main Result of Appdix A} which explicitly constructs manifolds $\cl{M}$ as in Subsection \ref{Subsect The Manifold} and metrics $G_{\mathbf 0}(t_0)$ on $\cl{M}$ satisfying Definition \ref{Defn G_0(t_0)} for any given (smooth, complete) asymptotically conical shrinker $(M, \ol{g}, f)$.
Informally, $\cl{M}$ is the double of a suitably large subset of $M$ and has a $\cl{S} \cong \mathbb{Z}_2$ symmetry given by interchanging the two pieces of the double.
The metrics $G_{\mathbf 0 }(t_0)$ are $\mathbb{Z}_2$-invariant metrics that, on each piece of the double, spatially interpolate from the soliton metric $(1 - t_0) \phi_{t_0}^* \ol{g}$ to the cone metric $g_{\cone}$ to a cylindrical metric of large radius.
The spatial interpolations can be done without increasing the curvature too much or collapsing the volume of balls along the interpolation.

Throughout this appendix, $(M, \ol{g}, f)$ continues to denote a smooth, complete, asymptotically conical, gradient shrinking Ricci soliton.
The first proposition here comes from \cite[Proposition 2.1]{KotschwarWang15} and its proof adapted to the case of multiple ends and with the diffeomorphisms $\Psi_t$ specified further.

\begin{prop} [Proposition 2.1 of \cite{KotschwarWang15}] \label{Prop 2.1 in KW15}
	Let $(M, \ol{g}, f)$ be a shrinker which is asymptotic to the (regular) cone 
	$\left( \cone( \Sigma), g_{\cone} \right)$ along an end $V$ (see Definition \ref{Defn Asymply Conical}).
	
	Then there exists $R_0 > 0$
	and a smooth family of maps $\Psi_t =  \phi_t \circ \Psi : \cone_{R_0} ( \Sigma)  \to V$
	such that
	\begin{enumerate}
		\item For all $t \in [0, 1)$, $\Psi_t =  \phi_t \circ \Psi $ is a diffeomorphism onto its image and
		$\Psi_t ( \cone_{R_0} ( \Sigma) )$ is an end of $M$,
		
		\item $(1 - t) \Psi_t^* \ol{g} = ( 1- t) \Psi^*  \phi_t^* \ol{g}$ extends smoothly as $t \nearrow 1$ to $g_{\cone}$ on $\cl{C}_{R_0}( \Sigma)$, and
		
		\item the function $\Psi^* f : \cone_{R_0} \to \R$ satisfies
			$$r^2 - \frac{N_0}{r^2} \le 4\Psi^* f \le r^2 + \frac{N_0}{r^2}.$$
	\end{enumerate}
	
	In particular, 
		$$ ( 1 - t) \Psi^* \phi_t^* \ol{g} \xrightarrow[t \nearrow 1]{ C^\infty_{loc} ( \cone_{R_0}( \Sigma), g_{\cone} ) }		g_{\cone}.$$
\end{prop}

\begin{remark}
	By taking $R_0$ slightly larger in Proposition \ref{Prop 2.1 in KW15}, we may assume without loss of generality that the conclusions apply on the set $\ol{\cone_R( \Sigma) } = [R_0, \infty) \times \Sigma$.
	Proposition \ref{Prop 2.1 in KW15} also extends naturally to the case of finitely many asymptotically conical ends.
\end{remark}

Let $( \Sigma^{n-1}, g_\Sigma)$ be a closed Riemannian manifold.
Consider the warped product 
	$$I \times_\psi \Sigma		\quad \text{ with metric } dr^2 + \psi(r)^2 g_\Sigma$$
where $I \subset \R$ is an interval with coordinate $r$.

Observe that $\psi \equiv r$ gives a cone metric and a constant warping function $\psi \equiv C$ gives a cylinder metric.
The next result says how to interpolate between these two metrics without increasing the curvature too much.

\begin{lem} \label{Lem Warped Product Cone to Cyl}
	Let $(\Sigma^{n-1}, g_\Sigma)$ be a closed Riemannian manifold
	and let $R > 0$.
	There exists a smooth function $\psi_R  : (0, \infty) \to (0, \infty)$, $\psi_R(r)$,
	such that
	$\psi_R$ is non-decreasing,
	$$\psi_R(r) \left\{ \begin{array}{ll}
		= r, 		& \text{ if } 0 < r \le R,	\\
		\in [ R, 2R], & \text{ if } R < r < 3R,	\\
		= 2R, 	& \text{ if } 3R \le r,	\\
	\end{array} \right.$$
	and the curvature of $(0, \infty) \times_{\psi_R} \Sigma$ with the associated warped product metric $g = g_R = dr^2 + \psi_R(r)^2 g_\Sigma$
	satisfies the estimate
		$$| {}^g \nabla^m Rm[g] |_g \lesssim_{n, \Sigma, g_\Sigma, m} \frac{1}{R^{2+m}} \qquad \text{on } \cone_{R/2} (\Sigma).$$
	Additionally, there exists $c = c(n, \Sigma, g_\Sigma) > 0$ such that
		$$Vol_{g_R} ( B_{g_R} (x, 1)) \ge c > 0 \qquad \forall x \in \cone_{R/2} (\Sigma), \, \forall R \ge 1.$$
\end{lem}
\begin{proof}
	There exists a smooth, non-decreasing function $\psi_1 : (0, \infty) \to (0 , \infty)$
	such that
	$$\psi_1(r) \left\{ \begin{array}{ll}
		= r, 		& \text{ if } 0 < r \le 1,	\\
		\in [ 1, 2], & \text{ if } 1 < r < 3, \text{ and}	\\
		= 2, 	& \text{ if } 3 \le r.	\\
	\end{array} \right.$$	
	Since $\psi_1$ is smooth and $\Sigma$ is closed, for each $m$ there exists a constant $C_m$ so that
		$$\left| \frac{d^m \psi_1}{dr^m}  \right| \le C_m		\qquad \text{ and } \qquad
		\left| {}^\Sigma \nabla^m Rm_\Sigma \right|_\Sigma \le C_m .$$
	It follows that the warped product metric
		$$g_1 \doteqdot dr^2 + \psi_1(r)^2 g_\Sigma$$
	satisfies curvature estimates of the form
		$$| {}^{g_1} \nabla^m Rm[ g_1 ] |_{g_1} \le C(n, \Sigma, g_\Sigma, m) \qquad \text{on } \cone_{1/2} (\Sigma).$$
		
	By scaling, the warped product metric
		$$g_R \doteqdot R^2 g_1 = R^2 dr^2 + [ R \psi_1 (r) ]^2 g_\Sigma
		= d \tilde{r}^2 + \left[ R \psi_1 \left( \tilde{r} / R \right) \right]^2 g_\Sigma 		\qquad ( \tilde{r} \doteqdot R r )$$
	has curvature bounds and warping function $\psi_R ( \tilde{r} ) \doteqdot R \psi_1 ( \tilde{r} / R )$ as claimed in the statement of the lemma.

	It remains to show $g_R$ satisfies the volume lower bound claimed in the statement of the lemma.
	First, observe that there exists $c = c(n, \Sigma, g_\Sigma) >0$ such that
		$$Vol_{g_1} B_{g_1}(x, 1/4) \ge c > 0 \qquad \forall x \in \cone_{1/2}(\Sigma),$$
	and $B_{g_1}(x, 1/4) \subset \cone_{1/4}(\Sigma)$ for all $x \in \cone_{1/2}(\Sigma)$.
	By curvature bounds for $g_1$ on $\cone_{1/4}(\Sigma)$ and volume monotonicity, it follows that 	
		$$Vol_{g_1} B_{g_1}(x, r) \ge c' r^n \qquad \forall x \in \cone_{1/2}(\Sigma), \, \forall 0 < r \le \frac14$$
	for some constant $c' = c'(n, \Sigma, g_\Sigma) > 0$.
	Therefore, by scaling $g_R = R^2 g_1$, 
	\begin{multline*}
		Vol_{g_R} B_{g_R} (x, 1)
		\ge Vol_{g_R} B_{g_R} (x, 1/4) 
		= R^n Vol_{g_1} B_{g_1} \left( x, \frac 1{4R} \right) \\
		\ge R^n c' \left( \frac1 {4R} \right)^n = \frac{c'}{4^n} > 0 
	\end{multline*}
	for all $x \in \cone_{R/2}(\Sigma)$ and all $R \ge 1$.
	This completes the proof.
\end{proof}

The next lemma looks at an interpolation between the metrics from Proposition \ref{Prop 2.1 in KW15} and Lemma \ref{Lem Warped Product Cone to Cyl}. It provides estimates for this new metric.

\begin{lem} \label{Lem Interpolated Metric on Cone Appdix}
Let $(M^n, \ol{g}, f)$ be a shrinker which is asymptotic to the cone $(\cone(\Sigma), g_{\cone}) $.
Consider $R_0>0$ and $\Psi_t = \Psi \circ \phi_t$ as in Proposition \ref{Prop 2.1 in KW15}.
Let $R_1, R_2, R_3 \in \R$ be such that
	$$0 < R_0 < R_1 < R_2  < R_2 + 2 < R_3.$$
Let $\eta(r) : (0, \infty) \to [0,1]$ be a smooth bump function that decreases from 1 to 0 over the interval $(R_1, R_2)$.
For all $0 \le t_0 < 1$, consider the metric $\check{G}(t_0)$ on $\cone_{R_0}( \Sigma)$ given by 
	$$\check G(t_0)  = \eta ( 1 - t_0) \Psi^* \phi_{t_0}^* \ol{g} + ( 1 - \eta) g_{R_3}$$
where $g_{R_3}$ is the warped product metric from Lemma \ref{Lem Warped Product Cone to Cyl}.

	Then
	\begin{enumerate}
	\item \label{lem interp metric on cone appdix cgce to cone}
	$$\check G(t_0) \xrightarrow[t_0 \nearrow 1]{} g_{\cone}
	\quad \text{in } C^\infty \left( \cone_{R_0}(\Sigma) \setminus \ol{ \cone_{R_3}( \Sigma ) } , g_\cone \right),$$

	\item \label{lem interp metric on cone appdix curv ests}
	for all $m \in \mathbb{N}$,
	if $0 < 1 - t_0 \ll 1$ is sufficiently small, 
	then
	$\check G(t_0)$ satisfies the curvature estimate
		$$| \check \nabla^m \check {Rm} |_{\check G(t_0) } \lesssim_{n, \Sigma, g_{\Sigma}, m} \max \left\{ \frac{1}{r^{2 + m} }, \frac{1}{R_3^{2 + m} } \right\}
		\qquad \text{on } \cone_{R_0}( \Sigma) ,$$
	and
	
	\item \label{lem interp metric on cone appdix non-collapsed}
	there exists $c = c(n, \Sigma, g_\Sigma) > 0$ such that 
		$$Vol_{\check G(t_0)} B_{\check G(t_0)}(x, 1) \ge c > 0 \qquad \text{for all } x \in \cone_{R_0 +1}(\Sigma)$$
	for all $0 < 1-t_0 \ll 1$ sufficiently small.
	\end{enumerate}
\end{lem}
\begin{proof}
	To simplify notation, for any $R < R'$, let $\cone_{R}^{R'}( \Sigma)$ denote the annular region
		$$\cone_{R}^{R'}( \Sigma) \doteqdot \cone_R ( \Sigma) \setminus \ol{ \cone_{R'} ( \Sigma ) } = (R, R') \times \Sigma$$
	and $\ol{ \cone_{R}^{R'} ( \Sigma) }$ its closure in $\cone(\Sigma)$, 
		$$\ol{ \cone_{R}^{R'} ( \Sigma) } = [R, R'] \times \Sigma.$$

	First, note that $\check G(t_0) = g_{\cone}$ on $\ol{ \cone_{R_2}^{R_3} (\Sigma) }$ by construction.
	Thus, it suffices to prove convergence to $g_{\cone}$ on the region $\ol{ \cone_{R_0}^{R_2}( \Sigma) }$.
	Here, we can write
	$$\check G(t_0) =  \eta ( 1 - t_0) \Psi^* \phi_{t_0}^* \ol{g} + ( 1 - \eta) g_{R_3}
	= \eta ( g_{\cone}  + h(t_0)  ) + (1 - \eta) g_{\cone}
	= g_{\cone} + \eta h( t_0) $$
	where
		$$h (t_0) = ( 1 - t_0) \Psi^* \phi_{t_0}^* \ol{g} - g_{\cone}.$$
	By Proposition \ref{Prop 2.1 in KW15},
	$(1-t_0) \Psi^* \phi_{t_0}^* \ol{g}$ converges to $g_{\cone}$ in $C^\infty_{loc}( \ol{\cone_{R_0}( \Sigma) },  g_{\cone} ) $ as $t_0 \nearrow 1$. 
	Hence, $\eta h(t_0)$ converges to $0$ in $C^\infty( \ol{ \cone_{R_0}^{R_2}( \Sigma) } , g_{\cone} ) $ as $t_0 \nearrow 1$.
	This proves the first conclusion \eqref{lem interp metric on cone appdix cgce to cone}.
	
	Next, we show the curvature estimates \eqref{lem interp metric on cone appdix curv ests}.
	Observe that, by construction, $\check G(t_0) \equiv g_{R_3}$ on the set $\ol{\cone_{R_2}( \Sigma) } \subset \{ (r, \sigma)  \in \cone( \Sigma) : \eta(r) = 0 \}$.
	Therefore, the curvature estimates from Lemma \ref{Lem Warped Product Cone to Cyl} imply that
		$$\left| \check \nabla^m \check{Rm} \right| \lesssim_{n, \Sigma, g_\Sigma, m} \frac{1}{R_3^{2 + m}}
		\qquad \text{on } \ol{ \cone_{R_3} (\Sigma) }.$$
	
	It remains to obtain the curvature estimate on $ \cone_{R_0}^{R_3}( \Sigma) $.
	Note that the corresponding curvature estimate holds for $g_{\cone}$, that is,
		$$\left| {}^{g_{\cone}} \nabla^m Rm_{\cone} \right|_{g_{\cone}} (r, \sigma)
		\lesssim_{n, \Sigma, g_\Sigma,  m} \frac{1}{r^{2+m}} 
		\qquad \text{for all  } (r, \sigma) \in \ol{\cone_{R_0}(\Sigma) }.$$
	$\check G(t_0)$ $g_{\cone}$-smoothly converges to $g_{\cone}$ on $\cone_{R_0}^{R_3} ( \Sigma)$.
	It therefore follows that, for any $m \in \mathbb{N}$, if $0 < 1- t_0 \ll 1$ is sufficiently small, then 
		$$| \check \nabla^m \check {Rm} |_{\check G(t_0) }(r, \sigma) 
		\lesssim_{n, \Sigma,  g_\Sigma, m} \frac{1}{r^{2+m}}
		\qquad \text{for all } (r, \sigma) \in \cone_{R_0}^{R_3} ( \Sigma) .$$
	This completes the proof of the curvature estimates \eqref{lem interp metric on cone appdix curv ests}.

	Finally, we show the volume lower bound \eqref{lem interp metric on cone appdix non-collapsed}.
	Since the cone $(\cone_{R_0}(\Sigma), g_\cone )$ satisfies a similar volume lower bound and $\check{G}(t_0)$ converges to $g_\cone$ in $\cone_{R_0}^{R_3} (\Sigma)$, it follows that there exists $c = c(n, \Sigma, g_\Sigma) > 0$ such that
	\begin{equation} \label{proof lem interp metric on cone appdix eqn 1}
		Vol_{\check G(t_0)} B_{\check G(t_0)} (x,1) \ge c >0 \qquad \text{for all } x \in \cone_{R_0+1}(\Sigma)
	\end{equation}
	for all $0 < 1 - t_0 \ll 1$.
	Since $R_3 > R_2 +2 \ge 1$, Lemma \ref{Lem Warped Product Cone to Cyl} implies there exists $c' = c'(n,\Sigma, g_\Sigma)>0$ such that
	\begin{multline} \label{proof lem interp metric on cone appdix eqn 2}
		Vol_{\check G(t_0)} B_{\check G(t_0)}(x,1) = Vol_{g_{R_3}} B_{g_{R_3}} (x,1) \ge c' > 0 \\ \text{for all } x \in \cone_{\max \{ R_3/2, R_2 +1 \} } (\Sigma).
	\end{multline}
	$2 < R_2 +2 < R_3$ implies $\max\{ R_3/2, R_2 +1 \} < R_3 -1$, and thus \eqref{proof lem interp metric on cone appdix eqn 1} and \eqref{proof lem interp metric on cone appdix eqn 2} imply 
	$$Vol_{\check G(t_0)} B_{\check G(t_0)} (x,1) \ge \min \{ c , c'\} > 0 \qquad \text{for all } x \in \cone_{R_0+1} (\Sigma).$$
	This completes the proof of item \eqref{lem interp metric on cone appdix non-collapsed} in the statement of the lemma.
\end{proof}

Let $\check G(t_0) , \Psi, R_0$ be as in Lemma \ref{Lem Interpolated Metric on Cone Appdix} and
consider now $(\Psi^{-1} )^* \check G(t_0)$ on $\Psi( \cone_{R_0} ( \Sigma)  ) \subset M$.
This metric can be extended by $( 1 - t_0) \phi_{t_0}^* \ol{g}$ on the complement of $\Psi ( \cone_{R_0} ( \Sigma)  )$ to yield a Riemannian metric on $M$.
We denote this metric by $G(t_0)$.

We first record the following lemma comparing regions in $M$ and $\cone_{R_0}$.

\begin{lem} \label{Lem f r Region Inclusion}
	If $R \gg 1$ is sufficiently large (depending on $n, M, \ol{g}, f, N_0, R_0$, and $\Psi$ from Proposition \ref{Prop 2.1 in KW15}), then
		$$\left \{ x \in M : f(x) > \frac{R^2}{2} \right \} \subset \Psi ( \cone_R ) \subset \left\{ x \in M : f(x) > \frac{R^2}{8} \right\}.$$
\end{lem}
\begin{proof}
	We first claim that there exists $\Gamma_* \in \R$ such that
		$$\{ x \in M : f(x) > \Gamma_* \} \subset  \Psi ( \cone_{R_0} ).$$
	By Proposition \ref{Prop 2.1 in KW15}, $\Psi ( \cone_{R_0} )$ is a union of the ends of $M$.
	Hence, there exists $\Omega \Subset M$ such that $M \setminus \Omega \subset \Psi( \cone_{R_0} )$.
	Since $\Omega \Subset M$, there exists $\Gamma_* \in \R$ such that 
		$$\Omega \subset \{ x \in M : f(x) \le  \Gamma_* \}.$$
	It follows that 
		$$\{ x \in M : f(x) > \Gamma_* \} \subset M \setminus \Omega \subset \Psi ( \cone_{R_0} ).$$

	Now, let $x \in \cone_R$.
	By Proposition \ref{Prop 2.1 in KW15},
	\begin{gather*}
		f( \Psi(x) ) 
		\ge \frac{1}{4} \left( r(x)^2 - \frac{N_0}{r(x)^2} \right) 
		> \frac{1}{4} \left( R^2 - \frac{N_0}{R^2} \right)
		\ge \frac{R^2}{8} 
	\end{gather*}
	if $R \ge (2N_0^2)^{1/4}$.
	Hence, if $R \ge (2N_0^2)^{1/4}$, then
		$$ \Psi( \cone_R ) \subset \left \{ x \in M : f(x) > \frac{R^2}{8}  \right \}.$$
		
	For the other inclusion, assume $x \in M$ and $f(x) > \frac{R^2}{2}$.
	If $\frac{R^2}{2} > \Gamma_*$, then $x = \Psi(x')$ for some $x' \in \cone_{R_0}$.
	It follows from Proposition \ref{Prop 2.1 in KW15} that
	\begin{gather*}
		r(x')^2 + \frac{N_0}{R_0^2}
		> r(x')^2 + \frac{N_0}{r(x')^2}
		\ge 4 f ( \Psi (x') )
		= 4 f (x)
		> 2 R^2.
	\end{gather*}
	If also $R^2  \ge \frac{N_0}{R_0^2}$, then
		$$r(x') > \sqrt{2R^2 - \frac{N_0}{R_0^2} } \ge R.$$
	It follows that for such $R$,
	$\{ x \in M : f(x) > \frac{R^2}{2} \} \subset \Psi ( \cone_R ).$
\end{proof}

\begin{prop} \label{Prop Interpolated Metric on M Appdix}
	Consider the metric $G(t_0)$ on $M$ which consists of $( \Psi^{-1})^* \check G(t_0)$ extended by $( 1- t_0) \phi_{t_0}^* \ol{g}$.
	If $\frac{1}{4 \sqrt{2} } R_3 = \frac{1}{2} R_2 = R_1 \gg 1$ is sufficiently large (depending on $n, M, \ol{g}, f, N_0, R_0, \Psi$),
	then the following hold:
	\begin{enumerate}
		\item \label{Time Invt Appdix}
		 for all $0 \le t_0, t_0' < 1$, 
			$$G(t_0) = G(t_0')		\qquad \text{on } 
			\left\{ x \in M : f(x) > 2 R_1^2 \right\},$$
			
		\item \label{Soliton Metric Appdix}
			$$G(t_0) = ( 1 - t_0) \phi_{t_0}^* \ol{g} 
			\qquad \text{on } \left \{ x \in M : f(x)  \le \frac{ R_1^2}{8} \right \},$$
			
		\item \label{Cgce to Cone Appdix}
		$G(t_0)$ smoothly converges to $( \Psi^{-1})^* g_{\cone}$ in the region 
			$$\left \{x \in M :  \frac{R_1^2}{16} < f(x) <  4 R_1^2 \right \} \subset \Psi ( \cone_{R_0} )$$
		as $t_0 \nearrow 1$,
			
		\item \label{Curv Est Appdix}
		for any $m \in \mathbb{N}$, if $0 < 1 - t_0 \ll 1$ is sufficiently small then 
			$$ | \nabla^m Rm |_{G(t_0) } \lesssim_{n , \Sigma, g_\Sigma, m} \frac{1}{R_1^{m + 2} } 
			\qquad \text{on } \{ x \in M : f(x) > R_1^2 / 16 \},$$
		and
		\item \label{Non-Collapsing Est Appdix}
		there exists $c = c(n, \Sigma, g_\Sigma) > 0$ such that
			$$Vol_{G(t_0)} B_{G(t_0)} (x, 1) \ge c > 0 \qquad \text{for all } x \in \{ x \in M : f(x) > R_1^2 /16 \} $$
		for all $0 < 1-t_0 \ll 1$ sufficiently small.
	\end{enumerate}
\end{prop}

\begin{proof}
	Item \eqref{Time Invt Appdix} follows from the fact that, for all $0 \le t_0, t_0' < 1$,
		$$\check G( t_0) = \check G( t_0')		\qquad \text{ on } \cone_{R_2}$$
	by construction and the fact that
		$$\Psi ( \cone_{R_2}  ) 
		\supset \left \{ x \in M : f(x) > \frac{R_2^2}{2} \right \} 
		= \{ x \in M : f(x) > 2 R_1^2 \} $$
	by Lemma \ref{Lem f r Region Inclusion} if $R_1 \gg 1$ is sufficiently large.
	
	For item \eqref{Soliton Metric Appdix}, observe that
		$$\check G(t_0) = ( 1 - t_0) \Psi^* \phi_{t_0}^* \ol{g} 
		\qquad \text{on } \cone_{R_0}^{R_1} $$
	by construction.
	It then follows from the definition of $G(t_0)$ and Lemma \ref{Lem f r Region Inclusion} that, if $R_1 \gg 1$ is sufficiently large, then
		$$G(t_0) \equiv ( 1 - t_0) \phi_{t_0}^* \ol{g} 
		\qquad \text{on }  \left\{ x \in M : f(x)  \le \frac{ R_1^2}{8} \right\}.$$
			
	Next, we prove item \eqref{Cgce to Cone Appdix}.
	By Lemma \ref{Lem Interpolated Metric on Cone Appdix}, 
	$\check G(t_0)$ smoothly converges to $g_{\cone}$ on $\cone_{R_0}^{R_3} = \cone_{R_0}^{4 \sqrt{2} R_1}$.
	If $R_1 \gg 1$ is sufficiently large, then $\cone_{R_0}^{4 \sqrt{2} R_1} \supset \cone_{R_1 / \sqrt{8} }^{4 \sqrt{2} R_1 }$.
	Lemma \ref{Lem f r Region Inclusion} then implies that 
	$ G(t_0)$ smoothly converges to $( \Psi^{-1})^* g_{\cone}$ on 
		$$\Psi ( \cone_{R_1 / \sqrt{8} }^{4 \sqrt{2} R_1} )  \supset \left\{ x \in M : \frac{R_1^2}{16} < f(x) < 4 R_1^2 \right\}$$
	as $t_0 \nearrow 1$.
	
	Next, we prove item \eqref{Curv Est Appdix}.
	By Lemma \ref{Lem Interpolated Metric on Cone Appdix}, 
	for any $m \in \mathbb{N}$, if $0 < 1 - t_0 \ll 1$ is sufficiently small then
	$\check G(t_0)$ satisfies the curvature estimate 
		$$| \check \nabla^m \check {Rm} |_{\check G(t_0) } \lesssim_{n, \Sigma, g_\Sigma, m}  \frac{1}{R_1^{2 + m} } 	
		\qquad \text{on } \cone_{R_1 / \sqrt{8} } = \cone_{R_3/ 16 }.$$
	The corresponding estimate
		$$| \nabla^m Rm |_{G(t_0) } \lesssim_{n, \Sigma, g_\Sigma,  m} \frac{1}{R_1^{2+m}}$$
	therefore holds on 
		$$\Psi ( \cone_{R_1 / \sqrt{8} } ) \supset \left \{ x \in M : f(x) > \frac{R_1^2}{16} \right \}$$
	by Lemma \ref{Lem f r Region Inclusion} if $R_1 \gg 1$ is sufficiently large.

	Finally, we prove item \eqref{Non-Collapsing Est Appdix}.
	If $R_1 \ge \sqrt 8 ( R_0 +1)$, then $\cone_{R_0+1}(\Sigma) \supset \cone_{R_1/ \sqrt 8}(\Sigma)$ and 
	Lemma \ref{Lem Interpolated Metric on Cone Appdix} implies there exists $c = c(n, \Sigma, g_\Sigma)>0$ such that
		$$Vol_{\check G(t_0)} B_{\check G(t_0)} (x, 1) \ge c > 0 \qquad \text{for all } x \in \cone_{R_1 / \sqrt 8} (\Sigma)$$
	for all $0 < 1 -t_0  \ll 1$.
	Thus, by Lemma \ref{Lem f r Region Inclusion}, the corresponding estimate
	holds on 
		$$\Psi ( \cone_{R_1 / \sqrt{8} } ) \supset \left \{ x \in M : f(x) > \frac{R_1^2}{16} \right \}$$
	if $R_1 \gg 1$ is sufficiently large.
	This proves item \eqref{Non-Collapsing Est Appdix} and completes the proof of the proposition.
\end{proof}

\begin{prop} \label{prop Main Result of Appdix A}
	Let $(M, \ol{g}, f)$ be a smooth, complete, asymptotically conical, gradient, shrinking Ricci soliton which is asymptotic to the cone $( \cone(\Sigma), g_\cone)$.
	Then there exists a closed manifold $\cl{M}$ satisfying the conditions of Subsection \ref{Subsect The Manifold} with $\cl{S} \cong \mathbb{Z}_2$ and $A = \{ 0,1\}$, and there exists a family of metrics $G_{\mathbf 0}(t_0) = G_{\mathbf 0}(\Gamma_0, t_0)$ on $\cl{M}$ satisfying Definition \ref{Defn G_0(t_0)}.
\end{prop}

\begin{proof}
	Fix $R_0 > 0$ and
		$$\Psi : \cone_{R_0}(\Sigma) \xrightarrow[]{\cong} V \subset M$$
	a diffeomorphism as in Proposition \ref{Prop 2.1 in KW15}.
	For any $R > R_0$, define
		$$M_R \doteqdot (M \setminus V) \cup \left( \Psi( (R_0, R) \times \Sigma ) \right) \subset M.$$
	Note that $M_R \subset M$ is an open submanifold with closure 
		$$\ol{M}_R = ( M \setminus V ) \cup \left( \Psi( (R_0, R] \times \Sigma ) \right) \subset M$$
	which is a compact manifold with boundary $\partial \ol{M}_R \cong \Sigma$.
	Let $\cl{M}_R$ be the double of $\ol{M}_R$, that is,
		$$\cl{M}_R \doteqdot (\ol{M}_R \times \{0\} ) \bigcup_{\substack{(x, 0) \sim (x, 1) \\ \text{ for } x \in \partial \ol{M}_R}} ( \ol{M}_R \times \{1 \} ).$$
	Then $\cl{M}_R$ inherits a topology and smooth structure from $\ol{M}_R$ which makes $\cl{M}_R$ into a closed manifold.
	Additionally, the map $\cl{M}_R \to \cl{M}_R$ given by $(x, i \mod 2) \mapsto (x, i+1 \mod 2)$ is a diffeomorphism of $\cl{M}_R$ that generates a subgroup $\cl{S}_R \le \text{Diff}(\cl{M}_R)$ such that $\cl{S}_R \cong \mathbb{Z}_2$ as groups.

	There exists a diffeomorphism from $(R, R+1)$ to $(R, \infty)$ which is the identity near $R$, and this induces a diffeomorphism $\iota_{R+1}'' : M_{R+1} \to M$ for any $R > R_0$.
	Note that $\iota_{R+1}''$ is not the inclusion map $M_{R+1} \subset M$, but $\iota_{R+1}''$ can be chosen so that it restricts to the inclusion map on $M_R$.
	Thus, for any $R > R_0$, $M_{R+2}$ contains a diffeomorphic copy of $M$ given by $M_{R+1}\subset M_{R+2}$ with diffeomorphism $\iota''_{R+1} : M_{R+1} \to M$.
	Hence, $\cl{M}_{R+2}$ contains two disjoint diffeomorphic copies $\cl{M}''_{R+2} = \cl{M}''_{R+2, 0} \sqcup \cl{M}''_{R+2, 1} \subset \cl{M}_{R+2}$ of $M$ corresponding to the $\iota''_{R+1} : M_{R+1} \xrightarrow[]{\cong} M$ in each piece of the double.
	Clearly, the diffeomorphisms in $\cl{S}_{R+2} \le \text{Diff}(\cl{M}_{R+2})$ preserve these identifications with $M$.

	Note that all this topological data $M_R, \cl{M}_R, \cl{S}_R, \iota''_R$ depends on $R > R_0$.
	However, because $(R_0, R)$ is diffeomorphic to $(R_0, R')$ for any $R, R' \in (R_0, \infty)$,
	it follows that $M_R \cong M_{R'}$, $\ol{M}_R \cong \ol{M}_{R'}$, and $\cl{M}_R \cong \cl{M}_{R'}$ for all $R, R' \in ( R_0, \infty)$.
	Thus, there exists a \emph{fixed} closed manifold $\cl{M}$ and diffeomorphisms $F_{R} : \cl{M} \xrightarrow[]{\cong} \cl{M}_{R}$ for all $R > R_0$.
	Moreover, the diffeomorphisms $F_{R+2} : \cl{M} \to \cl{M}_{R+2}$ can be chosen such that 
		$$\cl{M}'' \doteqdot \cl{M}''_0 \sqcup \cl{M}''_1 \doteqdot  F_{R+2}^{-1} ( \cl{M}''_{R+2, 0} ) \sqcup F_{R+2}^{-2} ( \cl{M}''_{R+2,1} ) \subset \cl{M} \text{ is indepedent of $R > R_0$,}$$
		$$\text{and } \iota_{R+1}'' \circ F_{R+2}|_{\cl{M}''_{R+2}} = \iota_{R'+1}'' \circ F_{R'+2}|_{\cl{M}''_{R'+2}} \quad \forall R , R' > R_0.$$
	For example, take $\cl{M} = \cl{M}_{R_0+3}$, use the relation
	$$\iota_{R_0+2}'' \circ F_{R_0+3}|_{\cl{M}''_{R_0+3}} = \iota_{R+1}'' \circ F_{R+2}|_{\cl{M}''_{R+2}}$$
	to define $F_{R+2}$ on $\cl{M}''_{R+2} \subset \cl{M}_{R+2}$, and extend this to a diffeomorphism of $\cl{M}_{R+2}$ to $\cl{M}_{R_0+3}$ using a suitable diffeomorphism between the intervals $(R+1, R+2)$ and $(R_0+2,R_0 +3)$.

	Therefore, by pulling back the $\iota''_{R+1}$ to the manifold $\cl{M}$ via the diffeomorphisms $F_{R+2} : \cl{M} \to \cl{M}_{R+2}$, we have a closed manifold $\cl{M}$ with two disjoint diffeomorphic copies $\cl{M}'' = \cl{M}''_0 \sqcup \cl{M}''_1$ of $M$ and \emph{fixed} diffeomorphisms $\iota_i'' : \cl{M}_i'' \to M$ $(i = 0, 1)$ which are independent of $R>R_0$.
	Since the $\mathbb{Z}_2$-action on $\cl{M}_{R+2}$ interchanges the two pieces of the double, this action also pulls back by $F_{R+2}$ to give a $\mathbb{Z}_2$-action on $\cl{M}$ by diffeomorphisms which is independent of $R>R_0$, interchanges the $\cl{M}_0''$ and $\cl{M}_1''$, and respects the identifications $\iota_i''$.
	In summary, we have a closed manifold $\cl{M}$ that satisfies the conditions of Subsection \ref{Subsect The Manifold} with $\cl{S} \cong \mathbb{Z}_2$.

	Let $G(t_0) = G(R_1, t_0)$ denote the metrics on $M$ from Proposition \ref{Prop Interpolated Metric on M Appdix}.
	Consider the restriction of $G(R_1, t_0)$ to $\ol{M}_{R+2}$ where $R = 12 \sqrt 2 R_1$.
	Then, in a neighborhood of $\partial \ol{M}_{R+2}$, $G(R_1, t_0)$ is a cylindrical metric after pulling back by $\Psi$.
	It follows that $G(R_1, t_0)$ induces a smooth $\mathbb{Z}_2$-invariant metric on the double $\cl{M}_{R+2}$ which we still denote by $G(R_1, t_0)$.
	Using that $\iota_{R+1}'' : M_{R+1} \to M$ restricts to the inclusion map on $M_R$, it now follows from Proposition \ref{Prop Interpolated Metric on M Appdix} and Lemma \ref{Lem f r Region Inclusion} that, after pulling these metrics back by $F_{R+2} : \cl{M} \xrightarrow[]{\cong} \cl{M}_{R+2}$ (with $R = 12 \sqrt 2 R_1$), we obtain a family of metrics $G_{\mathbf 0}(\Gamma_0, t_0) \doteqdot F^*_{48 \sqrt {\Gamma_0}+2} G(\sqrt{8 \Gamma_0}, t_0)$ on $\cl{M}$ satisfying Definition \ref{Defn G_0(t_0)} for all $R_1 = \sqrt{8 \Gamma_0} \gg 1$ sufficiently large depending on $n, M, \ol{g}, f, N_0, R_0, \Psi$.
	Note that since $N_0, R_0, \Psi$ were fixed and depend on $n, M, \ol{g}, f$,
	this dependence reduces to all $\Gamma_0 \gg 1$ sufficiently large depending only on $n, M, \ol{g}, f$.
	This completes the proof.
\end{proof}

\section{Harmonic Map Heat Flow} \label{Appdix HarMapFlow}

Throughout this section, we establish results for the harmonic map heat flow solutions appropriate to our setting.

The next result establishes short-time existence for the harmonic map heat flow.
It is a restatement of \cite[Proposition A.9]{BamlerKleiner18} in form more appropriate for our applications.

\begin{prop}[Short-time existence of the harmonic map heat flow] \label{Prop Existence HarMapFlow}
	For all $n \in \mathbb{N}$ and all $ 0 < \epsilon < \epsilon' \ll1$ sufficiently small depending only on $n$ 
	the following holds:
	
	If $(M^n, g(t))_{t \in [t_0, T] }, (\ol{M}^n, \ol{g}(t))_{t \in [t_0, T] }$ are smooth families of Riemannian manifolds and 
	$F_{t_0} : M \to \ol{M}$ is a smooth function such that
	\begin{enumerate}
		\item $(M, g(t_0))$ and $(\ol{M} , \ol{g}(t_0) )$ are complete,
		
		\item For all integers $0 \le m \le 10$, there exist constants $C_m \in \mathbb{R}$ such that 
			$$| {}^{g(t)} \nabla^m Rm_{g(t)} |_{g(t)}(x) \le C_m 		
			\quad \text{for all } (x,t) \in M \times [ t_0, T] \text{ and}$$
			$$| {}^{\ol{g}(t)} \nabla^m Rm_{\ol{g}(t)} |_{\ol{g}(t)}(x) \le C_m 		
			\quad \text{for all } (x,t) \in \ol{M} \times [ t_0, T],$$
			
		\item $(\ol{M} , \ol{g}(t) )_{t \in [t_0, T]}$ evolves by Ricci flow 
			$$\partial_t \ol{g} = - 2 Rc_{ \ol{g} }
				\qquad \text{ on } \ol{M} \times ( t_0, T) ,$$	
			
		\item there exists a constant $C_0'$ such that
			$$| \partial_t g + 2 Rc_{g(t)} |_{g(t)} \le C_0' 
			\qquad \text{for all } (x,t) \in M \times [ t_0, T] $$
		and
		\item $F_{t_0} : M \to \ol{M}$ is a diffeomorphism with
			$$| ( F_{t_0}^{-1} )^* g(t_0) - \ol{g}(t_0) |_{\ol{g}(t_0)} \le \epsilon,$$ 
	\end{enumerate}
	then, for some $T_F$ with $t_0 < T_F \le T$,
	there exists a smooth solution $F : M \times [t_0, T_F] \to \ol{M}$
	to the harmonic map heat flow 
		$$\partial_t F = \Delta_{g(t), \ol{g}(t)} F \qquad 
		\text{ with initial condition } F(\cdot, t_0) = F_{t_0}$$
	such that
	$F(\cdot , t) \doteqdot F_t : M \to \ol{M}$ is a diffeomorphism for all $t \in [t_0, T_F]$
	and
		$$| ( F_t^{-1} )^* g(t)  - \ol{g}(t) |_{\ol{g}(t)} \le \epsilon'
		\qquad \text{on } \ol{M} \times [ t_0, T_F].$$
	 
	 Moreover, 
	 $T_F \in (t_0, T ]$ can be chosen such that 
	 $T_F \ge  \min\{ t_0 + \tau, T \}$
	 where $\tau = \tau( n, \epsilon, \epsilon', C_0, C_0') > 0$ is a positive constant depending only on $n, \epsilon, \epsilon', C_0, C_0'$ 
	 and such that
	 	$$\sup_{\ol{M}} | ( F_{T_F}^{-1} )^* g(T_F)  - \ol{g}(T_F) |_{\ol{g}(T_F)} < \epsilon'
		\quad \implies \quad T_F = T.$$
\end{prop}
\begin{proof}
	This proposition is a direct consequence of \cite[Proposition A.9]{BamlerKleiner18} with
	$C = C_0$,
	$C' = \max_{0 \le m \le 10} C_m T^{m/2}$,
	$\delta = C_0'$,
	$\eta_0 =  \epsilon$,
	$\eta_1 =  \epsilon'$,
	$\eta' = \frac{1}{2} \sup_{\ol{M}}  | ( F_{T_F}^{-1} )^* g(T_F)  - \ol{g}(T_F) |_{\ol{g}(T_F)}  + \frac{1}{2} \epsilon'$,
	$\ol{\chi} = F_{t_0} $,
	$\chi_t  = F_{t_0 + t}$,
	$g'(t) = g(t_0 + t)$,
	$g(t) = \ol{g}(t_0 + t)$, and
	$T = T - t_0$.
\end{proof}

\begin{remark}
	In practice, $\ol{g}(t)$ will be the soliton solution $\ol{g}(t) = ( 1 - t) \phi_t^* \ol{g}$,
	$g(t)$ will be $\acute G_{\mathbf p }(t)$,
	and $F = \tl{\Phi}$. 
\end{remark}

\begin{lem}[Drift Control] \label{Lem Global Drift Control}
	Let $0 \le t_0 < T \le 1$.
	Let $(M^n, g(t))_{t \in [t_0, T]}$ and $( \ol{M}^n, \ol{g}(t) )_{t \in [t_0, T]}$ be smooth families of complete Riemannian manifolds and let $( F_t : M \to \ol{M} )_{t \in [t_0, T]} $ be a smooth family of diffeomorphisms such that
	$$ \partial_t \ol{g} = -2 \ol{Rc}	\quad \text{and} \quad
	\partial_t F_t = \Delta_{g, \ol{g}} F_t.$$
	Let
		$$h \doteqdot (F_t^{-1} )^* g(t) - \ol{g}(t).$$
	Assume that
	 \begin{gather*}
		 \sup_{\ol{M} \times [t_0, T]} | \ol{Rm} |_{\ol{g}} \le \ol{ \cl{K}} < \infty, \\
		\sup_{ (x,t) \in \ol{M} \times [t_0, T] } | h  |_{\ol{g}}(x,t) \le \epsilon, \text{ and}	\\
		\sup_{ (x,t) \in \ol{M} \times [t_0, T] } \sqrt{ 1 - t} | \ol{\nabla} h  |_{\ol{g}} (x,t) \le \epsilon.
	\end{gather*}
	If $0 < \epsilon \ll 1$ is sufficiently small depending only on $n$,
	then there exists $C$ depending only on $n$ and $\ol{\cl{K}}$ such that 
		$$d_{\ol{g}(t_0)} ( F_t(x), F_{t_0}(x) ) \le C \epsilon \left( \sqrt{ 1 - t_0} - \sqrt{ 1 - t} \right) \le C \epsilon \sqrt{1 - t_0}$$
	for all $(x,t) \in M \times [ t_0, T].$
\end{lem}
\begin{proof}
	For any $\ol{x} \in \ol{M}$, observe
		$$\partial_t F_t \circ F_t^{-1}( \ol{x} )  = - B_{\ol{g} } ( \ol{g} + h ) ( \ol{x} ) $$
	where $B_{\ol{g}} ( \ol{g} + h)$ is defined by \eqref{DeTurck Vector Field Eqn} in Lemma \ref{Lem g Evol Eqn}.
	If $0 < \epsilon \ll 1$ is sufficiently small depending only on $n$, then
		$$| \partial_t F_t \circ F_t^{-1} |_{\ol{g}} ( \ol{x} )  
		= | B_{\ol{g}} ( \ol{g} + h) |_{\ol{g}} ( \ol{x} ) 
		\lesssim_n | \ol{\nabla} h |_{\ol{g}} ( \ol{x} ) 
		\le \frac{ \epsilon}{ \sqrt{ 1 - t} }.$$	
	Since the $F_t : M \to \ol{M}$ are diffeomorphisms, 
		$$| \partial_t F_t (x) |_{\ol{g} } \lesssim_n \frac{ \epsilon}{ \sqrt{ 1- t}}
		\qquad \text{for all } x \in M.$$
	It follows that
	\begin{align*}
		  d_{\ol{g}(t_0)} ( F_t(x), F_{t_0} (x) ) 
		&= \int_{t_0}^t \frac{d}{ds}  d_{\ol{g}(t_0)} ( F_s(x), F_{t_0} (x) )ds	\\
		&= \int_{t_0}^t  \langle \nabla d , \partial_s F_s \rangle_{\ol{g}(t_0)}  ds	\\
		&\le  \int_{t_0}^t  \left| \partial_s F_s \right|_{\ol{g}(t_0)}  ds		.
	\end{align*}
	Since the Ricci flow solution $\ol{g}(t)$ has uniformly bounded curvature for $t \in [t_0, T] \subset [0, 1]$,
	the metrics $\ol{g}(t)$ are comparable for all $t \in [t_0, T]$ with
		$$\ol{g}(t') \lesssim_{n, \ol{K}} \ol{g} (t) \lesssim_{n, \ol{K}} \ol{g}(t')
		\qquad \text{for all } t, t' \in [t_0, T]$$
	(see for example \cite[proof of Theorem 8.1]{Hamilton95}).
	Hence,
	\begin{align*}
		d_{\ol{g}(t_0)} ( F_t(x), F_{t_0} (x) ) 
		&\le  \int_{t_0}^t  \left| \partial_s F_s \right|_{\ol{g}(t_0)}  ds		\\
		&\lesssim_{n, \ol{\cl{K}}}  \int_{t_0}^t  \left| \partial_s F_s \right|_{\ol{g}(s)}  ds	\\
		&\lesssim_n \int_{t_0}^t \frac{  \epsilon}{ \sqrt{ 1 - s}} ds		\\
		&=  \epsilon \left( 2 \sqrt{ 1 - t_0} - 2 \sqrt{ 1 - t} \right)	\\
		&\le 2  \epsilon \sqrt{ 1 - t_0}. 	
	\end{align*}	
\end{proof}

\begin{remark}
	The choice to use the distance function at time $t_0$ in the above Lemma \ref{Lem Global Drift Control} was arbitrary.
	In fact, since $\ol{Rm}$ is uniformly bounded and $0 \le T - t_0 \le 1$ in Lemma \ref{Lem Global Drift Control}, 
	the distance functions are all comparable, that is
		$$d_{\ol{g}(t)}(\ol{x}, \ol{y} ) \sim_{n, \ol{\cl{K}} } d_{\ol{g}(t') }(\ol{x}, \ol{y} )		
		\qquad \text{for all } \ol{x}, \ol{y} \in \ol{M}  \text{ and } t, t' \in [t_0, T],$$
	by \cite[Theorem 17.1]{Hamilton95}.
\end{remark}

Unfortunately, the drift control estimates of the previous lemma do not suffice for some of our purposes, and we need the following drift estimate where curvature bounds are assumed only on a subset $\Omega'$ of the manifold $\ol{M}$.

\begin{lem}[Local Drift Control] \label{Lem Local Drift Control}
	Let $0 \le t_0 < T \le 1$.
	Let $(M^n, g(t))_{t \in [t_0, T]}$ and $( \ol{M}^n, \ol{g}(t) )_{t \in [t_0, T]}$ be smooth families of complete Riemannian manifolds and let $( F_t : M \to \ol{M} )_{t \in [t_0, T]} $ be a smooth family of diffeomorphisms such that
	$$ \partial_t \ol{g} = -2 \ol{Rc}	\quad \text{and} \quad
	\partial_t F_t = \Delta_{g, \ol{g}} F_t.$$
	Let
		$$h \doteqdot (F_t^{-1} )^* g(t) - \ol{g}(t)$$
	and $\Omega \Subset \Omega' \subset \ol{M}$.
	Assume that
	 \begin{gather*}
		 \sup_{\Omega' \times [t_0, T]} | \ol{Rm} |_{\ol{g}} \le \ol{ \cl{K}} < \infty, \\
		\sup_{ (x,t) \in \ol{M} \times [t_0, T] } | h  |_{\ol{g}}(x,t) \le \epsilon, \text{ and}	\\
		\sup_{ (x,t) \in \ol{M} \times [t_0, T] } \sqrt{ 1 - t} | \ol{\nabla} h  |_{\ol{g}} (x,t) \le \epsilon.
	\end{gather*}
	If $0 < \epsilon \ll 1$ is sufficiently small depending only on $n$
	and $0 <  1 - t_0 \ll 1$ is sufficiently small depending on $n, \ol{ \cl{K}}, \dist_{\ol{g}(t_0)} ( \Omega, \ol{M} \setminus \Omega' )$,
	then
		$$F_{t_0}(x) \in \Omega \implies F_t (x) \in \Omega' \text{ for all } t \in [t_0, T].$$
\end{lem}
\begin{proof}
	Suppose for the sake of contradiction that there exists $x \in M$ and $t \in [t_0, T]$ such that 
	$F_{t_0} (x) \in \Omega$ and $F_t (x) \notin \Omega'$.
	Then we can define
		$$t_* \doteqdot \inf \{ t \in [t_0, T] : F_t(x) \notin \Omega' \}.$$
	It follows that
		$$d_{\ol{g}(t_0)} ( F_{t_0} (x) , F_{t_*}(x) ) \ge d_{\ol{g}(t_0)} ( \Omega , \ol{M} \setminus \Omega' ) > 0.$$
		
	By the same argument as in the beginning of the proof of Lemma \ref{Lem Global Drift Control}, 
		$$| \partial_t F_t  |_{\ol{g}} (F_t(x)) \lesssim_n \frac{ \epsilon } { \sqrt{ 1 - t} }$$
	for all $t \in [t_0, T]$ if $0 < \epsilon \ll 1$ is sufficiently small depending on $n$.
	Hence,
	\begin{align*}
		d_{\ol{g}(t_0)} ( F_{t_*}(x), F_{t_0} (x) ) 
		&= \int_{t_0}^{t_*} \frac{d}{ds}  d_{\ol{g}(t_0)} ( F_s(x), F_{t_0} (x) )ds	\\
		&= \int_{t_0}^{t_*}  \langle \nabla d , \partial_s F_s(x) \rangle_{\ol{g}(t_0)}  ds	\\
		&\le  \int_{t_0}^{t_*}  \left| \partial_s F_s \right|_{\ol{g}(t_0)} (F_s(x)) ds		.
	\end{align*}
	By definition of $t_*$, $F_s(x) \in \Omega'$ for all $s \in [t_0, t_*)$.
	Since the Ricci flow solution $\ol{g}(t)$ has uniformly bounded curvature on $\Omega' \times [t_0, T] \subset \Omega' \times [0,1]$,
	the metrics $\ol{g}(t)$ at $F_t(x) \in \ol{M}$ are comparable for all $t \in [t_0, t_*]$ with
		$$\ol{g}(t') \lesssim_{n, \ol{ \cl{K}}} \ol{g} (t) \lesssim_{n, \ol{ \cl{K}}} \ol{g}(t')
		\qquad \text{at } F_t(x) \in \ol{M}, \text{ for all } t, t' \in [t_0, t_*]$$
	(see for example \cite[proof of Theorem 8.1]{Hamilton95}).
	Hence,
	\begin{align*}
		d_{\ol{g}(t_0)} ( F_{t_*}(x), F_{t_0} (x) ) 
		&\le  \int_{t_0}^{t_*}  \left| \partial_s F_s \right|_{\ol{g}(t_0)} ( F_s(x))  ds		\\
		&\lesssim_{n, \ol{\cl{K}}}  \int_{t_0}^{t_*}  \left| \partial_s F_s \right|_{\ol{g}(s)}  (F_s(x))ds	\\
		&\lesssim_n \int_{t_0}^t \frac{  \epsilon}{ \sqrt{ 1 - s}} ds		\\
		&=  \epsilon \left( 2 \sqrt{ 1 - t_0} - 2 \sqrt{ 1 - t_*} \right)	\\
		&\le 2  \epsilon \sqrt{ 1 - t_0}. 	
	\end{align*}	
	In particular, if $\epsilon < 1$ and $0 < 1 - t_0 \ll 1$ is sufficiently small depending on $n, \ol{\cl{K}}, d_{\ol{g}(t_0)} ( \Omega, \ol{M} \setminus \Omega' )$, then 
		$$d_{\ol{g}(t_0)} ( F_{t_*}(x), F_{t_0} (x) ) < d_{\ol{g}(t_0)} ( \Omega, \ol{M} \setminus \Omega' ) \le d_{\ol{g}(t_0)} ( F_{t_*}(x), F_{t_0} (x) ),$$
	a contradiction.
\end{proof}

\begin{prop}[Uniqueness of the Harmonic Map Heat Flow Solution] \label{Prop Uniqueness HarMapFlow}
	Let $0 \le t_0 < T \le 1$.
	Let $(M^n, g(t))_{t \in [t_0, T]}$ and $( \ol{M}^n, \ol{g}(t) )_{t \in [t_0, T]}$ be smooth families of complete Riemannian manifolds such that $\ol{g}(t)$ evolves by Ricci flow
	$ \partial_t \ol{g} = -2 \ol{Rc}.$
	Let $ F_1, F_2 : M \times [t_0, T] \to \ol{M} $ be smooth functions such that
	\begin{gather*}
	\begin{aligned}
		\partial_t F_1 &= \Delta_{g, \ol{g} } F_1
		&& \text{on } M \times (t_0, T), \\
		\partial_t F_2 &= \Delta_{g, \ol{g} } F_2
		&& \text{on } M \times (t_0, T),  \\	
		F_1(x, t_0) &= F_2(x,t_0) 
		&& \text{for all } x \in M, \text{ and} \\
	\end{aligned} \\
		F_1( \cdot, t), F_2(\cdot , t) : M \to \ol{M}
		\text{ are diffeomorphisms for all } t \in [t_0, T].
	\end{gather*}
	Let
		$$h_1 \doteqdot ( F_1( \cdot , t)^{-1} )^* g - \ol{g} \qquad \text{ and } \qquad 
		h_2 \doteqdot ( F_2( \cdot , t)^{-1} )^* g - \ol{g}.$$
	Assume that 
	\begin{gather*}
		\sup_{\ol{M} \times [t_0, T] } | \ol{Rm} |_{\ol{g}} \le \ol{ \cl{K} } < \infty,	\qquad 
		\inf_{\ol{M} \times [t_0, T] } \inj(\ol{M} , \ol{g}(t) ) \ge i_0 > 0,	\\
		\sup_{\ol{M} \times [t_0, T] } \left( | h_1 |_{\ol{g}} (x,t) +  | h_2 |_{\ol{g}}(x,t) \right) \le \epsilon , 
		\text{ and} \\
		\sup_{\ol{M} \times [t_0, T] } \sqrt{ 1 - t} \left(| \ol{\nabla} h_1 |_{\ol{g}}(x,t)  +  | \ol{\nabla} h_2 |_{\ol{g}}(x,t) \right) \le \epsilon.
	\end{gather*}
	If $0 < \epsilon \ll 1$ is sufficiently small depending only on $n$, 
	 then $F_1(x,t) = F_2(x,t)$ for all $(x,t) \in M \times [t_0, T]$.
\end{prop}
\begin{proof}
	Suppose for the sake of contradiction that
		$$t_0' \doteqdot \max \{ t' \in [t_0, T] : F_1( \cdot , t ) = F_2(\cdot , t) \text{ for all } t \in [t_0, t'] \} < T.$$
	Consider the function $u : M \times [ t_0', T] \to \mathbb{R}$ given by the squared distance of $F_1$ to $F_2$, that is,
		$$u (x, t ) = \frac{1}{2} d_{\ol{g}(t)}^2 ( F_1(x,t), F_2(x,t) ).$$
	Note $u( x ,t_0') = 0$ for all $x \in M$.
	
	The uniform injectivity radius lower bound implies that there exists a sufficiently small neighborhood $U = B_{\ol{g}(t_0)} (D, r) $ of the diagonal $D \subset \ol{M} \times \ol{M}$ where the squared distance function $(\ol{x},\ol{y},t) \mapsto d_{\ol{g}(t)}^2 ( \ol{x}, \ol{y} )$ is twice differentiable.
	Additionally, the uniform curvature bounds
		$$\sup_{\ol{M} \times [ t_0, T] } | \ol{Rm} | \le \ol{\cl{K}}$$	
	yield estimates on the Hessian $\nabla^2 d = {}^{\ol{M} \times \ol{M}} \nabla^2 d$
	which in particular imply
	that the Hessian $\nabla^2 d$ is positive semidefinite on $U = B_{\ol{g}(t_0)} (D, r) \subset \ol{M} \times \ol{M}$ after possibly taking $r$ smaller.
	By the drift control estimate Lemma \ref{Lem Global Drift Control}, there exists $0 < \tau \ll 1$ sufficiently small depending on $n, \ol{\cl{K}}, i_0, \epsilon$
	such that, for all $t \in [t_0' , t_0' + \tau] \subset [t_0 , T]$,
	the image of $x \mapsto ( F_1( x, t), F_2(x ,t) )$ stays within this neighborhood $U$.

	We can now compute and estimate the evolution equation for $u$ on $M \times ( t_0' , t_0' + \tau )$ in terms of the induced connections:
	\begin{align*}
	\partial_t u 
	={}& d_{\ol{g}}  ( F_1, F_2 ) \cdot\left[ ( \partial_t d_{\ol{g}(t)} ) |_{ (F_1, F_2) }
	+ (\nabla d )( \partial_t F_1 \oplus \partial_t F_2 )	\right]\\
	={}& d_{\ol{g}} ( F_1, F_2 )\cdot \left[ ( \partial_t d_{\ol{g}(t)} )|_{ (F_1, F_2) }
	+ (\nabla d) ( \Delta_{g, \ol{g}} F_1 \oplus \Delta_{g, \ol{g}}  F_2 ) \right]	,\\
	\nabla_j u 
	={}& d_{\ol{g}} (F_1, F_2)\cdot \left[  (\nabla d) |_{(F_1, F_2)} ( \nabla_j F_1 \oplus \nabla_j F_2 ) \right], \text{ and}	\\ 
	\nabla_i \nabla_j u 
		={}& \left[  (\nabla d)  ( \nabla_i F_1 \oplus \nabla_i F_2 ) \right]
		\left[  (\nabla d)  ( \nabla_j F_1 \oplus \nabla_j F_2 ) \right]	\\
		&+ d_{\ol{g}} (F_1, F_2) \cdot \left[  (\nabla d)  ( \nabla_i \nabla_j F_1 \oplus \nabla_i \nabla_j F_2 ) \right]	\\
		&+ d_{\ol{g}} (F_1, F_2) \cdot \left[  (\nabla^2 d) ( \nabla_i F_1 \oplus \nabla_i F_2,  \nabla_j F_1 \oplus \nabla_j F_2 ) \right].	
	\end{align*}
	
	It follows that
	\begin{gather} \label{Uniqueness Proof u Evol Eqn} \begin{aligned}
		( \partial_t - \Delta_g ) u	
		={}& d_{\ol{g}} (F_1, F_2) \cdot ( \partial_t d_{\ol{g}(t)} )|_{ (F_1, F_2) }
		 - | ( \nabla d) ( \nabla F_1 \oplus \nabla F_2 ) |_{g}^2	\\
		& - d_{\ol{g}} (F_1, F_2) \cdot g^{ij} \left[  (\nabla^2 d) ( \nabla_i F_1 \oplus \nabla_i F_2,  \nabla_j F_1 \oplus \nabla_j F_2 ) \right].	
	\end{aligned}	\end{gather}

	It follows from \cite[Lemma 17.3]{Hamilton95} that
		$$d_{\ol{g}} ( F_1, F_2) \cdot ( \partial_t d_{\ol{g}(t)} ) (F_1, F_2) 
		\le \| \ol{Rc} \|_{C^0( \ol{M} )  }   d_{\ol{g}} ( F_1, F_2)^2
		\le C_n  \ol{\cl{K}} u .$$
	 Writing the last term in the evolution equation \eqref{Uniqueness Proof u Evol Eqn} in terms of a $g$-orthonormal basis $\{ e_i \}$,
	 the estimate
	 \begin{align*}
	 	&- d_{\ol{g}} (F_1, F_2) \cdot g^{ij} \left[  (\nabla^2 d) ( \nabla_i F_1 \oplus \nabla_i F_2,  \nabla_j F_1 \oplus \nabla_j F_2 ) \right]	\\
	 	={}&-d_{\ol{g}} (F_1, F_2) \cdot \sum_{i} 
		\nabla^2 d \left(  \nabla_{i} F_1 \oplus \nabla_{i} F_2 , \nabla_{i} F_1 \oplus \nabla_{i} F_2 \right) \\
	 	\le{}& 0  
	 \end{align*}
	holds since $\nabla^2 d$ is positive semidefinite on $U$.

	Applying these estimates to \eqref{Uniqueness Proof u Evol Eqn}, it follows that 
		$$( \partial_t - \Delta_g) u \le C_n \ol{ \cl{K}} u
		\qquad \text{on } M \times (t_0' , t_0' + \tau) .$$
	The drift control estimate Lemma \ref{Lem Global Drift Control} implies additionally that $u$ is uniformly bounded.
	Therefore, $u : M \times [ t_0', t_0' + \tau] \to \mathbb{R}$ is a bounded subsolution of
	\begin{align*}
		\partial_t u &\le \Delta_g u + C_n \ol{\cl{K}} u 
		&& \text{ on } M \times ( t_0' , t_0' + \tau), \\
		u( \cdot  , t_0') &= 0
		&& \text{ on } M . 
	\end{align*}
	The maximum principle then implies that $u \equiv 0 $ on $M \times [t_0', t_0' + \tau]$.
	The definition of $u$ then implies that $F_1(x,t) = F_2(x,t)$ for all $(x,t) \in M \times [ t_0', t_0' + \tau]$, 
	which contradicts the choice of $t_0'$.
\end{proof}

\section{H{\"o}lder Estimates in Euclidean Space} \label{Holder Ests in Eucl Space}

In this appendix, we provide the a priori H{\"o}lder estimates for parabolic equations on domains in Euclidean space.
The content here closely follows that of \cite[Subsection 2.5]{Bamler14} and \cite[Section 4]{Appleton18}.
While the results in this appendix are stated only for scalar-valued functions $u$ for simplicity, all the results in this appendix hold for vector-valued functions $u$ and the proofs carry over mutatis mutandis.

If $\Omega \subset \R^n \times \R$ is some parabolic neighborhood (e.g. $\Omega = B_r ( 0 ) \times [0, T]$) and
$m \in \mathbb{N}$, then $C^{2m}( \Omega)$ denotes the space of functions $u (x, t)$ on $\Omega$ such that $\nabla^{i} \partial_t^j u$ is continuous and $\sup_\Omega | \nabla^i \partial_t^j u | < \infty$
for all $j \in \mathbb{N}$ and multi-indices $i$ with $|i| + 2j \le m$.
Let $\alpha \in (0,1)$. We define weighted H{\"o}lder norms $\| \cdot \|_{C^{2m, \alpha}( \Omega) }$ which are invariant under parabolic dilation:
Assume
	$$r_\Omega = \min \{ r  : \Omega \subset B_r(p) \times [ t - r^2 , t] \text{ for some } (p,t) \in \R^n \times \R \} < \infty,$$
then
	$$\| u \|_{C^{2m, \alpha} ( \Omega) } \doteqdot \sum_{i + 2j \le 2m} 
	r_{\Omega}^{|i| + 2j} \left( \| \nabla^i \partial_t^j u \|_{C^0 ( \Omega) } + r_\Omega^{\alpha} [ \nabla^i \partial_t^j u ]_{\alpha, \alpha/2}\right) .$$
$C^{2m, \alpha}(\Omega)$ then refers to the space of functions $u(x,t) \in C^{2m}(\Omega)$ with $\| u \|_{C^{2m, \alpha} ( \Omega) } < \infty$.

Throughout $L$ will denote an  linear second order operator of the form
	$$Lu = a^{ij} (x,t) \nabla_i \nabla_j u + b^{i}(x,t) \nabla_i u + c(x,t) u.$$
We first recall some interior H{\"o}lder estimates for inhomogeneous linear parabolic equations.

\begin{lem} \label{Lem Lin Int Est}
	Let $r > 0$.
	Consider the parabolic neighborhoods $\Omega = B_r \times [ -r^2, 0]$ and $\Omega' = B_{2r} \times [ -4r^2, 0]$.
	Assume that $u \in C^{2m, \alpha} ( \Omega')$ satisfies the equation
	\begin{gather*}
		(\partial_t - L )u = f \qquad \text{on } \Omega'
	\end{gather*}
	where 
		$$Lu = a^{ij} (x,t) \nabla_i \nabla_j u + b^{i}(x,t) \nabla_i u + c(x,t) u,$$
	\begin{gather*}
		\frac{1}{\Lambda} \le a^{ij} \le \Lambda, \qquad
		\| a^{ij} \|_{C^{2m-2, \alpha} ( \Omega') } \le \Lambda,	\\
		\| b^i \|_{C^{2m-2, \alpha} ( \Omega') } \le r^{-1} \Lambda, 		\qquad
		\| c \|_{C^{2m-2, \alpha} ( \Omega') } \le r^{-2} \Lambda,
	\end{gather*}
	and $f \in C^{2m-2, \alpha}( \Omega')$.
	Then
		$$\| u \|_{C^{2m, \alpha} ( \Omega) } 
		\le C \left( r^2 \| f \|_{C^{2m-2 , \alpha} ( \Omega' ) } + \| u \|_{C^0 ( \Omega') } \right)$$
	where $C = C( n, m, \alpha, \Lambda) $ depends only on $n, m , \alpha, \Lambda$.
\end{lem}
\begin{proof}
	For $m = 1$, this lemma is identical to \cite[Theorem 8.11.1]{Krylov96}.
	The result for $m > 1$ follows from differentiating the evolution equation of $u$.
\end{proof}

\begin{lem} \label{Lem Lin Int Est+}
	Let $2r \ge s > 0$.
	Consider the parabolic neighborhoods $\Omega = B_r \times [ -s^2, 0]$ and $\Omega' = B_{2r} \times [ -s^2, 0] $.
	Assume that $u \in C^{2m, \alpha} ( \Omega')$ satisfies the equation
	\begin{gather*}
		(\partial_t - L )u = f \qquad \text{on } \Omega', \text{ and} \\
		u( \cdot, 0) = u_0 \qquad \text{on } B_{2r} \times \{ 0 \}
	\end{gather*}
	where 
		$$Lu = a^{ij} (x,t) \nabla_i \nabla_j u + b^{i}(x,t) \nabla_i u + c(x,t) u,$$
	\begin{gather*}
		\frac{1}{\Lambda} \le a^{ij} \le \Lambda, \qquad
		\| a^{ij} \|_{C^{2m-2, \alpha} ( \Omega') } \le \Lambda,	\\
		\| b^i \|_{C^{2m-2, \alpha} ( \Omega') } \le r^{-1} \Lambda, 		\qquad
		\| c \|_{C^{2m-2, \alpha} ( \Omega') } \le r^{-2} \Lambda,
	\end{gather*}
	$f \in C^{2m-2, \alpha}( \Omega')$, and $u_0 \in C^{2 m, \alpha} ( B_{2r}) $.
	Then 
		$$\| u \|_{C^{2m, \alpha} ( \Omega) } 
		\le C \left( r^2 \| f \|_{C^{2m-2 , \alpha} ( \Omega' ) } 
				+ \| u \|_{C^0 ( \Omega') } + \| u_0 \|_{C^{2m, \alpha} ( B_{2r} ) } \right)$$
	where $C = C( n, m, \alpha, \Lambda) $ depends only on $n, m , \alpha, \Lambda$.
\end{lem}
\begin{proof}
	For $m = 1$, the result follows from \cite[Exercise 9.2.5]{Krylov96} applied to $u - u_0$ combined with \cite[Remark 8.11.2]{Krylov96}.
	The result for $m > 1$ follows from differentiation.
\end{proof}

\begin{prop} \label{Prop Nonlin Int Est}
	Let $r > 0$.
	Consider the parabolic neighborhoods $\Omega = B_r \times [ -r^2 , 0]$ and $\Omega' = B_{2r} \times [ -4r^2 , 0]$.
	
	Assume that $u \in C^2 ( \Omega' )$ satisfies the equation
	\begin{gather*} \begin{aligned}
		( \partial_t - L ) u = Q[u] + F	
		={}& r^{-2} f_1( r^{-1} x, r^{-2} t, u ) \cdot u + r^{-1} f_2( r^{-1} x, r^{-2} t, u) \cdot \nabla u	\\
		&+ f_3 ( r^{-1} x, r^{-2} t, u) \cdot \nabla u \otimes \nabla u + f_4( r^{-1} x, r^{-2} t, u ) \cdot u \otimes \nabla^2 u	\\
		&+ F(x,t) 
	\end{aligned} \end{gather*}
	where $f_1, \dots, f_4$ are smooth functions such that $f_2, \dots , f_4$ can be paired with the tensors $u \otimes \nabla u, \nabla u \otimes \nabla u, u \otimes \nabla^2 u$ respectively.
	Assume that the linear operator $L$ has the form
		$$Lu = a^{ij}(x,t) \partial_{ij}^2 u + b^i(x,t) \partial_i u + c( x, t) u$$
	and that the coefficients satisfy the following bounds for $m \ge 1$ and $\alpha \in (0, 1)$:
	\begin{gather*}
		\frac{1}{\Lambda } \le a^{ij} \le \Lambda,	\qquad
		\| a^{ij} \|_{C^{2m-2,\alpha} ( \Omega') } \le \Lambda, \\
		\| b^{i} \|_{C^{2m-2,\alpha} ( \Omega') } \le r^{-1} \Lambda, \qquad
		\| c \|_{C^{2m-2,\alpha} ( \Omega') } \le r^{-2} \Lambda.
	\end{gather*}
	Then there are constants $\epsilon, C > 0$ depending only on $n, m, \alpha, \Lambda$ and the $f_i$ such that 
	\begin{gather*}
		 \| u \|_{C^0 ( \Omega' ) } + r^{2} \| F \|_{C^{2m - 2, \alpha} ( \Omega' ) } \le \epsilon	\implies	\\
		 \| u \|_{C^{2m, \alpha} ( \Omega)} \le C \left(  \| u \|_{C^0 ( \Omega' ) } + r^{2} \| F \|_{C^{2m - 2, \alpha} ( \Omega' ) } \right).
	\end{gather*}
\end{prop}

The proof of Proposition \ref{Prop Nonlin Int Est} follows verbatim the proof of Proposition 2.5 in \cite{Bamler14}
with
	$$H = \| u \|_{C^0 ( \Omega' ) } + r^{2} \| F \|_{C^{2m - 2, \alpha} ( \Omega' ) }$$
instead of $H = \| u \|_{C^0 ( \Omega') }$.
The upcoming proof of Proposition \ref{Prop Nonlin Int Est+} uses similar logic and we instead provide the full details there.

\begin{prop}	\label{Prop Nonlin Int Est+}
	Let $2r \ge s > 0$.
	Consider the parabolic neighborhoods
	$\Omega  = B_r \times [0, s^2]$ and $\Omega'  = B_{2r} \times [0, s^2]$.
	
	Assume that $u \in C^{2} ( \Omega' )$ satisfies the equation
	\begin{gather*}
	\begin{aligned}
		(\partial_t - L ) u = Q[u] + F	
		={}& r^{-2} f_1( r^{-1} x, r^{-2} t, u ) \cdot u + r^{-1} f_2( r^{-1} x, r^{-2} t, u) \cdot \nabla u	\\
		&+ f_3 ( r^{-1} x, r^{-2} t, u) \cdot \nabla u \otimes \nabla u + f_4( r^{-1} x, r^{-2} t, u ) \cdot u \otimes \nabla^2 u	\\
		&+ F(x,t) 
		\\
		\text{ and } u( \cdot, 0 ) ={}& u_0,
	\end{aligned} \end{gather*}
	where $f_1, \dots, f_4$ are smooth functions such that $f_2, \dots , f_4$ can be paired with the tensors $ \nabla u, \nabla u \otimes \nabla u, u \otimes \nabla^2 u$ respectively.
	Assume that the linear operator $L$ has the form
		$$Lu = a^{ij}(x,t) \partial_{ij}^2 u + b^i(x,t) \partial_i u + c( x, t) u$$
	and that the coefficients satisfy the following bounds for $m \ge 1$ and $\alpha \in (0, 1)$:
	\begin{gather*}
		\frac{1}{\Lambda } \le a^{ij} \le \Lambda,	\qquad
		\| a^{ij} \|_{C^{2m-2,\alpha} ( \Omega') } \le \Lambda, \\
		\| b^{i} \|_{C^{2m-2,\alpha} ( \Omega') } \le r^{-1} \Lambda, \qquad
		\| c \|_{C^{2m-2,\alpha} ( \Omega') } \le r^{-2} \Lambda.
	\end{gather*}
	
	Then there are constants $\epsilon, C > 0$ depending only on $n, m, \alpha, \Lambda,$ and the $f_i$ such that
	\begin{gather*}
		 \| u \|_{C^0 ( \Omega' ) } + \| u_0 \|_{C^{2m , \alpha} ( B_{2r} ) } + r^{2} \| F \|_{C^{2m - 2, \alpha} ( \Omega') } \le \epsilon \implies	\\
		\| u \|_{C^{2m, \alpha} ( \Omega ) } 
		\le C \left( \| u \|_{C^0 ( \Omega' ) } + \| u_0 \|_{C^{2m, \alpha} ( B_{2r} ) }
				+ r^{2} \| F \|_{C^{2m - 2, \alpha} ( \Omega') } \right).
	\end{gather*}
\end{prop}
\begin{proof}
	We adapt the proofs of \cite[Proposition 2.5]{Bamler14} and \cite[Lemma 4.4]{Appleton18}.
	Begin by introducing a weighted norm for any $\theta \in ( 0, 1]$ and domain $\Omega$ as
	$$\| u \|^{(\theta)}_{C^{2m, \alpha} ( \Omega) } \doteqdot \sum_{i + 2j \le 2m} 
	(r_{\Omega} \theta) ^{i + 2j} \left( \| \nabla^i \partial_t^j u \|_{C^0 ( \Omega) } + (r_\Omega \theta)^{\alpha} [ \nabla^i \partial_t^j u ]_{\alpha, \alpha/2}\right) .$$	
	This weighted norm with $\theta = 1$ equals the $C^{2m, \alpha}$ norm, that is $\| \cdot \|^{(1)}_{C^{2m, \alpha} ( \Omega) } = \| \cdot \|_{C^{2m, \alpha} ( \Omega) }$. 
	
	Choosing any $x \in \R^n$ such that $B_{\theta r}(x) \subset B_{r}$
	and applying
	Lemma \ref{Lem Lin Int Est} on $B_{\theta r} (x) \times [ t - (\theta r)^2, t]$ for all $t \in [ (2 \theta r)^2, s^2 ]$ and
	Lemma \ref{Lem Lin Int Est+} on $B_{\theta r} (x) \times  [0, (2 \theta r)^2]$ we deduce:
	
	In the setting of Lemma \ref{Lem Lin Int Est+},
	\begin{multline} \label{Lin Int Est+ theta}
		\| u \|_{C^{2m, \alpha} ( B_r \times [ 0, s^2]) }^{(\theta)}
		\le C \left( (r \theta)^2 \| f \|^{( \theta)}_{C^{2m-2, \alpha} ( B_{(1 + \theta) r} \times [ 0, s^2] )} \right.\\
		\left. + \| u \|^{(\theta)}_{C^{0} ( B_{(1 + \theta) r} \times [ 0, s^2] )}	
		 + \| u_0 \|^{(\theta)}_{C^{2m, \alpha} ( B_{(1 + \theta)r} )} \right).
	\end{multline}
	
	By scaling, we shall henceforth assume without loss of generality that $r = 1$. 
	Any constants $C$ in the rest of the proof will depend on $n, m, \alpha, \Lambda$ and the $f_i$.
	
	Define 
		$$H \doteqdot  \| u \|_{C^0 ( B_2 \times [0, s^2]  ) } + \| u_0 \|_{C^{2m , \alpha} ( B_{2} ) } +  \| F \|_{C^{2m - 2, \alpha} ( B_2 \times [ 0, s^2] ) } ,$$
		$$r_k \doteqdot \sum_{i = 0}^k 2^{-i} = 2 - 2^{-k},	\qquad
		\theta_k \doteqdot \frac{r_{k+1}}{r_k} - 1,	\quad \text{ and } \quad
		\Omega_k \doteqdot B_{r_k} \times [ 0, s^2].$$
	Note that $1 \le r_k \le 2$ and $s \le 2 \le 2 r_k$.
	\eqref{Lin Int Est+ theta} applied to $\Omega = \Omega_k$ implies
	\begin{equation} \label{a_k Eqn}
		a_k \doteqdot \| u \|_{C^{2m, \alpha} ( \Omega_k)}^{(\theta_k)}	
		\le C \left(  \theta_k^2 \| Q[u] \|_{C^{2m - 2, \alpha} ( \Omega_{k+1} ) }^{( \theta_{k+1})} 
			+ H \right).
	\end{equation}
	Additionally, because $\lim_{k \to \infty} \theta_k = 0$, $a_\infty \doteqdot \lim_{k \to \infty} a_k 	= \| u \|_{C^0 ( B_2 \times [ 0, s^2] ) } \le H$. 
	
	We now estimate $Q[u]$ in terms of $u$.
	First, observe that for all $i = 1, \dots, 4$,
		$$\| f_i (x,t, u) \|_{C^{2m - 2, \alpha} ( \Omega_{k+1} )}^{ (\theta_{k+1} ) }
		\le C \left( 1 + \left( \| u \|_{C^{2m - 2, \alpha} ( \Omega_{k+1} )}^{ (\theta_{k+1} ) } \right)^{2m - 1} \right)
		\le C \left( 1 + a_{k+1}^{2m -1} \right).$$
	Hence,
		$$\| f_1 \cdot u \|_{C^{2m - 2, \alpha} ( \Omega_{k+1} )}^{ (\theta_{k+1} ) }
		\le C \| f_1 \|_{C^{2m - 2, \alpha} ( \Omega_{k+1} )}^{ (\theta_{k+1} ) } \|u \|_{C^{2m - 2, \alpha} ( \Omega_{k+1} )}^{ (\theta_{k+1} ) }
		\le C ( a_{k+1} + a_{k+1}^{2m } ), $$
	and similarly
	\begin{align*}
		\| f_2 \cdot \nabla u \|_{C^{2m - 2, \alpha} ( \Omega_{k+1} )}^{ (\theta_{k+1} ) }
		&\le \| f_2 \|_{C^{2m - 2, \alpha} ( \Omega_{k+1} )}^{ (\theta_{k+1} ) } \theta_{k+1}^{-1} \| u \|_{C^{2m, \alpha} ( \Omega_{k+1} )}^{ (\theta_{k+1} ) }	\\
		&\le C \theta_{k+1}^{-1} (  a_{k+1} + a_{k+1}^{2m} ),
		\\
		\| f_3 \cdot \nabla u \otimes \nabla u \|_{C^{2m - 2, \alpha} ( \Omega_{k+1} )}^{ (\theta_{k+1} ) }
		&\le C\theta_{k+1}^{-2}( a_{k+1}^2 + a_{k+1}^{2m+1} ), \text{ and}
		\\
		\| f_3 \cdot  u \otimes \nabla^2 u \|_{C^{2m - 2, \alpha} ( \Omega_{k+1} )}^{ (\theta_{k+1} ) }
		&\le C \theta_{k+1}^{-2} ( a_{k+1}^2 + a_{k+1}^{2m+1} ).
	\end{align*}
	Therefore, using that $0 < \theta_{k+1} \le 1$, 
		$$\| Q[u] \|_{C^{2m - 2, \alpha} ( \Omega_{k+1} )}^{ (\theta_{k+1} ) } \le C
		\left( 	\theta_{k+1}^{-1} a_{k+1} + \theta_{k+1}^{-2} a_{k+1}^2 + \theta_{k+1}^{-1} a_{k+1}^{2m}
		+ \theta_{k+1}^{-2} a_{k+1}^{2m+1}			\right),$$
	and so, by \eqref{a_k Eqn},
		$$a_k \le C \left( \theta_{k+1} a_{k+1} +  a_{k+1}^2 + a_{k+1}^{2m}
		+  a_{k+1}^{2m+1}	 + H \right).$$
	Defining $b_k \doteqdot a_k  / H$, it follows that 
	\begin{equation} \label{b_k Recursion Eqn}
		b_k \le C_0 \left( \theta_{k+1} b_{k+1} +  H b_{k+1}^2 + H^{2m-1} b_{k+1}^{2m}
		+  H^{2m} b_{k+1}^{2m+1}	 + 1 \right).
	\end{equation}
	
	Assume without loss of generality that $C_0 > 1$ 
	and set $\epsilon = \frac{1}{16 C_0^2}$.
	Choose $k_0$ such that $\theta_{k+1} \le \epsilon$ for all $k \ge k_0$.
	Then for all $k \ge k_0$, $H \le \epsilon$ implies
		$$b_k \le
		\frac{1}{16} b_{k+1} + \frac{1}{16 C_0} b_{k+1}^2 + \frac{1}{16^{2m-1} C_0^{4m-3} } b_{k+1}^{2m} 
		+ \frac{1}{16^{2m} C_0^{4m-1} } b_{k+1}^{2m+1} + C_0.$$
	Thus, $b_{k+1} < 2C_0$ implies $b_k < 2C_0$ as well.
	By induction and the fact that $\lim_{k\to\infty} b_k = a_\infty / H \le 1 < 2 C_0$,
	it then follows that $b_{k_0} < 2 C_0$.
	
	Using \eqref{b_k Recursion Eqn}, the bound $b_{k_0} < 2 C_0$ yields a bound for $b_0$, say $b_0 \le C'$.
	Hence, $a_0 \le C' H$ which completes the proof.
\end{proof}

\section{Constants} \label{App Constants}
The following list summarizes the parameters used throughout the proof of \ref{Main Thm} for the reader's convenience.
It is ordered so that any constant depends only on the ones above it.

\begin{itemize}
	\item $(M^n,\ol{g}, f)$ is the smooth, complete, asymptotically conical, shrinking, gradient Ricci soliton fixed at the outset.
	
	\item $\lambda_* < 0$ governs the $L^2_f$ decay rate. $\lambda_*$ is chosen so that $\lambda_* \notin \{ \lambda_j : j \in \mathbb{N} \} = \sigma( \ol{\Delta}_f + 2 \ol{Rm} )$.
	
	\item $K = K (n, M, \ol{g}, f, \lambda_*) \in \mathbb{N}$ denotes the index of the eigenvalue $\lambda_K \in \sigma( \ol{\Delta}_f + 2 \ol{Rm} )$ such that $\lambda_{K+1} < \lambda_* < \lambda_K$.
	
	\item $\Gamma_0 \gg 1$ is a parameter that governs the transition from the diffeomorphic copies of $M$ and its soliton metric $(1-t_0) \phi_{t_0}^* \ol{g}$ to the rest of the closed manifold $\cl{M}$.
	
	\item $0 < \ol{p} \ll 1$ controls how much of the lower eigenmodes $\{ h_j : j \le K \}$ the initial perturbation $h_{\mathbf p }(t_0)$ contains via $| \mathbf p | \le \ol{p} e^{\lambda_* \tau_0}$.
	
	\item $0 < \epsilon_0, \epsilon_1, \epsilon_2 \ll 1$ govern the global $C^2$ estimates for $h_{\mathbf p}$. $0 < \epsilon_0 \ll 1$ may also be taken sufficiently small depending on $\epsilon_1, \epsilon_2$.
	
	\item $0 < \gamma_0 \ll 1$ is another macro level parameter used to distinguish scales smaller than $\Gamma_0$. $\gamma_0$ governs the support of the initial perturbation $h_{\mathbf p}(t_0)$ at initial time $t_0$.
	
	\item $0 < \mu_u, \mu_s \ll 1$ govern the $L^2_f$ decay estimates of $h_{\mathbf p }$.
	
	\item $t_0 \in [0, 1)$ with $0 < 1- t_0 \ll 1$ or equivalently $\tau_0 = - \ln ( 1 - t_0) \gg 1$ is the initial time of the Ricci flow $G_{\mathbf p }$.
\end{itemize}


\bibliographystyle{alpha}
\bibliography{parab4}

\end{document}